\documentclass[11pt,a4paper]{article}
\usepackage{xcolor}
\usepackage[T1]{fontenc}
\usepackage{lmodern}
\usepackage[margin=25mm]{geometry}
\usepackage{amsmath,amssymb,amsthm,mathtools,microtype,needspace}
\usepackage{enumitem,aligned-overset,aliascnt}
\usepackage[colorlinks=true,allcolors=blue,pdfencoding=auto]{hyperref}
\usepackage[nameinlink,noabbrev,capitalize]{cleveref}
\crefname{theorem}{Theorem}{Theorems}
\crefname{proposition}{Proposition}{Propositions}
\crefname{lemma}{Lemma}{Lemmas}
\crefname{corollary}{Corollary}{Corollaries}
\crefname{conjecture}{Conjecture}{Conjectures}
\crefname{claim}{Claim}{Claims}
\numberwithin{equation}{section}
\newtheorem{theorem}{Theorem}[section]
\newaliascnt{lemma}{theorem}
\newtheorem{lemma}[lemma]{Lemma}
\aliascntresetthe{lemma}
\newaliascnt{proposition}{theorem}
\newtheorem{proposition}[proposition]{Proposition}
\aliascntresetthe{proposition}
\newaliascnt{corollary}{theorem}
\newtheorem{corollary}[corollary]{Corollary}
\aliascntresetthe{corollary}
\newaliascnt{observation}{theorem}
\newtheorem{observation}[observation]{Observation}
\aliascntresetthe{observation}
\crefname{observation}{Observation}{Observations}
\Crefname{observation}{Observation}{Observations}
\newaliascnt{claim}{theorem}
\newtheorem{claim}[claim]{Claim}
\aliascntresetthe{claim}
\newaliascnt{fact}{theorem}
\newtheorem{fact}[fact]{Fact}
\aliascntresetthe{fact}
\crefname{fact}{Fact}{Facts}
\Crefname{fact}{Fact}{Facts}

\newaliascnt{conjecture}{theorem}
\newtheorem{conjecture}[conjecture]{Conjecture}
\aliascntresetthe{conjecture}
\theoremstyle{remark}
\newaliascnt{remark}{theorem}

\aliascntresetthe{remark}
\newcommand{\E}{\mathbb{E}}
\newcommand{\Prob}{\mathbb{P}}
\newcommand{\1}{\mathbf{1}}
\newcommand{\law}{\mathsf{L}}
\newcommand{\Pcal}{\mathcal{P}}
\newcommand{\rd}{\mathrm{d}}
\newcommand{\dd}{\,\rd}
\newcommand{\Given}{\,\middle|\,}
\DeclareMathOperator*{\esssup}{ess\,sup}
\DeclareMathOperator{\Var}{Var}
\DeclareMathOperator{\Cov}{Cov}
\DeclareMathOperator{\supp}{supp}
\allowdisplaybreaks[2]
\title{\fontsize{14pt}{20pt}\selectfont
A Proof of the Erd\H{o}s Matching Conjecture for Large Matchings\\
via Sharp Tail Probability for Sums of i.i.d. Nonnegative Random Variables}
\author{Cheng Chi$^\ast$
\and Shuang Sun$^\ast$
\and Yan Wang$^\ast$}
\date{}
\hypersetup{pdftitle={A Proof of the Erdős Matching Conjecture for Large Matchings via Sharp Tail Probability for Sums of i.i.d. Nonnegative Random Variables},pdfauthor={}}
\crefname{theorem}{Theorem}{Theorems}
\crefname{conjecture}{Conjecture}{Conjectures}
\Crefname{conjecture}{Conjecture}{Conjectures}
\crefname{lemma}{Lemma}{Lemmas}
\crefname{proposition}{Proposition}{Propositions}
\crefname{corollary}{Corollary}{Corollaries}
\crefname{claim}{Claim}{Claims}
\theoremstyle{remark}

\newcommand{\doi}[1]{\href{https://doi.org/#1}{\nolinkurl{doi:#1}}}
\begin{document}
\maketitle
{
\renewcommand{\thefootnote}{\fnsymbol{footnote}}
\footnotetext[1]{School of Mathematical Sciences, Shanghai Jiao Tong University, 800 Dongchuan Road, Shanghai 200240, China.
  Supported by National Key R\&D Program of China under grant No. 2022YFA1006400, National Natural Science Foundation of China No. 12571376 and Shanghai Institute for Mathematics and Interdisciplinary Sciences, SIMIS-ID-26-AMS-003.
  Email: chengchi@sjtu.edu.cn (C. Chi), yan.w@sjtu.edu.cn (Y. Wang) and chocolatesun@sjtu.edu.cn (S. Sun).}
}
\begin{abstract}
  We resolve a conjecture of \L{}uczak, Mieczkowska and \v{S}ileikis on the maximal tail probability of sums of independent and identically distributed nonnegative random variables.
  Specifically, for every integer $k\ge1$ and every $0\le x\le1/k$, we prove
  \[
    \sup_{\E X_1\le x}\Prob(X_1+\cdots+X_k\ge1)
    =\max\{1-(1-x)^k,(kx)^k\},
  \]
  where the supremum is over nonnegative i.i.d.\ random variables $X_1,\ldots,X_k$.
  We also characterize all maximizing distributions.

  Using this result, we prove the Erd\H{o}s Matching Conjecture for every fixed $k\ge3$, all sufficiently large $s$, and every $n\ge k(s+1)$.
  More specifically, we show that in this range, the maximum number of edges in an $n$-vertex $k$-uniform hypergraph containing no $s+1$ pairwise disjoint edges is
  \[
    \max\left\{
    \binom{k(s+1)-1}{k},
    \binom nk-\binom{n-s}{k}
    \right\}.
  \]
  Moreover, every extremal family consists either of all $k$-sets meeting a fixed $s$-set or of all $k$-sets contained in a fixed $(k(s+1)-1)$-set.
\end{abstract}

\section{Introduction}\label{sec:introduction}

\subsection{Maximum tail probability of i.i.d.\ sums of nonnegative random variables}

We consider the problem of determining the largest probability that the sum of $k$ independent and identically distributed (i.i.d.) nonnegative random variables is at least one when their common mean is at most a prescribed value.
For $k=1$, the answer follows from Markov's inequality.
In 1955, Hoeffding and Shrikhande~\cite{HoeffdingShrikhande} solved the case $k=2$.
Samuels~\cite{Samuels1966,Samuels1968,Samuels1969} subsequently studied the broader setting of independent nonnegative variables whose distributions and means need not be identical.

  {Alon, Frankl, Huang, R\"odl, Ruci\'nski and Sudakov~\cite{AFHRS} related tail inequalities for sums of nonnegative random variables to fractional matchings in uniform hypergraphs.}
Building on this connection, \L{}uczak, Mieczkowska and \v{S}ileikis~\cite{LMS} formulated a precise conjecture for the i.i.d.\ nonnegative random variables problem and proved its equivalence to the corresponding asymptotic fractional matching problem.

For a Borel probability law $\mathsf P$ on $[0,\infty)$, let $m(\mathsf P)=\int t\dd\mathsf P(t)$ and define
\[
  F_k(\mathsf P)=\Prob(X_1+\cdots+X_k\ge1)
  \text{ and }
  \Phi_k(x)=\sup_{m(\mathsf P)\le x}F_k(\mathsf P),
\]
where $X_1,\ldots,X_k$ are independent with common law $\mathsf P$, and the supremum is over all such laws with mean at most $x$.
  {If $x\ge1/k$, then $\Phi_k(x)=1$, since we may choose $X_i=1/k$ with probability one for each $i$, so that the mean constraint is satisfied and their sum equals one.}
For $0\le x\le1/k$, consider the two laws
\[
  \mathsf B_x=(1-x)\delta_0+x\delta_1
  \text{ and }
  \mathsf C_{k,x}=(1-kx)\delta_0+kx\delta_{1/k},
\]
where $\delta_t$ denotes the point mass at $t$.
Equivalently,
\[
  \mathsf B_x(\{t\})=
  \begin{cases}
    1-x, & t=0,             \\
    x,   & t=1,             \\
    0,   & \text{otherwise}
  \end{cases}
  \text{ and }
  \mathsf C_{k,x}(\{t\})=
  \begin{cases}
    1-kx, & t=0,             \\
    kx,   & t=1/k,           \\
    0,    & \text{otherwise}
  \end{cases}.
\]
Both laws have mean $x$.
For $\mathsf B_x$, the sum reaches one exactly when at least one observation equals one, so $F_k(\mathsf B_x)=1-(1-x)^k$.
For $\mathsf C_{k,x}$, the sum reaches one exactly when all $k$ observations equal $1/k$, so $F_k(\mathsf C_{k,x})=(kx)^k$.
Thus, $\Phi_k(x)\ge\max\{1-(1-x)^k,(kx)^k\}$.
The conjecture of \L{}uczak, Mieczkowska and \v{S}ileikis~\cite{LMS} asserts that this lower bound is sharp for every $k\ge2$ and $0\le x\le1/k$.

\begin{conjecture}[\L{}uczak--Mieczkowska--\v{S}ileikis~\cite{LMS}]
  \label{conj:lms}
  For every integer $k\ge2$ and every real $0\le x\le1/k$,
  \[
    \Phi_k(x)=\max\{1-(1-x)^k,(kx)^k\}.
  \]
\end{conjecture}

\L{}uczak, Mieczkowska and \v{S}ileikis~\cite{LMS} proved \cref{conj:lms} for $k=3$ and every $x$.
More recently, Frankl, Lu, Ma and Wu~\cite{FranklLuMaWu2026} and Hou, Hu and Liu~\cite{HouHuLiu2026} obtained results on the Erd\H{o}s matching conjecture for $4$-uniform hypergraphs.
Combining these results with the equivalence established by \L{}uczak, Mieczkowska and \v{S}ileikis~\cite{LMS} yields \cref{conj:lms} for $k=4$.

For $k\ge5$, \L{}uczak, Mieczkowska and \v{S}ileikis~\cite{LMS} proved \cref{conj:lms} when $x<1/(2k-1)$.
Frankl and Kupavskii~\cite{FranklKupavskii2022} extended the range $x<1/(2k-1)$ to $x\le3/(5k-2)$.
For every $k\ge2$, a result of Fu, Han, Wang, Yan, Zhang and Zhou~\cite{FHWYZZ} implies $\Phi_k(x)=1-(1-x)^k$ for $0\le x\le1/(k+1)$.
Cao, Liu and Zhang~\cite{CLZ} proved that $\mathsf B_x$ is the unique maximizing distribution in this range and used this characterization to prove a stability theorem for hypergraphs with bounded matching number.
Despite this progress, Conjecture~\ref{conj:lms} remains open in general for $k\ge5$.

We prove Conjecture~\ref{conj:lms} for every $k$ and determine all maximizing Borel distributions.
To describe the equality cases for $k\ge2$, let ${\beta_k}$ be the unique value of $x$ in $(0,1/k)$ at which $\mathsf B_x$ and $\mathsf C_{k,x}$ have the same success probability.
Equivalently, $1-(1-{\beta_k})^k=(k{\beta_k})^k$.
  {The existence and uniqueness of ${\beta_k}$ are proved in \cref{lem:crossing-parameters}.}
Our main result is as follows.

\begin{theorem}\label{thm:iid-tail}
  For every integer $k\ge1$ and every real $x\ge0$,
  \[
    \Phi_k(x)=
    \begin{cases}
      \max\{1-(1-x)^k,(kx)^k\}, & 0\le x\le1/k, \\
      1,                        & x\ge1/k.
    \end{cases}
  \]
  The maximizing distributions are characterized as follows.
  \begin{enumerate}[label=\textup{(\roman*)},leftmargin=*,itemsep=0pt]
    \item For $k=1$ and $0\le x\le1$, the unique maximizing distribution is $\mathsf B_x$.
    \item For $k\ge2$ and $0\le x<\beta_k$, the unique maximizing distribution is $\mathsf B_x$.
    \item For $k\ge2$ and $x=\beta_k$, the maximizing distributions are exactly $\mathsf B_{\beta_k}$ and $\mathsf C_{k,\beta_k}$.
    \item For $k\ge2$ and $\beta_k<x\le1/k$, the unique maximizing distribution is $\mathsf C_{k,x}$.
    \item For every $k\ge1$ and $x>1/k$, a distribution with mean at most $x$ is maximizing if and only if it assigns probability one to $[1/k,\infty)$.
  \end{enumerate}
\end{theorem}

\subsection{The Erd\H{o}s Matching Conjecture}
\label{sec:matching-intro}

We next apply \cref{thm:iid-tail} to hypergraphs with bounded matching number.
The matching number $\nu(\mathcal F)$ of a $k$-uniform hypergraph ($k$-graph, for brevity) $\mathcal F$ is the maximum number of pairwise disjoint edges in $\mathcal F$.
In 1965, Erd\H{o}s~\cite{Erdos1965} asked for the maximum number of edges in an $n$-vertex $k$-graph $\mathcal F$ with $\nu(\mathcal F)\le s$, that is, containing no $s+1$ pairwise disjoint edges.

For a positive integer $n$, write $[n]=\{1,\ldots,n\}$.
  {There are two natural constructions with many edges and matching number at most $s$.}
One is the \emph{clique family} $\binom{Q}{k}$, where $Q\subseteq[n]$ is a fixed $({k(s+1)-1})$-set.
The other is the \emph{cover family} $\mathcal A(S)=\{E\in\binom{[n]}k:E\cap S\ne\emptyset\}$, where $S\subseteq[n]$ is {any fixed $s$-set.}
The two families have sizes $C_k(s)=\binom{{k(s+1)-1}}{k}$ and $M_k(n,s)=\binom{n}{k}-\binom{n-s}{k}$, respectively.
  {Both have matching number at most $s$. The clique family has too few vertices to contain $s+1$ disjoint edges, while every edge of the cover family must meet the fixed $s$-set.}
Erd\H{o}s~\cite{Erdos1965} conjectured that the maximum size is $\max\{C_k(s),M_k(n,s)\}$.

\begin{conjecture}[Erd\H{o}s~\cite{Erdos1965}]
  \label{conj:emc}
  Let $n,k,s$ be positive integers with $n\ge k(s+1)$.
  If $\mathcal F\subseteq\binom{[n]}k$ and $\nu(\mathcal F)\le s$, then $|\mathcal F|\le\max\{C_k(s),M_k(n,s)\}$.
\end{conjecture}

Several special cases of \cref{conj:emc} have been studied.
For $s=1$, \cref{conj:emc} is the theorem of Erd\H{o}s, Ko and Rado~\cite{ErdosKoRado1961} on intersecting families. For $k=2$, Erd\H{o}s and Gallai~\cite{ErdosGallai1959} proved the corresponding matching theorem.
Kleitman~\cite{Kleitman1968} proved the {special case when $n=k(s+1)$.}
For $k=3$, Frankl, R\"odl and Ruci\'nski~\cite{FranklRodlRucinski} proved \cref{conj:emc} for $n\ge4(s+1)$.
\L{}uczak and Mieczkowska~\cite{LuczakMieczkowska2014} proved the conjecture for all sufficiently large $n$ and every $s\ge1$ satisfying $n\ge3(s+1)$.
Frankl~\cite{Frankl2017Triple} completely solved the conjecture for $k=3$.
  {When} $k=4$, Frankl, Lu, Ma and Wu~\cite{FranklLuMaWu2026} proved the conjecture for $n\ge5s$ and sufficiently large $n$.
The result of Hou, Hu and Liu~\cite{HouHuLiu2026} confirms \cref{conj:emc} for $k=4$ provided that $s\ge6004$.

For general uniformity, progress has followed two main directions.
On the cover side, Erd\H{o}s~\cite{Erdos1965} proved the conjecture when $n$ is sufficiently large in terms of $k$ and $s$.
Bollob\'as, Daykin and Erd\H{o}s~\cite{BDE} {proved the conjecture for $n\ge2k^3s$.}
Huang, Loh and Sudakov~\cite{HuangLohSudakov} {extended the range to $n\ge3k^2s$.}
Frankl, \L{}uczak and Mieczkowska~\cite{FranklLuczakMieczkowska2012} {further improved it to $n>2k^2s/\log k$.}
Frankl~\cite{Frankl2013} then {obtained} the linear range $n\ge(2k-1)s+k$.
For sufficiently large $s$, Frankl and Kupavskii~\cite{FranklKupavskii2022} reduced the leading coefficient to $(5k-2)/3$.
Cao, Liu and Zhang~\cite{CLZ} obtained the range $n\ge(k+1)s$ for every fixed $k$ and sufficiently large $s$, together with uniqueness and stability of the cover construction.
Using Ling's theorem on Samuels' conjecture~\cite{Ling2026}, Cao, Liu and Zhang~\cite[Section~6.1]{CLZ} further obtained the range $n\ge(k+0.6)s$ for fixed $k\ge5$ and sufficiently large $s$.

The other direction concerns the clique side, where $n$ is close to $k(s+1)$.
For $s>k\ge2$, Frankl~\cite{Frankl2017Range} proved \cref{conj:emc} for $k(s+1)\le n<(k+1/(2k^{2k+1}))(s+1)$.
Kolupaev and Kupavskii~\cite{KolupaevKupavskii2023} extended the range to $k(s+1)\le n<(k+1/(100k))(s+1)$ for $s>101k^3$.
For $k\ge5$, \cref{conj:emc} remains open in full generality.

The edge density of an $n$-vertex $k$-graph $\mathcal F$ is $|\mathcal F|/\binom nk$.
For fixed $k$, as $n\to\infty$ and $s/n\to x$, the edge densities of the cover and clique families tend to $F_k(\mathsf B_x)=1-(1-x)^k$ and $F_k(\mathsf C_{k,x})=(kx)^k$, respectively.
Thus, the two limiting densities agree when $n/s\to\rho_k$, where $\rho_k=1/\beta_k$.
The main difficulty is to resolve \cref{conj:emc} near this transition, where neither construction has a fixed positive advantage in edge density.

We use \cref{thm:iid-tail} to prove \cref{conj:emc} for every fixed $k\ge3$, all sufficiently large $s$, and every $n\ge k(s+1)$.
The following theorem covers both the range where the cover family dominates and the range where the clique family dominates, including the transition between them, and determines all extremal families.

  \begin{theorem}\label{thm:emc}
    For every fixed integer $k\ge3$, there is an integer $s_{\mathrm{EMC}}(k)$ such that, for all integers $s\ge s_{\mathrm{EMC}}(k)$ and $n\ge k(s+1)$, every family $\mathcal F\subseteq\binom{[n]}k$ with $\nu(\mathcal F)\le s$ satisfies
    \[
      |\mathcal F|\le\max\{C_k(s),M_k(n,s)\}.
    \]
    The equality cases are as follows.
    \begin{enumerate}[label=\textup{(\roman*)},leftmargin=*,itemsep=0pt]
      \item If $M_k(n,s)>C_k(s)$, then equality holds if and only if $\mathcal F=\mathcal A(S)$ for some $s$-set $S\subseteq[n]$.
      \item If $M_k(n,s)<C_k(s)$, then equality holds if and only if $\mathcal F=\binom Qk$ for some $(k(s+1)-1)$-set $Q\subseteq[n]$.
      \item If $M_k(n,s)=C_k(s)$, then equality holds if and only if $\mathcal F=\mathcal A(S)$ for some $s$-set $S\subseteq[n]$ or $\mathcal F=\binom Qk$ for some $(k(s+1)-1)$-set $Q\subseteq[n]$.
    \end{enumerate}
  \end{theorem}

We also describe the structure of nearly extremal families.
Fix $k\ge3$, and let $s\to\infty$ with $n\ge k(s+1)$ and $n/s$ bounded.
If $\mathcal F\subseteq\binom{[n]}k$ satisfies $\nu(\mathcal F)\le s$ and
\[
  |\mathcal F|\ge\max\{C_k(s),M_k(n,s)\}-o(n^k),
\]
then $\mathcal F$ differs in $o(n^k)$ members from a cover or clique family.
The conclusion remains valid when $s/n\to\beta_k$, where the two constructions have the same limiting edge density.
A unified formulation of this stability result for the Erd\H{o}s Matching Conjecture is given in \cref{prop:unified-stability}.

We also obtain two Hilton--Milner refinements for families with matching number at most $s$.
The first imposes the additional condition that no $s$-set meets every edge, or equivalently that the vertex cover number, the minimum size of a vertex set meeting every edge, is greater than $s$.
This excludes the cover construction.
For every fixed $k\ge3$ and all sufficiently large $s$, \cref{cor:hilton-milner} determines the maximum size and all extremal families under this condition throughout $n\ge k(s+1)$.

The other refinement requires at least $k(s+1)$ non-isolated vertices, which excludes the clique construction.
For every fixed $k\ge3$, \cref{cor:clique-hilton-milner} determines the maximum size and all extremal families throughout $k(s+1)\le n\le(1/\beta_k-o(1))s$ as $s\to\infty$.

\subsection{Proof outlines for Theorems~\ref{thm:iid-tail} and~\ref{thm:emc}}

We first outline the proof of \cref{thm:iid-tail} and then explain how it implies \cref{thm:emc}.

The case $k=1$ follows directly from Markov's inequality, as shown in \cref{lem:dimension-one}.
For $k\ge2$, mixtures with point masses at one and zero reduce the upper bound in the nontrivial mean range to the crossing mean $\beta_k$ by \cref{prop:mean-transfer}.
We first prove that the maximum is attained in \cref{lem:attainment}.
Using Kellerer's duality theorem (\cref{thm:kellerer-duality}), we then show in \cref{prop:normalization} that every maximizing distribution at this mean can be written as $\mathsf P=(1-k\beta_k)\delta_0+k\beta_k\mathsf Q$, where $\mathsf Q$ is a Borel probability law on $[0,1]$ with mean $1/k$.
Conditioning on the number of samples drawn from $\mathsf Q$, we express $F_k(\mathsf P)/(k\beta_k)^k$ as the full-sum success probability plus a weighted contribution from shorter sums, as stated in \cref{fact:retained-sums}\textup{(i)}.
For independent samples $Y_1,\ldots,Y_k$ from $\mathsf Q$, let $w=k\E(Y_1+\cdots+Y_k-1)_+$.
We choose the interpolation nodes $W_s=s(1-1/s)^s$, where $1\le s\le k$.
As functions of $w$, we construct an upper bound $A(w)$ for the shorter-sum contribution and a lower bound $H(w)$ for the full-sum failure probability in \cref{lem:allocation,lem:critical}.
Maximality requires $A(w)\ge H(w)$ by \cref{prop:strict-chord}.
We compare $A$ and $H$ by interpolation to derive a contradiction for values of $w$ in $(0,W_k)$.
By \cref{lem:critical}, we have $0\le w\le W_k$.
For every integer $s$ with $1\le s<k$, the function $A$ is affine and $H$ is strictly concave on $[W_s,W_{s+1}]$, as shown in \cref{prop:strict-chord}.
Thus, if $A(W_s)\le H(W_s)$ and $A(W_{s+1})\le H(W_{s+1})$, then $A(w)<H(w)$ for every $W_s<w<W_{s+1}$, which contradicts maximality and excludes all such values of $w$.
For $2\le k\le5$, the node comparisons in \cref{lem:initial-nodes} exclude every $0<w<W_k$, leaving only $w=0$ and $w=W_k$.
For $k\ge6$, they exclude $0<w\le W_4$, leaving $w=0$ or $W_4<w\le W_k$.

For $k\ge6$, suppose that a maximizing distribution under the mean bound $\beta_k$ has $0<w<W_k$.
Then $W_4<w<W_k$ by \cref{cor:remaining-excess-range}.
We use the optimality of the maximizing distribution to derive a contradiction.
Small perturbations of its probabilities that preserve the total mass and mean yield an optimality condition relating a fixed coordinate value to the probability of success in \cref{obs:contact-identities}.
Together with quantitative estimates for the sums, the optimality condition shows that the distribution assigns all its mass to two intervals separated by a positive gap, with positive probability in each interval, as shown in \cref{prop:support-separation}.
Let $X^{\operatorname{L}}$ be a random variable that has the conditional distribution on the lower interval.
Among samples whose sum is less than one, we separately count those with no observation in the upper interval and those with exactly one.
Estimating these contributions with the optimality condition and comparing their sum with the total failure probability, we obtain $\Prob(X_1^{\operatorname{L}}+\cdots+X_k^{\operatorname{L}}\ge1)>\eta/3$ by \cref{lem:exact-upper-count}, where $X_1^{\operatorname{L}},\ldots,X_k^{\operatorname{L}}$ are independent copies of $X^{\operatorname{L}}$.
Shevtsova's Berry--Esseen bound (\cref{lem:normal-comparison}) compares the centered lower-interval sum with a normal distribution.
Combined with the optimality condition, it rules out $\eta\ge1/10$ by establishing an upper variance bound incompatible with the lower variance bound, as shown in \cref{lem:conditional-variance-bounds,lem:averaged-contact-normal,lem:large-defect}.
In the remaining range $0<\eta<1/10$, \cref{lem:large-defect} in fact gives the stronger bound $\eta<k\beta_k/10$.
Sharper support and variance estimates, together with a centered exponential-moment bound, then yield an upper bound smaller than $\eta/3$ for the same tail probability by \cref{lem:conditional-variance-bounds,lem:conditional-support,lem:small-conditional-tail}, which again leads to a contradiction.
Thus, every maximizing distribution at $\beta_k$ has $w=0$ or $w=W_k$, as stated in \cref{prop:maximizing-endpoints}.
    At $w=0$ and $w=W_k$, the equality $A(w)=H(w)$ gives $\Phi_k(\beta_k)=(k\beta_k)^k$ by \cref{prop:strict-chord}.
    \Cref{prop:mean-transfer} then determines $\Phi_k(x)$ for every $x\ge0$.
    The values $w=0$ and $w=W_k$ correspond to $\mathsf P=\mathsf C_{k,\beta_k}$ and $\mathsf P=\mathsf B_{\beta_k}$, respectively, and \cref{prop:equality-transfer} determines all maximizing distributions for every mean bound $x\ge0$.

For \cref{thm:emc}, we use the connection between tail probabilities and fractional vertex covers developed by Alon, Frankl, Huang, R\"odl, Ruci\'nski and Sudakov~\cite{AFHRS}.
For a family $\mathcal F\subseteq\binom{[n]}k$, a fractional vertex cover assigns weights in $[0,1]$ to the vertices so that every edge has total weight at least one.
Sampling $k$ vertices independently and uniformly with replacement turns their weights into i.i.d.\ nonnegative random variables, whose common mean is the average vertex weight.
Each edge contributes $k!$ successful ordered samples.
Hence, if the total vertex weight is at most $xn$, then \eqref{eq:cover-sampling} gives, for fixed $k$ as $n\to\infty$,
\[
  \frac{|\mathcal F|}{\binom nk}
  \le\frac{n^k}{k!\binom nk}\Phi_k(x)
  =(1+o(1))\Phi_k(x).
\]

Assigning weight one to the $s$ central vertices of a cover family and zero to the other vertices produces $\mathsf B_{s/n}$.
Assigning weight $1/k$ to the ${k(s+1)-1}$ vertices of a clique family and zero elsewhere produces a distribution tending to $\mathsf C_{k,x}$ when $n\to\infty$ and $s/n\to x$.
\L{}uczak, Mieczkowska and \v{S}ileikis~\cite{LMS} proved that \cref{conj:lms} is equivalent to the asymptotic fractional matching conjecture.
For ordinary matchings, we use the approximate-cover lemma of Cao, Liu and Zhang~\cite[Lemma~3.2]{CLZ}.
For fixed $k$ and a sequence with $\nu(\mathcal F)\le s$, $s\to\infty$ and $s/n\to x\in(0,1/k)$, \cref{lem:approx-cover} yields weights in $[0,1]$ with average at most $s/n+o(1)$ that cover all but $o(n^k)$ edges.
If $|\mathcal F|\ge\max\{C_k(s),M_k(n,s)\}-o(n^k)$, then the sampling argument forces the tail probability of the weight distribution to tend to $\Phi_k(x)$.
By weak compactness, every subsequence of the weight distributions has a further weakly convergent subsequence.
The limit of any such convergent subsequence has mean at most $x$ and, by the Portmanteau theorem (\cref{thm:portmanteau}\textup{(iii)}), attains $\Phi_k(x)$.
By \cref{thm:iid-tail}, the only possible limits are $\mathsf B_x$ and $\mathsf C_{k,x}$, including at $x=\beta_k$, where both are maximizing.
  Convergence to $\mathsf B_x$ implies that all but $o(n)$ vertex weights lie near $0$ or $1$, with $s+o(n)$ weights near $1$, while convergence to $\mathsf C_{k,x}$ implies that all but $o(n)$ vertex weights lie near $0$ or $1/k$, with $ks+o(n)$ weights near $1/k$.
  Adding or removing $o(n)$ vertices produces an $s$-set $S$ from the vertices with weights near $1$, or a $(k(s+1)-1)$-set $Q$ from those with weights near $1/k$.
  With $S$ or $Q$ chosen as above, \cref{lem:approx-cover} and $|\mathcal F|\ge\max\{C_k(s),M_k(n,s)\}-o(n^k)$ yield $|\mathcal F\triangle\mathcal A(S)|=o(n^k)$ or $|\mathcal F\triangle\binom Qk|=o(n^k)$, respectively.
  To prove \cref{thm:emc}, we first apply \cref{lem:sparse-emc} to handle $n\ge4k^3s$, including the equality cases.
  Suppose that, for arbitrarily large $s$, there is a family $\mathcal F\subseteq\binom{[n]}k$ with $n\ge k(s+1)$, $\nu(\mathcal F)\le s$, and $|\mathcal F|\ge\max\{C_k(s),M_k(n,s)\}$ that is neither a cover family nor a clique family.
  \Cref{lem:sparse-emc} gives $n<4k^3s$.
  We choose a sequence with $s\to\infty$ and pass to a subsequence such that $n/s\to R\in[k,4k^3]$.
  \Cref{lem:matching-sequence} gives a further subsequence satisfying $|\mathcal F\triangle\mathcal A(S)|=o(n^k)$ with $R\ge\rho_k>k$ or $|\mathcal F\triangle\binom Qk|=o(n^k)$ with $R\le\rho_k<k+1/2$.
  \Cref{lem:local-cover} then gives $|\mathcal F|\le M_k(n,s)$ with equality only for cover families, while \cref{lem:local-clique} gives $|\mathcal F|\le C_k(s)$ with equality only for clique families.
  Since $|\mathcal F|\ge\max\{C_k(s),M_k(n,s)\}$, equality is forced and $\mathcal F$ must itself be a cover or clique family, a contradiction.

  The paper is organized as follows.
  \Crefrange{sec:crossing-reduction}{sec:equality-cases} prove \cref{thm:iid-tail}, including the classification of all maximizing distributions.
  \Cref{sec:emc} proves cover--clique stability and \cref{thm:emc}.
  \Cref{sec:concluding} gives a unified stability statement and two Hilton--Milner refinements.

\section{Reduction and decomposition of maximizing laws}\label{sec:crossing-reduction}
\label{sec:dimension-one}

We begin this section by proving \cref{thm:iid-tail} for $k=1$.

\begin{lemma}\label{lem:dimension-one}
  For every $x\ge0$, we have $\Phi_1(x)=\min\{x,1\}$.
  For $0\le x\le1$, the unique maximizing law is $\mathsf B_x=(1-x)\delta_0+x\delta_1$.
  For $x>1$, a probability law on $[0,\infty)$ with mean at most $x$ is maximizing if and only if it is supported on $[1,\infty)$.
\end{lemma}

\begin{proof}
  Let $0\le x\le1$, and let $X$ be a nonnegative random variable with $\E X\le x$.
  Markov's inequality yields
  \[
    \Prob(X\ge1)=\E\1_{\{X\ge1\}}\le\E X\le x.
  \]
  The law $\mathsf B_x$ has mean $x$ and success probability $x$, so $\Phi_1(x)=x$.

  If $\Prob(X\ge1)=x$, then $\E X=\E\1_{\{X\ge1\}}=x$, so the nonnegative random variable $X-\1_{\{X\ge1\}}$ has expectation zero and hence vanishes almost surely.
  Thus, $X$ takes only the values zero and one, with $\Prob(X=1)=x$, so its law is $\mathsf B_x$.

  For $x>1$, the law $\delta_1$ is feasible and has success probability one, so $\Phi_1(x)=1$.
  For any nonnegative random variable $X$ with $\E X\le x$, its law attains this value exactly when $\Prob(X\ge1)=1$, which is equivalent to being supported on $[1,\infty)$.
\end{proof}

For the remainder of this section, assume $k\ge2$.
We reduce the upper bound to the mean $x=\beta_k$, at which the two candidate laws, $\mathsf B_{\beta_k}$ and $\mathsf C_{k,\beta_k}$, have the same success probability.
Adding mass at one or at zero converts the bound at $\beta_k$ into bounds for smaller or larger means.
We then prove that the maximum exists and that every maximizing law at $x=\beta_k$ is a mixture of a point mass at zero and a distribution of mean $1/k$.
Finally, by truncating the random variables one at a time, we obtain the bounds for the failure probability needed in \cref{sec:low-dimensions}.

All reductions apply to nonnegative Borel probability laws.
  {We first collect some necessary notation.}

  For a metric space $K$, let $\Pcal(K)$ denote the set of Borel probability measures on $K$.
  For a $K$-valued random variable $X$ on a probability space $(\Omega,\mathcal F,\Prob)$, its \emph{law} $\law(X)$ is the measure defined by $\law(X)(A)=\Prob(X\in A)$ for every Borel set $A\subseteq K$.
  For an event $E$ with $\Prob(E)>0$, the conditional law $\law(X\mid E)$ is defined by $\law(X\mid E)(A)=\Prob(X\in A\mid E)$ for every Borel set $A\subseteq K$.
  For $\mathsf P\in\Pcal(K)$ and a positive integer $j$, the \emph{product measure} $\mathsf P^{\otimes j}$ is the joint law of $j$ independent random variables with common law $\mathsf P$.
  For a nonnegative real-valued law $\mathsf P$, write $m(\mathsf P)=\int_{[0,\infty)}z\dd\mathsf P(z)$ for its \emph{mean}.
  For $t\in K$, the \emph{Dirac law} $\delta_t$ assigns mass one to $t$.
  It satisfies $\delta_t(A)=1$ if $t\in A$ and $\delta_t(A)=0$ otherwise.
  For real numbers $u,v$, write $u_+=\max\{u,0\}$ and $u\wedge v=\min\{u,v\}$.
  A function $f$ on an interval is \emph{affine} if $f(t)=At+B$ for some constants $A,B\in\mathbb R$.
  For an event $A$, its indicator $\1_A$ equals one on $A$ and zero outside $A$.
  The \emph{topological support} $\supp\mathsf P$ consists of the points $t\in K$ for which every open neighbourhood $U$ of $t$ in $K$ satisfies $\mathsf P(U)>0$.
  For a random variable $X$, write $\supp X=\supp\law(X)$.
  We say that $\mathsf P$ is \emph{supported on} a Borel set $A$ when $\mathsf P(A)=1$.
  A property holds for $\mathsf P$-almost every $z$ if the set of points where it fails has $\mathsf P$-measure zero.
  Equality of Borel laws means equality on every Borel set.
  Conditional expectations are taken on events of positive probability.

For independent $X_1,X_2,\ldots$ with a common nonnegative Borel law $\mathsf P$, write $S_j=\sum_{i=1}^jX_i$ and $F_j(\mathsf P)=\Prob(S_j\ge1)$ for every positive integer $j$.
When considering partial sums of $S_k$, we take $1\le j\le k$.
As in the introduction, for $x\ge0$, let
\[
  \Phi_k(x)=\sup\{F_k(\mathsf P):\mathsf P\in\Pcal([0,\infty)),\ m(\mathsf P)\le x\}.
\]
For $\mathsf P\in\Pcal([0,1])$,
\[
  F_j(\mathsf P)
  =\int_{[0,1]^j}
  \1_{\{z_1+\cdots+z_j\ge1\}}\dd\mathsf P^{\otimes j}(z).
\]
We use $H_k=\{z\in[0,1]^k:z_1+\cdots+z_k\ge1\}$ for the success set in the truncated space.
Throughout, whenever $S$ denotes the sum of the random variables under consideration, we refer to the event $S\ge1$ as success and to $S<1$ as failure.
The candidate laws are $\mathsf B_x=(1-x)\delta_0+x\delta_1$ for $0\le x\le1$ and $\mathsf C_{k,x}=(1-kx)\delta_0+kx\delta_{1/k}$ for $0\le x\le1/k$.

\subsection{\texorpdfstring{Reduction to $x=\beta_k$}{Reduction to x = beta k}}
\label{sec:mix-one}\label{sec:mean-transfer-section}
\label{sec:crossing-notation}

  For $k\ge2$, we show that it suffices to prove the upper bound in \cref{thm:iid-tail} at $x=\beta_k$ (\cref{prop:mean-transfer}), where the two distributions $\mathsf B_x$ and $\mathsf C_{k,x}$ have the same success probability.
  We first verify that the equation $1-(1-x)^k=(kx)^k$ has a unique solution $\beta_k\in(0,1/k)$.

  For $k\ge2$ and $0<x<1/k$, we use the notation $a=kx$, $\delta=1-a$, and $c=\delta/x$ throughout Sections~\ref{sec:crossing-reduction}--\ref{sec:exclusion}.
  Thus, $x=1/(k+c)$ and $\delta=cx$.
  For $k\ge2$ and $c\ge0$, define
  \begin{equation}\label{eq:sec3-crossing-function}
    g_k(c)=\left(1+\frac ck\right)^k-\left(1+\frac{c-1}{k}\right)^k.
  \end{equation}
  The substitution $x=1/(k+c)$ transforms $1-(1-x)^k=(kx)^k$ into $g_k(c)=1$.
  Let $c_k\in(0,1)$ be the solution of $g_k(c)=1$.
  The existence and uniqueness of $c_k$ are proved in \cref{lem:crossing-parameters}.
  The parameter $\beta_k$ defined in \cref{sec:introduction} therefore satisfies $\beta_k=1/(k+c_k)$.
  At $x=\beta_k$, we have $c=c_k$ and $a^k=1-(1-\beta_k)^k$.

  \begin{lemma}\label{lem:crossing-parameters}
    For every integer $k\ge2$, the equation $g_k(c)=1$ has a unique solution $c=c_k$ in $[0,\infty)$, and ${\beta_k}$ is the unique solution of $1-(1-x)^k=(kx)^k$ in $(0,1/k]$.
    Moreover, $c_2=1/2$ and $0<c_k<1/2$ for $k\ge3$.
  \end{lemma}

  \begin{proof}
    Differentiating \eqref{eq:sec3-crossing-function}, we obtain $g_k'(c)=(1+c/k)^{k-1}-(1+(c-1)/k)^{k-1}>0$ for $c\ge0$, so $g_k$ is strictly increasing on $[0,\infty)$.
    We have $g_k(0)=1-(1-1/k)^k<1$ and $g_k(1)=(1+1/k)^k-1>1$.
    Hence, there is a unique $c_k\in(0,1)$ satisfying $g_k(c_k)=1$.
    Multiplying $g_k(c_k)=1$ by $(k/(k+c_k))^k$, we obtain $1-(1-{\beta_k})^k=(k{\beta_k})^k$.
    The substitution $c=1/x-k$ is a bijection from $(0,1/k]$ to $[0,\infty)$ and transforms $1-(1-x)^k=(kx)^k$ into $g_k(c)=1$.
            Since $g_k(c)=1$ has the unique solution $c=c_k$, the equation $1-(1-x)^k=(kx)^k$ has the unique solution $x=1/(k+c_k)=\beta_k$ in $(0,1/k]$.
    For $k=2$, the equation $g_2(c)=1$ has the unique solution $c_2=1/2$, and the corresponding crossing mean is $\beta_2=2/5$.
    For $k\ge3$,
      \[
        \left(1+\frac1{2k}\right)^k-\left(1-\frac1{2k}\right)^k
        =1+2\sum_{\substack{3\le j\le k\\j\text{ odd}}}
        \binom{k}{j}(2k)^{-j}>1.
      \]
    This shows that $0<c_k<1/2$ for $k\ge3$, since $g_k$ is strictly increasing.
  \end{proof}

  For a nonnegative law $\mathsf P$ of finite mean and parameters $0\le\theta,q\le1$, define the two mixtures $\mathsf P_\theta^+=(1-\theta)\mathsf P+\theta\delta_1$ and $\mathsf P_q^-=q\mathsf P+(1-q)\delta_0$.
  The first replaces a sample from $\mathsf P$ by one with probability $\theta$.
  The second retains it with probability $q$ and replaces it by zero otherwise.

  \begin{fact}\label{fact:mixtures}
    Let $k\ge2$, and let $\mathsf P$ be a nonnegative Borel law of finite mean.
    The following statements hold.
    \begin{enumerate}[label=\textup{(\roman*)},leftmargin=*,itemsep=0pt]
      \item\label{it:mixtures-one} For $0\le\theta\le1$, we have $m(\mathsf P_\theta^+)=(1-\theta)m(\mathsf P)+\theta$ and $1-F_k(\mathsf P_\theta^+)=(1-\theta)^k[1-F_k(\mathsf P)]$.
      Consequently, $F_k(\mathsf P_\theta^+)\ge F_k(\mathsf P)$, with equality if and only if $\theta=0$ or $F_k(\mathsf P)=1$.
      \item\label{it:mixtures-zero} For $0\le q\le1$, we have $m(\mathsf P_q^-)=qm(\mathsf P)$ and
      \[
        F_k(\mathsf P_q^-)=\sum_{j=1}^k\binom{k}{j}q^j(1-q)^{k-j}F_j(\mathsf P)
        \ge q^kF_k(\mathsf P).
      \]
      If $0<q<1$, then equality in the inequality holds if and only if $F_j(\mathsf P)=0$ for every $1\le j<k$.
    \end{enumerate}
  \end{fact}

  \begin{proof}
    The mean formulas follow immediately by linearity.
    For \textup{(i)}, independently replace each of $k$ independent samples from $\mathsf P$ by one with probability $\theta$.
    The resulting sum is less than one exactly when no sample is replaced and the original sum is less than one.
    This proves \textup{(i)}.
    For \textup{(ii)}, retain each sample independently with probability $q$.
    For a fixed set of $j\ge1$ retained indices, independence gives the contribution $q^j(1-q)^{k-j}F_j(\mathsf P)$ to the success probability.
    There are $\binom{k}{j}$ choices of these indices, and the sum is zero when no sample is retained.
    Thus,
    \[
      F_k(\mathsf P_q^-)
      =q^kF_k(\mathsf P)
      +\sum_{j=1}^{k-1}\binom{k}{j}q^j(1-q)^{k-j}F_j(\mathsf P)
      \ge q^kF_k(\mathsf P).
    \]
    If $0<q<1$, then the coefficient $\binom{k}{j}q^j(1-q)^{k-j}$ of $F_j(\mathsf P)$ is positive for every $1\le j<k$.
    Since $F_j(\mathsf P)\ge0$, the sum over $1\le j<k$ vanishes if and only if $F_j(\mathsf P)=0$ for every $1\le j<k$.
  \end{proof}

  The next proposition determines the maximum of $F_k(\mathsf P)$ under $m(\mathsf P)\le x$ for every $x\ge0$.
  The two mixtures reduce the nontrivial range $0<x<1/k$ to $x=\beta_k$.

\begin{proposition}\label{prop:mean-transfer}
  Fix $k\ge2$.
  Suppose that every nonnegative Borel law $\mathsf P$ with $m(\mathsf P)\le\beta_k$ satisfies $F_k(\mathsf P)\le1-(1-\beta_k)^k=(k\beta_k)^k$.
  Then, for every $x\ge0$,
  \[
    \Phi_k(x)=
    \begin{cases}
      \max\{1-(1-x)^k,(kx)^k\}, & 0\le x\le1/k, \\
      1,                        & x\ge1/k.
    \end{cases}
  \]
\end{proposition}

\begin{proof}
  Fix any $x\in(0,\beta_k)$, and let $\mathsf P$ be any nonnegative Borel law with $m(\mathsf P)\le x$.
  Take $\theta=({\beta_k}-x)/(1-x)$.
  Then $m(\mathsf P_\theta^+)\le\beta_k$.
  By \cref{fact:mixtures}\textup{(i)} and the assumed bound at $\beta_k$,
  \[
    (1-\theta)^k[1-F_k(\mathsf P)]
    =1-F_k(\mathsf P_\theta^+)\ge(1-\beta_k)^k.
  \]
  Since $1-\theta=(1-\beta_k)/(1-x)$, we obtain $1-F_k(\mathsf P)\ge(1-x)^k$, with equality attained by $\mathsf B_x$.

  Now fix any $x\in(\beta_k,1/k)$, and let $\mathsf P$ be any nonnegative Borel law with $m(\mathsf P)\le x$.
  Take $q={\beta_k}/x\in(0,1)$, so $m(\mathsf P_q^-)\le {\beta_k}$.
  Combining \cref{fact:mixtures}\textup{(ii)} with the assumed bound at ${\beta_k}$, we obtain
    \[
      q^kF_k(\mathsf P)\le F_k(\mathsf P_q^-)\le(k{\beta_k})^k.
    \]
  Consequently, $F_k(\mathsf P)\le q^{-k}(k{\beta_k})^k=(kx)^k$, with equality attained by $\mathsf C_{k,x}$.

  The two candidate probabilities $F_k(\mathsf B_x)=1-(1-x)^k$ and $F_k(\mathsf C_{k,x})=(kx)^k$ are equal at $x=\beta_k$.
    For $0<x<\beta_k$, we have $F_k(\mathsf B_x)>F_k(\mathsf C_{k,x})$.
    For $\beta_k<x<1/k$, we have $F_k(\mathsf C_{k,x})>F_k(\mathsf B_x)$.
  The conclusion for $x=\beta_k$ follows from the assumed bound and the two candidate laws.
  At $x=0$, nonnegativity forces $\mathsf P=\delta_0$.
  For $x\ge1/k$, the feasible law $\delta_{1/k}$ attains one.
\end{proof}

\subsection[Decomposition of maximizing laws]{Decomposition of maximizing laws}
\label{sec:normalization}\label{sec:notation}

We prove that the supremum defining $\Phi_k(x)$ is attained.
For $k\ge2$ and $0<x<1/k$, we then show that every maximizing distribution is a mixture of a point mass at zero and a distribution of mean $1/k$.

For nonnegative coordinates, truncation at $1$ leaves the success
event unchanged, since $\sum_i z_i\ge1$ if and only if
$\sum_i(z_i\wedge1)\ge1$.
Moreover, $z_i\wedge1\le z_i$ for every $i$, so this truncation
will not increase the mean.

  Let $E$ be a metric space, and let $\mathsf P,\mathsf P_1,\mathsf P_2,\ldots\in\Pcal(E)$ be Borel probability measures.
  We write $\mathsf P_j\Rightarrow\mathsf P$ and say that $\mathsf P_j$ \emph{converges weakly} to $\mathsf P$ if $\int f\dd\mathsf P_j\to\int f\dd\mathsf P$ for every bounded continuous function $f:E\to\mathbb R$.
  For a Borel set $A\subseteq E$, its \emph{boundary} $\partial_E A$ is the intersection of the closures in $E$ of $A$ and $E\setminus A$.
  We call $A$ a \emph{$\mathsf P$-continuity set} if $\mathsf P(\partial_E A)=0$.

We shall use the following form of the Portmanteau theorem.

\begin{theorem}[Portmanteau theorem; Billingsley~{\cite[Theorem~2.1]{Billingsley}}]
  \label{thm:portmanteau}
  Let $\mathsf P,\mathsf P_1,\mathsf P_2,\ldots$ be Borel probability measures
  on a metric space $S$. The following conditions are equivalent:
  \begin{enumerate}[label=\textup{(\roman*)},leftmargin=*,itemsep=2pt]
    \item The laws $\mathsf P_j$ converge weakly to $\mathsf P$.
    \item For every bounded uniformly continuous function
          $f:S\to\mathbb R$, we have $\int_S f\dd\mathsf P_j\longrightarrow\int_S f\dd\mathsf P$.
    \item Each closed set $C\subseteq S$ satisfies $
            \displaystyle\limsup_{j\to\infty}\mathsf P_j(C)\le\mathsf P(C)$.
    \item Each open set $G\subseteq S$ satisfies $\displaystyle\liminf_{j\to\infty}\mathsf P_j(G)\ge\mathsf P(G)$.
    \item For every Borel set $A\subseteq S$ with
          $\mathsf P(\partial_S A)=0$, one has $\mathsf P_j(A)\longrightarrow\mathsf P(A)$.
  \end{enumerate}
\end{theorem}

The sets in \cref{thm:portmanteau}\textup{(iii)--(v)} are fixed as
$j\to\infty$.
We will use \textup{(iii)} for success probabilities and \textup{(v)} for
probabilities of fixed weight intervals.

We record the following standard consequences of Billingsley~\cite[Theorems~2.8(ii), 5.1 and~6.8]{Billingsley} and \cref{thm:portmanteau}.

\begin{fact}\label{fact:weak-convergence}
  The following statements hold for Borel probability laws.
  \begin{enumerate}[label=\textup{(\roman*)},leftmargin=*,itemsep=0pt]
    \item\label{it:weak-compactness} For every compact metric space $K$, the space $\Pcal(K)$ is compact and metrizable in the topology of weak convergence.
    In particular, every sequence in $\Pcal(K)$ has a weakly convergent subsequence.
    \item\label{it:weak-products} If $\mathsf P_j,\mathsf P\in\Pcal([0,1])$ and $\mathsf P_j\Rightarrow\mathsf P$, then $m(\mathsf P_j)\to m(\mathsf P)$ and $\mathsf P_j^{\otimes r}\Rightarrow\mathsf P^{\otimes r}$ for every fixed positive integer $r$.
    \item\label{it:weak-upper-semicontinuity} If $\mathsf P_j,\mathsf P\in\Pcal([0,1])$ and $\mathsf P_j\Rightarrow\mathsf P$, then $\limsup_{j\to\infty}F_k(\mathsf P_j)\le F_k(\mathsf P)$.
  \end{enumerate}
\end{fact}

\Cref{fact:weak-convergence}\textup{(iii)} follows by applying \cref{thm:portmanteau}\textup{(iii)} to the product laws in \textup{(ii)} and $H_k=\{(z_1,\ldots,z_k)\in[0,1]^k:\sum_{i=1}^kz_i\ge1\}$, which is closed in $[0,1]^k$.

We next apply the inequality $\limsup_{j\to\infty}F_k(\mathsf P_j)\le F_k(\mathsf P)$ to a weakly convergent maximizing sequence to prove that $\Phi_k(x)$ is attained.

\begin{lemma}\label{lem:attainment}
  For every $x\ge0$, the supremum defining $\Phi_k(x)$ is attained by a law $\mathsf P\in\Pcal([0,1])$ with $m(\mathsf P)\le x$.
\end{lemma}

\begin{proof}
  By truncation at one, choose laws $\mathsf P_j\in\Pcal([0,1])$ with $m(\mathsf P_j)\le x$ and $F_k(\mathsf P_j)\to\Phi_k(x)$.
  By \cref{fact:weak-convergence}\textup{(i)}, we may pass to a subsequence such that $\mathsf P_j\Rightarrow\mathsf P$ for some $\mathsf P\in\Pcal([0,1])$.
  By \cref{fact:weak-convergence}\textup{(ii)}, we have $m(\mathsf P)=\lim_{j\to\infty}m(\mathsf P_j)\le x$, so $\mathsf P$ is feasible.
  Applying \cref{fact:weak-convergence}\textup{(iii)} gives
  \[
    \Phi_k(x)
    =\lim_{j\to\infty}F_k(\mathsf P_j)
    \le F_k(\mathsf P)
    \le\Phi_k(x).
  \]
  Thus, $\mathsf P$ attains $\Phi_k(x)$.
\end{proof}

We next normalize a maximizing law.
For $k\ge2$ and $0<x<1/k$, \cref{prop:normalization} shows that every $\mathsf P\in\Pcal([0,1])$ maximizing $F_k$ under $m(\mathsf P)\le x$ can be written as $\mathsf P=(1-a)\delta_0+a\mathsf Q$, where $\mathsf Q\in\Pcal([0,1])$ and $m(\mathsf Q)=1/k$.

\emph{Coupling} probability laws $\mathsf P_1,\ldots,\mathsf P_j$ means constructing random variables $Y_1,\ldots,Y_j$ on a common probability space with $\law(Y_i)=\mathsf P_i$ for every $i$.
For example, let $X\sim\mathsf P$, and independently choose $I\in\{0,1\}$ with $\Prob(I=1)=q$.
Then $IX$ has law $\mathsf P_q^-=q\mathsf P+(1-q)\delta_0$, so $(X,IX)$ couples $\mathsf P$ and $\mathsf P_q^-$ and satisfies $IX\le X$ almost surely.
It should be noticed that the coordinates of a coupling need not be independent.

For $\mathsf P\in\Pcal([0,1])$, let $\Pi_k(\mathsf P)=\{\mathsf G\in\mathcal P([0,1]^k):\mathsf G_i=\mathsf P,\ 1\le i\le k\}$, where $\mathsf G_i$ denotes the $i$-th marginal of $\mathsf G$.
  Thus, if $(X_1,\ldots,X_k)\sim\mathsf G\in\Pi_k(\mathsf P)$, then $X_i\sim\mathsf P$ for every $1\le i\le k$, while the variables need not be independent.
  Since $H_k=\{(z_1,\ldots,z_k)\in[0,1]^k:\sum_{i=1}^k z_i\ge1\}$, we have $\mathsf G(H_k)=\Prob_{\mathsf G}(\sum_{i=1}^k X_i\ge1)$.
  Define $T_k(\mathsf P)=\max_{\mathsf G\in\Pi_k(\mathsf P)}\mathsf G(H_k)$.
  The independent coupling gives $F_k(\mathsf P)=\mathsf P^{\otimes k}(H_k)\le T_k(\mathsf P)$.
The marginal constraints are preserved by weak convergence, so $\Pi_k(\mathsf P)$ is a closed subset of $\Pcal([0,1]^k)$ and is compact by \cref{fact:weak-convergence}\textup{\ref{it:weak-compactness}}.
By \cref{thm:portmanteau}\textup{(iii)}, the function $\mathsf G\mapsto\mathsf G(H_k)$ is upper semicontinuous, so the maximum defining $T_k(\mathsf P)$ exists.
The product law is an admissible coupling, and every admissible coupling has expected sum $km(\mathsf P)$.
Thus, by Markov's inequality, we have $F_k(\mathsf P)\le T_k(\mathsf P)\le km(\mathsf P)$.

Kellerer~\cite[Theorem~2.6 and Proposition~1.31]{Kellerer} established the following duality theorem.

\begin{theorem}[Kellerer's duality theorem]\label{thm:kellerer-duality}
  Let $\mathsf P\in\Pcal([0,1])$, and let $h:[0,1]^k\to\mathbb R$ be bounded and upper semicontinuous.
  Then
  \[
    \sup_{\mathsf G\in\Pi_k(\mathsf P)}\int h\dd\mathsf G
    =
    \inf\left\{
    \sum_{i=1}^k\int\varphi_i\dd\mathsf P:
    \begin{array}{l}
      \varphi_i:[0,1]\to\mathbb R\text{ are upper semicontinuous}, \\
      \varphi_i\text{ are integrable},                             \\
      \sum_{i=1}^k\varphi_i(z_i)\ge h(z)\text{ for every }z\in[0,1]^k
    \end{array}
    \right\}.
  \]
\end{theorem}

In the duality argument, a \emph{covering function} is a Borel map $w:[0,1]\to[0,1]$ such that $\sum_{i=1}^k w(z_i)\ge1$ for every $z\in H_k$.
A covering function assigns weights to coordinate values so that every successful tuple has total weight at least one.

\begin{lemma}\label{lem:duality}
  For every $\mathsf P\in\Pcal([0,1])$,
  \[
    T_k(\mathsf P)
    =\inf\left\{k\int w\dd\mathsf P:
    \begin{array}{l}
      w:[0,1]\to[0,1]\text{ is Borel measurable}, \\
      \sum_{i=1}^k w(z_i)\ge1\text{ for every }z\in H_k
    \end{array}\right\}.
  \]
\end{lemma}

\begin{proof}
  Applying \cref{thm:kellerer-duality} with $h=\1_{H_k}$, we obtain
  \[
    T_k(\mathsf P)
    =
    \inf\left\{
    \sum_{i=1}^k\int\varphi_i\dd\mathsf P:
    \sum_{i=1}^k\varphi_i(z_i)\ge \1_{H_k}(z)
    \text{ for every }z\in[0,1]^k
    \right\}.
  \]
  Here the infimum is over upper semicontinuous, $\mathsf P$-integrable functions $\varphi_i:[0,1]\to\mathbb R$, $1\le i\le k$.
  For any admissible family $\varphi_1,\ldots,\varphi_k$, set $w=k^{-1}\sum_i\varphi_i$.
  Averaging $\sum_i\varphi_i(z_i)\ge\1_{H_k}(z)$ over all coordinate permutations, we obtain $\sum_iw(z_i)\ge\1_{H_k}(z)$ for every $z\in[0,1]^k$.
  Taking $z=(t,\ldots,t)$ shows $kw(t)\ge0$, so $w\ge0$.
  Replacing $w$ by $w\wedge1$ preserves the covering condition and does not increase its integral.
  \[
    \inf_{\text{covers}} k\int w\dd\mathsf P
    \le
    \inf_{\text{dual feasible }(\varphi_i)}
    \sum_{i=1}^k \int \varphi_i\dd\mathsf P
    =
    T_k(\mathsf P).
  \]

  Conversely, let $w:[0,1]\to[0,1]$ be any Borel measurable function satisfying $\sum_{i=1}^k w(z_i)\ge1$ for every $z\in H_k$.
  For every $\mathsf G\in\Pi_k(\mathsf P)$, we have
  \[
    \mathsf G(H_k)
    \le\int\sum_{i=1}^k w(z_i)\dd\mathsf G(z)
    =k\int w\dd\mathsf P.
  \]
  Taking the maximum over $\mathsf G\in\Pi_k(\mathsf P)$ and then the infimum over all such $w$ gives $T_k(\mathsf P)\le\inf_w k\int w\dd\mathsf P$, completing the proof.
\end{proof}

\begin{proposition}\label{prop:normalization}
  Fix $k\ge2$ and $0<x<1/k$.
  Let $\mathsf P$ be a nonnegative Borel law maximizing $F_k$ under $m(\mathsf P)\le x$, and set $a=kx$.
  Then the following statements hold.
  \begin{enumerate}[label=\textup{(\roman*)},leftmargin=*,itemsep=0pt]
    \item\label{it:normalization-support} The law $\mathsf P$ is supported on $[0,1]$ and has mean $x$.
    \item\label{it:normalization-mixture} There is a law $\mathsf Q\in\Pcal([0,1])$ of mean $1/k$ such that $\mathsf P=(1-a)\delta_0+a\mathsf Q$.
    In particular, $\mathsf P(\{0\})\ge1-a$.
    \item\label{it:normalization-coupling} The law $\mathsf Q$ admits an exchangeable coupling whose sum is one almost surely.
  \end{enumerate}
\end{proposition}

\begin{proof}
  First suppose that $\mathsf P$ is supported on $[0,1]$.
  {If $m(\mathsf P)<x$, then take $\theta=(x-m(\mathsf P))/(1-m(\mathsf P))\in(0,1)$.
  Then $m(\mathsf P_\theta^+)=x$.
  Since $F_k(\mathsf P)\le kx<1$ and $0<\theta<1$, it follows from $1-F_k(\mathsf P_\theta^+)=(1-\theta)^k[1-F_k(\mathsf P)]$ that $F_k(\mathsf P_\theta^+)>F_k(\mathsf P)$, a contradiction.
  Thus, $m(\mathsf P)=x$.}

  Suppose that $T_k(\mathsf P)<kx$.
  By \cref{lem:duality}, there exists a Borel measurable map $w:[0,1]\to[0,1]$ such that $\int w\dd\mathsf P<x$ and
  \[
    \sum_{i=1}^k w(z_i)\ge1
    \text{ for every }(z_1,\ldots,z_k)\in[0,1]^k
    \text{ with }\sum_{i=1}^k z_i\ge1.
  \]
  For independent $X_1,\ldots,X_k$ with law $\mathsf P$, the variables $w(X_i)$ are independent with the common law $\law(w(X_1))$, whose mean is less than $x$.
  The covering condition implies that $\sum_iX_i\ge1$ forces $\sum_iw(X_i)\ge1$, so this new law has success probability at least $F_k(\mathsf P)$.
  This new law is therefore also maximizing, but its mean is strictly less than $x$, a contradiction.
  Therefore, $T_k(\mathsf P)=kx$.

  Average an optimal coupling over all coordinate permutations, and denote the resulting coordinates by $\widetilde X_1,\ldots,\widetilde X_k$.
  The averaged coupling is exchangeable and still has marginal law $\mathsf P$ in every coordinate.
  Coordinate permutations preserve the sum, so the averaged coupling remains optimal.
  Writing $\widetilde S=\sum_{i=1}^k\widetilde X_i$, we have $\E\widetilde S=kx=T_k(\mathsf P)=\Prob(\widetilde S\ge1)$, while
  \[
    \E\widetilde S-\Prob(\widetilde S\ge1)
    =\E[\widetilde S\1_{\{\widetilde S<1\}}]
    +\E[(\widetilde S-1)\1_{\{\widetilde S\ge1\}}].
  \]
  Both terms on the right are nonnegative and therefore vanish.
  Thus, $\widetilde S\in\{0,1\}$ almost surely.
  On $\{\widetilde S=0\}$ every coordinate is zero.
  The event $\{\widetilde S=1\}$ has probability $kx=a$, and exchangeability makes its conditional marginals equal to a law $\mathsf Q$ of mean $1/k$.
  Decomposing a marginal over $\{\widetilde S=0\}$ and $\{\widetilde S=1\}$, we obtain $\mathsf P=(1-a)\delta_0+a\mathsf Q$.
  The conditional joint law on $\{\widetilde S=1\}$ is the required constant-sum coupling.
  The mixture formula also gives $\mathsf P(\{0\})=1-a+a\mathsf Q(\{0\})\ge1-a$.

  For a general maximizing law, let $X$ have law $\mathsf P$.
  Truncation at one preserves success, so the law of $X\wedge1$ is maximizing and, by the bounded case, has mean $x$.
  Since $x=\E(X\wedge1)\le\E X\le x$, we have $\E(X-1)_+=0$.
  Thus, $\mathsf P$ is supported on $[0,1]$, and all the conclusions follow from the bounded case.
\end{proof}

\subsection{Lower bounds for the failure probability}\label{sec:cap}

For independent bounded nonnegative random variables $X_1,\ldots,X_j$, we derive lower bounds for the failure probability, that is, $\Prob(X_1+\cdots+X_j<1)$.
We also obtain an explicit bound when $\E X_1=\cdots=\E X_j$.
In \cref{sec:envelope}, \cref{lem:gamma} will be applied with common mean $1/k$ to bound $\Prob(X_1+\cdots+X_j<1)$ in terms of $\E(X_1+\cdots+X_j-1)_+$.

For independent bounded nonnegative variables $X_1,\ldots,X_j$, not necessarily identically distributed, write $S=\sum_{i=1}^jX_i$ and $p=\Prob(S<1)$.
Whenever $p>0$, write $r=1-\E[S\mid S<1]$ for the \emph{conditional deficit} and $B=\E S-\E[S\mid S<1]$ for the difference between the unconditional mean and the mean conditional on failure.
The word \emph{coordinate} refers to one of the random variables $X_i$.
When all $X_i$ have the same mean, we write $m=\E X_i$ for every $i$.
We continue to use $m(\mathsf P)$ for the mean of a law.
Assume $p>0$.
Then $0<r\le1$, and
$(1-S)_+=(1-S)\1_{\{S<1\}}=1-S+(S-1)_+$ implies
$pr=\E(1-S)_+=1-\E S+\E(S-1)_+$.
Thus $r$ is the average deficit below one conditional on failure.
By the definition of $B$, we have $r=1-\E S+B$.

We first prove \cref{lem:safe-cap}, which provides a truncation that does not decrease the conditional deficit $r$ for finitely supported variables.
Each truncation replaces one coordinate $X_i$ by $X_i\wedge \tau$, where $\tau\ge0$ is a constant chosen from the current distributions, not from the sampled values.
Since the new variable depends only on $X_i$, independence is preserved.
Truncating only $X_i$ can destroy equality of the coordinate means, so \cref{lem:safe-cap,lem:product} do not assume equal means.

\begin{lemma}\label{lem:safe-cap}
  Let $X_1,\ldots,X_j$ be independent nonnegative variables with finite support, sum $S$, and $p=\Prob(S<1)>0$.
  Suppose that $\max_{1\le i\le j}|\supp X_i|\ge2$.
  Then there exist $i\in\{1,\ldots,j\}$ and $\tau\in\supp X_i$ with $\tau<\max\supp X_i$ such that the following statements hold after replacing $X_i$ by $X_i'=X_i\wedge\tau$ and leaving all other coordinates unchanged.
  Write $S'=\sum_hX_h'$, $p'=\Prob(S'<1)$, $r'=1-\E[S'\mid S'<1]$, $B'=\E S'-\E[S'\mid S'<1]$, and $\zeta=B-B'$.
  \begin{enumerate}[label=\textup{(\roman*)},leftmargin=*,itemsep=0pt]
    \item\label{it:safe-cap-deficit} We have $r'\ge r$, $|\supp X_i'|<|\supp X_i|$, and $\sum_{h=1}^j|\supp X_h'|<\sum_{h=1}^j|\supp X_h|$.
    \item\label{it:safe-cap-loss} We have $p'\ge p>0$ and
    \[
      r'-r=\E(X_i-\tau)_+-\zeta,\text{ }
      p'(r-\zeta)\le pr
      \text{ and }0\le\zeta\le\E(X_i-\tau)_+.
    \]
  \end{enumerate}
\end{lemma}

\begin{proof}
  First consider any index $i$ and cap $\tau\ge0$, using the primed notation in the statement.
  Then $S'=S-X_i+(X_i\wedge\tau)$.
  Since $S'\le S$, we have $p'\ge p>0$, so these parameters, namely $r'$, $B'$, and $\zeta$, are defined.
  The loss of mean is
  $\E S-\E S'=\E(X_i-\tau)_+$.
  For each sample, the increase in $(1-S)_+$ is zero unless
  $X_i>\tau$ and $S-X_i+\tau<1$.
  When both conditions hold, the increase is at most $X_i-\tau$.
  By independence of $X_i$ and $S-X_i$, we therefore have
  \[
    \begin{aligned}
      p'r'-pr
       & =\E[(1-S')_+-(1-S)_+]                    \\
       & \le\E[(X_i-\tau)_+\1_{\{S-X_i+\tau<1\}}] \\
       & =\E(X_i-\tau)_+\Prob(S-X_i+\tau<1)       \\
       & \le p'\E(X_i-\tau)_+.
    \end{aligned}
  \]
  The last inequality follows from
  $\{S-X_i+\tau<1\}\subseteq\{S'<1\}$, since $X_i\wedge \tau\le \tau$.
  Using $r=1-\E S+B$ and $r'=1-\E S'+B'$, we obtain
  \begin{equation}\label{eq:cap-basic}
    r'-r=\E(X_i-\tau)_+-\zeta
    \text{ and }
    p'(r-\zeta)\le pr.
  \end{equation}
  Since $p'\ge p>0$, the second inequality implies
  $\zeta\ge r(1-p/p')\ge0$.
  If $r'\ge r$, then the first identity also implies
  $\zeta\le\E(X_i-\tau)_+$.
  Thus a truncation that does not decrease $r$ satisfies
  $0\le\zeta\le\E(X_i-\tau)_+$.

  Recall that, for a finitely supported random variable $X$, $\supp X=\{t:\Prob(X=t)>0\}$ and $|\supp X|$ is the number of points in this set.
  In particular, $|\supp X|=1$ means that $X$ is almost surely constant.
  Since
  $\sum_{h=1}^j\Cov(S,X_h\mid S<1)=\Var(S\mid S<1)\ge0$
  and $\Cov(S,X_h\mid S<1)=0$ whenever $|\supp X_h|=1$,
  there is an index $i\in\{1,\ldots,j\}$ such that
  $|\supp X_i|\ge2$ and $\Cov(S,X_i\mid S<1)\ge0$.
  Fix one such index $i$.

  Write $T=S-X_i$ and $\supp X_i=\{x_1,\ldots,x_N\}$,
  where $0\le x_1<\cdots<x_N$ and $N\ge2$.
  For $t\ge0$, define
  $K(t)=\E[(1-r-T-t)\1_{\{T+t<1\}}]$,
  where the expectation is over $T$.
  By independence of $T$ and $X_i$ and the equality $r=1-\E[S\mid S<1]$, we have
  \[
    \begin{aligned}
      \E K(X_i)
       & =
      \E\left[
        \E_T\left[
          (1-r-T-X_i)\1_{\{T+X_i<1\}}
          \Given X_i
          \right]
      \right]                        \\
       & =
      \E\left[
        (1-r-T-X_i)\1_{\{T+X_i<1\}}
      \right]                        \\
       & =
      \E\left[
        (1-r-S)\1_{\{S<1\}}
      \right]                        \\
       & =
      \Prob(S<1)
      \left(1-r-\E[S\mid S<1]\right) \\
       & =0.
    \end{aligned}
  \]

  Set $A_0=0$ and
  $A_q=\E[K(X_i)\1_{\{X_i\le x_q\}}]$ for $1\le q\le N$.
  Then $A_N=\E K(X_i)=0$.
  By summation by parts and the choice of $i$, we have
  \[
    \sum_{q=1}^{N-1}(x_{q+1}-x_q)A_q
    =-\E[X_iK(X_i)]
    =p\Cov(S,X_i\mid S<1)\ge0.
  \]
  Every coefficient $x_{q+1}-x_q$ is positive.
  Hence there is a $q\in\{1,\ldots,N-1\}$ such that $A_q\ge0$.
  Choose the least such $q$.
  Then $A_{q-1}\le0$, and
  $K(x_q)=(A_q-A_{q-1})/\Prob(X_i=x_q)\ge0$.

  Set $\tau=x_q$, so that $\tau\in\supp X_i$ and
  $\tau<\max\supp X_i$.
  Define $X_i'=X_i\wedge\tau$ and $X_h'=X_h$ for $h\ne i$.
  The variable $X_i'$ remains independent of $T$.
  Since $S'=\sum_{h=1}^jX_h'=T+X_i'$, we obtain
  \[
    \begin{aligned}
      p'(r'-r)
       & =p'r'-p'r                                 \\
       & =\E[(1-S')\1_{\{S'<1\}}]
      -\E[r\1_{\{S'<1\}}]                          \\
       & =\E[(1-r-S')\1_{\{S'<1\}}]                \\
       & =\E[(1-r-T-X_i')\1_{\{T+X_i'<1\}}]        \\
       & =\E K(X_i')                               \\
       & =\E\left[K(X_i)\1_{\{X_i\le x_q\}}\right]
      +K(x_q)\Prob(X_i>x_q)                        \\
       & =A_q+\Prob(X_i>x_q)K(x_q)                 \\
       & \ge0,
    \end{aligned}
  \]
  Thus $r'\ge r$, since $p'\ge p>0$.
  The truncation leaves each value $x_1,\ldots,x_q$ unchanged
  and replaces every value $x_{q+1},\ldots,x_N$ by $x_q$.
  Therefore $\supp X_i'=\{x_1,\ldots,x_q\}$, so
  $|\supp X_i'|=q<N=|\supp X_i|$.
  For every $h\ne i$, we have $\supp X_h'=\supp X_h$.
  Consequently,
  $\sum_{h=1}^j|\supp X_h'|<\sum_{h=1}^j|\supp X_h|$.
  Together with \eqref{eq:cap-basic}, $p'\ge p>0$ and $r'\ge r$ give \textup{(ii)}.
\end{proof}

We next use \cref{lem:safe-cap} to prove a lower bound for $p$.

\begin{lemma}\label{lem:product}
  Let $j\ge1$, and let $X_1,\ldots,X_j$ be independent bounded nonnegative variables with sum $S$.
  Suppose that $p=\Prob(S<1)>0$, and write $r=1-\E[S\mid S<1]$.
  If $r>\max_{1\le i\le j}\E X_i$, then there are $z_i\in[0,\E X_i]$, $1\le i\le j$, such that
  $\sum_{i=1}^jz_i=\E S-\E[S\mid S<1]$ and
  $
    p\ge\prod_{i=1}^j\left(1-\frac{z_i}{r}\right).
  $
\end{lemma}

\begin{proof}
  First suppose that every coordinate has finite support.
  If $|\supp X_i|=1$ for every $i$, then $S$ is constant.
  Since $p>0$, we have $p=1$ and $B=0$, so the conclusion holds
  with $z_i=0$ for every $i$.

  Otherwise, set $X_i^{(0)}=X_i$.
  For the variables after $n$ truncations, define
  $S^{(n)}=\sum_{i=1}^jX_i^{(n)}$,
  $p^{(n)}=\Prob(S^{(n)}<1)$,
  $r^{(n)}=1-\E[S^{(n)}\mid S^{(n)}<1]$,
  $B^{(n)}=\E S^{(n)}-\E[S^{(n)}\mid S^{(n)}<1]$, and
  $L_n=\sum_{i=1}^j|\supp X_i^{(n)}|$.
  Initially, $S^{(0)}=S$, $p^{(0)}=p$, $r^{(0)}=r$,
  $B^{(0)}=B$, and $L_0>j$.

  Whenever $L_n>j$, some coordinate has at least two support points.
  By \cref{lem:safe-cap}, we can choose an index $i_n$ and a point
  $\tau_n\in\supp X_{i_n}^{(n)}$ with
  $\tau_n<\max\supp X_{i_n}^{(n)}$ such that, setting
  $X_{i_n}^{(n+1)}=X_{i_n}^{(n)}\wedge\tau_n$ and
  $X_h^{(n+1)}=X_h^{(n)}$ for $h\ne i_n$, we have
  $r^{(n+1)}\ge r^{(n)}$ and $L_{n+1}\le L_n-1$.
  Each truncation preserves independence, nonnegativity, and finite
  support. Since $S^{(n+1)}\le S^{(n)}$, we also have
  $p^{(n+1)}\ge p^{(n)}\ge p>0$.
  Thus, the conditional parameters remain defined, and the current variables satisfy the hypotheses of \cref{lem:safe-cap} whenever $L_n>j$.
  Moreover, $r^{(n)}\ge r>0$ throughout the procedure.

  Since $L_n\ge j$, the procedure stops after some
  $N\le L_0-j$ steps with $L_N=j$.
  At each step, the index $i_n$ and threshold $\tau_n$ are chosen from the current distributions, not from the sampled values.
    Once these choices and the initial distributions are fixed, $N$ is a fixed integer and does not depend on the sample.
  Each support is nonempty, so $|\supp X_i^{(N)}|=1$ for every $i$.
  Hence $S^{(N)}$ is constant.
  Since $p^{(N)}>0$, this constant is less than one, so
  $p^{(N)}=1$ and
  $\E[S^{(N)}\mid S^{(N)}<1]=\E S^{(N)}$.
  Therefore $B^{(N)}=0$.

  For $0\le n<N$, set $\zeta_n=B^{(n)}-B^{(n+1)}$.
  Since $p^{(n+1)}\ge p^{(n)}>0$ and $r^{(n+1)}\ge r^{(n)}$, \cref{lem:safe-cap} gives
  \[
    0\le r^{(n)}\left(1-\frac{p^{(n)}}{p^{(n+1)}}\right)
    \le\zeta_n
    \le\E X_{i_n}^{(n)}-\E X_{i_n}^{(n+1)}.
  \]
  Moreover, since $r^{(n)}\ge r>0$ and $\zeta_n\ge0$, \cref{lem:safe-cap} yields
  \[
    p^{(n)}
    \ge p^{(n+1)}\left(1-\frac{\zeta_n}{r^{(n)}}\right)
    \ge p^{(n+1)}\left(1-\frac{\zeta_n}{r}\right).
  \]

  For each $i\in\{1,\ldots,j\}$, set
  $z_i=\sum_{0\le n<N:\,i_n=i}\zeta_n$.
  Since $X_i^{(n+1)}=X_i^{(n)}$ whenever $i_n\ne i$, summing the mean losses at the steps with $i_n=i$ gives
  \[
    0\le z_i\le\E X_i-\E X_i^{(N)}\le\E X_i.
  \]
  Each $\zeta_n$ occurs in exactly one $z_i$, so
  \[
    \sum_{i=1}^jz_i=\sum_{n=0}^{N-1}\zeta_n=B^{(0)}-B^{(N)}=B.
  \]

  For each $i$, we have
  $\sum_{0\le n<N:\,i_n=i}\zeta_n/r=z_i/r\le\E X_i/r<1$.
  Since every $\zeta_n$ is nonnegative, all factors
  $1-\zeta_n/r$ are positive.
  Iterating $p^{(n)}\ge p^{(n+1)}(1-\zeta_n/r)$ for $0\le n<N$, using $p^{(N)}=1$, and grouping the factors by coordinate, we obtain
  \[
    \begin{aligned}
      p=p^{(0)}
       & \ge p^{(N)}
      \prod_{n=0}^{N-1}\left(1-\frac{\zeta_n}{r}\right)=\prod_{i=1}^j
      \prod_{\substack{0\le n<N        \\i_n=i}}
      \left(1-\frac{\zeta_n}{r}\right) \\
       & \ge\prod_{i=1}^j
      \left(1-\sum_{\substack{0\le n<N \\i_n=i}}
      \frac{\zeta_n}{r}\right)=\prod_{i=1}^j\left(1-\frac{z_i}{r}\right).
    \end{aligned}
  \]
  The last inequality applies
  $\prod_h(1-u_h)\ge1-\sum_hu_h$
  to each coordinate separately, with $u_h\ge0$ and
  $\sum_hu_h<1$.
  For a coordinate that is never truncated, the corresponding
  sum is zero and the product is one.
  This proves the finite-support case.

  Now consider arbitrary bounded nonnegative coordinates.
  For each integer $h\ge1$, set $X_i^{(h)}=2^{-h}\lceil2^{h}X_i\rceil$ and
  $S_{h}=\sum_{i=1}^jX_i^{(h)}$.
  Here $h$ is a new approximation index, distinct from the fixed number of truncations $N$ above.
  These rounded variables remain independent and have finite support.
  Moreover, $X_i^{(h)}\downarrow X_i$ and
  $0\le X_i^{(h)}-X_i<2^{-h}$.
  It follows that $\{S_{h}<1\}\uparrow\{S<1\}$ and
  $(1-S_{h})_+\uparrow(1-S)_+$.
  In particular, a sample with $S=1$ belongs to none of the rounded
  failure events.

  Write $p_{h}=\Prob(S_{h}<1)$.
  By monotone convergence, we have $p_{h}\uparrow p$ and
  $\E(1-S_{h})_+\to pr$.
  For all sufficiently large $h$, we therefore have $p_{h}>0$ and
  $r_{h}=\E(1-S_{h})_+/p_{h}\to r$.
  The bound on the rounding error also implies
  $\E X_i^{(h)}\to\E X_i$ for each $i$.
  Since $r>\max_i\E X_i$, we have
  $r_{h}>\max_i\E X_i^{(h)}$ for all sufficiently large $h$.

For each sufficiently large $h$, there exist numbers
$z_i^{(h)}\in[0,\E X_i^{(h)}]$, $1\le i\le j$, such that
$\sum_{i=1}^j z_i^{(h)}=\E S_h-1+r_h$ and
$p_h\ge\prod_{i=1}^j\left(1-z_i^{(h)}/r_h\right)$.
  These vectors are uniformly bounded, so a subsequence converges
  to a vector $(z_i)_{i=1}^j$.
  Passing to the limit along this subsequence in $0\le z_i^{(h)}\le\E X_i^{(h)}$ and $\sum_i z_i^{(h)}=\E S_{h}-1+r_{h}$, we obtain $0\le z_i\le\E X_i$ and $\sum_i z_i=\E S-1+r=B$.
  Since $p_{h}\to p$ and $r_{h}\to r>0$, we also obtain $p\ge\prod_i(1-z_i/r)$.
  This proves the result for the original variables.
\end{proof}

When all coordinate means are equal, minimizing the product yields
the following explicit bound.

\begin{lemma}\label{lem:gamma}
  Under the hypotheses and notation of \cref{lem:product}, suppose that $\E X_i=m>0$ for $1\le i\le j$ and $jm\le1$.
  For $\ell\in\{0,\ldots,j-1\}$, if $1-(j-\ell)m\le r\le1-(j-\ell-1)m$, then
  \[
    p\ge\left(1-\frac{m}{r}\right)^\ell
    \frac{1-(j-\ell)m}{r}.
  \]
\end{lemma}

\begin{proof}
  By \cref{lem:product}, choose numbers $z_i\in[0,m]$, $1\le i\le j$, satisfying $\sum_i z_i=B=r-1+jm$ and $p\ge\prod_i(1-z_i/r)$.
  If two coordinates satisfy $0<u\le v<m$, then set $\varepsilon=\min\{u,m-v\}>0$ and replace them by $u-\varepsilon,v+\varepsilon$.
  Their sum is unchanged, and
  \[
    \left(1-\frac{u-\varepsilon}{r}\right)
    \left(1-\frac{v+\varepsilon}{r}\right)
    =\left(1-\frac ur\right)\left(1-\frac vr\right)
    -\frac{\varepsilon(v-u+\varepsilon)}{r^2}.
  \]
  Since $r>m$, all other factors are positive, so the product does not increase.
  Each step moves at least one coordinate in $(0,m)$ to $0$ or $m$.
  After finitely many steps, at most one coordinate lies in $(0,m)$.
  The hypothesis on $r$ is equivalent to $\ell m\le B\le(\ell+1)m$.
  Thus the resulting coordinates, up to permutation, are $\ell$ copies of $m$, one copy of $B-\ell m$, and zeros.
  Consequently,
  \[
    p\ge\left(1-\frac mr\right)^\ell
    \left(1-\frac{B-\ell m}{r}\right)
    =\left(1-\frac mr\right)^\ell\frac{1-(j-\ell)m}{r}.
  \]
  The same result holds at either interval endpoint, where the remaining coordinate is $0$ or $m$.
\end{proof}

\section[Linear interpolation and node comparisons]{Linear interpolation and node comparisons}
\label{sec:low-dimensions}

For $k\ge2$, \cref{prop:mean-transfer} reduces the upper bound in \cref{thm:iid-tail} to $x=\beta_k$.
  Let $Y_1,\ldots,Y_k$ be independent with a common law $\mathsf Q\in\Pcal([0,1])$ of mean $1/k$.
  For $1\le j\le k$, write $R_j=\sum_{i=1}^jY_i$ and $f_j=\Prob(R_j\ge1)$, and set $w=k\E(R_k-1)_+$.
  With the weights $b_j$ defined in \cref{sec:common-bounds}, we construct functions $A$ and $H$ such that $\sum_{j=1}^{k-1}b_jf_j\le A(w)$ and $1-f_k=\Prob(R_k<1)\ge H(w)$.
The function $A$ is obtained by linear interpolation at $W_s=s(1-1/s)^s$, $1\le s\le k$, while $H$ is strictly concave between consecutive nodes.
Thus, $A\le H$ at two consecutive nodes implies $A<H$ on the open interval between them.

The number $W_s$ is the largest value of $s\E(U_1+\cdots+U_s-1)_+$ for independent $[0,1]$-valued random variables $U_1,\ldots,U_s$ with $\E U_i=1/s$.
By convexity in each variable, this maximum is attained by independent Bernoulli random variables with mean $1/s$, each taking the value $1$ with probability $1/s$ and $0$ otherwise.
For $N_s\sim Bin(s,1/s)$, we have $\E N_s=1$, and hence
\[
  s\E(N_s-1)_+=s\Prob(N_s=0)
  =s\left(1-\frac1s\right)^s=W_s.
\]

At $x=\beta_k$, we prove $A(W_s)<H(W_s)$ for $2\le s\le\min\{4,k-1\}$, with equality at $W_1$ and $W_k$.
For $2\le k\le5$, these comparisons cover every interior node and yield the upper bound in \cref{thm:iid-tail}.
The comparison does not hold at all interior nodes in general for larger $k$.
For example, when $k=6$ and $x=\beta_6$, we have $A(W_5)>H(W_5)$.
We therefore treat the remaining case $k\ge6$ using further properties of maximizing distributions in Sections~\ref{sec:contact-shortfall}--\ref{sec:exclusion}.
We also state the upper bound for $c_k$ used in the node comparisons.

\subsection[Retained sums and comparison distributions]{Retained sums and comparison distributions}
\label{sec:common-bounds}
\label{sec:comparison-nodes}

We first express $F_k(\mathsf P)$ in terms of probabilities for partial sums.
We then construct comparison distributions for sums of $k$ independent random variables, each with mean $1/k$, and interpolate the associated probabilities and expected excesses.

For the preliminary estimates, let $0<x<1/k$, let $\mathsf Q\in\Pcal([0,1])$ have mean $1/k$, and set $a=kx$ and $\mathsf P=(1-a)\delta_0+a\mathsf Q$.
By \cref{lem:attainment,prop:normalization}, a maximizing distribution under the mean bound $x$ exists and can be written in this form.
Let $Y_1,\ldots,Y_k$ be independent with law $\mathsf Q$.
Let $I_1,\ldots,I_k$ be independent Bernoulli variables with mean $a$, independent of $Y_1,\ldots,Y_k$, and set $X_i=I_iY_i$ for $1\le i\le k$.
Then $X_1,\ldots,X_k$ are independent with common law $\mathsf P=(1-a)\delta_0+a\mathsf Q$.
We call $Y_i$ \emph{retained} when $I_i=1$.
For $1\le j\le k$, write $R_j=\sum_{i=1}^jY_i$, $f_j=\Prob(R_j\ge1)$ and $e_j=k\E(R_j-1)_+$.
In particular, $w=e_k$.
As shown in \cref{lem:critical}, $0\le w\le W_k$.

For $1\le j<k$, define
\[
  b_j=\binom{k}{j}\left(\frac{1-a}{a}\right)^{k-j}.
\]
\begin{fact}\label{fact:retained-sums}
  Let $k\ge2$ and $0<x<1/k$, and let $\mathsf Q\in\Pcal([0,1])$ have mean $1/k$.
  Set $a=kx$ and $\mathsf P=(1-a)\delta_0+a\mathsf Q$, and use the notation $R_j,f_j,e_j,w,b_j$ above.
  For independent $X_1,\ldots,X_k$ with law $\mathsf P$, set $S_k=\sum_{i=1}^kX_i$.
  Then the following statements hold.
  \begin{enumerate}[label=\textup{(\roman*)},leftmargin=*,itemsep=0pt]
    \item\label{it:retained-sums-mixture} We have
    \[
      \begin{aligned}
        F_k(\mathsf P)     & =a^k\left(f_k+\sum_{j<k}b_jf_j\right),         \\
        F_{k-1}(\mathsf P) & =a^{k-1}\sum_{j=1}^{k-1}\binom{k-1}{j}
        \left(\frac{1-a}{a}\right)^{k-1-j}f_j,                              \\
        \E(S_k-1)_+        & =\frac{a^k}{k}\left(w+\sum_{j<k}b_je_j\right).
      \end{aligned}
    \]
    \item\label{it:retained-sums-budget} We have $e_1=0$ and $e_{j+1}\ge e_j+f_j$ for $1\le j<k$.
    In particular, $\sum_{j<k}f_j\le w$.
    \item\label{it:retained-sums-maximality} If $\mathsf P$ is maximizing under the mean bound $x$, then $\sum_{j<k}b_jf_j\ge1-f_k$.
  \end{enumerate}
\end{fact}

\begin{proof}
  For $1\le j\le k$, the probability that exactly $j$ samples are retained is $\binom{k}{j}a^j(1-a)^{k-j}$.
  Conditional on exactly $j$ samples being retained, $S_k$ has the same distribution as $R_j$, because the $Y_i$'s are i.i.d. and independent of the $I_i$'s.
  When no samples are retained, $S_k=0$, so summing over $1\le j\le k$ gives
  \begin{equation}\label{eq:thinning}
    \begin{aligned}
      F_k(\mathsf P)
       & =\sum_{j=1}^k\binom{k}{j}a^j(1-a)^{k-j}f_j
      =a^k\left(f_k+\sum_{j<k}b_jf_j\right),            \\
      \E(S_k-1)_+
       & =\frac{a^k}{k}\left(w+\sum_{j<k}b_je_j\right).
    \end{aligned}
  \end{equation}
  The same conditioning with $k-1$ samples gives the stated formula for $F_{k-1}(\mathsf P)$.

  Since $Y_1\le1$, we have $e_1=0$.
  For $1\le j<k$, the increment $(R_{j+1}-1)_+-(R_j-1)_+$ equals $Y_{j+1}$ on $\{R_j\ge1\}$ and is nonnegative otherwise.
  Independence and $\E Y_{j+1}=1/k$ therefore give $e_{j+1}-e_j\ge k\E[Y_{j+1}\1_{\{R_j\ge1\}}]=f_j$.
  Consequently,
  \[
    e_1=0\text{ and }e_{j+1}\ge e_j+f_j
  \]
  for $1\le j\le k-1$.
  In particular, summing yields $\sum_{j<k}f_j\le w$.

  If $\mathsf P$ is maximizing, then $F_k(\mathsf P)\ge F_k(\mathsf C_{k,x})=a^k$, so \eqref{eq:thinning} implies $\sum_{j<k}b_jf_j\ge1-f_k$.
\end{proof}

At $x=\beta_k$, we will construct a piecewise affine upper bound $A(w)$ for $\sum_{j<k}b_jf_j$ and a lower bound $H(w)$ for $1-f_k$.
Whenever $A(w)<H(w)$, we obtain
\[
  \sum_{j<k}b_jf_j\le A(w)<H(w)\le1-f_k,
\]
contradicting \cref{fact:retained-sums}\textup{(iii)}.

For $\mathsf C_{k,x}$, every retained sample equals $1/k$, whereas for $\mathsf B_x$ the retained samples are Bernoulli variables with mean $1/k$.
For each $1\le s\le k$, the first $k-s$ random variables are constant and equal to $1/k$.
The remaining $s$ random variables take the value $s/k$ with probability $1/s$ and zero otherwise.
Their sum has the distribution of $(k-s)/k+(s/k)N_s$, and hence has expected excess $W_s/k$ above one.
The first nonzero value among the last $s$ variables brings the sum exactly to one, and only further nonzero values create positive excess.
For $1<s<k$, these random variables are not identically distributed and are used only for comparison.

For $1\le s\le k$, take $k$ independent Bernoulli variables $J_1,\ldots,J_k$ with mean $1/s$ and set
\[
  Y_i^{(s)}=
  \begin{cases}
    1/k,      & 1\le i\le k-s, \\
    (s/k)J_i, & k-s<i\le k.
  \end{cases}
\]
Each $Y_i^{(s)}$ has mean $1/k$.
Write $R_j^{(s)}=\sum_{i=1}^jY_i^{(s)}$, $f_j^{(s)}=\Prob(R_j^{(s)}\ge1)$, and $e_j^{(s)}=k\E(R_j^{(s)}-1)_+$.
We have $W_1<W_2<\cdots<W_k$, since $W_1=0$ and differentiation shows that $s(1-1/s)^s$ is strictly increasing for real $s>1$.

We use $W_1,\ldots,W_k$ as interpolation nodes for the pairs $(e_j^{(s)},f_j^{(s)})$.
For each integer $s$ with $1\le s<k$ and every $u\in[W_s,W_{s+1}]$, define, for $1\le j<k$,
\begin{equation}\label{eq:interpolated-arrays}
  \begin{aligned}
    e_j^*(u)
     & =\frac{W_{s+1}-u}{W_{s+1}-W_s}e_j^{(s)}
    +\frac{u-W_s}{W_{s+1}-W_s}e_j^{(s+1)},     \\
    f_j^*(u)
     & =\frac{W_{s+1}-u}{W_{s+1}-W_s}f_j^{(s)}
    +\frac{u-W_s}{W_{s+1}-W_s}f_j^{(s+1)}.
  \end{aligned}
\end{equation}
The formulas agree at every common endpoint, so they define $e_j^*$ and $f_j^*$ on $[0,W_k]$ for every $1\le j<k$.
For every $u\in[0,W_k]$, set $e_k^*(u)=u$.

\begin{fact}\label{fact:comparison-arrays}
  Let $k\ge2$, and use the comparison values $e_j^{(s)},f_j^{(s)}$ and their interpolants $e_j^*,f_j^*$ defined above.
  \begin{enumerate}[label=\textup{(\roman*)},leftmargin=*,itemsep=0pt]
    \item\label{it:comparison-node-arrays} For $1\le s\le k$ and $1\le j\le k$,
    \[
      (e_j^{(s)},f_j^{(s)})=
      \begin{cases}
        (0,0),                                                & j\le k-s, \\
        \bigl(j-k+s(1-1/s)^{j-k+s},\,1-(1-1/s)^{j-k+s}\bigr), & j>k-s.
      \end{cases}
    \]
    Moreover, $e_1^{(s)}=0$, $e_{j+1}^{(s)}=e_j^{(s)}+f_j^{(s)}$ for $1\le j<k$, and $e_k^{(s)}=\sum_{j<k}f_j^{(s)}=W_s$.
    \item\label{it:comparison-interpolated-budget} For $0\le w\le W_k$, we have
    \[
      \begin{aligned}
        e_1^*(w) & =0,\text{ }e_{j+1}^*(w)=e_j^*(w)+f_j^*(w)\quad(1\le j<k), \\
        e_k^*(w) & =w=\sum_{j<k}f_j^*(w).
      \end{aligned}
    \]
  \end{enumerate}
\end{fact}

\begin{proof}
  For $j\le k-s$, we have $R_j^{(s)}=j/k\le(k-s)/k<1$.
  For $j>k-s$, set $N=\sum_{i=k-s+1}^jJ_i$.
  Then $N$ is binomial with parameters $j-k+s$ and $1/s$, and $R_j^{(s)}=(k-s)/k+(s/k)N$.
  Since $R_j^{(s)}-1=(s/k)(N-1)$, we have $f_j^{(s)}=\Prob(N\ge1)=1-\Prob(N=0)$.
  Since $N$ is a nonnegative integer, $(N-1)_+=N-1+\1_{\{N=0\}}$, so $e_j^{(s)}=s\E(N-1)_+=s(\E N-1+\Prob(N=0))$.
  Substituting $\E N=(j-k+s)/s$ and $\Prob(N=0)=(1-1/s)^{j-k+s}$ gives the formula in \textup{(i)}.
  Taking differences $e_{j+1}^{(s)}-e_j^{(s)}$ for $1\le j<k$ and evaluating $e_j^{(s)}$ at $j=k$, we obtain
  \begin{equation}\label{eq:alignment}
    e_{j+1}^{(s)}=e_j^{(s)}+f_j^{(s)}
    \text{ for $1\le j\le k-1$ and }e_k^{(s)}=W_s.
  \end{equation}
  Summing \eqref{eq:alignment} and using $e_1^{(s)}=0$, we obtain $\sum_{j<k}f_j^{(s)}=W_s$ for every integer $s$ with $1\le s\le k$.
  \textup{(ii)} follows by taking the same affine combinations of the relations in \textup{(i)}.
\end{proof}

\subsection[Reduction to node comparisons]{Reduction to node comparisons}
\label{sec:envelope}\label{sec:proof-partial-sum-capacities}\label{sec:structural}

We bound $\sum_{j=1}^{k-1}b_jf_j$ from above and $1-f_k=\Prob(R_k<1)$ from below.
We then use the fact that $A$ is affine and $H$ is strictly concave between consecutive nodes to reduce their comparison to their values at $W_1,\ldots,W_k$.

If $\mathsf P$ is maximizing, then \cref{fact:retained-sums}\textup{(iii)} gives $\sum_{j<k}b_jf_j\ge1-f_k$.
For fixed $k$ and $x$, define
\[
  A(w)=\sum_{j<k}b_jf_j^*(w).
\]
We first show that, whenever $b_1\le\cdots\le b_{k-1}$, we have
\begin{equation}\label{eq:weighted-upper-bound}
  \sum_{j<k}b_jf_j\le\sum_{j<k}b_jf_j^*(w).
\end{equation}
The weights are nondecreasing if $(1-a)/a\le2/(k-1)$, as verified in the proof of \cref{prop:strict-chord}.
At $x=\beta_k$, we have $(1-a)/a=c_k/k\le1/(2k)<2/(k-1)$ by \cref{lem:crossing-parameters}, so the weights are nondecreasing.
We also obtain a lower bound $H(w)$ for $1-f_k$ in \cref{lem:critical}.
When $b_1\le\cdots\le b_{k-1}$, we combine \eqref{eq:weighted-upper-bound}, \cref{lem:critical}, and \cref{fact:retained-sums}\textup{(i)} to obtain
\[
  a^{-k}F_k(\mathsf P)\le1-H(w)+A(w).
\]
By \eqref{eq:interpolated-arrays}, each $f_j^*$ is affine on $[W_s,W_{s+1}]$; hence $A=\sum_{j<k}b_jf_j^*$ is affine there. We will show that $H$ is strictly concave on the same interval.
Consequently, if $A(W_s)\le H(W_s)$ and $A(W_{s+1})\le H(W_{s+1})$, then $A(w)<H(w)$ for every $W_s<w<W_{s+1}$.
Hence, $F_k(\mathsf P)<a^k$ throughout the open interval, excluding a maximizing distribution with such a value of $w$.

We first derive bounds for the individual probabilities $f_j$ in terms of $e_j$, using \cref{lem:gamma}.
We use the following inequality to bound $e_j$ and later to identify the law when $w=W_k$.

\begin{lemma}\label{lem:product-shortfall}
  Let $k\ge2$.
  The following statements hold.
  \begin{enumerate}[label=\textup{(\roman*)},leftmargin=*,itemsep=0pt]
    \item\label{it:product-shortfall-pointwise} For $y_1,\ldots,y_k\in[0,1]$, we have
    \[
      \left(1-\sum_{i=1}^ky_i\right)_+\le\prod_{i=1}^k(1-y_i).
    \]
    The inequality is strict if $0<y_i<1$ for every $1\le i\le k$.
    \item\label{it:product-shortfall-expectation} Let $Y_1,\ldots,Y_j$ be independent $[0,1]$-valued variables of mean $1/k$, where $1\le j\le k$.
    Set $R_j=\sum_{i=1}^jY_i$ and $e_j=k\E(R_j-1)_+$.
    Then
    \[
      \begin{aligned}
        \E(1-R_j)_+ & =\frac{k-j+e_j}{k}\le\left(1-\frac1k\right)^j, \\
        0\le e_j    & \le k\left(1-\frac1k\right)^j-(k-j).
      \end{aligned}
    \]
  \end{enumerate}
\end{lemma}

\begin{proof}
  If $\sum_{i=1}^k y_i\ge1$, then the left-hand side is zero and the right-hand side is nonnegative, so the inequality holds.
  If $\sum_{i=1}^k y_i<1$, then induction using $(1-u)(1-v)=1-u-v+uv\ge1-u-v$ proves the inequality.

  Now suppose that $0<y_i<1$ for every $1\le i\le k$.
  If $\sum_{i=1}^k y_i\ge1$, then $(1-\sum_{i=1}^k y_i)_+=0<\prod_{i=1}^k(1-y_i)$.
  If $\sum_{i=1}^k y_i<1$, then applying the non-strict inequality to $0,y_2,\ldots,y_k$ gives
  \[
    \prod_{i=1}^k(1-y_i)
    \ge(1-y_1)\left(1-\sum_{i=2}^k y_i\right)
    >1-\sum_{i=1}^k y_i,
  \]
  since $y_1\sum_{i=2}^k y_i>0$.

  For \textup{(ii)}, by \textup{(i)} applied with $y_i=Y_i$ for $1\le i\le j$ and $y_i=0$ for $j<i\le k$, we have $(1-R_j)_+\le\prod_{i=1}^j(1-Y_i)$ almost surely.
  Taking expectations and using the independence of $Y_1,\ldots,Y_j$ and their common mean $1/k$, we obtain
  \[
    \frac{k-j+e_j}{k}
    =\E(1-R_j)_+
    \le\E\prod_{i=1}^j(1-Y_i)
    =\prod_{i=1}^j\E(1-Y_i)
    =\left(1-\frac1k\right)^j.
  \]
  The bound on $e_j$ follows by rearranging.
\end{proof}

We introduce the following functions to express the lower bound for $1-f_j$ from \cref{lem:gamma} in terms of $e_j$.
For an integer $\ell$ with $0\le \ell\le k-2$, define $\Psi_\ell$ on $[\ell+1,k]$ as follows.
For each integer $s$ with $\ell+1\le s<k$ and every $v\in[s,s+1]$, define
\[
  \Psi_\ell(v)=s\left(1-\frac1v\right)^{s-\ell}.
\]
For each integer $s$ with $\ell+1\le s\le k-2$, the formulas on $[s,s+1]$ and $[s+1,s+2]$ agree at $v=s+1$, since $s(s/(s+1))^{s-\ell}=(s+1)(s/(s+1))^{s+1-\ell}$.
The derivative of $\Psi_\ell$ is positive on $(s,s+1)$ for every integer $s$ with $\ell+1\le s<k$.
Thus, $\Psi_\ell$ is continuous and strictly increasing, with range $[\ell,k(1-1/k)^{k-\ell}]$.
Therefore, $\Psi_\ell$ has an inverse function $\Psi_\ell^{-1}$ on $[\ell,k(1-1/k)^{k-\ell}]$.
For each integer $s$ with $\ell+1\le s<k$ and every $u\in[\Psi_\ell(s),\Psi_\ell(s+1)]$, we have
\[
  \Psi_\ell^{-1}(u)=\frac{1}{1-(u/s)^{1/(s-\ell)}}.
\]
The denominator is positive because $u\le\Psi_\ell(s+1)<s$.

\begin{fact}\label{fact:retained-shortfall}
  Let $k\ge2$ and $1\le j\le k$.
  Let $Y_1,\ldots,Y_j$ be independent $[0,1]$-valued variables of mean $1/k$.
  Set $R_j=\sum_{i=1}^jY_i$, $f_j=\Prob(R_j\ge1)$, and $e_j=k\E(R_j-1)_+$.
  Suppose that $f_j<1$, and write $r_j=\E[1-R_j\mid R_j<1]$.
  Then the following statements hold.
  \begin{enumerate}[label=\textup{(\roman*)},leftmargin=*,itemsep=0pt]
    \item\label{it:retained-shortfall-mean} We have $k(1-f_j)r_j=k-j+e_j$.
    \item\label{it:retained-shortfall-inversion} If $2\le j\le k$ and $\ell=k-j$, then $\ell+e_j$ belongs to the domain of $\Psi_\ell^{-1}$ and
    \[
      kr_j\le\Psi_\ell^{-1}(\ell+e_j)
      \text{ and }
      1-f_j\ge\frac{\ell+e_j}{\Psi_\ell^{-1}(\ell+e_j)}.
    \]
  \end{enumerate}
\end{fact}

\begin{proof}
  Since $\E R_j=j/k$ and $(1-z)_+=1-z+(z-1)_+$, we have $\E(1-R_j)_+=(k-j+e_j)/k$.
  By the definitions of $f_j$ and $r_j$, we have $(1-f_j)r_j=\E(1-R_j)_+$, and hence
  \begin{equation}\label{eq:retained-shortfall-identity}
    k(1-f_j)r_j=k-j+e_j.
  \end{equation}

  Fix $2\le j\le k$ with $f_j<1$, and set $\ell=k-j$.
  To prove \textup{(ii)}, it suffices by \eqref{eq:retained-shortfall-identity} to show that $kr_j\le\Psi_\ell^{-1}(\ell+e_j)$.
  By \cref{lem:product-shortfall}\textup{(ii)}, $\ell+e_j$ lies in the domain of $\Psi_\ell^{-1}$, and $\Psi_\ell^{-1}(\ell+e_j)\ge \ell+1$.
  If $kr_j\le \ell+1$, then $kr_j\le\Psi_\ell^{-1}(\ell+e_j)$, proving the desired result.

  Suppose that $kr_j>\ell+1$.
  Since $r_j\le1$, choose an integer $\sigma$ with $\ell+1\le\sigma<k$ such that $\sigma\le kr_j\le\sigma+1$.
  Applying \cref{lem:gamma} to $Y_1,\ldots,Y_j$ with $m=1/k$ and index $\sigma-\ell\in\{1,\ldots,j-1\}$, and using \eqref{eq:retained-shortfall-identity}, we obtain
  \[
    \ell+e_j=k(1-f_j)r_j
    \ge\sigma\left(1-\frac1{kr_j}\right)^{\sigma-\ell}
    =\Psi_\ell(kr_j).
  \]
  Since $\Psi_\ell^{-1}$ is increasing, we again have $kr_j\le\Psi_\ell^{-1}(\ell+e_j)$.
  Thus, \textup{(ii)} follows from \eqref{eq:retained-shortfall-identity}.
\end{proof}

For each $1\le j<k$, set $\ell=k-j$ and define $C_j$ on $[0,\infty)$ by
\[
  C_j(u)=
  \begin{cases}
    f_j^{(s)}+
    \dfrac{f_j^{(s+1)}-f_j^{(s)}}{e_j^{(s+1)}-e_j^{(s)}}
    \bigl(u-e_j^{(s)}\bigr),
               & \begin{aligned}
                    & e_j^{(s)}\le u<e_j^{(s+1)}, \\
                    & s=\ell+1,\ldots,k-1,
                 \end{aligned} \\[6pt]
    f_j^{(k)}, & u\ge e_j^{(k)}.
  \end{cases}
\]
In particular, $C_1\equiv1/k$ on $[0,\infty)$.
        We first prove some basic properties of $C_j$.

        \begin{lemma}\label{lem:capacities}
          For every $1\le j<k$, the function $C_j$ is nonnegative, nondecreasing, and concave, and satisfies $f_j\le C_j(e_j)$.
        \end{lemma}

        \begin{proof}
          The function $C_1$ is constant, and $f_1=\Prob(Y_1=1)\le\E Y_1=1/k=C_1(e_1)$.
          Thus, it suffices to prove the desired result for $j\ge2$.
          Fix $2\le j<k$ from now on.

          We first show that $C_j$ is nondecreasing and concave.
          To compare the slopes of the segments defining $C_j$, we regard $e_j^{(s)}$ and $f_j^{(s)}$ as functions of a real variable $s\ge \ell+1$ using \cref{fact:comparison-arrays}\textup{(i)}.
          Let
          \[
            g(s)=-\frac{\rd}{\rd s}\log(1-f_j^{(s)})
            =\log\frac{s}{s-1}-\frac{s-\ell}{s(s-1)}.
          \]
          Since $\ell=k-j\ge1$ and $s\ge\ell+1\ge2$, the power series for $-\log(1-1/s)$ and direct differentiation give
          \[
            g(s)=\sum_{m=2}^{\infty}\frac{\ell-1+1/m}{s^m}
            \text{ and }
            g'(s)=\frac{\ell-(2\ell-1)s}{s^2(s-1)^2}.
          \]
          The lower bound for $g(s)$ follows from the first term, while comparison with $\ell\sum_{m=2}^{\infty}s^{-m}=\ell/(s(s-1))$ proves the strict upper bound because $0<\ell-1+1/m<\ell$ for $m\ge2$.
          Substituting $g(s)<\ell/(s(s-1))$ into $g'(s)+g(s)^2$ and using $s\ge\ell+1$, we obtain
          \[
            \begin{aligned}
              \frac{\ell-1/2}{s^2}
               & \le g(s)<\frac{\ell}{s(s-1)}\le\frac1s,    \\
              g'(s)+g(s)^2
               & <\frac{\ell^2+\ell-(2\ell-1)s}{s^2(s-1)^2}
              \le\frac{1-\ell^2}{s^2(s-1)^2}\le0.
            \end{aligned}
          \]
          Consequently,
          \[
            \begin{aligned}
              \frac{\rd}{\rd s}e_j^{(s)} & =(1-f_j^{(s)})(1-sg(s))>0,           \\
              \frac{\rd}{\rd s}f_j^{(s)} & =(1-f_j^{(s)})g(s)>0,                \\
              \frac{\rd}{\rd s}\left(\frac{\rd f_j^{(s)}/\rd s}{\rd e_j^{(s)}/\rd s}\right)
                                         & =\frac{g'(s)+g(s)^2}{(1-sg(s))^2}<0.
            \end{aligned}
          \]
          By the Cauchy mean value theorem, $C_j$ is increasing and concave on $[0,e_j^{(k)}]$.
          Since $C_j$ is constant on $[e_j^{(k)},\infty)$ and $C_j(0)=f_j^{(\ell+1)}\ge0$, it is nonnegative, nondecreasing, and concave on $[0,\infty)$.

          It remains to prove $f_j\le C_j(e_j)$.
          Since $\E R_j=j/k<1$, we have $f_j<1$.
          Moreover, \cref{lem:product-shortfall}\textup{(ii)} implies that $0\le e_j\le e_j^{(k)}$.
          Since $0=e_j^{(\ell+1)}<\cdots<e_j^{(k)}$, we can choose an integer $s$ with $\ell+1\le s<k$ such that $e_j^{(s)}\le e_j\le e_j^{(s+1)}$.
          Since $\ell+e_j^{(s)}=\Psi_\ell(s)$, \cref{fact:retained-shortfall}\textup{(ii)} and the explicit formula for $\Psi_\ell^{-1}$ imply $f_j\le Q_{\ell,s}(e_j)$, where
          \[
            Q_{\ell,s}(v)
            =1-(\ell+v)+s^{-1/(s-\ell)}(\ell+v)^{1+1/(s-\ell)}.
          \]
          On $[e_j^{(s)},e_j^{(s+1)}]$, this function is convex, since
          \[
            Q_{\ell,s}''(v)
            =\frac{s-\ell+1}{(s-\ell)^2}s^{-1/(s-\ell)}(\ell+v)^{1/(s-\ell)-1}>0.
          \]

          By \cref{fact:comparison-arrays}\textup{(i)} and the definition of $C_j$, for $t=s,s+1$, we have $Q_{\ell,s}(e_j^{(t)})=f_j^{(t)}=C_j(e_j^{(t)})$.
          Thus, $C_j$ is the chord joining the endpoint values of the convex function $Q_{\ell,s}$ on $[e_j^{(s)},e_j^{(s+1)}]$, so $Q_{\ell,s}(v)\le C_j(v)$ throughout this interval.
          Taking $v=e_j$, we obtain $f_j\le Q_{\ell,s}(e_j)\le C_j(e_j)$.
        \end{proof}

        To prove \eqref{eq:weighted-upper-bound}, we use \cref{fact:retained-sums}\textup{(ii)} together with the bounds $f_j\le C_j(e_j)$.
        We establish the weighted estimate in the following lemma for arbitrary nonnegative nondecreasing weights.

        \begin{lemma}\label{le:allocation}\label{lem:allocation}
          Let $0\le\kappa_1\le\cdots\le\kappa_{k-1}$ and $w\ge0$.
          Suppose that the nonnegative numbers $e_1,\ldots,e_k$ and $f_1,\ldots,f_{k-1}$ satisfy $e_1=0$, $e_{j+1}\ge e_j+f_j$ for $1\le j<k$, $f_j\le C_j(e_j)$ for $1\le j<k$, and $e_k\le w$.
          Then
          \[
            \sum_{j<k}\kappa_jf_j
            \le
            \begin{cases}
              \displaystyle\sum_{j<k}\kappa_jf_j^*(w),
               & 0\le w\le W_k, \\[4pt]
              \displaystyle\sum_{j<k}\kappa_jf_j^{(k)},
               & w>W_k.
            \end{cases}
          \]
          For $0\le w\le W_k$, equality holds when $e_j=e_j^*(w)$ and $f_j=f_j^*(w)$.
          For $w>W_k$, equality holds when $e_j=e_j^{(k)}$ and $f_j=f_j^{(k)}$.
        \end{lemma}

        \begin{proof}
          For $w>W_k$, the monotonicity of $C_j$ and the definition of $C_j$ imply $C_j(u)\le f_j^{(k)}$ for every $u\ge0$.
          Thus, $f_j\le f_j^{(k)}$ for each $j<k$, and summing with the nonnegative weights $\kappa_j$ yields $\sum_{j<k}\kappa_jf_j\le\sum_{j<k}\kappa_jf_j^{(k)}$.
          Equality holds when $e_j=e_j^{(k)}$ for $1\le j\le k$ and $f_j=f_j^{(k)}$ for $1\le j<k$.
          Indeed, \cref{fact:comparison-arrays}\textup{(i)} gives $e_1=0$ and $e_{j+1}=e_j+f_j$, while the definition of $C_j$ gives $f_j=C_j(e_j)$ for $j<k$, with $e_k=W_k<w$.

          Now suppose that $0\le w\le W_k$ and choose $1\le s<k$ with $W_s\le w\le W_{s+1}$.
          We first verify that the bound in \cref{lem:allocation} is attained when $e_j=e_j^*(w)$ for $1\le j\le k$ and $f_j=f_j^*(w)$ for $1\le j<k$.
          For $1\le j<k$, the numbers $e_j^*(w)$ and $f_j^*(w)$ are nonnegative by the definitions of $e_j^*$ and $f_j^*$.
          We have $e_1^*(w)=0$, $e_{j+1}^*(w)=e_j^*(w)+f_j^*(w)$, and $e_k^*(w)=w$ by \cref{fact:comparison-arrays}\textup{(ii)}.
          For $j<k-s$, both $e_j^*(w)$ and $f_j^*(w)$ are zero.
          At $j=k-s$, we have $e_j^*(w)=0$ and $0\le f_j^*(w)\le C_j(0)$.
          For $k-s<j<k$, substituting $e_j^*(w)$ from the definitions of $e_j^*$ and $f_j^*$ into the definition of $C_j$, we obtain $C_j(e_j^*(w))=f_j^*(w)$.
          Thus, all the hypotheses of \cref{lem:allocation} hold with $e_j=e_j^*(w)$ for $1\le j\le k$ and $f_j=f_j^*(w)$ for $1\le j<k$.
          The corresponding weighted sum is $\sum_{j<k}\kappa_jf_j^*(w)$.

          For arbitrary nonnegative numbers $e_1,\ldots,e_k$ and $f_1,\ldots,f_{k-1}$ satisfying the hypotheses of \cref{lem:allocation}, it remains to prove
          \[
            \sum_{j<k}\kappa_j(f_j-f_j^*(w))\le0.
          \]
          Set $d_j=e_j-e_j^*(w)$ and $g_j=f_j-f_j^*(w)$.
          Subtracting the equalities in \cref{fact:comparison-arrays}\textup{(ii)} from the assumed inequalities $e_{j+1}\ge e_j+f_j$, we find $d_{j+1}\ge d_j+g_j$, while $d_k=e_k-w\le0$.
          We first bound the terms through $j=k-s$ in terms of $d_{k-s+1}$.

          Since $f_j^*(w)=0$ for $j<k-s$, $e_{k-s+1}^*(w)=f_{k-s}^*(w)$, and the weights are nondecreasing, summing the assumed inequalities $e_{j+1}\ge e_j+f_j$ yields
          \begin{equation}\label{eq:sec3-allocation-initial}
            \begin{aligned}
              \sum_{j\le k-s}\kappa_jg_j
               & \le\kappa_{k-s}\left(\sum_{j\le k-s}f_j-f_{k-s}^*(w)\right)\le\kappa_{k-s}(e_{k-s+1}-e_{k-s+1}^*(w)) \\
               & =\kappa_{k-s}d_{k-s+1}.
            \end{aligned}
          \end{equation}
          For $k-s<j<k$, let $\alpha_j$ be the slope of the segment of $C_j$ joining nodes $s$ and $s+1$, that is, $\alpha_j=\frac{f_j^{(s+1)}-f_j^{(s)}}{e_j^{(s+1)}-e_j^{(s)}}$.
          By the concavity of $C_j$ and $f_j^*(w)=C_j(e_j^*(w))$,
          \[
            g_j\le C_j(e_j)-C_j(e_j^*(w))\le\alpha_jd_j.
          \]
          Set $\Gamma_{k-s}=\kappa_{k-s}$ and $\Gamma_j=(\Gamma_{j-1}+\kappa_j\alpha_j)/(1+\alpha_j)$ for $k-s<j<k$.
          Since $0\le\Gamma_{j-1}\le\kappa_{j-1}\le\kappa_j$, induction shows that $0\le\Gamma_j\le\kappa_j$.
          By the definition of $\Gamma_j$, we obtain
          \[
            \Gamma_{j-1}d_j+\kappa_jg_j =\Gamma_j(d_j+g_j)
            +(\kappa_j-\Gamma_j)(g_j-\alpha_jd_j)\le\Gamma_jd_{j+1}.
          \]
          Applying the preceding inequality repeatedly for $k-s<j<k$ and using \eqref{eq:sec3-allocation-initial}, we obtain
          \[
            \sum_{j<k}\kappa_jg_j
            \le\Gamma_{k-s}d_{k-s+1}+\sum_{k-s<j<k}\kappa_jg_j
            \le\Gamma_{k-1}d_k\le0.
          \]
          This proves the asserted maximum on $[W_s,W_{s+1}]$.
        \end{proof}

        We next bound $1-f_k$ in terms of $w$.
        For $0\le w\le W_k$, define
        \[
          H(w)=\frac{w}{\Psi_0^{-1}(w)}.
        \]
        \begin{lemma}\label{lem:critical}
          Let $k\ge2$, and let $Y_1,\ldots,Y_k$ be independent with a common law $\mathsf Q\in\Pcal([0,1])$ of mean $1/k$.
          Set $R_k=\sum_{i=1}^kY_i$, $f_k=\Prob(R_k\ge1)$, and $w=e_k=k\E(R_k-1)_+$.
          Then the following statements hold.
          \begin{enumerate}[label=\textup{(\roman*)},leftmargin=*,itemsep=0pt]
            \item\label{it:critical-tail} We have $0\le w=e_k\le W_k$ and $1-f_k\ge H(w)$.
            \item\label{it:critical-shape} The function $H$ is continuous on $[0,W_k]$, satisfies $H(W_s)=(1-1/s)^s$ for $1\le s\le k$, and is strictly concave on each $[W_s,W_{s+1}]$ with $1\le s<k$.
          \end{enumerate}
        \end{lemma}

        \begin{proof}
          By \cref{lem:product-shortfall}\textup{(ii)} with $j=k$, we have $0\le w\le W_k$.
          We next prove $1-f_k\ge H(w)$.
          If $w=0$, then $H(w)=0\le1-f_k$ as desired.
          If $w>0$, then $\E R_k=1$ implies $\E(1-R_k)_+=w/k>0$, so $f_k<1$.
          Applying \cref{fact:retained-shortfall}\textup{(ii)} with $j=k$ and $\ell=0$, we obtain $1-f_k\ge w/\Psi_0^{-1}(w)=H(w)$.

          The function $H$ is continuous on $[0,W_k]$ and satisfies $H(W_s)=(1-1/s)^s$ for $1\le s\le k$.
          For each integer $s$ with $1\le s<k$, the function $H$ is strictly concave on $[W_s,W_{s+1}]$, since
          \[
            H''(u)=-(s+1)s^{-2-1/s}u^{1/s-1}<0
          \]
          for every $u\in(W_s,W_{s+1})$.
          This completes the proof.
        \end{proof}

        The following proposition bounds $F_k(\mathsf P)$ from above using the values of $H$ and $A$ at $W_s$ and $W_{s+1}$.

        \begin{proposition}\label{prop:strict-chord}
          Let $k\ge2$ and $0<x<1/k$.
          Let $\mathsf P=(1-kx)\delta_0+kx\mathsf Q$, where $\mathsf Q$ is a probability distribution on $[0,1]$ with mean $1/k$.
          Suppose that $(1-a)/a\le2/(k-1)$.
          Then
          \[
            a^{-k}F_k(\mathsf P)\le1-H(w)+A(w).
          \]
          Moreover, if $W_s\le w\le W_{s+1}$ for an integer $1\le s<k$, then
          \[
            \begin{aligned}
              1-H(w)+A(w)
               & \le(1-\theta)\bigl[1-H(W_s)+A(W_s)\bigr]         \\
               & \quad+\theta\bigl[1-H(W_{s+1})+A(W_{s+1})\bigr],
            \end{aligned}
          \]
          where $\theta=(w-W_s)/(W_{s+1}-W_s)$.
          The second inequality is strict when $W_s<w<W_{s+1}$.
        \end{proposition}

        \begin{proof}
          We first check that $b_1\le\cdots\le b_{k-1}$, as required by \cref{lem:allocation}.
          If $k\ge3$ and $1\le j\le k-2$, then by the definition of $b_j$ we have
          \[
            \frac{b_{j+1}}{b_j}
            =\frac{(k-j)a}{(j+1)(1-a)}
            \ge\frac{2a}{(k-1)(1-a)}
            \ge1,
          \]
          since $k-j\ge2$, $j+1\le k-1$, and $(1-a)/a\le2/(k-1)$.
          For $k=2$, the sequence consists only of $b_1$.
          Thus, $b_1\le\cdots\le b_{k-1}$ holds for all $k$.

          By \cref{lem:allocation}, $\sum_{j<k}b_jf_j\le A(w)$, while by \cref{lem:critical} we have $f_k\le1-H(w)$.
          Substitution into \cref{fact:retained-sums}\textup{(i)} proves the first inequality.
          For the second inequality, fix any integer $s$ with $1\le s<k$ and $W_s\le w\le W_{s+1}$, and set $\theta=(w-W_s)/(W_{s+1}-W_s)$.
          Then $w=(1-\theta)W_s+\theta W_{s+1}$ and $0\le\theta\le1$.
          The function $A$ is affine on $[W_s,W_{s+1}]$ by definition, while $H$ is strictly concave there by \cref{lem:critical}\textup{(ii)}.
          Hence,
          \[
            \begin{aligned}
              1-H(w)+A(w)
               & =1-H(w)+(1-\theta)A(W_s)
              +\theta A(W_{s+1})                         \\
               & \le1-(1-\theta)H(W_s)-\theta H(W_{s+1})
              +(1-\theta)A(W_s)
              +\theta A(W_{s+1})                         \\
               & =(1-\theta)\bigl[1-H(W_s)+A(W_s)\bigr]
              +\theta\bigl[1-H(W_{s+1})+A(W_{s+1})\bigr].
            \end{aligned}
          \]
          The inequality in the second line is strict when $0<\theta<1$.
        \end{proof}

        \subsection[Node comparisons for maximizing laws]{Node comparisons for maximizing laws}
        \label{sec:sec3-crossing-estimates}\label{sec:dimensions-two-five}

        At $x=\beta_k$, we compare $A(W_s)$ and $H(W_s)$ to prove the upper bound in \cref{thm:iid-tail} for $2\le k\le5$.
        For $k\ge6$, we show that any maximizing distribution with $0<w<W_k$ must satisfy $W_4<w<W_k$.

        Since $x=\beta_k=1/(k+c_k)$ and $c=1/x-k$, we have $c=c_k$.
        Thus,
        \[
          x=\frac1{k+c},\text{ } a=\frac{k}{k+c},\text{ }
          \delta=\frac{c}{k+c}\text{ and }\frac{\delta}{a}=\frac ck.
        \]
        Recall that $g_k(c_k)=1$.
        Since
        \[
          g_k'(u)=\left(1+\frac uk\right)^{k-1}
          -\left(1+\frac{u-1}{k}\right)^{k-1}>0
        \]
        for $k\ge2$ and $u\ge0$,
        we have $c_k<u$ whenever $g_k(u)>1$.

        We use the following upper bound for $c_k$.
        This bound will be used to prove the quantitative estimate for $H(W_4)-A(W_4)$ in \cref{lem:initial-nodes}\textup{(iii)}.
        We prove \cref{lem:crossing-upper} in \cref{app:crossing-upper}.

        \begin{lemma}\label{lem:crossing-upper}
          Let $k\ge5$ and set $\gamma=-\log(1-\mathrm{e}^{-1})$.
          Then
          \[
            c_k<\gamma+\frac{41}{500k}<\frac{11467}{25000}+\frac{41}{500k}.
          \]
        \end{lemma}

        To rule out a given value of $w$, it suffices to show that $A(w)<H(w)$.
        Indeed, \cref{prop:strict-chord} would then imply $F_k(\mathsf P)<a^k=F_k(\mathsf B_{\beta_k})$, contradicting the maximality of $\mathsf P$.

        The following lemma also gives $H(W_4)-A(W_4)>3/5000-1/(400k)$ for $k\ge5$.
        We will use this lower bound in the proof of {\cref{lem:node-shortfall-margin}} in \cref{app:node-shortfall-margin}.

        \begin{lemma}\label{lem:initial-nodes}
          Let $k\ge2$ and $x=\beta_k$.
          \begin{enumerate}[label=\textup{(\roman*)},leftmargin=*,itemsep=0pt]
            \item For every integer $s$ with $2\le s\le\min\{4,k-1\}$, we have $A(W_s)<H(W_s)=((s-1)/s)^s$.
            \item We have $A(W_1)=H(W_1)=0$ and $A(W_k)=H(W_k)=(1-1/k)^k$.
                  \item\label{it:initial-fourth-gap} For $k\ge5$, we have
                  \[
                    H(W_4)-A(W_4)>\frac3{5000}-\frac1{400k}.
                  \]
            \item If $2\le k\le5$, then every nonnegative Borel law $\mathsf P$ maximizing $F_k$ under $m(\mathsf P)\le x$ has $w=0$ or $w=W_k$.
                  If $k\ge6$, then every nonnegative Borel law $\mathsf P$ maximizing $F_k$ under $m(\mathsf P)\le x$ has $w=0$ or $W_4<w\le W_k$.
          \end{enumerate}
        \end{lemma}

        \begin{proof}
          We first prove $c>\gamma>5/12$, where $\gamma=-\log(1-\mathrm{e}^{-1})$.
          Since $k\ge2$, we have $0<1-1/k<1$.
          By the definition of $g_k$, we have
          \[
            g_k(\gamma)
            =\int_0^1\left(1+\frac{\gamma-u}{k}\right)^{k-1}\dd u
            <\int_0^1\mathrm{e}^{(1-1/k)(\gamma-u)}\dd u
            \le\left(\int_0^1\mathrm{e}^{\gamma-u}\dd u\right)^{1-1/k}=1.
          \]
          The first inequality follows from $\log(1+v)\le v$ for $v>-1$, with strict inequality in the integrands whenever $u\ne\gamma$.
          The second inequality follows from the concavity of $y\mapsto y^{1-1/k}$.
          Thus, $c>\gamma$ by the strict monotonicity of $g_k$.
          Since $\mathrm{e}<11/4$, we also have $\gamma>\mathrm{e}^{-1}+\tfrac12\mathrm{e}^{-2}>52/121>5/12$.

          By the definitions of $e_j^*$ and $f_j^*$, we have $f_j^*(W_s)=f_j^{(s)}$ at every node $W_s$.
          Also, the definition of $b_j$ and $(1-a)/a=c/k$ imply $b_j=\binom{k}{j}(c/k)^{k-j}$.
          Substituting the values of $f_j^{(s)}$ from \cref{fact:comparison-arrays}\textup{(i)} into the definition of $A$, we obtain
          \[
            A(W_s)
            =\sum_{h=1}^{s-1}\binom{k}{h}\left(\frac ck\right)^h
            \left[1-\left(1-\frac1s\right)^{s-h}\right]
            \qquad(1\le s\le k).
          \]
          Taking $s=1$, we obtain $A(W_1)=H(W_1)=0$.
          For $s=k$, by separating the two terms in the summand and applying the binomial theorem, we have
          \[
            \begin{aligned}
              A(W_k)
               & =\left(1+\frac ck\right)^k
              -\left(1+\frac{c-1}{k}\right)^k-1+\left(1-\frac1k\right)^k \\
               & =g_k(c)-1+\left(1-\frac1k\right)^k=H(W_k),
            \end{aligned}
          \]
          where we used $g_k(c)=1$.
          For $2\le s\le\min\{4,k\}$, define $Q_s=A(W_s)/H(W_s)$.
          To prove \textup{(i)}, we must show that $Q_s<1$ for $2\le s\le\min\{4,k-1\}$.
          Evaluating the sum for $A(W_s)$ at $s=2,3,4$ with $s\le k$ and dividing by $H(W_s)=(1-1/s)^s$, we obtain
          \begin{equation}\label{eq:sec3-initial-node-values}
            \begin{aligned}
              Q_2 & =2c,                                     \\
              Q_3 & =\frac{15c}{8}+\frac{9(1-1/k)c^2}{16},   \\
              Q_4 & =\frac{148c}{81}+\frac{56(1-1/k)c^2}{81}
              +\frac{32(1-1/k)(1-2/k)c^3}{243},
            \end{aligned}
          \end{equation}
          and hence
          \[
            \begin{aligned}
              Q_3-Q_2
               & =\frac{c}{16}\bigl(9(1-1/k)c-2\bigr) &          & (k\ge3), \\
              Q_4-Q_3
               & =\frac{c}{3888}
              \bigl(501(1-1/k)c+512(1-1/k)(1-2/k)c^2-186\bigr)
               &                                      & (k\ge4).
            \end{aligned}
          \]
          For $k\ge3$, the bounds $c>5/12$ and $1-1/k\ge2/3$ imply $Q_3-Q_2>c/32>0$.
          For $k\ge4$, using $c>5/12$, $1-1/k\ge3/4$, and $1-2/k\ge1/2$, we obtain $Q_4-Q_3>(c/3888)(187/48)>0$.
          When $k=2$, there is no integer $s$ with $2\le s<k$.
          Since $A(W_k)=H(W_k)$, we have $Q_2<Q_3=1$ when $k=3$ and $Q_2<Q_3<Q_4=1$ when $k=4$.
          Thus, \textup{(i)} holds for $2\le k\le4$.

          For $k\ge5$, the inequalities $Q_2<Q_3<Q_4$ show that it suffices to prove $Q_4<1$.
          We prove the stronger bound in \textup{(iii)}.

          Set $C=41/500$ and $c_+(u)=11467/25000+Cu$ for $0\le u\le1/5$.
          Define
          \begin{equation}\label{eq:sec3-fourth-comparison-function}
            R(u)=\frac{37c_+(u)}{64}+\frac{7(1-u)c_+(u)^2}{32}
            +\frac{(1-u)(1-2u)c_+(u)^3}{24}.
          \end{equation}
          By \cref{lem:crossing-upper}, $c<c_+(1/k)$ for $k\ge5$.
          Since the expression for $Q_4$ in \eqref{eq:sec3-initial-node-values} is strictly increasing in $c\ge0$ and $H(W_4)=81/256$, we obtain $A(W_4)<R(1/k)$.
          On $[0,1/5]$, we have $5/12<c_+(u)<1/2$.
          Differentiating \eqref{eq:sec3-fourth-comparison-function} and using $5/12<c_+(u)<1/2$ and $0\le u\le1/5$, we obtain
          \[
            \begin{aligned}
              R''(u)
               & =-\frac{7Cc_+(u)}8+\frac{7C^2(1-u)}{16}
              +\frac{c_+(u)^3}{6}-\frac{C(3-4u)c_+(u)^2}{4}
              +\frac{C^2(1-u)(1-2u)c_+(u)}4                       \\
               & \le-\frac{35C}{96}+\frac{9C^2}{16}+\frac1{48}<0.
            \end{aligned}
          \]
          Direct calculation shows that $R(1/5)<81/256-1/10000$ and $R'(1/5)>1/400$.
          Since $R''<0$ on $[0,1/5]$, we have $R'(u)\ge R'(1/5)>1/400$ throughout this interval.
          For $k\ge5$, using $A(W_4)<R(1/k)$ and integrating $R'$ over $[1/k,1/5]$, we obtain
          \[
            \begin{aligned}
              H(W_4)-A(W_4)
               & >\frac{81}{256}-R(1/k)                            \\
               & =\frac{81}{256}-R(1/5)+\int_{1/k}^{1/5}R'(u)\dd u \\
               & \ge\frac{81}{256}-R(1/5)+R'(1/5)(1/5-1/k)         \\
               & >\frac1{10000}+\frac{1/5-1/k}{400}
              =\frac3{5000}-\frac1{400k}>0.
            \end{aligned}
          \]
          This proves \cref{lem:initial-nodes}\textup{(iii)} and $Q_4<1$.
          The inequalities $Q_2<Q_3<Q_4$ now imply $A(W_s)<H(W_s)$ for $s=2,3,4$ when $k\ge5$.

          Finally, let $\mathsf P$ be any law maximizing $F_k$ among laws with mean at most $\beta_k$.
          By \cref{prop:normalization}, we may write $\mathsf P=(1-a)\delta_0+a\mathsf Q$, where $a=k\beta_k$ and $\mathsf Q\in\Pcal([0,1])$ has mean $1/k$.
          By \cref{lem:critical}, $0\le w\le W_k$.
          By \cref{lem:crossing-parameters}, $(1-a)/a=c/k\le1/(2k)<2/(k-1)$, so \cref{prop:strict-chord} applies.
          Since $\mathsf B_{\beta_k}$ has mean $\beta_k$ and attains $F_k(\mathsf B_{\beta_k})=a^k$, maximality and \cref{prop:strict-chord} imply
          \[
            1\le a^{-k}F_k(\mathsf P)\le1-H(w)+A(w),
          \]
          and hence $H(w)-A(w)\le0$.
          The difference $H-A$ is zero at $W_1=0$ and $W_k$, and positive at $W_s$ for $2\le s\le\min\{4,k-1\}$.
          Fix any integer $s$ with $1\le s<k$ such that $H(W_s)\ge A(W_s)$ and $H(W_{s+1})\ge A(W_{s+1})$.
          By \cref{lem:critical}\textup{(ii)} and the definitions of $e_j^*$ and $f_j^*$ and $A$, for every $u\in(W_s,W_{s+1})$ we have
          \[
            \begin{aligned}
              H(u)
               & >\frac{W_{s+1}-u}{W_{s+1}-W_s}H(W_s)
              +\frac{u-W_s}{W_{s+1}-W_s}H(W_{s+1})      \\
               & \ge\frac{W_{s+1}-u}{W_{s+1}-W_s}A(W_s)
              +\frac{u-W_s}{W_{s+1}-W_s}A(W_{s+1})
              =A(u).
            \end{aligned}
          \]
          For $2\le k\le5$, this proves $H(u)-A(u)>0$ for every $0<u<W_k$.
          Since $0\le w\le W_k$ and $H(w)-A(w)\le0$, we have $w=0$ or $w=W_k$.
          For $k\ge6$, the same argument excludes $0<w\le W_4$, leaving $w=0$ or $W_4<w\le W_k$.
        \end{proof}

        \begin{lemma}\label{lem:low-dimensional-crossing}
          Let $2\le k\le5$.
          For every nonnegative Borel law $\mathsf P$ of mean at most $\beta_k$, $F_k(\mathsf P)\le(k\beta_k)^k=1-(1-\beta_k)^k$.
          Moreover, every maximizing law has $w=0$ or $w=W_k$.
        \end{lemma}

        \begin{proof}
          Let $\mathsf P$ be any maximizing law under the mean bound $\beta_k$.
          Such a law exists by \cref{lem:attainment}.
          By \cref{prop:normalization}, $\mathsf P$ has mean $\beta_k$ and has the form $\mathsf P=(1-a)\delta_0+a\mathsf Q$ with $a=k\beta_k$ and $m(\mathsf Q)=1/k$.
          Let $Y_1,\ldots,Y_k$ be independent with law $\mathsf Q$ and set $w=k\E(Y_1+\cdots+Y_k-1)_+$.
          By \cref{lem:initial-nodes}, $w=0$ or $w=W_k$.
          By \cref{lem:crossing-parameters}, $(1-a)/a=c_k/k\le1/(2k)<2/(k-1)$, so \cref{prop:strict-chord} applies and gives
          \[
            a^{-k}F_k(\mathsf P)\le1-H(w)+A(w)=1,
          \]
          where the equality follows from \cref{lem:initial-nodes}\textup{(ii)} at $w=0,W_k$.
          Both candidate laws attain $(k\beta_k)^k$, proving the maximum.
        \end{proof}

        \Cref{lem:dimension-one,lem:low-dimensional-crossing,prop:mean-transfer} establish the upper bound in \cref{thm:iid-tail} for $1\le k\le5$.
        For $k\ge6$, we shall use the following corollary.

        \begin{corollary}\label{cor:remaining-excess-range}
          Let $k\ge6$ and let $\mathsf P$ be a maximizing law with $m(\mathsf P)\le \beta_k$.
          Write $\mathsf P=(1-k\beta_k)\delta_0+k\beta_k\mathsf Q$ as in \cref{prop:normalization}, and let $Y_1,\ldots,Y_k$ be independent with law $\mathsf Q$.
          If $w=k\E(Y_1+\cdots+Y_k-1)_+$ satisfies $0<w<W_k$, then $W_4<w<W_k$.
        \end{corollary}

        \begin{proof}
          By \cref{lem:initial-nodes}\textup{(iv)}, we have $W_4<w<W_k$.
        \end{proof}

        We conclude this section with an example illustrating the necessity for further analysis of maximizing laws when $k\ge6$.
        For $k=6$ and $x=\beta_6$, we have $A(W_5)>H(W_5)$.
        Indeed, by the definition of $A$ and \cref{fact:comparison-arrays}\textup{(i)}, a direct calculation using $c_6>17/36$ gives
        \[
          A(W_5)-H(W_5)
          =\frac{369c_6}{625}+\frac{61c_6^2}{300}
          +\frac{c_6^3}{30}+\frac{c_6^4}{432}
          -\frac{1024}{3125}>\frac1{12500}>0.
        \]
        Thus, we need a more detailed analysis of maximizing laws to rule out the remaining range $W_4<w<W_k$.

        \section{Mean and variance of sums conditional on failure}
        \label{sec:contact-shortfall}

        For $1\le k\le5$ and $0\le x\le1/k$, \cref{lem:dimension-one,lem:low-dimensional-crossing,prop:mean-transfer} imply
        \[
          F_k(\mathsf P)\le\max\{1-(1-x)^k,(kx)^k\}
        \]
        for every nonnegative Borel probability law $\mathsf P$ with $m(\mathsf P)\le x$.
        For $k\ge6$ and $x=\beta_k$, \cref{cor:remaining-excess-range} leaves $W_4<w<W_k$ for any maximizing law $\mathsf P$ with $0<w<W_k$.
        We now use maximality to show the properties that the distribution conditional on failure must satisfy.

        For such a maximizing law $\mathsf P$, let $X_1,\ldots,X_k$ be independent with law $\mathsf P$, and write $S_k=\sum_{i=1}^kX_i$.
        We study $1-S_k$ conditional on $S_k<1$.
        We will prove $\Var(S_k\mid S_k<1)<\Var(S_k)/20$ by \cref{lem:contact-moment}\textup{\ref{item:contact-moment-bounds}} and \cref{prop:finite-shortfall-margin}.

        In \cref{sec:contact}, we prove $\Prob(S_{k-1}<1-z)=p(M-z)/D$ for $\mathsf P$-almost every $z$ and $\Var(S_k)=D(r-\delta)$, with $p,M,D,r,\delta$ defined there.
        A covariance calculation then proves the conditional variance bound in \cref{lem:contact-moment}.
        In \cref{sec:sec3-energy-comparison}, we choose $L(w)>0$ and prove
        \[
          \E(S_k-1)_+-L(w)F_{k-1}(\mathsf P)\ge E_*(w)-L(w)V_*(w),
        \]
        where $E_*$ and $V_*$ are the functions defined in that subsection.
        We combine this inequality with \cref{lem:node-shortfall-margin} in \cref{sec:node-margin-transfer} to prove $M-r<(M-x)/20-3x/20$, as required by \cref{prop:finite-shortfall-margin}.

        \subsection[Moment identities from perturbations in the probabilities]{Moment identities from perturbations in the probabilities}
        \label{sec:contact}

        {We perturb the probabilities in $\mathsf P$ while keeping their total mass and mean fixed. This yields the formulas for $h$ in \cref{obs:contact-identities}. We then compute the moments in \cref{cor:contact-moments} and bound $\Var(S_k\mid S_k<1)$ in \cref{lem:contact-moment}.}
        Let $k\ge6$, let $0<x<1/k$, and let $\mathsf P$ maximize $F_k$ among nonnegative Borel probability laws of mean at most $x$.
        By \cref{prop:normalization}, we have $\mathsf P([0,1])=1$, $m(\mathsf P)=x$, and $\mathsf P(\{0\})\ge1-kx>0$.
        For independent $X_i\sim\mathsf P$, write $S_j=\sum_{i=1}^jX_i$ and $p=1-F_k(\mathsf P)$.
        Since $\E S_k=kx<1$, we have $p>0$.
        Write $\E[S_k\mid S_k<1]=1-r$.
        Let $h(z)=\Prob(S_{k-1}+z\ge1)$.

        The function $h(z)$ can be seen as the success probability when one coordinate is fixed at $z$.
          {For $z\in[0,1]$, the following signed measure describes a perturbation in the probabilities:}
        \[
          \rho_z=\delta_z-\frac zx\mathsf P+\left(\frac zx-1\right)\delta_0.
        \]
        The total mass of $\rho_z$ is $1-z/x+(z/x-1)=0$, and its first moment is $z-(z/x)m(\mathsf P)=0$.
        Thus, adding $\epsilon\rho_z$ preserves both the total mass and the mean.
        For $0<\epsilon\le\frac12\min\{x,\mathsf P(\{0\})\}$, the restriction of $\mathsf P+\epsilon\rho_z$ to $(0,1]$ is nonnegative, since $1-\epsilon z/x\ge1/2$.
Its mass at zero is at least $\mathsf P(\{0\})-\epsilon>0$.
Hence, $\mathsf P+\epsilon\rho_z$ is a probability law of mean $x$.
Using $\int h\dd\mathsf P=F_k(\mathsf P)$ and $h(0)=F_{k-1}(\mathsf P)$, differentiating at $\epsilon=0+$ and applying maximality, we obtain
\begin{equation}\label{eq:contact-first-variation}
  \begin{aligned}
    0 & \ge\left.\frac{\rd}{\rd\epsilon}F_k(\mathsf P+\epsilon\rho_z)\right|_{0+}
    =k\int h\dd\rho_z                                                             \\
      & =k\left[h(z)-F_{k-1}(\mathsf P)
    -\frac{F_k(\mathsf P)-F_{k-1}(\mathsf P)}x z\right].
  \end{aligned}
\end{equation}
Taking $z=1$ and using $h(1)=1$ and $F_{k-1}(\mathsf P)\le(k-1)x<1$, we obtain $(F_k(\mathsf P)-F_{k-1}(\mathsf P))/x\ge1-F_{k-1}(\mathsf P)>0$.
Thus, the following definitions have positive denominators,
\[
  M=\frac{x(1-F_{k-1}(\mathsf P))}{F_k(\mathsf P)-F_{k-1}(\mathsf P)}
  \text{ and }
  D=M-x=\frac{xp}{F_k(\mathsf P)-F_{k-1}(\mathsf P)}>0.
\]
We retain $h,p,r,M,D$ for the arguments below.

The following observation gives an affine upper bound for $h$ on $[0,1]$,
with equality $\mathsf P$-almost everywhere.
It also shows that $\mathsf P$ is supported on $[0,M]$ and gives an explicit
formula for the failure probability $\Prob(S_{k-1}<1-z)$
for $\mathsf P$-almost every $z$.

\begin{observation}\label{obs:contact-identities}
  Let $k\ge6$ and $0<x<1/k$, and suppose that $\mathsf P$ maximizes $F_k$ among nonnegative Borel probability laws of mean at most $x$.
  With $h,p,M,D$ as defined above, the following statements hold.
  \begin{enumerate}[label=\textup{(\roman*)},leftmargin=*,itemsep=0pt]
    \item\label{item:contact-majorant}\label{item:contact-equality}
    For every $z\in[0,1]$,
    \[
      h(z)\le F_{k-1}(\mathsf P)
      +\frac{F_k(\mathsf P)-F_{k-1}(\mathsf P)}x z,
    \]
    with equality for $\mathsf P$-almost every $z$.

    \item\label{item:contact-failure} We have $x<M\le1$ and $\mathsf P([0,M])=1$.
    Moreover, for $\mathsf P$-almost every $z$,
    \[
      1-h(z)=\Prob(S_{k-1}<1-z)=\frac{p(M-z)}D.
    \]
  \end{enumerate}
\end{observation}

\begin{proof}
  The upper bound in \ref{item:contact-majorant} follows from \eqref{eq:contact-first-variation}.
  Since $\int z\dd\mathsf P(z)=x$ and $\int h\dd\mathsf P=F_k(\mathsf P)$, we have
  \[
    \int\left(
    F_{k-1}(\mathsf P)
    +\frac{F_k(\mathsf P)-F_{k-1}(\mathsf P)}x z-h(z)
    \right)\dd\mathsf P(z)=0.
  \]
  The integrand is nonnegative by \ref{item:contact-majorant}, so it vanishes $\mathsf P$-almost everywhere.
  This proves the equality assertion in \ref{item:contact-equality}.

  We have $M>x$ because $D>0$.
  Taking $z=1$ in \ref{item:contact-majorant} shows that $M\le1$.
  Also, the definitions imply $p/D=(F_k(\mathsf P)-F_{k-1}(\mathsf P))/x$ and $1-F_{k-1}(\mathsf P)=pM/D$.
  By \ref{item:contact-equality}, for $\mathsf P$-almost every $z$,
  we have
  $h(z)=F_{k-1}(\mathsf P)+(F_k(\mathsf P)-F_{k-1}(\mathsf P))z/x
    =F_{k-1}(\mathsf P)+pz/D$.
  Using $1-F_{k-1}(\mathsf P)=pM/D$, we obtain
  $1-h(z)=pM/D-pz/D=p(M-z)/D$.
  Since $1-h(z)=\Prob(S_{k-1}<1-z)$, this proves the formula
  in \ref{item:contact-failure}.
  Since its left-hand side is nonnegative and $p/D>0$, we have $z\le M$ for $\mathsf P$-almost every $z$.
  Hence, $\mathsf P([0,M])=1$, proving \ref{item:contact-failure}.
\end{proof}

The random variable $S_k-X_1$ is independent of $X_1$ and has the same law as $S_{k-1}$.
For every bounded Borel function $\varphi:[0,1]\to\mathbb R$, conditioning on $X_1$ and applying {\cref{obs:contact-identities}\textup{\ref{item:contact-failure}}}, we obtain
\begin{equation}\label{eq:conditional-reweighting}
  \begin{aligned}
    \E[\varphi(X_1)\mid S_k<1]
     & =\frac{\E[\varphi(X_1)(1-h(X_1))]}p \\
     & =\frac{\E[(M-X_1)\varphi(X_1)]}D.
  \end{aligned}
\end{equation}

\noindent Equivalently, the conditional law of $X_1$ given $S_k<1$ is $(M-z)\mathsf P(\rd z)/D$.

Set $\delta=1-kx$.
The decrease in the mean of $S_k$ after conditioning on $S_k<1$ is $B=\E S_k-\E[S_k\mid S_k<1]=r-\delta$.
We retain this notation in the following corollary and subsequent arguments.
The following corollary shows that $B>0$ and expresses the variances of
$X_1$ and $S_k$ in terms of $D$ and $B$.
It also gives identities for $pr$ and $1-F_{k-1}(\mathsf P)$.

\begin{corollary}\label{cor:contact-moments}
  Let $k\ge6$ and $0<x<1/k$, and suppose that $\mathsf P$ maximizes $F_k$ among nonnegative Borel probability laws of mean at most $x$.
  With $p,r,M,D,\delta,B$ as defined above, the following statements hold.
  \begin{enumerate}[label=\textup{(\roman*)},leftmargin=*,itemsep=0pt]
    \item\label{item:contact-moment-identity} We have $B>0$ and
    $
      \Var(X_1)=\frac{DB}{k}\text{ and }
      \Var(S_k)=DB=(M-x)(r-\delta).
    $
    \item\label{item:contact-excess-identities} We have
    $
      pr=\delta+\E(S_k-1)_+
      \text{ and }1-F_{k-1}(\mathsf P)=\frac{pM}{M-x}.
    $
  \end{enumerate}
\end{corollary}

\begin{proof}
  The event $S_k<1$ is unchanged by permuting the coordinates, so $\E[X_1\mid S_k<1]=(1-r)/k$.
  Taking $\varphi(z)=z$ in \eqref{eq:conditional-reweighting}, we obtain
  \[
    \frac{1-r}{k}=\frac{Mx-\E X_1^2}{D}
    =x-\frac{\Var(X_1)}D.
  \]
  Subtracting $(1-r)/k$ from $x$ in the last equality and multiplying by $D>0$, we obtain
  \[
    \begin{aligned}
      \Var(X_1)
       & =D\left(x-\frac{1-r}{k}\right)
      =\frac Dk(kx-1+r)                 \\
       & =\frac Dk(r-\delta)
      =\frac{DB}{k},
    \end{aligned}
  \]
  where we used $\delta=1-kx$ and $B=r-\delta$.
  The variables $X_1,\ldots,X_k$ are independent and have the same variance, so
  \[
    \Var(S_k)=\sum_{i=1}^k\Var(X_i)
    =k\Var(X_1)=DB=(M-x)(r-\delta).
  \]
  Since $x>0$ and $\mathsf P(\{0\})>0$, the law is nonconstant.
  Thus, $\Var(X_1)>0$, and $D>0$ implies $B>0$.
  This proves \ref{item:contact-moment-identity}.

  Taking expectations in $(1-S_k)_+-(S_k-1)_+=1-S_k$ and using $\E(1-S_k)_+=pr$ and $\E S_k=1-\delta$ proves the first equality in \ref{item:contact-excess-identities}.
  The second follows from the definitions of $M$ and $D$.
  This proves \ref{item:contact-excess-identities}.
\end{proof}

Write $v=\Var(S_k\mid S_k<1)$.
The following lemma expresses $B(M-r)$ as the sum of $v$ and a
nonnegative term, yielding $v\le B(M-r)$ and $r\le M$.
\begin{lemma}\label{lem:contact-moment}
  Let $k\ge6$ and $0<x<1/k$, and suppose that $\mathsf P$ maximizes $F_k$ among nonnegative Borel probability laws of mean at most $x$.
    {Use the notation above.}
  \begin{enumerate}[label=\textup{(\roman*)},leftmargin=*,itemsep=0pt]
    \item\label{item:contact-moment-formula} For independent $X,X'\sim\mathsf P$, we have
    \[
      B(M-r)=v+\frac{k}{2p}\E\left[(X-X')
      \int_{X'}^X(h(t)-F_{k-1}(\mathsf P))\dd t\right].
    \]
    \item\label{item:contact-moment-bounds} We have $v\le B(M-r)$, $r\le M$, and
    \[
      \frac{\Var(S_k\mid S_k<1)}{\Var(S_k)}
      \le\frac{M-r}{M-x}.
    \]
  \end{enumerate}
\end{lemma}

\begin{proof}
  Let $T=S_k-X_1$ and $Q(z)=\E(1-T-z)_+$.
  Since $T$ is independent of $X_1$, we have $Q(X_1)=\E[(1-S_k)_+\mid X_1]$.
  Symmetry and $S_k-kx=\delta-(1-S_k)$ give
  \[
    \begin{aligned}
      \Cov(X_1,Q(X_1))
       & =\frac1k\E[(S_k-kx)(1-S_k)_+] \\
       & =\frac pk(\delta r-r^2-v)
      =-\frac pk(v+rB).
    \end{aligned}
  \]
  By Fubini's theorem, $Q(z)-Q(z')=-\int_{z'}^z(1-h(t))\dd t$.
  For independent $X,X'\sim\mathsf P$, we have $2\Cov(X,Q(X))=\E[(X-X')(Q(X)-Q(X'))]$.
  Using $1-h(0)=pM/D$ and $\E(X-X')^2=2DB/k$ from {\cref{cor:contact-moments}\textup{\ref{item:contact-excess-identities}} and \textup{\ref{item:contact-moment-identity}}, respectively}, we obtain
  \[
    \begin{aligned}
      -\frac pk(v+rB)
       & =-\frac12\E\left[(X-X')\int_{X'}^X(1-h(t))\dd t\right] \\
       & =-\frac{pMB}{k}
      +\frac12\E\left[(X-X')\int_{X'}^X(h(t)-h(0))\dd t\right].
    \end{aligned}
  \]
  Rearranging and using $h(0)=F_{k-1}(\mathsf P)$ proves \ref{item:contact-moment-formula}.
  Since $h(t)\ge h(0)$ for $t\ge0$, we have $(X-X')\int_{X'}^X(h(t)-h(0))\dd t\ge0$.
  Thus, $v\le B(M-r)$, and $r\le M$ follows from $v\ge0$ and $B>0$.
  Since $D>0$, division by $\Var(S_k)=DB$ gives $v/\Var(S_k)\le(M-r)/D$.
\end{proof}

We now bound $M-r$ in terms of $D=M-x$.
The following lemma gives the bound $0\le M-r<(D-3x)/20$
for maximizing laws at $x=\beta_k$ with $0<w<W_k$.
In particular, it shows that $D>3x$.

\begin{lemma}\label{prop:finite-shortfall-margin}
  Let $k\ge6$ and $x=\beta_k$.
  Suppose that $\mathsf P$ maximizes $F_k$ among Borel probability laws on $[0,\infty)$ of mean at most $x$.
  Assume $0<w<W_k$, with $w=k\E(R_k-1)_+$ as in \cref{sec:common-bounds}.
  Use $M$, $D$ and $r$ from \cref{sec:contact}.
  \begin{enumerate}[label=\textup{(\roman*)},leftmargin=*,itemsep=0pt]
    \item\label{item:shortfall-margin} We have
    \[
      0\le M-r<\frac{M-x}{20}-\frac{3x}{20}.
    \]
    \item\label{item:shortfall-ratio} We have $D>3x$ and
    \[
      0\le\frac{M-r}{D}<\frac1{20}-\frac{3x}{20D}.
    \]
  \end{enumerate}
\end{lemma}

The inequality $M-r<(D-3x)/20$ and {\cref{lem:contact-moment}\textup{\ref{item:contact-moment-bounds}}} give
\[
  \frac{\Var(S_k\mid S_k<1)}{\Var(S_k)}
  \le\frac{M-r}{D}
  <\frac1{20}-\frac{3x}{20D}.
\]
We will prove \cref{prop:finite-shortfall-margin} in \cref{sec:node-margin-transfer}.

\subsection[A joint bound for the expected excess and the tail probability]{A joint bound for $\E(S_k-1)_+$ and $\Prob(S_{k-1}\ge1)$}
\label{sec:sec3-energy-comparison}
\label{sec:sec3-multipliers}

The inequality $M-r<(D-3x)/20$ is equivalent to $r-19M/20>x/5$.
{To prove $r-19M/20>x/5$, we will bound $\E(S_k-1)_+-L(w)F_{k-1}(\mathsf P)$ from below by $E_*(w)-L(w)V_*(w)$, with $L,E_*,V_*$ defined below.}
Let $k\ge6$ and $0<x<1/k$, and set $a=kx$.
Let $\mathsf P=(1-a)\delta_0+a\mathsf Q$, where $\mathsf Q\in\Pcal([0,1])$ has mean $1/k$.
As in \cref{sec:common-bounds}, let $Y_1,\ldots,Y_k$ be independent with law $\mathsf Q$, and let $I_1,\ldots,I_k$ be independent Bernoulli variables with $\Prob(I_i=1)=a$, independent of $Y_1,\ldots,Y_k$.
Set $X_i=I_iY_i$.
Thus, each $Y_i$ is retained with probability $a$ and replaced by zero otherwise, independently for different $i$, so $X_1,\ldots,X_k$ are independent with common law $\mathsf P$.
We retain the notation $R_j,f_j,e_j,w,b_j$ from \cref{sec:common-bounds} and $S_j$ from \cref{sec:contact}.
For $1\le j\le k-1$, define the weights
\[
  v_j=\binom{k-1}{j}a^j(1-a)^{k-1-j}.
\]
Define the functions $E_*$ and $V_*$ by
\[
  E_*(w)=\frac{a^k}{k}\left(w+\sum_{j<k}b_je_j^*(w)\right)\text{ and }
  V_*(w)=\sum_{j<k}v_jf_j^*(w).
\]
By \cref{fact:comparison-arrays}, both $E_*$ and $V_*$ are affine on each interval $[W_s,W_{s+1}]$.

We bound how far $\E(S_k-1)_+$ can fall below $E_*(w)$ in terms of the gap $V_*(w)-F_{k-1}(\mathsf P)$.
For each fixed integer $k\ge6$, define the piecewise constant function $L$ on $[W_4,W_k)$ by
\[
  L(w)=\frac{s(s-1)}{2k}
  \text{ when }W_s\le w<W_{s+1},
\]
where $4\le s<k$.
  {We shall combine the inequality $\E(S_k-1)_+-L(w)F_{k-1}(\mathsf P)\ge E_*(w)-L(w)V_*(w)$ with the bounds in \cref{lem:node-shortfall-margin} to prove \cref{prop:finite-shortfall-margin}.}

For $k\ge6$ and $w\in[W_4,W_k)$, let $s$ be the unique integer with $4\le s<k$ and $W_s\le w<W_{s+1}$.
For $k-s<j<k$, define
\[
  \alpha_j=\frac{f_j^{(s+1)}-f_j^{(s)}}
  {e_j^{(s+1)}-e_j^{(s)}}>0.
\]
Define
\[
  \Gamma_{k-s} =L(w)v_{k-s} \text{ and }\Gamma_j=\frac{\Gamma_{j-1}-a^kb_j/k+L(w)v_j\alpha_j}{1+\alpha_j}
  \text{ for $k-s<j<k$.}
\]
The following lemma bounds $F_{k-1}(\mathsf P)$ from above by $V_*(w)$.
For $w\in[W_4,W_k)$, it also gives a lower bound for
$\E(S_k-1)_+-L(w)F_{k-1}(\mathsf P)$ when the multipliers
$\Gamma_j$ are nonnegative.
\begin{lemma}\label{lem:coupled-energy}
  Let $k\ge6$, $0<a<1$, and $\mathsf P=(1-a)\delta_0+a\mathsf Q$, where $\mathsf Q\in\Pcal([0,1])$ has mean $1/k$.
  For independent $Y_i\sim\mathsf Q$, set $w=k\E(Y_1+\cdots+Y_k-1)_+\in[0,W_k]$.
  Use the weights $v_j$ and the functions $E_*$ and $V_*$ defined above, and suppose that $v_1\le\cdots\le v_{k-1}$.
  For independent $X_i\sim\mathsf P$, write $S_k=\sum_{i=1}^kX_i$.
  \begin{enumerate}[label=\textup{(\roman*)},leftmargin=*,itemsep=0pt]
    \item\label{item:coupled-profile-general} We have $F_{k-1}(\mathsf P)\le V_*(w)$.
    \item\label{item:coupled-energy-bound} Suppose in addition that $w\in[W_4,W_k)$, let $s$ satisfy $4\le s<k$ and $W_s\le w<W_{s+1}$, and use $L(w)$ and $\Gamma_j$ defined above.
    If $\Gamma_j\ge0$ for every $k-s\le j<k$, then
    \[
      \E(S_k-1)_+-L(w)F_{k-1}(\mathsf P)\ge E_*(w)-L(w)V_*(w).
    \]
  \end{enumerate}
\end{lemma}

\begin{proof}
  By \cref{fact:retained-sums} with $x=a/k$, we have $F_{k-1}(\mathsf P)=\sum_{j<k}v_jf_j$ and $e_{j+1}\ge e_j+f_j$.
  Together with $f_j\le C_j(e_j)$ from \cref{lem:capacities}, these conditions allow us to apply \cref{lem:allocation} with weights $v_j$ and obtain $F_{k-1}(\mathsf P)\le\sum_{j<k}v_jf_j^*(w)=V_*(w)$.
  This proves \ref{item:coupled-profile-general}.
  We now assume the additional hypotheses of \ref{item:coupled-energy-bound}.
  Set $d_j=e_j-e_j^*(w)$ for $1\le j\le k$, and $g_j=f_j-f_j^*(w)$ for $1\le j<k$.
  By \cref{fact:retained-sums}\textup{(ii)}, we have
  $e_{j+1}\ge e_j+f_j$ for $1\le j<k$.
  By \cref{fact:comparison-arrays}\textup{(ii)}, we have
  $e_{j+1}^*(w)=e_j^*(w)+f_j^*(w)$.
  Subtracting this equality from $e_{j+1}\ge e_j+f_j$,
  we obtain
  $d_{j+1}=e_{j+1}-e_{j+1}^*(w)
    \ge(e_j-e_j^*(w))+(f_j-f_j^*(w))
    =d_j+g_j$.
  Moreover, $d_k=0$, since
  $e_k=e_k^*(w)=w$.
  For $k-s<j<k$, \cref{lem:capacities} gives that $C_j$ is concave and satisfies $f_j\le C_j(e_j)$.
  Also, $f_j^*(w)=C_j(e_j^*(w))$ by the definition of $C_j$ and  \cref{fact:comparison-arrays}.
  The function $C_j$ is affine with slope $\alpha_j$ on $[e_j^{(s)},e_j^{(s+1)}]$, which contains $e_j^*(w)$.
  Hence, by concavity,
  \[
    g_j\le C_j(e_j)-C_j(e_j^*(w))\le\alpha_jd_j.
  \]

  By \cref{fact:retained-sums}\textup{(i)} and the definitions of $E_*$ and $V_*$, the inequality
  \[
    \E(S_k-1)_+-L(w)F_{k-1}(\mathsf P)\ge E_*(w)-L(w)V_*(w)
  \]
  is equivalent to $\sum_{j<k}(L(w)v_jg_j-a^kb_jd_j/k)\le0$.
  We first handle the indices $j\le k-s$.
  By \cref{fact:comparison-arrays}, we have $e_j^*(w)=0$ and hence $d_j\ge0$ for $j\le k-s$, and $f_j^*(w)=0$ for $j<k-s$.
  Also, $e_{k-s+1}^*(w)=f_{k-s}^*(w)$ by \cref{fact:comparison-arrays}.
  {Summing $e_{j+1}-e_j\ge f_j$ from \cref{fact:retained-sums}\textup{(ii)} over $1\le j\le k-s$ and using $e_1=0$, we obtain $\sum_{j\le k-s}f_j\le e_{k-s+1}$.}
  Using $v_j\le v_{k-s}$ for $j\le k-s$ and $\Gamma_{k-s}=L(w)v_{k-s}$, we obtain
  \[
    \sum_{j\le k-s}\left(L(w)v_jg_j-\frac{a^kb_j}{k}d_j\right)
    \le L(w)v_{k-s}\left(\sum_{j\le k-s}f_j-f_{k-s}^*(w)\right)
    \le\Gamma_{k-s} d_{k-s+1}.
  \]
  By hypothesis, $\Gamma_j\ge0$.
  To prove the upper bound $\Gamma_j\le L(w)v_j$, start with $\Gamma_{k-s}=L(w)v_{k-s}$.
  If $\Gamma_{j-1}\le L(w)v_{j-1}\le L(w)v_j$, then, by the definition of $\Gamma_j$,
  \[
    (1+\alpha_j)\Gamma_j
    =\Gamma_{j-1}-\frac{a^kb_j}{k}+L(w)v_j\alpha_j
    \le L(w)v_j(1+\alpha_j).
  \]
  Induction yields $0\le\Gamma_j\le L(w)v_j$ for every $k-s\le j<k$.
  For $k-s<j<k$, using $d_j+g_j\le d_{j+1}$, $g_j\le\alpha_jd_j$, and the definition of $\Gamma_j$, we obtain
  \[
    \begin{aligned}
      L(w)v_jg_j-\frac{a^kb_j}{k}d_j
       & =\Gamma_j(d_j+g_j)-\Gamma_{j-1}d_j
      +(L(w)v_j-\Gamma_j)(g_j-\alpha_jd_j)   \\
       & \le\Gamma_jd_{j+1}-\Gamma_{j-1}d_j.
    \end{aligned}
  \]
  Summing over $k-s<j<k$, we obtain
  \[
    \sum_{k-s<j<k}\left(L(w)v_jg_j-\frac{a^kb_j}{k}d_j\right)
    \le\Gamma_{k-1}d_k-\Gamma_{k-s} d_{k-s+1}.
  \]
  Combining the estimates for $j\le k-s$ and $k-s<j<k$ and using $d_k=0$, we obtain $\sum_{j<k}(L(w)v_jg_j-a^kb_jd_j/k)\le0$.
\end{proof}

The next lemma proves $\Gamma_j>0$ for $k-s\le j<k$,
$r_b>x$, and $L(w)x/(r_b-x)^2<9/10$.
It also establishes $F_{k-1}(\mathsf P)\le V_*(w)$ and
$\E(S_k-1)_+-L(w)F_{k-1}(\mathsf P)
  \ge E_*(w)-L(w)V_*(w)$.

\begin{lemma}\label{lem:coupled-multipliers}
  Let $k\ge6$ and $4\le s<k$ be integers, and suppose that $11/24<c<1/2$.
  Write $x=1/(k+c)$, $a=kx$, and $p_0=(1-x)^k$.
  Let $w\in[W_s,W_{s+1})$, set $r_b=x(s+c)H(W_s)/p_0$, and use $L$, $E_*$, $V_*$, and $\Gamma_j$ defined above.
  \begin{enumerate}[label=\textup{(\roman*)},leftmargin=*,itemsep=0pt]
    \item\label{item:coupled-multipliers} We have $\Gamma_j>0$ for every $k-s\le j<k$.
    \item\label{item:coupled-ratio} We have $r_b>x$ and $L(w)x/(r_b-x)^2<9/10$.
    \item\label{item:coupled-comparison} Let $\mathsf P=(1-a)\delta_0+a\mathsf Q$, where $\mathsf Q\in\Pcal([0,1])$ has mean $1/k$, and suppose that $k\E(Y_1+\cdots+Y_k-1)_+=w$ for independent $Y_i\sim\mathsf Q$.
    For independent $X_i\sim\mathsf P$ and $S_k=\sum_{i=1}^kX_i$, we have
    $
      F_{k-1}(\mathsf P)\le V_*(w)
    $
    and
    \[
      \E(S_k-1)_+-L(w)F_{k-1}(\mathsf P)
      \ge E_*(w)-L(w)V_*(w).
    \]
  \end{enumerate}
\end{lemma}

For ease of exposition, the proof of {\cref{lem:coupled-multipliers}} is given in \cref{app:coupled-multipliers}.

To apply {\cref{lem:coupled-multipliers}} at $x=\beta_k$, we use the following lower bound for $c_k$.

\begin{lemma}\label{lem:crossing-lower-finite}
  For every integer $k\ge6$, we have
  \[
    c_k>\frac{11}{24}+\frac{2}{25k}.
  \]
\end{lemma}

The proof of \cref{lem:crossing-lower-finite} is given in \cref{app:crossing-lower-finite}.

At $x=\beta_k$, we have $11/24<c<1/2$ by {\cref{lem:crossing-lower-finite,lem:crossing-parameters}}, so {\cref{lem:coupled-multipliers}} applies.

\subsection{An interpolation estimate for the mean conditional on failure}
\label{sec:sec3-node-margin}

Let $k\ge6$ and $x=\beta_k$. Recall that $a=kx$ and $\delta=1-a$.
We estimate $E_*$ and $V_*$ at the interpolation nodes $W_i$ for $4\le i\le k$ and extend the resulting margin across each interpolation interval.
Recall that $a^k=1-(1-x)^k$ and $p_0=1-a^k$.
For $4\le s\le k$, write $E_s=E_*(W_s)$ and $V_s=V_*(W_s)$.
The following parameter bounds will be used throughout this subsection.

\begin{lemma}\label{lem:sec4-parameter-bounds}
  Let $k\ge6$ and $x=\beta_k$.
  Then
  \[
    \frac{11}{24}+\frac{2}{25k}<c<\frac{19}{40},\text{ }
    \frac{12}{13}<a<1\text{ and }\frac13<p_0<\mathrm{e}^{-1}.
  \]
\end{lemma}

\begin{proof}
  The bounds on $c$ follow from {\cref{lem:crossing-lower-finite,lem:crossing-upper}} and $k\ge6$.
  Since $0<c<1/2$, we have $1>a=k/(k+c)>k/(k+1/2)\ge12/13$.
  Since $c<1/2$, we have $p_0<((2k-1)/(2k+1))^k<\mathrm{e}^{-1}$.
  The lower bound $p_0>1/3$ follows from $p_0>(1-1/k)^k\ge(5/6)^6>1/3$.
\end{proof}

\begin{lemma}\label{lem:node-shortfall-margin}
  Let $k\ge6$, $x=\beta_k$, $a=kx$, $\delta=1-a$, and $p_0=1-a^k=(1-x)^k$.
  Use the functions $E_*$ and $V_*$ defined in \cref{sec:sec3-energy-comparison}.
  \begin{enumerate}[label=\textup{(\roman*)},leftmargin=*,itemsep=0pt]
    \item\label{item:profile-positive} For every $t\in[W_4,W_k]$, we have $V_*(t)<a^k$.
    \item\label{item:profile-margin} Define
    \[
      r_*(t)=\frac{\delta+E_*(t)}{p_0}\text{ and }
      M_*(t)=\frac{x(1-V_*(t))}{a^k-V_*(t)}
      =x+\frac{xp_0}{a^k-V_*(t)}.
    \]
    For every $t\in[W_4,W_k]$, we have
    \[
      r_*(t)-\frac{19M_*(t)}{20}>\frac{x}{5}.
    \]
  \end{enumerate}
\end{lemma}

The proof of {\cref{lem:node-shortfall-margin}} is given in \cref{app:node-shortfall-margin}.
We first prove the estimate at the nodes and then use the concavity of $r_*-19M_*/20$ on each interpolation interval.

\subsection{Proof of Lemma~{\ref{prop:finite-shortfall-margin}}}
\label{sec:node-margin-transfer}

We will prove $M\le M_*(w)$ and $r_*(w)-r\le(9/10)(M_*(w)-M)$.
Since $9/10<19/20$, {\cref{lem:node-shortfall-margin}\textup{\ref{item:profile-margin}}} will then give $r-19M/20>x/5$.

\begin{proof}
  Recall that $a=kx$, $\delta=1-a$, and $p_0=1-a^k=(1-x)^k$.
  By \cref{prop:normalization}, write $\mathsf P=(1-a)\delta_0+a\mathsf Q$ with $m(\mathsf Q)=1/k$.
  We have $0<p\le p_0$ since $\E S_k=a<1$ and $F_k(\mathsf P)\ge a^k$ by maximality.
  By \cref{cor:remaining-excess-range}, $W_4<w<W_k$.
  Let $s$ be the unique integer with $4\le s<k$ and $W_s\le w<W_{s+1}$.
  Set $r_b=x(s+c)H(W_s)/p_0$, and write $r_*=r_*(w)$ and $M_*=M_*(w)$.

  We first show that $x<r_b\le r\le M\le M_*$.
  Recall that $X_i=I_iY_i$, where the $I_i$ are independent Bernoulli variables of mean $a$ and are independent of $Y_1,\ldots,Y_k$.
  Thus, $\E[S_k\mid Y_1,\ldots,Y_k]=aR_k$.
  Since $(1-S_k)_+\ge0$ and $(1-S_k)_+\ge1-S_k$, we have
  \[
    \E[(1-S_k)_+\mid Y_1,\ldots,Y_k]\ge\max\{0,1-aR_k\}\ge a(1-R_k)_++\delta\1_{\{R_k<1\}}.
  \]
  The last inequality is an equality when $R_k<1$, while its right-hand side is zero when $R_k\ge1$.
  Since $\E R_k=km(\mathsf Q)=1$ and $e_k=w$, we have $\E(1-R_k)_+=\E(R_k-1)_+=w/k$.
  Taking expectations, we obtain
  \[
    \begin{aligned}
      pr=\E(1-S_k)_+
       & \ge\frac{aw}{k}+\delta\Prob(R_k<1) \\
       & =\frac{aw}{k}+\delta(1-f_k)        \\
       & \ge\frac{aw}{k}+\delta H(w).
    \end{aligned}
  \]
  The last inequality uses $1-f_k\ge H(w)$ from \cref{lem:critical}.
  By concavity of $H$ on $[W_s,W_{s+1}]$, we have
  \[
    H(w)\ge\min\{H(W_s),H(W_{s+1})\}=H(W_s),
  \]
  where the equality uses $H(W_{s+1})\ge H(W_s)$.
  Using $w\ge W_s=sH(W_s)$, $a/k=x$, and $\delta=cx$, we obtain
  \[
    pr\ge x(s+c)H(W_s),
  \]
  and hence $r\ge r_b$, since $p\le p_0$.

  By {\cref{lem:contact-moment}}, $r\le M$.
  Also, {\cref{lem:coupled-multipliers}\textup{\ref{item:coupled-comparison}}}, {\cref{lem:node-shortfall-margin}\textup{\ref{item:profile-positive}}}, and $p\le p_0$ imply
  \[
    M=x+\frac{xp}{F_k(\mathsf P)-F_{k-1}(\mathsf P)}
    \le x+\frac{xp_0}{a^k-V_*(w)}=M_*.
  \]
  Together with {\cref{lem:coupled-multipliers}\textup{\ref{item:coupled-ratio}}}, this proves
  \begin{equation}\label{eq:sec4-parameter-order}
    x<r_b\le r\le M\le M_*.
  \end{equation}

  Combining {\cref{lem:coupled-multipliers}\textup{\ref{item:coupled-comparison}}} and {\cref{cor:contact-moments}\textup{\ref{item:contact-excess-identities}}}, we obtain
  \[
    \begin{aligned}
      r+\frac{L(w)M}{M-x}
       & \ge\frac{\delta+E_*(w)+L(w)(1-V_*(w))}{p}   \\
       & \ge\frac{\delta+E_*(w)+L(w)(1-V_*(w))}{p_0}
      =r_*+\frac{L(w)M_*}{M_*-x}.
    \end{aligned}
  \]
  The second inequality uses $0<p\le p_0$ and $\delta+E_*(w)+L(w)(1-V_*(w))>0$.
  The last equality follows from the definitions of $r_*$ and $M_*$ in {\cref{lem:node-shortfall-margin}}.
  By {\cref{lem:coupled-multipliers}\textup{\ref{item:coupled-ratio}}} and {\eqref{eq:sec4-parameter-order}},
  \begin{equation}\label{eq:sec4-parameter-loss}
    r_*-r
    \le\frac{L(w)x(M_*-M)}{(M-x)(M_*-x)}
    \le\frac9{10}(M_*-M).
  \end{equation}
  Combining {\eqref{eq:sec4-parameter-order}}, {\eqref{eq:sec4-parameter-loss}}, and {\cref{lem:node-shortfall-margin}\textup{\ref{item:profile-margin}}}, we obtain
  \[
    r-\frac{19M}{20}
    \ge r_*-\frac{19M_*}{20}+\frac{M_*-M}{20}
    >\frac{x}{5}.
  \]
  Rearranging and using $r\le M$ proves \ref{item:shortfall-margin}.
  Since $M-r\ge0$, the upper bound $M-r<(D-3x)/20$ implies $D>3x$.
  Dividing by $D$ proves \ref{item:shortfall-ratio}.
\end{proof}

\section{Lower bounds for the tail probability of sums of samples from the lower interval}
\label{sec:support}

For $k\ge6$ and $x=\beta_k$, we consider a maximizing law
$\mathsf P$ with $0<w<W_k$.
In \cref{sec:support-gap-and-cap}, we show that $\mathsf P$
is supported on two separated intervals,
$[0,x+D/16]$ and $(x+17D/20,M]$, each with positive mass.
We call $[0,x+D/16]$ the \emph{lower interval} and $(x+17D/20,M]$ the \emph{upper interval}.
In \cref{sec:conditional-laws}, we decompose $\mathsf P$
into its conditional distributions on these two intervals
and introduce the corresponding random variables
$X^{\operatorname{L}}$ and $X^{\operatorname{U}}$.
In \cref{sec:conditional-moments}, we first prove that
$\lambda=k(1-q)<a$, where $q$ is the probability of the lower interval.
We then count the cases in which $S_k<1$ and zero or exactly one
of the $X_i$ lies in the upper interval to obtain
$T_0>\eta/3$, where
$T_0=\Prob(X_1^{\operatorname{L}}+\cdots+
  X_k^{\operatorname{L}}\ge1)$.
Finally, we derive bounds for
$\mathcal V=k\Var(X^{\operatorname{L}})/A^2$
using the variance identity for $S_k$ and two estimates for
$\Prob(X^{\operatorname{U}}+
  X_1^{\operatorname{L}}+\cdots+
  X_{k-1}^{\operatorname{L}}<1)$.

\subsection{Separation of support}
\label{sec:support-gap-and-cap}

Fix $k\ge6$ and write $x=\beta_k=(k+c_k)^{-1}$ and $c=c_k$.
By {\cref{lem:crossing-lower-finite,lem:crossing-upper}}, we have $11/24<c<1/2$.
Let $\mathsf P$ maximize $F_k$ among nonnegative Borel probability laws of mean at most $x$.
  By \cref{prop:normalization}, $\mathsf P$ has mean $x$, is supported on $[0,1]$, and can be written as $\mathsf P=(1-a)\delta_0+a\mathsf Q$, where $\mathsf Q\in\Pcal([0,1])$ has mean $1/k$.
  We retain the notation $w=e_k$ from \cref{sec:common-bounds}.
  By \cref{lem:critical}, we have $0\le w\le W_k$.
  If $w=0$ or $w=W_k$, then \cref{prop:strict-chord} and \cref{lem:initial-nodes}\textup{(ii)} give $F_k(\mathsf P)\le a^k$.
  We therefore study the remaining case $0<w<W_k$, for which \cref{cor:remaining-excess-range} gives $W_4<w<W_k$.
For independent $X_1,\ldots,X_k$ with law $\mathsf P$, we retain the notation $S_j,p$ from \cref{sec:contact}.
Recall that $a=kx$, $\delta=1-a=cx$, and $p_0=(1-x)^k=1-a^k$.
Since $\E S_k=a<1$, we have $p>0$.
Recall also that $r=\E[1-S_k\mid S_k<1]$, $M=x(1-F_{k-1}(\mathsf P))/(F_k(\mathsf P)-F_{k-1}(\mathsf P))$, and $D=M-x$.
By {\cref{obs:contact-identities}\textup{\ref{item:contact-failure}}}, we have $x<M\le1$ and $\mathsf P([0,M])=1$.

\begin{proposition}\label{prop:support-separation}
  Let $k\ge6$ be an integer and $x=\beta_k$.
  Suppose that $\mathsf P$ maximizes $F_k$ among nonnegative Borel probability laws of mean at most $x$.
  Write $\mathsf P=(1-kx)\delta_0+kx\mathsf Q$, where $\mathsf Q\in\Pcal([0,1])$ has mean $1/k$.
    For independent $Y_1,\ldots,Y_k$ with law $\mathsf Q$, set $w=k\E(Y_1+\cdots+Y_k-1)_+$ and assume $0<w<W_k$.
  Then $\mathsf P([x+D/16,x+17D/20])=0$ and $\mathsf P((x+17D/20,M])>0$.
\end{proposition}

To prove that $\mathsf P((x+17D/20,M])>0$, we use the following normal comparison.
Let $\Phi$ denote the distribution function of the standard normal distribution.

\begin{lemma}[Shevtsova~{\cite[p.~125]{Shevtsova2013}}]
  \label{lem:normal-comparison}
  Let $n\ge1$ be an integer, and let $\widetilde Z_1,\ldots,\widetilde Z_n$ be independent and identically distributed random variables with mean zero, positive variance, and finite third absolute moment.
  Write $U=\sum_{i=1}^n\widetilde Z_i$, $s^2=\sum_{i=1}^n\E\widetilde Z_i^2$, and $B_3=\sum_{i=1}^n\E|\widetilde Z_i|^3$.
  For every $t\in\mathbb R$, we have
  \[
    \left|\Prob(U<t)-\Phi(t/s)\right|
    \le
    0.3328\left(
    \frac{B_3}{s^3}+\frac{0.429}{\sqrt n}
    \right)
    \le\frac{B_3}{3s^3}+\frac1{7\sqrt n}.
  \]
\end{lemma}

\begin{proof}[Proof of \cref{prop:support-separation}]
  Write $d=D/x$, $\mu=(M-r)/D$, and $v=\Var(S_k\mid S_k<1)$.
  By {\cref{prop:finite-shortfall-margin}\textup{\ref{item:shortfall-margin}}},
  \[
    0\le M-r<\frac{D-3x}{20}.
  \]
  Consequently, $d>3$ and
  \[
    0\le\mu<\mu_0:=\frac{d-3}{20d}<\frac1{20}.
  \]

  We first prove that $\mathsf P([x+D/16,x+17D/20])=0$.
  By {\cref{lem:contact-moment}\textup{\ref{item:contact-moment-bounds}}}, we have $v\le(r-\delta)(M-r)$.
  Since $M-r=\mu D$, $M=x+D$, $\delta=cx$, and $d=D/x$, we have $(r-\delta)/D=1-\mu+(1-c)/d$.
  Hence, using $c\ge11/24$,
  \[
    \frac v{D^2}
    \le\mu\left(1-\mu+\frac{1-c}{d}\right)
    \le\mu\left(1-\mu+\frac{13}{24d}\right).
  \]
  For fixed $d>3$, the function $u\mapsto u(1-u+13/(24d))$ is strictly increasing on $[0,1/20]$.
  Since $\mu<\mu_0$, we obtain
  \[
    \frac v{D^2}
    <\mu_0\left(1-\mu_0+\frac{13}{24d}\right)
    <\frac{19}{400}.
  \]

  Let $W$ have the distribution of $1-S_k$ conditional on $S_k<1$.
  Then $\E W=r$ and $\Var(W)=v$.
  Suppose, for a contradiction, that there exists $y\in[x+D/16,x+17D/20]$ such that
  \[
    \Prob(S_{k-1}<1-y)=p\frac{M-y}{D}.
  \]
  Write $y=x+\theta D$, where $1/16\le\theta\le17/20$.
  Since $\mu<1/20$, we have $r-y=(1-\mu-\theta)D>0$.
  Since $y>0$ and $S_{k-1}\le S_k$, it follows that
  \[
    \Prob(W>y)
    =\frac{\Prob(S_k<1-y)}p
    \le\frac{\Prob(S_{k-1}<1-y)}p
    =\frac{M-y}{D}
    =1-\theta.
  \]
  Since $r-y>0$, Cantelli's inequality gives
  \[
    \theta\le\Prob(W\le y)
    =\Prob(W-r\le-(r-y))
    \le\frac{v}{v+(r-y)^2}.
  \]
  Since $0<\theta<1$,
  \[
    \frac v{D^2}
    \ge
    \frac{\theta(1-\mu-\theta)^2}{1-\theta}
    \ge
    \frac{\theta(19/20-\theta)^2}{1-\theta}.
  \]
  The second inequality uses $\mu<1/20$ and $19/20-\theta>0$.

  Let $f(\theta)=\theta(19/20-\theta)^2/(1-\theta)$ for $1/16\le\theta\le17/20$.
  Then
  \[
    f'(\theta)
    =\frac{(19/20-\theta)(2\theta^2-3\theta+19/20)}
    {(1-\theta)^2}.
  \]
  Set $\theta_*=(15-\sqrt{35})/20$.
  We have $f'(\theta)>0$ for $1/16\le\theta<\theta_*$ and $f'(\theta)<0$ for $\theta_*<\theta\le17/20$.
  Since $f(1/16)=5041/96000$ and $f(17/20)=17/300$, we obtain
  \[
    \frac v{D^2}\ge f(\theta)
    \ge\min\{f(1/16),f(17/20)\}
    =\frac{5041}{96000}>\frac{19}{400},
  \]
  contrary to $v/D^2<19/400$.
  Thus, no $y\in[x+D/16,x+17D/20]$ satisfies $\Prob(S_{k-1}<1-y)=p(M-y)/D$.
  By {\cref{obs:contact-identities}\textup{\ref{item:contact-failure}}}, $\Prob(S_{k-1}<1-y)=p(M-y)/D$ holds for $\mathsf P$-almost every $y$, so $\mathsf P([x+D/16,x+17D/20])=0$.

  Suppose that $\mathsf P((x+17D/20,M])=0$.
  Then $\mathsf P([0,x+D/16))=1$.
  For $X\sim\mathsf P$, set $H=(X-x)/x$ and write $s^2=k\E H^2$ and $B_3=k\E|H|^3$.
  Then $\E H=0$ and $-1\le H\le d/16$ almost surely.
  By {\cref{cor:contact-moments}\textup{\ref{item:contact-moment-identity}}}, $c<1/2$, and $d\mu<(d-3)/20$,
  \begin{equation}\label{eq:variance-floor}
    s^2=d(d+1-c-d\mu)
    >d\left(\frac{19d}{20}+\frac{13}{20}\right).
  \end{equation}
  Since $d>3$, \eqref{eq:variance-floor} gives $s>3$ and $s>19d/20$.
  Thus, $0<c/s<1/6$ and $B_3/s^3\le\max\{1,d/16\}/s$.
  The standard normal density exceeds $1/3$ on $[0,1/6]$, so $\Phi(-t)\le1/2-t/3$ for $0\le t\le1/6$.

  For $1\le i\le k$, set $H_i=(X_i-x)/x$.
  Then $S_k\ge1$ if and only if $\sum_{i=1}^kH_i\ge c$.
  Lyapunov's inequality gives $B_3/s^3\ge1/\sqrt{k}$.
  Hence, $B_3/(3s^3)+1/(7\sqrt{k})\le10B_3/(21s^3)<B_3/(2s^3)$.
    By \cref{lem:normal-comparison} and a straightforward calculation using $c>11/24$, $s>3$, and $s>19d/20$, we obtain
    \[
      \begin{aligned}
        F_k(\mathsf P)
         & =\Prob\left(\sum_{i=1}^kH_i\ge c\right) \\
         & \le\Phi(-c/s)+\frac{B_3}{2s^3}<\frac58.
      \end{aligned}
    \]
  But maximality and {\cref{lem:sec4-parameter-bounds}} give $F_k(\mathsf P)\ge a^k=1-p_0>5/8$, a contradiction.
\end{proof}

\subsection{Lower and upper conditional laws}
\label{sec:conditional-laws}

To prepare the counting and moment estimates in \cref{sec:conditional-moments}, we describe the conditional decomposition of a general probability law supported on two separated intervals.

Let $k\ge2$ be an integer and $c>0$, and write $x=1/(k+c)$.
Fix $M$, $\varepsilon$, and $\zeta$ with $x<M\le1$ and $0\le\varepsilon<\zeta<1$.
Write $D=M-x$.
Let $\mathsf P$ have mean $x$ and be supported on $[0,x+\varepsilon D]\cup(x+\zeta D,M]$.
We refer to $[0,x+\varepsilon D]$ and $(x+\zeta D,M]$ as the lower and upper intervals, respectively.
Suppose that $\mathsf P((x+\zeta D,M])>0$.
We retain the notation $a=kx$, $\delta=1-a=cx$, and $p_0=(1-x)^k$, and set $d=D/x$.
For independent $X_1,\ldots,X_k$ with law $\mathsf P$, we retain the notation $S_j=\sum_{i=1}^jX_i$, $p=\Prob(S_k<1)$, and $r=1-\E[S_k\mid S_k<1]$.
Since $\E S_k=a<1$, we have $p>0$.

Write $q=\mathsf P([0,x+\varepsilon D])$.
Since $\mathsf P((x+\zeta D,M])>0$, we have $q<1$.
If $q=0$, then $X_1>x$ almost surely, which contradicts $\E X_1=x$.
Hence, $0<q<1$.
Let $X^{\operatorname{L}}$ and $X^{\operatorname{U}}$ be random variables with distributions specified by
\[
  \begin{aligned}
    \Prob(X^{\operatorname{L}}\in E)
     & =\frac{\mathsf P(E\cap[0,x+\varepsilon D])}{q}, \\
    \Prob(X^{\operatorname{U}}\in E)
     & =\frac{\mathsf P(E\cap(x+\zeta D,M])}{1-q}
  \end{aligned}
\]
for every Borel set $E\subseteq\mathbb R$.
Then $\mathsf P=q\law(X^{\operatorname{L}})+(1-q)\law(X^{\operatorname{U}})$.
Write $x_0=\E X^{\operatorname{L}}$.
Taking means, we obtain
\[
  x=qx_0+(1-q)\E X^{\operatorname{U}}.
\]
Since $X^{\operatorname{U}}>x$ almost surely and
$x=qx_0+(1-q)\E X^{\operatorname{U}}$, we have $0\le x_0<x$.
Set $A=M-x_0$ and $J=X^{\operatorname{U}}-x_0$, so $D<A\le1$.
  Write $\lambda=k(1-q)$ for the expected number of upper observations among $k$ independent samples.
  Set $m=k(x-x_0)$ and $\eta=a-m/A$.
  The mean equation gives $m=\lambda\E J$.
  Since $J=A-(M-X^{\operatorname{U}})$, we have
  \[
    \eta=a-\lambda+\frac{\lambda\E(M-X^{\operatorname{U}})}A.
  \]
Since $0\le X^{\operatorname{L}}\le x+\varepsilon D$, we also have $-x_0/A\le(X^{\operatorname{L}}-x_0)/A\le(x+\varepsilon D-x_0)/A$ almost surely.

For independent copies $X_1^{\operatorname{L}},\ldots,X_k^{\operatorname{L}}$ of $X^{\operatorname{L}}$, let $T_0$ denote the probability that their sum is at least one, so
\[
  T_0=\Prob(X_1^{\operatorname{L}}+\cdots+X_k^{\operatorname{L}}\ge1).
\]
We will derive a lower bound for $T_0$ in \cref{lem:exact-upper-count} and, for the maximizing distributions under consideration, obtain a contradictory upper bound in \cref{sec:exclusion}.
Let $X_1^{\operatorname{U}},\ldots,X_k^{\operatorname{U}}$ be independent copies of $X^{\operatorname{U}}$, and let $I_1,\ldots,I_k$ be independent Bernoulli$(1-q)$ variables.
Let all $3k$ variables $X_i^{\operatorname{L}},X_i^{\operatorname{U}},I_i$ be mutually independent.
We may then realize independent samples from $\mathsf P$ as
\[
  X_i=(1-I_i)X_i^{\operatorname{L}}+I_iX_i^{\operatorname{U}}\qquad(1\le i\le k).
\]
Since $X_i^{\operatorname{L}}\le x+\varepsilon D<X_i^{\operatorname{U}}$ almost surely, we have $X_i\ge X_i^{\operatorname{L}}$ almost surely for every $i$.

By \cref{prop:support-separation}, for $k\ge6$ and $x=\beta_k$ this conditional decomposition applies to every maximizing distribution $\mathsf P$ with $0<w<W_k$, taking $\varepsilon=1/16$, $\zeta=17/20$, and $M=x(1-F_{k-1}(\mathsf P))/(F_k(\mathsf P)-F_{k-1}(\mathsf P))$.
By {\cref{cor:contact-moments}\textup{\ref{item:contact-moment-identity}}} and {\cref{obs:contact-identities}\textup{\ref{item:contact-failure}}}, these maximizing distributions also satisfy
\[
  \begin{aligned}
    \Var(S_k)          & =D(r-\delta),   \\
    \Prob(S_{k-1}<1-y) & =p\frac{M-y}{D}
    \quad\text{for $\mathsf P$-almost every $y$.}
  \end{aligned}
\]

\subsection{Bounds for the tail probability and variance for lower-interval sums}
\label{sec:conditional-moments}

We use the notation from \cref{sec:conditional-laws}, with the parameter ranges specified in each lemma below.
For independent copies $X_1^{\operatorname{L}},\ldots,X_k^{\operatorname{L}}$ of $X^{\operatorname{L}}$, write $B_j=\sum_{i=1}^j(X_i^{\operatorname{L}}-x_0)$ for $1\le j\le k$ and set $\mathcal V=k\Var(X^{\operatorname{L}})/A^2=\Var(B_k/A)$.

We prove a lower bound for $T_0$ and bounds for $\mathcal V$.
We first bound the expected number of samples from the upper interval and then estimate $T_0$.
Finally, in \cref{sec:lower-variance}, \cref{lem:conditional-variance-bounds,lem:averaged-contact-normal} establish bounds for $\mathcal V$.

\subsubsection{Bounding the expected number of samples from the upper interval}
\label{sec:conditional-counts}

Recall that $\lambda=k(1-q)$ is the expected number of upper observations among $k$ independent samples from $\mathsf P$. The next lemma will be used to show that $\lambda<a$ unless $\mathsf P=\mathsf B_x$.

\begin{lemma}\label{lem:upper-intensity}
  Let $k\ge6$ be an integer, $11/24\le c\le1/2$, $x=(k+c)^{-1}$, and $x<M\le1$.
  Suppose that $0\le\varepsilon\le1/10$ and $\zeta=17/20$.
  Let $\mathsf P$ be a probability distribution of mean $x$, supported on $[0,x+\varepsilon D]\cup(x+\zeta D,M]$, with $\mathsf P((x+\zeta D,M])>0$.
  Assume that $p\le p_0$, $D>3x$, and $r\le M$.
  Suppose also that $\Var(S_k)=D(r-\delta)$ and $\Prob(S_{k-1}<1-y)=p(M-y)/D$ for $\mathsf P$-almost every $y$.
  Then the following statements hold.
  \begin{enumerate}[label=\textup{(\roman*)},leftmargin=*,itemsep=0pt]
    \item\label{item:upper-intensity-zero} If $x_0=0$, then $M=1$ and $\mathsf P=(1-x)\delta_0+x\delta_1=\mathsf B_x$.
    \item\label{item:upper-intensity-parameters} If $x_0>0$, then $J>m+\delta-x_0/5$ almost surely and
    \[
      0<a-\lambda\le\eta<a,\text{ }
      q\le\frac DA\text{ and }
      0\le\frac{\E(M-X^{\operatorname{U}})}{A}<\frac{3D}{20A}.
    \]
    \item\label{item:upper-intensity-ratio} If $x_0>0$, then
    \[
      \frac{p_0A}{Dq^{k-1}}\le\mathrm{e}^{-(a-\lambda)}\le1.
    \]
  \end{enumerate}
\end{lemma}

To prove \cref{lem:upper-intensity} when $x_0>0$, we need to estimate the probability that $k-1$ lower observations have sum less than $1-X^{\operatorname{U}}$.
We use the following inequality when $1-X^{\operatorname{U}}$ exceeds their mean.
For $v>0$, define
\[
  g(v)=
  \begin{cases}
    v\mathrm{e}^{-v}, & 0<v<1, \\
    \mathrm{e}^{-1},  & v\ge1.
  \end{cases}
\]

\begin{lemma}[Fu, Han, Wang, Yan, Zhang and Zhou~{\cite[Corollary~1.2]{FHWYZZ}}]
  \label{lem:positive-deficit-tail}
  Let $n\ge1$ be an integer, and let $X_1,\ldots,X_n$ be independent nonnegative random variables with $\E X_i\le1$ for every $i$.
  Write $S=X_1+\cdots+X_n$.
  Then $\Prob(S<\E S+v)>g(v)$ for every $v>0$.
\end{lemma}

\begin{proof}[Proof of \cref{lem:upper-intensity}]
  Since $\Var(S_k)=D(r-\delta)$, $r\le M$, $d>3$, and $c\ge11/24$,
  \[
    \frac{\Var(S_k)}{D^2}
    \le1+\frac{1-c}{d}<\frac{85}{72}.
  \]
  By independence and $\mathsf P=q\law(X^{\operatorname{L}})+(1-q)\law(X^{\operatorname{U}})$,
  \[
    \Var(S_k)\ge k(1-q)q(\E X^{\operatorname{U}}-x_0)^2
    =\frac{\lambda}{q}(\E X^{\operatorname{U}}-x)^2.
  \]
  Here $\E X^{\operatorname{U}}-x=q(\E X^{\operatorname{U}}-x_0)>17D/20$.
  Hence, $\lambda/q<(85/72)(20/17)^2=250/153$.

  Since $\log q>-(1-q)/q$ and $\log p_0\le k\log((k-1/2)/(k+1/2))<-1$, we obtain
  \begin{equation}\label{eq:prelude-count-ratio}
    \frac{q^{k-1}}{p_0}
    >\exp\left(1-\frac{k-1}{k}\frac{\lambda}{q}\right)
    >\mathrm{e}^{1-250/153}.
  \end{equation}
  We will use
  \[
    \mathrm{e}^{1-250/153}g(11/24)
    >\frac{11}{24}\mathrm{e}^{-11/10}
    >\frac{11}{24}\cdot\frac4{11}\cdot\frac9{10}
    =\frac3{20},
  \]
  where the inequalities follow from $250/153+11/24<21/10$, $\mathrm{e}^{-1}>4/11$, and $\mathrm{e}^{-1/10}>9/10$.
  Since $g(11/24)<1$, we also have $\mathrm{e}^{1-250/153}>3/20$.

  For every $y\in\mathbb R$, in the representation $X_i=(1-I_i)X_i^{\operatorname{L}}+I_iX_i^{\operatorname{U}}$, the probability that all $k-1$ samples lie in $[0,x+\varepsilon D]$ and have sum less than $1-y$ is $q^{k-1}\Prob(X_1^{\operatorname{L}}+\cdots+X_{k-1}^{\operatorname{L}}<1-y)$.
  Thus, $q^{k-1}\Prob(X_1^{\operatorname{L}}+\cdots+X_{k-1}^{\operatorname{L}}<1-y)\le\Prob(S_{k-1}<1-y)$.
  Using $\Prob(S_{k-1}<1-y)=p(M-y)/D$ for $\mathsf P$-almost every $y$, we obtain, for $\law(X^{\operatorname{U}})$-almost every $y$,
  \begin{equation}\label{eq:prelude-contact}
    p\frac{M-y}{D}
    \ge q^{k-1}\Prob(B_{k-1}<1-y-(k-1)x_0).
  \end{equation}
  Since $x+17D/20<X^{\operatorname{U}}\le M$ almost surely and $D<A$,
  \begin{equation}\label{eq:prelude-raw-cap}
    \frac{M-X^{\operatorname{U}}}{D}<\frac3{20}
    \text{ and }
    \frac JA=1-\frac{M-X^{\operatorname{U}}}A>\frac{17}{20}
    \quad\text{almost surely.}
  \end{equation}

  First suppose that $x_0=0$.
  Then $X^{\operatorname{L}}=0$ and $B_{k-1}=0$ almost surely, so $\Prob(B_{k-1}<1-y)=1$ for every $y<1$.
  If $\Prob(X^{\operatorname{U}}<1)>0$, then \eqref{eq:prelude-contact}, \eqref{eq:prelude-count-ratio}, and $p\le p_0$ give
  \[
    \frac{M-y}D\ge\frac{q^{k-1}}p
    >\mathrm{e}^{1-250/153}>\frac3{20}
  \]
  for $\law(X^{\operatorname{U}})$-almost every $y<1$.
  After substituting $y=X^{\operatorname{U}}$, the preceding inequality holds almost surely on $\{X^{\operatorname{U}}<1\}$, contrary to \eqref{eq:prelude-raw-cap}.
  Hence, $\Prob(X^{\operatorname{U}}<1)=0$.
  Since $X^{\operatorname{U}}\le M\le1$, we have $X^{\operatorname{U}}=1$ almost surely and $M=1$.
  Taking means then gives $1-q=x$, so $\mathsf P=(1-x)\delta_0+x\delta_1$.

  Henceforth assume $x_0>0$.
  For every $y\in\mathbb R$, write $\vartheta_y=(1-y-(k-1)x_0)/x_0$.
  For every fixed $y$ with $\vartheta_y>0$, apply \cref{lem:positive-deficit-tail} to $X_1^{\operatorname{L}}/x_0,\ldots,X_{k-1}^{\operatorname{L}}/x_0$, which are independent, nonnegative, and have mean one.
  Since $(1-y)/x_0=(k-1)+\vartheta_y$, we obtain $\Prob(B_{k-1}<1-y-(k-1)x_0)>g(\vartheta_y)$.
  Combining $\Prob(B_{k-1}<1-y-(k-1)x_0)>g(\vartheta_y)$ with \eqref{eq:prelude-contact} and $p\le p_0$ gives, for $\law(X^{\operatorname{U}})$-almost every $y$ with $\vartheta_y>0$,
  \begin{equation}\label{eq:prelude-tail-comparison}
    \frac{M-y}{D}
    >\frac{q^{k-1}}{p_0}g(\vartheta_y).
  \end{equation}
  Since $X^{\operatorname{U}}=x_0+J$ and $1-kx_0=\delta+m$, we have $\vartheta_{X^{\operatorname{U}}}=(\delta+m-J)/x_0$.

  We claim that $J>m/a$ almost surely.
  If $\Prob(J\le m/a)>0$, then, since $\delta=cx_0+cm/k$ and $a^{-1}=1+c/k$,
  \[
    \vartheta_{X^{\operatorname{U}}}\ge\frac{\delta+m-m/a}{x_0}=c
  \]
  on $\{J\le m/a\}$.
  Since $g$ is nondecreasing, substituting $y=X^{\operatorname{U}}$ in \eqref{eq:prelude-tail-comparison} and using \eqref{eq:prelude-count-ratio} gives
  \[
    \frac{M-X^{\operatorname{U}}}D
    >\mathrm{e}^{1-250/153}g(c)
    \ge\mathrm{e}^{1-250/153}g(11/24)
    >\frac3{20}
  \]
  almost surely on $\{J\le m/a\}$, contrary to \eqref{eq:prelude-raw-cap}.
  Taking expectations in $J>m/a$ gives $\E J>m/a=(\lambda/a)\E J$, and hence $\lambda<a$.

  In particular, $q>1-x$, so $q^{k-1}/p_0>1$.
  If $\Prob(J\le m+\delta-x_0/5)>0$, then $\vartheta_{X^{\operatorname{U}}}\ge1/5$ on $\{J\le m+\delta-x_0/5\}$.
  Substituting $y=X^{\operatorname{U}}$ in \eqref{eq:prelude-tail-comparison} therefore gives
  \[
    \frac{M-X^{\operatorname{U}}}D>g(1/5)
    =\frac{\mathrm{e}^{-1/5}}5>\frac4{25}>\frac3{20}
  \]
  almost surely on $\{J\le m+\delta-x_0/5\}$, contrary to \eqref{eq:prelude-raw-cap}.
  Thus, $J>m+\delta-x_0/5$ almost surely.

  Since $0<\E J\le A$ and $m=\lambda\E J$, we have $0<m/A\le\lambda$, so $0<a-\lambda\le\eta<a$.
  Also, $D=A-m/k$ implies $D/A=1-m/(kA)\ge1-\lambda/k=q$.
  Since $x+17D/20<X^{\operatorname{U}}\le M$ almost surely, taking expectations in $0\le M-X^{\operatorname{U}}<3D/20$ completes the proof of \textup{(ii)}.

  To prove \textup{(iii)}, note that $q-(1-x)=(a-\lambda)/k$.
  Since $q\le D/A$ and $1-(a-\lambda)/(kq)=(1-x)/q>0$, we obtain
  \[
    \frac{p_0A}{Dq^{k-1}}
    =\left(1-\frac{a-\lambda}{kq}\right)^k\frac{Aq}{D}
    \le\mathrm{e}^{-(a-\lambda)/q}\le\mathrm{e}^{-(a-\lambda)}\le1.
  \]
\end{proof}

The following corollary rules out $x_0=0$ and shows that
$\lambda<a$ and $0<\eta<a$ for maximizing laws with $0<w<W_k$.
  \begin{corollary}\label{cor:positive-lower-mean}
    Let $k\ge6$ and $x=\beta_k$, and let $\mathsf P$ maximize $F_k$ among nonnegative Borel laws of mean at most $x$.
    Assume $0<w<W_k$, with $w=k\E(R_k-1)_+$ as in \cref{sec:common-bounds}.
    Use the conditional decomposition from \cref{sec:conditional-laws} with $\varepsilon=1/16$ and $\zeta=17/20$.
    Then $x_0>0$, $\lambda<a$, and $0<\eta<a$.
  \end{corollary}
  \begin{proof}
    Maximality gives $p\le p_0$.
    Together with \cref{prop:finite-shortfall-margin,prop:support-separation}, the moment identities in \cref{sec:contact}, and the bounds on $c_k$ in \cref{lem:crossing-lower-finite,lem:crossing-upper}, this verifies the hypotheses of \cref{lem:upper-intensity}.
    If $x_0=0$, then \cref{lem:upper-intensity}\textup{\ref{item:upper-intensity-zero}} gives $\mathsf P=\mathsf B_x$.
    The decomposition $\mathsf P=(1-a)\delta_0+a\mathsf Q$ then gives $\mathsf Q=(1-1/k)\delta_0+(1/k)\delta_1$, so $w=k(1-1/k)^k=W_k$, a contradiction.
    Thus $x_0>0$, and \cref{lem:upper-intensity}\textup{\ref{item:upper-intensity-parameters}} gives $\lambda<a$ and $0<\eta<a$.
  \end{proof}

\subsubsection{A lower bound for the tail probability by counting upper observations}
\label{sec:conditional-success-lower}

  By counting failures with zero or one upper observation, we will have
  \[
    T_0\ge p_0(a-\lambda)
    +\frac{p_0\lambda\E(M-X^{\operatorname{U}})}A
    =p_0\eta.
  \]
  We also compare $T_0=\Prob(B_k\ge\delta+m)$ with $\Prob(B_k/A\ge1-\eta)$, using $\delta+m=A(1-\eta)+\delta(1-A)\ge A(1-\eta)$.

\begin{lemma}\label{lem:exact-upper-count}
  Let $k\ge6$ be an integer, $11/24\le c\le1/2$,
  $x=(k+c)^{-1}$, and $x<M\le1$.
  Suppose that $0\le\varepsilon\le1/10$ and $\zeta=17/20$.
  Let $\mathsf P$ be a probability distribution of mean $x$, supported on
  $[0,x+\varepsilon D]\cup(x+\zeta D,M]$, with
  $\mathsf P((x+\zeta D,M])>0$ and positive lower conditional mean $x_0$.
  Assume that $p\le p_0$, $D>3x$, and $r\le M$.
  For independent $X_1,\ldots,X_k$ with distribution $\mathsf P$,
  suppose that $\Var(X_1+\cdots+X_k)=D(r-\delta)$ and
  $\Prob(X_1+\cdots+X_{k-1}<1-y)=p(M-y)/D$ for
  $\mathsf P$-almost every $y$.
  Then the following statements hold.
  \begin{enumerate}[label=\textup{(\roman*)},leftmargin=*,itemsep=0pt]
    \item\label{item:exact-upper-count-positive} We have $0<\eta<a$ and $\mathcal V>0$.
    \item\label{item:exact-upper-count-tail} We have
    \[
      p_0\eta\le T_0
      \le\Prob\left(\frac{B_k}{A}\ge1-\eta\right).
    \]
     \end{enumerate}
     \noindent In particular, $T_0>\eta/3$.
\end{lemma}

\begin{proof}
  By \cref{lem:upper-intensity}\textup{\ref{item:upper-intensity-parameters}}, we have $\lambda<a$, $0<\eta<a$, $0<q\le D/A<1$, and $\E(M-X^{\operatorname{U}})\ge0$.
  We estimate the contributions to $\Prob(S_k<1)$ from samples with zero or one value in $(x+\zeta D,M]$.

  In the representation $X_i=(1-I_i)X_i^{\operatorname{L}}+I_iX_i^{\operatorname{U}}$, the probability that all $k$ samples are lower and $S_k<1$ is $q^k(1-T_0)$.
  The probability of exactly one upper sample is
  $k(1-q)q^{k-1}=\lambda q^{k-1}$.
  Conditional on $\sum_{i=1}^k I_i=1$, the sum has the distribution of $X^{\operatorname{U}}+X_1^{\operatorname{L}}+\cdots+X_{k-1}^{\operatorname{L}}$, with all variables independent.

  Since the coupling satisfies $X_i\ge X_i^{\operatorname{L}}$ almost surely, we have
  $\Prob(X_1^{\operatorname{L}}+\cdots+X_{k-1}^{\operatorname{L}}<1-y)\ge\Prob(S_{k-1}<1-y)$
  for every $y$.
  The hypothesis $\Prob(S_{k-1}<1-y)=p(M-y)/D$ holds for $\law(X^{\operatorname{U}})$-almost every $y$.
  Averaging over $X^{\operatorname{U}}$, independently of the lower samples, we obtain
  \[
    \Prob(X^{\operatorname{U}}+X_1^{\operatorname{L}}+\cdots+X_{k-1}^{\operatorname{L}}<1)
    \ge\frac pD\E(M-X^{\operatorname{U}}).
  \]
  The event that all $k$ samples are lower and $S_k<1$ is disjoint from the event that exactly one sample is upper and $S_k<1$, so
  \[
    p\ge q^k(1-T_0)
    +\lambda q^{k-1}\frac{p\E(M-X^{\operatorname{U}})}D.
  \]

  Since $\E S_k=a<1$, we have $p>0$.
  Dividing by $p$ and using $p\le p_0$ and $1-T_0\ge0$, we obtain
  \[
    \frac{q^k}{p_0}(1-T_0)
    +\frac{\lambda q^{k-1}\E(M-X^{\operatorname{U}})}D\le1.
  \]
  Rearranging, we have
  \[
    T_0\ge1-\frac{p_0}{q^k}
    +\frac{p_0\lambda\E(M-X^{\operatorname{U}})}{qD}.
  \]

  Since $q=1-x+(a-\lambda)/k$ and $a-\lambda>0$, the binomial
  theorem implies
  \[
    q^k-p_0\ge(a-\lambda)(1-x)^{k-1}.
  \]
  Using $q<1$, $D<A$, $(1-x)^{k-1}\ge p_0$, and
  $\E(M-X^{\operatorname{U}})\ge0$, we conclude that
  \[
    \begin{aligned}
      T_0
       & \ge
      \frac{(a-\lambda)(1-x)^{k-1}}{q^k}
      +\frac{p_0\lambda\E(M-X^{\operatorname{U}})}{qD}                           \\
       & \ge p_0\left(a-\lambda+\frac{\lambda\E(M-X^{\operatorname{U}})}A\right)
      =p_0\eta.
    \end{aligned}
  \]
  The last equality uses $m=\lambda\E J$ and $J=A-(M-X^{\operatorname{U}})$, which give $\eta=a-\lambda+\lambda\E(M-X^{\operatorname{U}})/A$.

  Since $\delta+m=A(1-\eta)+\delta(1-A)$, with $\delta>0$ and $A\le1$, we have $\delta+m\ge A(1-\eta)$.
  As $T_0=\Prob(B_k\ge\delta+m)$, this proves the upper bound in \textup{(ii)}.

  Since $c>0$, we have $p_0>(1-1/k)^k$.
  The sequence $(1-1/k)^k$ is increasing, and $(5/6)^6>1/3$.
  Thus $p_0>1/3$ for $k\ge6$, and hence $T_0>\eta/3$.

  Finally, if $\mathcal V=0$, then $X^{\operatorname{L}}=x_0$ almost surely.
  Since $x_0<x<1/k$, this would imply $T_0=0$,
  contradicting $T_0>\eta/3>0$.
  Therefore $\mathcal V>0$.
\end{proof}

\subsubsection{Bounds for the variance of lower-interval observations}
\label{sec:lower-variance}\label{sec:averaged-contact}

  We derive bounds for the normalized variance $\mathcal V=k\Var(X^{\operatorname{L}})/A^2$.
  The next lemma gives an upper bound from the lower support interval, a lower bound from the moment identity, and the estimate $\mathcal V<5\eta$ under its additional hypotheses.
\begin{lemma}\label{lem:conditional-variance-bounds}
  Let $k\ge2$ be an integer, $c>0$, $x=(k+c)^{-1}$, $x<M\le1$, and $0\le\varepsilon<\zeta<1$.
  Let $\mathsf P$ be a probability distribution of mean $x$, supported on $[0,x+\varepsilon D]\cup(x+\zeta D,M]$, with $\mathsf P((x+\zeta D,M])>0$.
  For independent $X_1,\ldots,X_k$ with distribution $\mathsf P$, assume that $\Var(X_1+\cdots+X_k)=D(r-\delta)$ and $\Prob(X_1+\cdots+X_{k-1}<1-y)=p(M-y)/D$ for $\mathsf P$-almost every $y$.
  Then the following statements hold.
  \begin{enumerate}[label=\textup{(\roman*)},leftmargin=*,itemsep=0pt]
    \item\label{item:conditional-variance-cap} Almost surely,
    \[
      -\frac{x_0}{A}
      \le\frac{X^{\operatorname{L}}-x_0}{A}
      \le\frac{x+\varepsilon D-x_0}{A}.
    \]
    Moreover, $\mathcal V\le kx_0(x+\varepsilon D-x_0)/A^2$.
    \item\label{item:conditional-variance-lower} We have
    \[
      \mathcal V
      \ge\frac{D}{Aq}
      \left[\frac{\eta}{a}+\frac{(1-c)x_0}{A}-\frac{M-r}{A}\right]
      +\frac{\lambda\E(M-X^{\operatorname{U}})}{Aq}
      \left[1-(1-\zeta)\frac DA\right].
    \]
    \item\label{item:conditional-variance-small} If, in addition, $k\ge6$, $11/24\le c\le1/2$, $r\le M$, $x_0>0$, $\lambda<a$, and $J>m+\delta-x_0/5$ almost surely, then $\mathcal V<5\eta$.
  \end{enumerate}
\end{lemma}

\begin{proof}
  Subtracting $x_0$ from $0\le X^{\operatorname{L}}\le x+\varepsilon D$ and dividing by $A>0$ proves the almost-sure bounds in \textup{(i)}.
  Also, $(X^{\operatorname{L}})^2\le(x+\varepsilon D)X^{\operatorname{L}}$ almost surely.
  Taking expectations and using $\E X^{\operatorname{L}}=x_0$, we obtain
  $\Var(X^{\operatorname{L}})=\E[(X^{\operatorname{L}})^2]-x_0^2\le x_0(x+\varepsilon D-x_0)$.
  Thus $\mathcal V=k\Var(X^{\operatorname{L}})/A^2\le kx_0(x+\varepsilon D-x_0)/A^2$.

  Since $\mathsf P=q\law(X^{\operatorname{L}})+(1-q)\law(X^{\operatorname{U}})$, $\E X^{\operatorname{L}}=x_0$, and $J=X^{\operatorname{U}}-x_0$, we have
  $\E(X_1-x_0)^2=q\Var(X^{\operatorname{L}})+(1-q)\E J^2$.
  Also, $\E(X_1-x_0)=x-x_0=m/k$.
  By independence, $\lambda=k(1-q)$, and $A^2\mathcal V=k\Var(X^{\operatorname{L}})$,
  \[
    \Var(S_k)
    =k\E(X_1-x_0)^2-\frac{m^2}{k}
    =qA^2\mathcal V+\lambda\E J^2-\frac{m^2}{k}.
  \]

  Write $\Theta=\lambda\E[(J/A)(1-J/A)]$.
  Since $\lambda\E J=m$, we have
  $\lambda\E J^2/A^2=m/A-\Theta$.
  Substituting $\Var(S_k)=D(r-\delta)$ and
  $\lambda\E J^2/A^2=m/A-\Theta$, we obtain
  $q\mathcal V=D(r-\delta)/A^2-m/A+m^2/(kA^2)+\Theta$.
  Since $M=A+x_0$ and $\delta=cx_0+cm/k$, we have
  $(r-\delta)/A=1+(1-c)x_0/A-cm/(kA)-(M-r)/A$.
  Also, $m/A-m^2/(kA^2)=Dm/A^2$, since $D=A-m/k$.
  Consequently,
  \begin{equation}\label{eq:normalized-variance-identity}
    \begin{aligned}
      q\mathcal V
       & =\frac DA\left[\eta+\frac{(1-c)x_0}{A}
      +\frac{c\eta}{k}-\frac{M-r}{A}\right]+\Theta        \\
       & =\frac DA\left[\frac{\eta}{a}+\frac{(1-c)x_0}{A}
        -\frac{M-r}{A}\right]+\Theta.
    \end{aligned}
  \end{equation}
  Here $1-(m/A)(1+c/k)=\eta(1+c/k)=\eta/a$, because
  $\eta=a-m/A$ and $a^{-1}=1+c/k$.

  Since $x+\zeta D<X^{\operatorname{U}}\le M$, we have
  $1-(1-\zeta)D/A<J/A\le1$.
  Multiplying $J/A>1-(1-\zeta)D/A$ by
  $\lambda(1-J/A)\ge0$ and taking expectations, we obtain
  \[
    \Theta
    \ge\lambda\left[1-(1-\zeta)\frac DA\right]\E(1-J/A)
    =\frac{\lambda\E(M-X^{\operatorname{U}})}A\left[1-(1-\zeta)\frac DA\right].
  \]
  The equality follows from $1-J/A=(M-X^{\operatorname{U}})/A$.
  Substituting
  $\Theta\ge\lambda\E(M-X^{\operatorname{U}})[1-(1-\zeta)D/A]/A$
  into \eqref{eq:normalized-variance-identity} and dividing by
  $q>0$ proves the lower bound for $\mathcal V$.

  We now prove $\mathcal V<5\eta$.
  Assume that $k\ge6$, $11/24\le c\le1/2$, $r\le M$,
  $x_0>0$, $\lambda<a$, and
  $J>m+\delta-x_0/5$ almost surely.
  Taking expectations in $J>m+\delta-x_0/5$ and using
  $m=\lambda\E J$, we obtain
  $(1-\lambda)\E J-\delta>-x_0/5$.
  Since $kx_0=a-\lambda\E J$, we have
  \[
    \frac{(a-\lambda)\E J}{a}-cx_0
    =(1-\lambda)\E J-\delta
    >-\frac{x_0}{5}.
  \]
  Hence $(a-\lambda)\E J>a(c-1/5)x_0$.
  Since $a-\lambda>0$ and $\E J\le A$, it follows that
  $x_0/A<(a-\lambda)/[a(c-1/5)]$.
  Moreover, $0<m/A=\lambda\E J/A\le\lambda<a$ implies
  $0<a-\lambda\le\eta<a$.

  Since $(J-A)^2\ge0$, we have $J^2\ge2AJ-A^2$.
  Taking expectations and using $\lambda\E J=m$, we obtain
  $\lambda\E J^2/A^2\ge2m/A-\lambda$.
  Therefore,
  $\Theta=m/A-\lambda\E J^2/A^2\le\lambda-m/A=\eta-(a-\lambda)$.
  Since $M-r\ge0$, $D/A\le1$, $\eta>0$, and $1-c>0$,
  \eqref{eq:normalized-variance-identity} implies
  \[
    q\mathcal V
    \le\left(2+\frac ck\right)\eta-(a-\lambda)+\frac{(1-c)x_0}{A}.
  \]

  For $11/24\le c\le1/2$ and $a<1$, we have
  $(1-c)/[a(c-1/5)]-1>0$.
  Substituting $x_0/A<(a-\lambda)/[a(c-1/5)]$ and then
  using $a-\lambda\le\eta$, we obtain
  \[
    \mathcal V
    <\frac{1+(1-c)/[a(c-1/5)]+c/k}{q}\eta.
  \]
  Since $k\ge6$ and $11/24\le c\le1/2$, we have
  $a\ge12/13$, $(1-c)/[a(c-1/5)]<5/2$, and $c/k\le1/12$.
  Also, $\lambda<a<1$ implies $q=1-\lambda/k>5/6$.
  Since $(6/5)(1+5/2+1/12)<5$, we conclude that $\mathcal V<5\eta$.
\end{proof}

  The estimate $\mathcal V<5\eta$ gives a good upper bound only when $\eta$ is small.
  We therefore obtain another upper bound for $\mathcal V$ by estimating
  $\Prob(X^{\operatorname{U}}+X_1^{\operatorname{L}}+\cdots+X_{k-1}^{\operatorname{L}}<1)$
  from above using $\Prob(S_{k-1}<1-y)=p(M-y)/D$ and from below using a normal approximation.
Write $Q=\E(J-m-\delta)_+/A$.

\begin{lemma}\label{lem:averaged-contact-normal}
  Let $k\ge6$ be an integer, $11/24\le c\le1/2$, $x=(k+c)^{-1}$, and $x<M\le1$.
  Suppose that $0\le\varepsilon\le1/10$ and $\zeta=17/20$.
  Let $\mathsf P$ be a probability distribution of mean $x$, supported on $[0,x+\varepsilon D]\cup(x+\zeta D,M]$, with $\mathsf P((x+\zeta D,M])>0$ and positive lower conditional mean $x_0$.
  Assume that $p\le p_0$, $D>3x$, and $r\le M$.
  For independent $X_1,\ldots,X_k$ with distribution $\mathsf P$, suppose that $\Var(X_1+\cdots+X_k)=D(r-\delta)$ and $\Prob(X_1+\cdots+X_{k-1}<1-y)=p(M-y)/D$ for $\mathsf P$-almost every $y$.
  If $B_*>0$ and $|X^{\operatorname{L}}-x_0|\le AB_*$ almost surely, then
  \[
    \left(\frac12-\frac{p_0\E(M-X^{\operatorname{U}})}{Dq^{k-1}}\right)
    \sqrt{\frac{k-1}{k}\mathcal V}
    \le\frac25Q+\frac{B_*}{3}
    +\frac17\sqrt{\frac{\mathcal V}{k}}.
  \]
  Moreover, $Q<\eta/a-\E(M-X^{\operatorname{U}})/A-(c-1/5)x_0/A$.
\end{lemma}

\begin{proof}
  Take $X^{\operatorname{U}}$ independently of $X_1^{\operatorname{L}},\ldots,X_{k-1}^{\operatorname{L}}$.
  We first estimate $\Prob(B_{k-1}+J<\delta+m)$ from above.
  In the representation $X_i=(1-I_i)X_i^{\operatorname{L}}+I_iX_i^{\operatorname{U}}$, the event that all $k-1$ observations are lower has probability $q^{k-1}$.
  Conditional on $I_1=\cdots=I_{k-1}=0$, their sum has the distribution of $X_1^{\operatorname{L}}+\cdots+X_{k-1}^{\operatorname{L}}$.
    {Thus, for each fixed $y$, the probability that all $k-1$ observations are lower and their sum is less than $1-y$ is at most $\Prob(S_{k-1}<1-y)$.}
  The hypothesis $\Prob(S_{k-1}<1-y)=p(M-y)/D$ holds for $\law(X^{\operatorname{U}})$-almost every $y$.
  Consequently,
  \[
    q^{k-1}\Prob(X_1^{\operatorname{L}}+\cdots+X_{k-1}^{\operatorname{L}}<1-y)
    \le p\frac{M-y}{D}.
  \]

  Since $X^{\operatorname{U}}=x_0+J$ and $1-kx_0=\delta+m$, the event $X^{\operatorname{U}}+X_1^{\operatorname{L}}+\cdots+X_{k-1}^{\operatorname{L}}<1$ is equivalent to $B_{k-1}+J<\delta+m$.
  Dividing the preceding inequality by $q^{k-1}$ and averaging over $X^{\operatorname{U}}$, we obtain
  \begin{equation}\label{eq:conditional-failure-average}
    \Prob(B_{k-1}+J<\delta+m)
    \le\frac{p\E(M-X^{\operatorname{U}})}{Dq^{k-1}}
    \le\frac{p_0\E(M-X^{\operatorname{U}})}{Dq^{k-1}}.
  \end{equation}
  The last inequality follows from $p\le p_0$ and $\E(M-X^{\operatorname{U}})\ge0$.

  We now estimate the same probability from below.
  By \cref{lem:exact-upper-count}\textup{\ref{item:exact-upper-count-positive}}, we have $\mathcal V>0$.
  Write $\sigma^2=\Var(B_{k-1}/A)=(k-1)\mathcal V/k$.
  {The variables $(X_i^{\operatorname{L}}-x_0)/A$} are independent, identically distributed, and have mean zero.
  Since their absolute values are at most $B_*$, we have
  \[
    (k-1)\E\left|\frac{X^{\operatorname{L}}-x_0}{A}\right|^3
    \le B_*(k-1)\E\left(\frac{X^{\operatorname{L}}-x_0}{A}\right)^2
    =B_*\sigma^2.
  \]
  By \cref{lem:normal-comparison}, the absolute difference between $\Prob(B_{k-1}/A<s)$ and $\Phi(s/\sigma)$ is at most $B_*/(3\sigma)+1/(7\sqrt{k-1})$ for every real $s$.
  Here $\Phi$ denotes the standard normal distribution function.

  Since $J$ is independent of the lower observations, conditioning on $J$ does not change the distribution of $B_{k-1}/A$.
  We may therefore apply this inequality at $s=(\delta+m-J)/A$ and then take expectations over $J$.
    {To estimate $\E\Phi((\delta+m-J)/(A\sigma))$, note that} $\Phi(-t)\ge1/2-(2/5)t_+$ for every real $t$, since $\Phi'(u)=(2\pi)^{-1/2}\mathrm{e}^{-u^2/2}<2/5$ for every $u\in\mathbb R$.
  Applying this inequality with $t=(J-m-\delta)/(A\sigma)$ and using the definition of $Q$, we obtain $\E\Phi((\delta+m-J)/(A\sigma))\ge1/2-2Q/(5\sigma)$.
  Together with \eqref{eq:conditional-failure-average}, these estimates imply
  \[
    \begin{aligned}
      \frac{p_0\E(M-X^{\operatorname{U}})}{Dq^{k-1}}
       & \ge
      \E\Phi\left(\frac{\delta+m-J}{A\sigma}\right)
      -\frac{B_*}{3\sigma}-\frac1{7\sqrt{k-1}} \\
       & \ge
      \frac12-\frac{2Q}{5\sigma}
      -\frac{B_*}{3\sigma}-\frac1{7\sqrt{k-1}}.
    \end{aligned}
  \]
  Rearranging and multiplying by $\sigma>0$ proves the first assertion, since $\sigma=\sqrt{(k-1)\mathcal V/k}$ and $\sigma/\sqrt{k-1}=\sqrt{\mathcal V/k}$.

  It remains to prove $Q<\eta/a-\E(M-X^{\operatorname{U}})/A-(c-1/5)x_0/A$.
  We first compute $\E(J-m-\delta)$, and then estimate the contribution from the event $J<m+\delta$.
  Since $x=x_0+m/k$ and $\delta=cx$, we have $\delta/A=cx_0/A+cm/(kA)$.
  Using $\E J=A-\E(M-X^{\operatorname{U}})$, $a^{-1}=1+c/k$, and $\eta=a-m/A$, we obtain
  \[
    \begin{aligned}
      \frac{\E(J-m-\delta)}A
       & =1-\frac{\E(M-X^{\operatorname{U}})}A-\frac mA-\frac{\delta}{A}      \\
       & =1-\frac{\E(M-X^{\operatorname{U}})}A-\frac{cx_0}{A}
      -\frac mA\left(1+\frac ck\right)                                        \\
       & =\frac{\eta}{a}-\frac{\E(M-X^{\operatorname{U}})}A-\frac{cx_0}{A}.
    \end{aligned}
  \]

  By \cref{lem:upper-intensity}\textup{\ref{item:upper-intensity-parameters}}, we have $J>m+\delta-x_0/5$ almost surely.
  Since $x_0>0$, it follows that $(m+\delta-J)_+<x_0/5$ almost surely, and hence $\E(m+\delta-J)_+<x_0/5$.
  Applying the identity $s_+=s+(-s)_+$ to $s=J-m-\delta$ and taking expectations, we conclude that
  \[
    Q
    =\frac{\E(J-m-\delta)}A
    +\frac{\E(m+\delta-J)_+}{A}
    <\frac{\eta}{a}-\frac{\E(M-X^{\operatorname{U}})}A-\left(c-\frac15\right)\frac{x_0}{A}.
  \]
\end{proof}

\section{\texorpdfstring{Determining $\Phi_k(x)$}{Determining Phi(k,x)}}\label{sec:exclusion}

  In this section, we prove the upper bound in \cref{thm:iid-tail}.
  By \cref{lem:dimension-one,lem:low-dimensional-crossing,prop:mean-transfer}, it remains to treat $k\ge6$ at $x=\beta_k$.
  We exclude $0<w<W_k$ for every maximizing distribution and then use the endpoint bounds and \cref{prop:mean-transfer} to determine $\Phi_k(x)$ for every $x\ge0$.
  Fix $k\ge6$ and $x=\beta_k$, and suppose for a contradiction that $\mathsf P$ is a maximizing distribution with $0<w<W_k$.
By \cref{prop:support-separation}, the conditional decomposition in \cref{sec:conditional-laws} applies with $\varepsilon=1/16$ and $\zeta=17/20$.
For independent copies $X_1^{\operatorname{L}},\ldots,X_k^{\operatorname{L}}$ of the lower conditional variable, recall that $B_k=\sum_{i=1}^k(X_i^{\operatorname{L}}-x_0)$ and $\mathcal V=k\Var(X^{\operatorname{L}})/A^2$.
If $x_0=0$, then $\mathsf P=\mathsf B_x$ by \cref{lem:upper-intensity}\textup{\ref{item:upper-intensity-zero}}, and hence $w=W_k$, a contradiction.
Thus, $x_0>0$, and \cref{lem:upper-intensity}\textup{\ref{item:upper-intensity-parameters}} gives $\lambda<a$ and $J>m+\delta-x_0/5$ almost surely.
Consequently, \cref{lem:exact-upper-count,lem:conditional-variance-bounds} give
\[
  0<\eta<a<1,\text{ }T_0>\frac{\eta}{3}\text{ and }0<\mathcal V<5\eta.
\]
Also, \cref{lem:exact-upper-count}\textup{\ref{item:exact-upper-count-tail}} gives $T_0\le\Prob(B_k/A\ge1-\eta)$.

The essential supremum of $X^{\operatorname{L}}$ is its least almost-sure upper bound, given by
\[
  \esssup X^{\operatorname{L}}
  =\inf\{b\in\mathbb R:\Prob(X^{\operatorname{L}}\le b)=1\}.
\]
We first prove $\eta/a<1/10$, which implies $\eta<1/10$ and $1-\eta>9/10$ since $a<1$.
Then in \cref{lem:conditional-support}, we prove $\esssup X^{\operatorname{L}}<A/10$ and $\mathcal V<1/25$.
Consequently, $|X_i^{\operatorname{L}}-x_0|/A<1/10$ almost surely for every $i$.
Applying \cref{lem:small-conditional-tail} to $(X_i^{\operatorname{L}}-x_0)/A$ and using \cref{lem:exact-upper-count}\textup{\ref{item:exact-upper-count-tail}}, we obtain
\[
  T_0
  \le\Prob\left(\frac{B_k}{A}\ge1-\eta\right)
  <\frac{\eta}{3},
\]
contradicting the bound $T_0>\eta/3$ from \cref{lem:exact-upper-count}.

\subsection{Incompatible variance bounds from moments and normal approximation}
\label{sec:large-defect-exclusion}\label{sec:tail-estimates}

We prove $\eta/a<1/10$ by showing that the lower variance bound in \cref{lem:conditional-variance-bounds}\textup{\ref{item:conditional-variance-lower}} and the normal approximation inequality in \cref{lem:averaged-contact-normal} are incompatible when $\eta/a\ge1/10$.
By bounding $p_0A/(Dq^{k-1})$ from above, we obtain a positive lower bound for the coefficient of $\sqrt{\mathcal V}$ in \cref{lem:averaged-contact-normal}.

\begin{lemma}\label{lem:large-defect}
  Let $k\ge6$ be an integer, $11/24\le c\le1/2$, $x=(k+c)^{-1}$, and $x<M\le1$.
  Set $D=M-x$, $\varepsilon=1/16$, and $\zeta=17/20$.
  Suppose that $\mathsf P$ is a probability distribution of mean $x$ such that
  \[
    \mathsf P\bigl([0,x+D/16]\cup(x+17D/20,M]\bigr)=1
    \text{ and }\mathsf P((x+17D/20,M])>0.
  \]
  Suppose also that $x_0>0$.
  Assume that $p\le p_0$, $D>3x$, $r\le M$, and $(M-r)/D<1/20-3x/(20D)$, and that $\Var(S_k)=D(r-\delta)$ and $\Prob(S_{k-1}<1-y)=p(M-y)/D$ for $\mathsf P$-almost every $y$.
  Then $\eta/a<1/10$.
\end{lemma}

\begin{proof}
  Suppose that $\eta/a\ge1/10$, and set $t=1/k$ and $\rho=3/20$.
  By \cref{lem:upper-intensity}\textup{\ref{item:upper-intensity-parameters}} and the assumptions on $k$ and $c$, we have $0<t\le1/6$, $1/10\le\eta/a<1$, and $12/13\le a<1$.
  Also, $x_0<x$ and $A>D>3x$ imply $0<x_0/A<1/3$.
  We have $m/A=a(1-\eta/a)$ and $D/A=1-tm/A$.
  Set $\theta=\E(M-X^{\operatorname{U}})/(\rho D)$.
  Since $0\le M-X^{\operatorname{U}}<\rho D$ almost surely, we have $0\le\theta<1$.
  Also, $(x+D/16-x_0)/A=1/16+15tm/(16A)$.

  We will derive
  \[
    (1-\theta)P_0(t,\eta/a,x_0/A)
    +\theta P_1(t,\eta/a,x_0/A)<0,
  \]
  where $P_0$ and $P_1$ will be defined in \eqref{eq:large-defect-endpoint-functions}.
  However, \cref{clm:large-defect-endpoint-positive} gives $P_0(t,\eta/a,x_0/A)>0$ and $P_1(t,\eta/a,x_0/A)>0$.
  Since $0\le\theta<1$, their weighted sum is positive, a contradiction.

  To estimate $\mathcal V$ from below, set
  \[
    \mathcal V_0=\frac{\eta}{a}-\frac1{20}+\frac{13x_0}{20A}
    +\frac{t(1-\eta/a)}6,
    \text{ and }\mathcal V_1=\mathcal V_0+\frac{1-\eta/a}{9}.
  \]
  We have $\mathcal V_1\ge \mathcal V_0\ge1/20$.
  Since $x/A=x_0/A+tm/A$ and $D/A=1-tm/A$, the assumption $(M-r)/D<1/20-3x/(20D)$ implies
  \[
    \begin{aligned}
      \frac{\eta}{a}+\frac{(1-c)x_0}{A}-\frac{M-r}{A}
       & >\frac{\eta}{a}-\frac1{20}
      +\left(1-c+\frac3{20}\right)\frac{x_0}{A}
      +\frac{tm}{5A}                  \\
       & =\frac{\eta}{a}-\frac1{20}
      +\left(1-c+\frac3{20}\right)\frac{x_0}{A}
      +\frac{at(1-\eta/a)}5
      \ge \mathcal V_0.
    \end{aligned}
  \]
  The last inequality uses $c\le1/2$ and $a/5\ge12/65>1/6$.
  Using $q\le D/A$, $\lambda\ge m/A$, and $\E(M-X^{\operatorname{U}})=\theta\rho D$ in \cref{lem:conditional-variance-bounds}\textup{\ref{item:conditional-variance-lower}}, we obtain
  \begin{equation}\label{eq:large-defect-variance-lower}
    \begin{aligned}
      \mathcal V
       & >\frac{D\mathcal V_0}{Aq}
      +\frac{\lambda\E(M-X^{\operatorname{U}})}{Aq}\left(1-\frac{\rho D}{A}\right) \\
       & \ge \mathcal V_0+\rho(1-\rho)\frac mA\theta
      \ge(1-\theta)\mathcal V_0+\theta \mathcal V_1.
    \end{aligned}
  \end{equation}
  The last inequality follows from $m/A=a(1-\eta/a)$ and $\rho(1-\rho)a>1/9$.

  We next estimate $Q$ in \cref{lem:averaged-contact-normal}.
  We have $0\le(M-X^{\operatorname{U}})/A<\rho D/A$, $\E(M-X^{\operatorname{U}})/A=\theta\rho D/A$, and $Q=\E(\eta/a-cx_0/A-(M-X^{\operatorname{U}})/A)_+$.
  Set $Q_0=(\eta/a-11x_0/(24A))_+$ and $Q_1=(\eta/a-11x_0/(24A)-\rho+\rho t(1-\eta/a))_+$.
  By convexity of $s\mapsto(\eta/a-cx_0/A-s)_+$ on $[0,\rho D/A]$ and the equality $\E(M-X^{\operatorname{U}})/(\rho D)=\theta$, we obtain
  \[
    \begin{aligned}
      Q & \le(1-\theta)\left(\frac{\eta}{a}-\frac{cx_0}{A}\right)_+
      +\theta\left(\frac{\eta}{a}-\frac{cx_0}{A}-\frac{\rho D}{A}\right)_+ \\
        & \le(1-\theta)Q_0+\theta Q_1.
    \end{aligned}
  \]
  The second inequality uses $c\ge11/24$ and $\rho D/A\ge\rho-\rho t(1-\eta/a)$.

  To bound $B_*/3+(1/7)\sqrt{\mathcal V/k}$ in \cref{lem:averaged-contact-normal}, write
  \[
    E=\frac1{16}+\frac{15t(1-\eta/a)}{16},
    \text{ and }R=\frac{\max\{x_0/A,E\}}3+\frac{x_0/A+E}{14}.
  \]
  By \cref{lem:conditional-variance-bounds}\textup{\ref{item:conditional-variance-cap}}, we may take $B_*=\max\{x_0/A,(x+D/16-x_0)/A\}$ in \cref{lem:averaged-contact-normal}.
  We have $(x+D/16-x_0)/A\le E$, and $\mathcal V\le kx_0(x+D/16-x_0)/A^2$ by \cref{lem:conditional-variance-bounds}\textup{\ref{item:conditional-variance-cap}}.
  The arithmetic--geometric mean inequality therefore yields
  \[
    \begin{aligned}
      \frac{B_*}{3}+\frac17\sqrt{\frac{\mathcal V}{k}}
       & \le\frac{\max\{x_0/A,E\}}3
      +\frac{\sqrt{x_0(x+D/16-x_0)}}{7A} \\
       & \le\frac{\max\{x_0/A,E\}}3
      +\frac17\sqrt{\frac{x_0}{A}E}      \\
       & \le\frac{\max\{x_0/A,E\}}3
      +\frac{x_0/A+E}{14}=R.
    \end{aligned}
  \]

  Define $\beta(s)=\sqrt{(1+s)/11}$ for $1/10\le s\le1$.
  We will prove
  \[
    \frac12-\rho\frac{p_0A}{Dq^{k-1}}>\beta(\eta/a).
  \]
  By \cref{lem:upper-intensity}\textup{\ref{item:upper-intensity-ratio}}, $p_0A/(Dq^{k-1})\le\mathrm{e}^{-(a-\lambda)}\le1$.
  Since $a\ge12/13$ and
  \[
    \mathrm{e}^{12/13}
    >1+\frac{12}{13}+\frac12\left(\frac{12}{13}\right)^2
    =\frac{397}{169}>\frac94,
  \]
  we have $\mathrm{e}^{-a}<4/9$.
  Convexity of $z\mapsto\mathrm{e}^{-z}$ on $[0,a]$ and $0<a-\lambda<a$ give
  \[
    \frac{p_0A}{Dq^{k-1}}
    \le\mathrm{e}^{-(a-\lambda)}
    \le1-(1-\mathrm{e}^{-a})\frac{a-\lambda}{a}
    <1-\frac{5(a-\lambda)}{9a}.
  \]
  Since $m=\lambda\E J$, $J=A-(M-X^{\operatorname{U}})$, and $\eta=a-m/A$, we have
  \[
    \frac{a-\lambda}{a}
    =\frac{\eta/a-\E(M-X^{\operatorname{U}})/A}{1-\E(M-X^{\operatorname{U}})/A}.
  \]
  If $\eta/a\ge\rho$, then $\E(M-X^{\operatorname{U}})/A<\rho$ implies $(a-\lambda)/a\ge(\eta/a-\rho)/(1-\rho)$, and hence
  \[
    \begin{aligned}
      \frac12-\rho\frac{p_0A}{Dq^{k-1}}
       & \ge\frac7{20}+\frac5{51}\left(\frac{\eta}{a}-\frac3{20}\right) \\
       & =\frac13+\frac{\eta}{10a}+\frac{1-\eta/a}{510}
      \ge\frac13+\frac{\eta}{10a}.
    \end{aligned}
  \]
  If $\eta/a\le\rho$, then $p_0A/(Dq^{k-1})\le1$ gives $1/2-\rho p_0A/(Dq^{k-1})\ge7/20>1/3+\eta/(10a)$.
  For $1/10\le s\le1$, the arithmetic--geometric mean inequality and $\sqrt{2/11}<13/30$ give
  \[
    \beta(s)
    =\sqrt{\frac2{11}}\sqrt{\frac{1+s}{2}}
    \le\sqrt{\frac2{11}}\frac{3+s}{4}
    <\frac{13(3+s)}{120}
    \le\frac13+\frac{s}{10}.
  \]
  The last inequality is equivalent to $s\le1$.
  Thus, $1/2-\rho p_0A/(Dq^{k-1})>\beta(\eta/a)$, and $0<\beta(s)<1/2$ for $1/10\le s\le1$.
  Since $\E(M-X^{\operatorname{U}})=\rho D\theta$ and $D/A<1$, we obtain
  \[
    \begin{aligned}
      \frac12-\frac{p_0\E(M-X^{\operatorname{U}})}{Dq^{k-1}}
       & =\frac12-\rho\theta\frac DA\frac{p_0A}{Dq^{k-1}} \\
       & \ge\frac{1-\theta}{2}
      +\theta\left(\frac12-\rho\frac{p_0A}{Dq^{k-1}}\right)
      \ge\frac{1-\theta}{2}+\theta\beta(\eta/a)>0.
    \end{aligned}
  \]
  We may therefore substitute \eqref{eq:large-defect-variance-lower} into the left-hand side of the inequality in \cref{lem:averaged-contact-normal}.
  Using also $Q\le(1-\theta)Q_0+\theta Q_1$ and $B_*/3+(1/7)\sqrt{\mathcal V/k}\le R$, we obtain
  \begin{equation}\label{eq:large-defect-normal}
    \begin{aligned}
       & \left(\frac{1-\theta}{2}+\theta\beta(\eta/a)\right)
      \sqrt{(1-t)\bigl((1-\theta)\mathcal V_0+\theta \mathcal V_1\bigr)} \\
       & \qquad<\frac25\bigl((1-\theta)Q_0+\theta Q_1\bigr)+R.
    \end{aligned}
  \end{equation}
  Define
  \begin{equation}\label{eq:large-defect-endpoint-functions}
    \begin{aligned}
      P_0(t,\eta/a,x_0/A)
       & =\frac12\sqrt{(1-t)\mathcal V_0}-\frac25Q_0-R,         \\
      P_1(t,\eta/a,x_0/A)
       & =\beta(\eta/a)\sqrt{(1-t)\mathcal V_1}-\frac25Q_1-R.
    \end{aligned}
  \end{equation}
  Since $0<\beta(\eta/a)<1/2$ and $\mathcal V_1\ge \mathcal V_0$, concavity of the square root gives
  \[
    \begin{aligned}
       & \left(\frac{1-\theta}{2}+\theta\beta(\eta/a)\right)
      \sqrt{(1-\theta)\mathcal V_0+\theta \mathcal V_1}                   \\
       & \quad\ge\left(\frac{1-\theta}{2}+\theta\beta(\eta/a)\right)
      \bigl((1-\theta)\sqrt{\mathcal V_0}+\theta\sqrt{\mathcal V_1}\bigr) \\
       & \quad=\frac{1-\theta}{2}\sqrt{\mathcal V_0}
      +\theta\beta(\eta/a)\sqrt{\mathcal V_1}                           \\
       & \qquad\quad+\theta(1-\theta)\bigl(1/2-\beta(\eta/a)\bigr)
      \bigl(\sqrt{\mathcal V_1}-\sqrt{\mathcal V_0}\bigr)                 \\
       & \quad\ge\frac{1-\theta}{2}\sqrt{\mathcal V_0}
      +\theta\beta(\eta/a)\sqrt{\mathcal V_1}.
    \end{aligned}
  \]
  After multiplication by $\sqrt{1-t}$, \eqref{eq:large-defect-normal} yields
  \begin{equation}\label{eq:large-defect-contact-obstruction}
    (1-\theta)P_0(t,\eta/a,x_0/A)
    +\theta P_1(t,\eta/a,x_0/A)<0.
  \end{equation}
  We extend the defining formulas for $P_0,P_1$ and their auxiliary functions by allowing $t$, $\eta/a$, and $x_0/A$ to vary independently over $[0,1/6]\times[1/10,1]\times[0,1/3]$.
  \begin{claim}\label{clm:large-defect-endpoint-positive}
    The functions $P_0$ and $P_1$ defined in \eqref{eq:large-defect-endpoint-functions} are positive throughout $[0,1/6]\times[1/10,1]\times[0,1/3]$.
  \end{claim}
  The proof of \cref{clm:large-defect-endpoint-positive} is given in \cref{app:large-defect-endpoint-positive}.
  Since $0\le\theta<1$, \cref{clm:large-defect-endpoint-positive} contradicts \eqref{eq:large-defect-contact-obstruction}.
  Therefore, $\eta/a<1/10$.
\end{proof}

  \begin{corollary}\label{cor:small-omega}
    Let $k\ge6$ and $x=\beta_k$, and let $\mathsf P$ maximize $F_k$ among nonnegative Borel laws of mean at most $x$.
    If $0<w<W_k$, then the conditional decomposition with $\varepsilon=1/16$ and $\zeta=17/20$ satisfies $0<\eta<a/10<1/10$.
  \end{corollary}
  \begin{proof}
    By \cref{cor:positive-lower-mean}, we have $x_0>0$ and $\eta>0$.
    The moment identities in \cref{sec:contact}, \cref{prop:finite-shortfall-margin,prop:support-separation}, and the bounds on $c_k$ verify the remaining hypotheses of \cref{lem:large-defect}.
    Hence $\eta/a<1/10$.
    Since $a=k/(k+c_k)<1$, we obtain $0<\eta<a/10<1/10$.
  \end{proof}

\subsection{\texorpdfstring{{Upper bounds for $\esssup X^{\operatorname{L}}$ and $\mathcal V$}}{Upper bounds for ess sup X(L) and V}}
\label{sec:small-defect-exclusion}

We first prove $b_0:=\esssup X^{\operatorname{L}}<A/10$.
Since $0\le X^{\operatorname{L}}\le b_0$ almost surely and $x_0=\E X^{\operatorname{L}}\in[0,b_0]$, this will also give $|X^{\operatorname{L}}-x_0|\le b_0<A/10$ almost surely.
The bound $|X^{\operatorname{L}}-x_0|<A/10$ allows us to take $B_*=1/10$ in \cref{lem:averaged-contact-normal} and prove $\mathcal V<1/25$.

  \begin{lemma}\label{lem:conditional-support}
    Let $k\ge6$ be an integer, $11/24\le c\le1/2$, $x=(k+c)^{-1}$, $x<M\le1$, and $0\le\varepsilon\le1/10$, and set $D=M-x$ and $\zeta=17/20$.
    Let $\mathsf P$ be a probability law of mean $x$, supported on $[0,x+\varepsilon D]\cup(x+17D/20,M]$, with positive mass on the upper interval.
    Use the conditional notation from \cref{sec:conditional-laws}, and set $\mathcal V=k\Var(X^{\operatorname{L}})/A^2$.
    For independent $X_1,\ldots,X_k$ with law $\mathsf P$, write $h(t)=\Prob(S_{k-1}+t\ge1)$.
    Assume that $p\le p_0$, $0\le M-r<(D-3x)/20$, and $h(y)-h(0)=py/D$ for $\mathsf P$-almost every $y$.
    If $0<\eta\le1/10$, then $\esssup X^{\operatorname{L}}<A/10$ and $\mathcal V<1/25$.
  \end{lemma}

\begin{proof}
  The assumption $0\le M-r<(D-3x)/20$ implies $r\le M$ and $D>3x$.
  Also, $x_0>0$: otherwise $m=a$ and $A=M\le1$, so $\eta=a-a/M\le0$, contrary to $\eta>0$.
  Taking expectations in $h(X_1)-h(0)=pX_1/D$ gives
  \[
    1-p-h(0)=\E[h(X_1)-h(0)]=\frac{px}{D}.
  \]
  Since $D+x=M$, it follows that
  \[
    1-h(0)=\frac{pM}{D},\text{ and }
    \Prob(S_{k-1}<1-y)=\frac{p(M-y)}D
    \quad\text{for $\mathsf P$-almost every $y$.}
  \]
    The contact identity gives, for every bounded Borel function $\varphi$,
    \[
      \E[\varphi(X_1)\mid S_k<1]
      =\frac{\E[(M-X_1)\varphi(X_1)]}{D}.
    \]
    Taking $\varphi(y)=y$ and using $\E[X_1\mid S_k<1]=(1-r)/k$ yields $\Var(S_k)=D(r-\delta)$.
    Write $v=\Var(S_k\mid S_k<1)$, and let $X,X'$ be independent with law $\mathsf P$.
    The covariance calculation in the proof of \cref{lem:contact-moment} uses only independence and these identities, and therefore yields
    \[
      v=(r-\delta)(M-r)
      -\frac{k}{2p}\E\left[(X-X')\int_{X'}^X(h(t)-h(0))\dd t\right].
    \]
  Applying \cref{lem:upper-intensity}\textup{\ref{item:upper-intensity-parameters}}, we obtain $\lambda<a<1$ and $0\le\E(M-X^{\operatorname{U}})/A<3/20$.
  Since $k\ge6$, $a\ge12/13$, and $\eta\le1/10$, we have $q=1-\lambda/k>5/6$ and $m/A=a-\eta\ge12/13-1/10>4/5$.
  Since $x=qx_0+(1-q)\E X^{\operatorname{U}}$, we also have $\lambda(\E X^{\operatorname{U}}-x_0)=m$.

  Set $b_0=\esssup X^{\operatorname{L}}\ge x_0>0$.
  Since $v\ge0$, we have
  \[
    \frac{k}{2p}\E\left[
      (X-X')\int_{X'}^X\bigl(h(t)-h(0)\bigr)\dd t
      \right]\le(r-\delta)(M-r).
  \]
  We will prove $\frac12Ab_0<(r-\delta)(M-r)<A^2/20$, and hence $b_0<A/10$.
  Since $D<A$ and $x_0<x$, we have
  \[
    0\le M-r<\frac{D-3x}{20}<\frac{A-3x_0}{20}.
  \]
  Consequently, $A>3x_0$.
  Since $r-\delta=\Var(S_k)/D\ge0$ and $r-\delta\le M=A+x_0$, we have
  \[
    (r-\delta)(M-r)
    <\frac{(A+x_0)(A-3x_0)}{20}<\frac{A^2}{20}.
  \]

  For each integer $n\ge1$, choose $y_n\in[0,x+\varepsilon D]\cap(b_0-1/n,b_0]$ such that $h(y_n)-h(0)=py_n/D$.
  Such a choice is possible because $\Prob(X^{\operatorname{L}}>b_0-1/n)>0$ and $h(y)-h(0)=py/D$ for almost every $y$ under the law of $X^{\operatorname{L}}$.
  For every $t>b_0$, monotonicity gives $h(t)-h(0)\ge py_n/D$ for all $n$, and letting $n\to\infty$ yields $h(t)-h(0)\ge pb_0/D$.
  Since $h(t)-h(0)\ge0$ for $t\ge0$, independent variables $X^{\operatorname{L}}$ and $X^{\operatorname{U}}$ satisfy
  \[
    \int_{X^{\operatorname{L}}}^{X^{\operatorname{U}}}(h(t)-h(0))\dd t\ge\frac{pb_0}{D}(X^{\operatorname{U}}-b_0),
    \text{ and }
    X^{\operatorname{U}}-b_0>\left(\frac{17}{20}-\varepsilon\right)D
    \ge\frac{3D}{4}.
  \]
  The product $(X-X')\int_{X'}^X(h(t)-h(0))\dd t$ is nonnegative and unchanged when $X$ and $X'$ are interchanged.
  We retain only the pairs satisfying $X\in[0,x+\varepsilon D]$ and $X'\in(x+17D/20,M]$, together with the pairs satisfying the reversed conditions.
  The two events have total probability $2q(1-q)$.
  Using $\lambda=k(1-q)$, $\E(X^{\operatorname{U}}-X^{\operatorname{L}})=\E X^{\operatorname{U}}-x_0$, and $\lambda(\E X^{\operatorname{U}}-x_0)=m$, we obtain
  \[
    \begin{aligned}
      \frac12Ab_0
       & <\frac34mqb_0
      <\frac{\lambda qb_0}{D}\E[(X^{\operatorname{U}}-X^{\operatorname{L}})(X^{\operatorname{U}}-b_0)] \\
       & \le(r-\delta)(M-r)<\frac{A^2}{20}.
    \end{aligned}
  \]
  The first inequality uses $m/A>4/5$ and $q>5/6$, and the second uses $X^{\operatorname{U}}-b_0>3D/4$.
  Thus, $b_0<A/10$.

  To prove $\mathcal V<1/25$, we apply \cref{lem:averaged-contact-normal} with $B_*=1/10$, since $0\le X^{\operatorname{L}},x_0\le b_0<A/10$.
  In particular, $Q<\eta/a-\E(M-X^{\operatorname{U}})/A-(c-1/5)x_0/A<\eta/a-\E(M-X^{\operatorname{U}})/A\le13/120-\E(M-X^{\operatorname{U}})/A$.
  By \cref{lem:upper-intensity}\textup{\ref{item:upper-intensity-ratio}}, we have $p_0A/(Dq^{k-1})\le1$, so $p_0\E(M-X^{\operatorname{U}})/(Dq^{k-1})\le\E(M-X^{\operatorname{U}})/A$.
  Using $\sqrt{(k-1)/k}>9/10$ and $1/(7\sqrt{k})<3/50$ in \cref{lem:averaged-contact-normal}, we obtain
  \[
    \frac9{10}\left(\frac12-\frac{\E(M-X^{\operatorname{U}})}A\right)\sqrt{\mathcal V}
    <\frac25\left(\frac{13}{120}-\frac{\E(M-X^{\operatorname{U}})}A\right)
    +\frac1{30}+\frac3{50}\sqrt{\mathcal V}.
  \]
  Rearranging gives
  \begin{equation}\label{eq:small-defect-variance-comparison}
    \begin{aligned}
      \left(\frac{39}{100}-\frac{9\E(M-X^{\operatorname{U}})}{10A}\right)\sqrt{\mathcal V}
       & <\frac{23}{300}-\frac{2\E(M-X^{\operatorname{U}})}{5A}                       \\
       & <\frac15\left(\frac{39}{100}-\frac{9\E(M-X^{\operatorname{U}})}{10A}\right).
    \end{aligned}
  \end{equation}
  The second inequality follows because the difference between its right- and left-hand sides is $1/750+11\E(M-X^{\operatorname{U}})/(50A)>0$.
  As $\E(M-X^{\operatorname{U}})/A<3/20$, the coefficient $39/100-9\E(M-X^{\operatorname{U}})/(10A)$ is positive.
  Dividing \eqref{eq:small-defect-variance-comparison} by this coefficient yields $\sqrt{\mathcal V}<1/5$, and hence $\mathcal V<1/25$.
\end{proof}

\subsection{An upper tail bound for sums of independent bounded mean-zero variables}
\label{sec:conditional-success-upper}

The bounds $|X^{\operatorname{L}}-x_0|<A/10$ and $\mathcal V<\min\{5\eta,1/25\}$ now yield an upper bound for $T_0$.
The next estimate applies to independent variables, which need not be identically distributed.

\begin{lemma}\label{lem:small-conditional-tail}
  Let $k\ge1$ be an integer and $0<\eta\le1/10$.
  Suppose that $V_1,\ldots,V_k$ are independent random variables satisfying $\E V_i=0$ and $|V_i|\le1/10$ almost surely for every $i$.
  If $\sum_{i=1}^k\Var(V_i)<\min\{5\eta,1/25\}$, then $\Prob(V_1+\cdots+V_k\ge1-\eta)<\eta/3$.
\end{lemma}

\begin{proof}
  Write $S=\sum_{i=1}^kV_i$, $v=\sum_{i=1}^k\E V_i^2$, and $\psi(t)=\mathrm{e}^t-1-t$.
  We prove $\Prob(S\ge1-\eta)\le v/16$, which gives the result because $v/16<5\eta/16<\eta/3$.
  We will apply Markov's inequality to $\psi(10S)$.
  Since $\E S=0$, we have $\E\psi(10S)=\E\mathrm{e}^{10S}-1$.
  Taylor's formula gives, for $|t|\le1$,
  \[
    0\le\psi(t)
    =t^2\int_0^1(1-s)\mathrm{e}^{st}\dd s
    \le\frac34t^2,
  \]
  since $\int_0^1(1-s)\mathrm{e}^s\dd s=\mathrm{e}-2<3/4$.
  Thus, $\E\mathrm{e}^{10V_i}\le1+75\E V_i^2$.
  By independence,
  \[
    \E\mathrm{e}^{10S}
    \le\prod_{i=1}^k(1+75\E V_i^2)\le\mathrm{e}^{75v}.
  \]
  Since the function $s\mapsto\mathrm{e}^{75s}-1$ is convex and $0\le v<1/25$, we obtain
  \[
    \E\psi(10S)=\E\mathrm{e}^{10S}-1
    \le\mathrm{e}^{75v}-1\le25v(\mathrm{e}^3-1).
  \]
  The function $\psi$ is nonnegative on $\mathbb R$ and increasing on $[0,\infty)$.
  As $1-\eta\ge9/10$, the event $S\ge1-\eta$ implies $\psi(10S)\ge\psi(9)$.
  Markov's inequality therefore yields
  \[
    \Prob(S\ge1-\eta)
    \le\frac{\E\psi(10S)}{\psi(9)}
    \le25v\frac{\mathrm{e}^3-1}{\mathrm{e}^9-10}
    \le\frac v{16}<\frac{5\eta}{16}<\frac{\eta}{3}.
  \]
  The numerical bound $(\mathrm{e}^3-1)/(\mathrm{e}^9-10)<\mathrm{e}^{-6}<1/400$ follows from $\mathrm{e}^3>20$.
\end{proof}

\subsection{\texorpdfstring{Proof of the upper bound for $\Phi_k(x)$}{Proof of the upper bound for Phi(k,x)}}
\label{sec:other-probability-consequences}

  We first prove that every maximizing law at $x=\beta_k$ has $w=0$ or $w=W_k$.
  At either endpoint, \cref{lem:initial-nodes}\textup{(ii)} gives $A(w)=H(w)$, so \cref{prop:strict-chord} yields $F_k(\mathsf P)\le(k\beta_k)^k$.
  Since a maximizing law exists, this proves the required upper bound for $\Phi_k(\beta_k)$.
  We then use \cref{prop:mean-transfer} to obtain the bound for every $x\ge0$.

\begin{proposition}\label{prop:maximizing-endpoints}
  Let $k\ge2$ and $x=\beta_k$.
  Suppose that $\mathsf P$ maximizes $F_k$ among nonnegative Borel probability laws of mean at most $x$.
  Write $\mathsf P=(1-a)\delta_0+a\mathsf Q$ as in \cref{prop:normalization}, where $a=kx$.
  For independent $Y_1,\ldots,Y_k$ with law $\mathsf Q$, set $w=k\E(Y_1+\cdots+Y_k-1)_+$.
  Then $w=0$ or $w=W_k$.
\end{proposition}

\begin{proof}
  For $2\le k\le5$, the conclusion follows from \cref{lem:initial-nodes}\textup{(iv)}.
  Fix $k\ge6$ and write $x=\beta_k=(k+c_k)^{-1}$.
  Let $\mathsf P$ be any maximizing distribution under the constraint $m(\mathsf P)\le x$.
  By \cref{prop:normalization}, $\mathsf P$ is supported on $[0,1]$, has mean $x$, and can be written as
  \[
    \mathsf P=(1-a)\delta_0+a\mathsf Q,
  \]
  where $a=kx$ and $\mathsf Q\in\Pcal([0,1])$ has mean $1/k$.
  In particular, $\mathsf P(\{0\})\ge1-a>0$.
  For independent $X_1,\ldots,X_k$ with law $\mathsf P$, we retain the notation $S_j,p,h,r$ from \cref{sec:contact}.
  Since $\E S_k=a<1$, we have $p>0$.
  Both $\mathsf B_x$ and $\mathsf C_{k,x}$ attain $a^k=1-(1-x)^k$.
  Hence, maximality gives $F_k(\mathsf P)\ge a^k$, or $p\le p_0=(1-x)^k$.
  Suppose for a contradiction that $0<w<W_k$.
  We will apply \cref{lem:upper-intensity,lem:large-defect} and then compare the upper and lower bounds for $T_0$ to exclude this possibility.

  By {\cref{lem:crossing-lower-finite,lem:crossing-upper}}, we have $11/24<c_k<1/2$.
  Recall that $\delta=1-a$.
  Let $M$ and $D$ be as in \cref{sec:contact}, so
  \[
    M=\frac{x(1-F_{k-1}(\mathsf P))}{F_k(\mathsf P)-F_{k-1}(\mathsf P)}
    \text{ and }
    D=M-x.
  \]
  By {\cref{obs:contact-identities}\textup{\ref{item:contact-failure}} and \cref{cor:contact-moments}\textup{\ref{item:contact-moment-identity}}}, we have $x<M\le1$, $\mathsf P([0,M])=1$, and
  \[
    \Var(S_k)=D(r-\delta),\text{ and }
    \Prob(S_{k-1}<1-y)=\frac{p(M-y)}D
    \quad\text{for $\mathsf P$-almost every $y$.}
  \]
  By {\cref{lem:contact-moment}\textup{\ref{item:contact-moment-bounds}}}, we have $r\le M$.
  By {\cref{prop:finite-shortfall-margin}\textup{\ref{item:shortfall-ratio}}}, we have $D>3x$ and
  \[
    0\le\frac{M-r}{D}<\frac1{20}-\frac{3x}{20D}.
  \]

  By \cref{prop:support-separation}, we have $\mathsf P([x+D/16,x+17D/20])=0$ and $\mathsf P((x+17D/20,M])>0$.
  We use $q,X^{\operatorname{L}},X^{\operatorname{U}},x_0,\lambda$ from the conditional decomposition in \cref{sec:conditional-laws} with $\varepsilon=1/16$ and $\zeta=17/20$.
  The upper interval $(x+17D/20,M]$ has positive probability, so $q<1$, while $q\ge\mathsf P(\{0\})\ge1-a>0$.
  Recall that $A=M-x_0$, $J=X^{\operatorname{U}}-x_0$, $m=k(x-x_0)$, and $\eta=a-m/A$.
  Apply \cref{lem:upper-intensity} with $\varepsilon=1/16$ and $\zeta=17/20$.
  If $x_0=0$, then $\mathsf P=\mathsf B_x$ by \cref{lem:upper-intensity}\textup{\ref{item:upper-intensity-zero}}.
  Since $a=kx$, the decomposition $\mathsf P=(1-a)\delta_0+a\mathsf Q$ then implies $\mathsf Q=(1-1/k)\delta_0+(1/k)\delta_1$.
  Recall that $R_k=Y_1+\cdots+Y_k$.
  Since $\E R_k=1$ and $R_k$ takes nonnegative integer values,
  \[
    w=k\E(R_k-1)_+
    =k\E(1-R_k)_+
    =k\Prob(R_k=0)
    =k(1-1/k)^k
    =W_k,
  \]
  contrary to $0<w<W_k$.
  Thus, $x_0>0$, and \cref{lem:upper-intensity}\textup{\ref{item:upper-intensity-parameters}} gives $\lambda<a$, $J>m+\delta-x_0/5$ almost surely, and $0<\eta<a$.

  By the almost-everywhere equality in {\cref{obs:contact-identities}\textup{\ref{item:contact-equality}}} and the definition of $D$, we have $h(y)-h(0)=py/D$ for $\mathsf P$-almost every $y$.
  For independent $X,X'\sim\mathsf P$, {\cref{lem:contact-moment}\textup{\ref{item:contact-moment-formula}}} gives
  \[
    v=(r-\delta)(M-r)
    -\frac{k}{2p}\E\left[
      (X-X')\int_{X'}^X\bigl(h(t)-h(0)\bigr)\dd t
      \right],
  \]
  where $v=\Var(S_k\mid S_k<1)$.

  Applying \cref{lem:large-defect} to $\mathsf P$, we obtain $\eta/a<1/10$.
  Since $0<\eta<a<1$, it follows that $0<\eta<1/10$.

  \begin{claim}\label{lem:small-defect-exclusion}
    We have $T_0<\eta/3$.
  \end{claim}
  \begin{proof}
    Applying \cref{lem:conditional-variance-bounds}\textup{\ref{item:conditional-variance-small}} and \cref{lem:conditional-support} with $\varepsilon=1/16$ and $\zeta=17/20$, we obtain
    \[
      b_0:=\esssup X^{\operatorname{L}}<\frac A{10},\text{ and }
      \mathcal V<\min\left\{5\eta,\frac1{25}\right\}.
    \]
    For independent copies $X_1^{\operatorname{L}},\ldots,X_k^{\operatorname{L}}$, the variables $(X_i^{\operatorname{L}}-x_0)/A$ are independent and centered.
    Since $0\le X_i^{\operatorname{L}},x_0\le b_0$, their absolute values are at most $b_0/A<1/10$, and their total variance is $\mathcal V$.
    Applying \cref{lem:exact-upper-count}\textup{\ref{item:exact-upper-count-tail}} and \cref{lem:small-conditional-tail}, we obtain
    \[
      T_0
      \le\Prob\left(\frac{B_k}{A}\ge1-\eta\right)
      <\frac{\eta}{3}.
    \]
  \end{proof}
  However, \cref{lem:exact-upper-count}\textup{\ref{item:exact-upper-count-tail}} gives $T_0>\eta/3$, a contradiction.

  By \cref{lem:critical}, every maximizing distribution has $0\le w\le W_k$.
  Since $0<w<W_k$ has been excluded, we have $w=0$ or $w=W_k$.
\end{proof}

  \begin{proof}[Proof of the upper bound for $\Phi_k(x)$ in \cref{thm:iid-tail}]
    For $k=1$, \cref{lem:dimension-one} gives $\Phi_1(x)=\min\{x,1\}$.
    Fix $k\ge2$ and set $a=k\beta_k$.
    By \cref{lem:attainment}, choose a law $\mathsf P$ maximizing $F_k$ among nonnegative Borel laws of mean at most $\beta_k$.
    By \cref{prop:normalization}, write $\mathsf P=(1-a)\delta_0+a\mathsf Q$, where $\mathsf Q\in\Pcal([0,1])$ has mean $1/k$.
    By \cref{prop:maximizing-endpoints}, we have $w=0$ or $w=W_k$.
    Since $(1-a)/a=c_k/k\le1/(2k)<2/(k-1)$, \cref{prop:strict-chord} applies.
    At both endpoints, \cref{lem:initial-nodes}\textup{(ii)} gives $A(w)=H(w)$, and hence
    \[
      a^{-k}\Phi_k(\beta_k)
      =a^{-k}F_k(\mathsf P)
      \le1-H(w)+A(w)=1.
    \]
    Therefore $\Phi_k(\beta_k)\le(k\beta_k)^k$.
    Both $\mathsf B_{\beta_k}$ and $\mathsf C_{k,\beta_k}$ attain $(k\beta_k)^k=1-(1-\beta_k)^k$, so equality holds.
    The hypothesis of \cref{prop:mean-transfer} is now verified.
    It follows that $\Phi_k(x)=1-(1-x)^k$ for $0\le x<\beta_k$, that $\Phi_k(x)=(kx)^k$ for $\beta_k<x\le1/k$, and that $\Phi_k(x)=1$ for $x\ge1/k$.
    Together with the value at $\beta_k$, this proves the formula in \cref{thm:iid-tail} for every $k\ge1$ and $x\ge0$.
  \end{proof}

\section{Proof of Theorem~\ref{thm:iid-tail}}\label{sec:equality-cases}

    The formula for $\Phi_k(x)$ was proved in \cref{sec:other-probability-consequences}.
    In this section, we determine all maximizing distributions, completing the proof of \cref{thm:iid-tail}.
    We first treat $x=\beta_k$ in \cref{sec:crossing-proof} and then every mean bound $x\ge0$ in \cref{sec:all-means-proof}.

\subsection{\texorpdfstring{Maximum and maximizing laws at $x=\beta_k$}{Maximum and maximizing laws at the crossing mean}}
\label{sec:crossing-proof}

By \cref{prop:maximizing-endpoints}, every maximizing law at $x=\beta_k$ has $w=0$ or $w=W_k$.
  We identify the distributions corresponding to these values and obtain all equality cases in \cref{prop:crossing-equality}.

\begin{lemma}\label{lem:endpoint-laws}
  Let $k\ge2$ and let $\mathsf Q\in\Pcal([0,1])$ have mean $1/k$.
  For independent $Y_1,\ldots,Y_k$ with law $\mathsf Q$, write $R_k=Y_1+\cdots+Y_k$ and $w=k\E(R_k-1)_+$.
  Then $w=0$ implies $\mathsf Q=\delta_{1/k}$, while $w=W_k$ implies $\mathsf Q=(1-1/k)\delta_0+(1/k)\delta_1$.
\end{lemma}

\begin{proof}
  If $w=0$, then $\E(R_k-1)_+=0$, so $R_k\le1$ almost surely.
  Since $\E R_k=1$, we conclude that $R_k=1$ almost surely.
  By independence, we have $k\Var(Y_1)=0$, and hence $\mathsf Q=\delta_{1/k}$.

  If $w=W_k$, then by independence and $\E Y_i=1/k$ we have
  \[
    \E(1-R_k)_+=\frac{w}{k}=\left(1-\frac1k\right)^k=\E\prod_{i=1}^k(1-Y_i).
  \]
  By \cref{lem:product-shortfall}\textup{(i)} with $y_i=Y_i$, we have $(1-R_k)_+\le\prod_{i=1}^k(1-Y_i)$ almost surely.
  Since the two sides have equal expectations, we obtain $(1-R_k)_+=\prod_{i=1}^k(1-Y_i)$ almost surely.
  If $\mathsf Q((0,1))>0$, then all $Y_1,\ldots,Y_k$ lie in $(0,1)$ with probability $\mathsf Q((0,1))^k>0$ by independence.
  On this event, \cref{lem:product-shortfall}\textup{(i)} implies $(1-R_k)_+<\prod_{i=1}^k(1-Y_i)$, contradicting the almost-sure equality $(1-R_k)_+=\prod_{i=1}^k(1-Y_i)$.
  Thus, $Y_1\in\{0,1\}$ almost surely, and $\Prob(Y_1=1)=\E Y_1=1/k$.
  Therefore, $\mathsf Q=(1-1/k)\delta_0+(1/k)\delta_1$.
\end{proof}

\begin{proposition}\label{prop:crossing-equality}
  Let $k\ge2$.
  Then $\Phi_k(\beta_k)=1-(1-\beta_k)^k=(k\beta_k)^k$.
  The maximizing nonnegative Borel laws under $m(\mathsf P)\le\beta_k$ are exactly $\mathsf B_{\beta_k}$ and $\mathsf C_{k,\beta_k}$.
\end{proposition}

\begin{proof}
  By \cref{lem:attainment}, $F_k$ attains its maximum among nonnegative Borel laws with mean at most $\beta_k$.
  Fix an arbitrary maximizing law $\mathsf P$ under $m(\mathsf P)\le\beta_k$.
  By \cref{prop:normalization}\textup{(i)}, $\mathsf P$ is supported on $[0,1]$ and has mean $\beta_k$.
  By \cref{prop:normalization}, we can write $\mathsf P=(1-a)\delta_0+a\mathsf Q$ with $a=k\beta_k$ and $m(\mathsf Q)=1/k$.
  Let $Y_1,\ldots,Y_k$ be independent with law $\mathsf Q$ and set $w=k\E(Y_1+\cdots+Y_k-1)_+$.
  By \cref{prop:maximizing-endpoints}, $w=0$ or $w=W_k$.
    The numerical formula proved in \cref{sec:other-probability-consequences} gives $\Phi_k(\beta_k)=(k\beta_k)^k$.
    Both $\mathsf B_{\beta_k}$ and $\mathsf C_{k,\beta_k}$ have mean $\beta_k$, and
    \[
      F_k(\mathsf B_{\beta_k})
      =1-(1-\beta_k)^k
      =(k\beta_k)^k
      =F_k(\mathsf C_{k,\beta_k}).
    \]
    Thus both distributions are maximizing.

    By \cref{lem:endpoint-laws}, $w=0$ implies $\mathsf Q=\delta_{1/k}$, while $w=W_k$ implies $\mathsf Q=(1-1/k)\delta_0+(1/k)\delta_1$.
    Substituting into $\mathsf P=(1-a)\delta_0+a\mathsf Q$, we obtain $\mathsf P=\mathsf C_{k,\beta_k}$ when $w=0$ and $\mathsf P=\mathsf B_{\beta_k}$ when $w=W_k$.
    Since $\mathsf P$ was an arbitrary maximizing law, no other nonnegative Borel law with mean at most $\beta_k$ attains $\Phi_k(\beta_k)$.
    Conversely, the two displayed success probabilities show that both $\mathsf B_{\beta_k}$ and $\mathsf C_{k,\beta_k}$ attain this maximum.
    This proves the formula and all equality cases in \cref{prop:crossing-equality}.
\end{proof}

\subsection{\texorpdfstring{Maximum and maximizing laws for every mean bound $x\ge0$}{Maximum and maximizing laws for every mean bound}}
\label{sec:all-means-proof}

  By \cref{prop:mean-transfer,prop:crossing-equality}, the formula for $\Phi_k(x)$ holds for every $k\ge2$ and $x\ge0$.
  We now determine all maximizing laws.
  For $0<x<\beta_k$, we use $\mathsf P_\theta^+=(1-\theta)\mathsf P+\theta\delta_1$, which replaces an observation by one with probability $\theta$.
  For $\beta_k<x<1/k$, we use $\mathsf P_q^-=q\mathsf P+(1-q)\delta_0$, which replaces an observation by zero with probability $1-q$.
  As in \cref{prop:mean-transfer}, the parameters are chosen so that the resulting law has mean at most $\beta_k$.
  The cases $x=0$, $x=1/k$, and $x>1/k$ are treated separately at the end of the proof.

\Needspace{12\baselineskip}
\begin{proposition}\label{prop:equality-transfer}
  Fix $k\ge2$.
  Suppose that the maximizing laws under $m(\mathsf P)\le\beta_k$ are exactly $\mathsf B_{\beta_k}$ and $\mathsf C_{k,\beta_k}$.
  Then the maximizing laws for every mean bound $x\ge0$ are characterized as follows.

  \begin{enumerate}[label=\textup{(\roman*)},leftmargin=*,itemsep=0pt]

    \item If $x=0$, then the unique maximizing law is $\delta_0$.

    \item If $0<x<{\beta_k}$, then the unique maximizing law is $\mathsf B_x$.

    \item If $x={\beta_k}$, then the maximizing laws are exactly $\mathsf B_{{\beta_k}}$ and $\mathsf C_{k,{\beta_k}}$.

    \item If $\beta_k<x<1/k$, then the unique maximizing law is $\mathsf C_{k,x}$.

    \item If $x=1/k$, then the unique maximizing law is $\delta_{1/k}$.

    \item If $x>1/k$, then a feasible law is maximizing if and only if it is supported on $[1/k,\infty)$.
  \end{enumerate}
\end{proposition}

\begin{proof}
  Fix an arbitrary $x\in(0,\beta_k)$ and an arbitrary maximizing law $\mathsf P$ under $m(\mathsf P)\le x$.
  Take $\theta=({\beta_k}-x)/(1-x)$.
  Then $m(\mathsf P_\theta^+)\le{\beta_k}$ and $1-F_k(\mathsf P_\theta^+)=(1-\theta)^k[1-F_k(\mathsf P)]$.
  Since $1-F_k(\mathsf P)=(1-x)^k$ and $1-\theta=(1-{\beta_k})/(1-x)$, we obtain $1-F_k(\mathsf P_\theta^+)=(1-{\beta_k})^k$.
  Thus, $\mathsf P_\theta^+$ maximizes $F_k$ under the mean bound $\beta_k$ and has positive mass at one.
  Since $\mathsf C_{k,{\beta_k}}$ has no mass at one, while $\mathsf P_\theta^+(\{1\})>0$, we must have $\mathsf P_\theta^+=\mathsf B_{{\beta_k}}$, which implies $\mathsf P=\mathsf B_x$.

  Now fix an arbitrary $x\in(\beta_k,1/k)$ and an arbitrary maximizing law $\mathsf P$ under $m(\mathsf P)\le x$.
  Take $q={\beta_k}/x\in(0,1)$, so $m(\mathsf P_q^-)\le{\beta_k}$.
  Since $m(\mathsf P_q^-)\le\beta_k$, the hypothesis gives $F_k(\mathsf P_q^-)\le(k\beta_k)^k$.
  Thus, \cref{fact:mixtures}\textup{(ii)} gives
  \[
    (k\beta_k)^k=q^kF_k(\mathsf P)\le F_k(\mathsf P_q^-)\le(k\beta_k)^k.
  \]
  Thus, equality holds, and \cref{fact:mixtures}\textup{(ii)} gives $F_j(\mathsf P)=0$ for every $1\le j<k$.
  In particular, $F_1(\mathsf P)=0$, so $\mathsf P(\{1\})=\mathsf P_q^-(\{1\})=0$.
  Since $m(\mathsf P_q^-)\le\beta_k$ and $F_k(\mathsf P_q^-)=(k\beta_k)^k=1-(1-\beta_k)^k$, the hypothesis of \cref{prop:equality-transfer} implies that $\mathsf P_q^-$ is either $\mathsf B_{\beta_k}$ or $\mathsf C_{k,\beta_k}$.
  Since $\mathsf P_q^-(\{1\})=0$ and $\mathsf B_{\beta_k}(\{1\})=\beta_k>0$, we have $\mathsf P_q^-=\mathsf C_{k,\beta_k}$, and hence $\mathsf P=\mathsf C_{k,x}$.

  Thus, among Borel probability laws $\mathsf P$ on $[0,\infty)$ with $m(\mathsf P)\le x$, the only possible maximizer is $\mathsf B_x$ when $0<x<\beta_k$ and $\mathsf C_{k,x}$ when $\beta_k<x<1/k$.
  Conversely,
  \[
    \begin{aligned}
      F_k(\mathsf B_x)
       & =1-(1-x)^k=\Phi_k(x)
       &                      & (0<x<\beta_k),   \\
      F_k(\mathsf C_{k,x})
       & =(kx)^k=\Phi_k(x)
       &                      & (\beta_k<x<1/k).
    \end{aligned}
  \]

  At $x=\beta_k$, the hypothesis states that, among Borel probability laws $\mathsf P$ on $[0,\infty)$ with $m(\mathsf P)\le\beta_k$, the equality $F_k(\mathsf P)=1-(1-\beta_k)^k=(k\beta_k)^k$ holds if and only if $\mathsf P=\mathsf B_{\beta_k}$ or $\mathsf P=\mathsf C_{k,\beta_k}$.
  At $x=0$, nonnegativity forces $\mathsf P=\delta_0$.
  For $x\ge1/k$, the feasible law $\delta_{1/k}$ attains one.
  An i.i.d.\ law has success probability one exactly when it is supported on $[1/k,\infty)$.
  Indeed, if it assigns positive mass to $[0,1/k)$, then all $k$ observations lie there with positive probability and their sum is less than one.
  At mean bound $1/k$, support on $[1/k,\infty)$ forces $\delta_{1/k}$.
  For larger bounds it characterizes precisely the stated maximizing laws.
\end{proof}

  \begin{proof}[Proof of \cref{thm:iid-tail}]
    By \cref{lem:dimension-one}, $\Phi_1(x)=\min\{x,1\}$ for every $x\ge0$.
    For $0\le x\le1$, the unique maximizing law is $\mathsf B_x$.
    For $x>1$, a law with mean at most $x$ is maximizing if and only if it is supported on $[1,\infty)$.
    Thus, the formula and all equality statements hold for $k=1$.

    Fix $k\ge2$.
    By \cref{prop:crossing-equality}, every nonnegative Borel law $\mathsf P$ with $m(\mathsf P)\le\beta_k$ satisfies $F_k(\mathsf P)\le1-(1-\beta_k)^k=(k\beta_k)^k$, with equality exactly for $\mathsf B_{\beta_k}$ and $\mathsf C_{k,\beta_k}$.
    By \cref{prop:mean-transfer}, this bound at $\beta_k$ implies $F_k(\mathsf P)\le1-(1-x)^k$ for $0<x<\beta_k$ and $F_k(\mathsf P)\le(kx)^k$ for $\beta_k<x<1/k$, whenever $\mathsf P$ is a nonnegative Borel law with $m(\mathsf P)\le x$.
    These bounds are attained by $\mathsf B_x$ and $\mathsf C_{k,x}$, respectively.
    At $x=0$, nonnegativity forces $\mathsf P=\delta_0$, while for $x\ge1/k$, the feasible law $\delta_{1/k}$ attains probability one.
    Together with the value at $\beta_k$ and the comparison of the two candidate probabilities in \cref{prop:mean-transfer}, this proves
    \[
      \Phi_k(x)=
      \begin{cases}
        \max\{1-(1-x)^k,(kx)^k\}, & 0\le x\le1/k, \\
        1,                        & x\ge1/k.
      \end{cases}
    \]

    It remains to verify the equality statements for $k\ge2$.
    The classification at $\beta_k$ in \cref{prop:crossing-equality} verifies the hypothesis of \cref{prop:equality-transfer}.
    By \cref{prop:equality-transfer}\textup{(i)(ii)}, the unique maximizing law is $\mathsf B_x$ for $0\le x<\beta_k$, where $\mathsf B_0=\delta_0$.
    At $x=\beta_k$, \cref{prop:crossing-equality} identifies exactly the two maximizing laws $\mathsf B_{\beta_k}$ and $\mathsf C_{k,\beta_k}$.
    By \cref{prop:equality-transfer}\textup{(iv)(v)}, the unique maximizing law is $\mathsf C_{k,x}$ for $\beta_k<x\le1/k$, including $\mathsf C_{k,1/k}=\delta_{1/k}$ at $x=1/k$.
    These are precisely \textup{(ii)--(iv)} of \cref{thm:iid-tail}.
    Finally, by \cref{prop:equality-transfer}\textup{(vi)}, for $x>1/k$ a law with mean at most $x$ is maximizing if and only if it assigns probability one to $[1/k,\infty)$, proving \textup{(v)} for $k\ge2$.
    Together with the case $k=1$, this completes the proof of the formula and all equality cases in \cref{thm:iid-tail}.
  \end{proof}
\section{The Erd\H{o}s matching conjecture}\label{sec:emc}
We use \cref{thm:iid-tail} to prove \cref{thm:emc}.
We sample the vertex weights of an approximate fractional cover
independently, relating edge density to
the tail probability of a sum of i.i.d.\ nonnegative
random variables.

Throughout this section, fix $k\ge3$.
Recall that $\rho_k=1/\beta_k=k+c_k$.
By \cref{sec:crossing-notation}, $k<\rho_k<k+1/2$.
For a finite vertex set $V$ and $S\subseteq V$, write $\mathcal A_V(S)=\{E\in\binom{V}{k}:E\cap S\ne\emptyset\}$.
In particular, $\mathcal A_{[n]}(S)=\mathcal A(S)$, as defined in \cref{sec:matching-intro}.
For an $s$-set $S$ and a $(ks+k-1)$-set $Q$ in $[n]$, the families $\mathcal A(S)$ and $\binom Qk$ have sizes $M_k(n,s)$ and $C_k(s)$, respectively.
Write $D_k(n,s)=\max\{C_k(s),M_k(n,s)\}$, and, for $U\subseteq[n]$, set $\mathcal F[U]=\mathcal F\cap\binom Uk$.
It is easy to verify the following fact.

\begin{fact}\label{fact:construction-densities}
  For fixed $k\ge3$ and positive integers $n,s$ with $n\ge k(s+1)$, the following estimates hold uniformly.
  \begin{enumerate}[label=\textup{(\roman*)},leftmargin=*,itemsep=0pt]
    \item\label{it:construction-cover-density} We have $M_k(n,s)/\binom nk=1-(1-s/n)^k+O_k(n^{-1})$.
    \item\label{it:construction-clique-density} We have $C_k(s)/\binom nk=(ks/n)^k+O_k(n^{-1})$.
  \end{enumerate}
\end{fact}

\subsection{Cover--clique stability}\label{sec:matching-sequences}\label{sec:probability-stability}

\subsubsection{Approximate fractional covers}\label{sec:approximate-covers}
A fractional vertex cover of $\mathcal F$ is a function $w:[n]\to[0,1]$ such that $\sum_{v\in E}w(v)\ge1$ for every $E\in\mathcal F$.
Let $\tau^*(\mathcal F)$ denote the minimum total weight of a fractional vertex cover of $\mathcal F$.
A fractional matching assigns nonnegative edge weights whose sum at each vertex is at most one.
Let $\nu^*(\mathcal F)$ denote the maximum total weight of a fractional matching of $\mathcal F$.
By linear-programming duality (see, for example, Schrijver~\cite[Chapter~7]{Schrijver}), we have $\nu(\mathcal F)\le\nu^*(\mathcal F)=\tau^*(\mathcal F)$.
Alon, Frankl, Huang, R\"odl, Ruci\'nski and Sudakov~\cite{AFHRS}, Alon, Huang and Sudakov~\cite{AlonHuangSudakov2012}, and \L{}uczak, Mieczkowska and \v{S}ileikis~\cite{LMS} established the connection between fractional vertex covers and inequalities for sums of nonnegative variables.

We use the following approximate covering result.

\begin{lemma}[Cao--Liu--Zhang~{\cite[Lemma~3.2]{CLZ}}]
  \label{lem:approx-cover}
  Fix $k\ge2$, $0\le\xi<1/k$ and $\rho>0$.
  There exists an integer $n_0=n_0(k,\xi,\rho)$ such that, for every integer $n\ge n_0$, the following holds.
  If $0\le x\le\xi$ and $\mathcal H\subseteq\binom{[n]}k$ satisfies $\nu(\mathcal H)\le xn$, then there is a function $w:[n]\to[0,1]$ such that
  \[
    \sum_{v=1}^n w(v)\le(x+\rho)n,
  \]
  and all but at most $\rho n^k$ edges $E\in\mathcal H$ satisfy
  \[
    \sum_{v\in E}w(v)\ge1.
  \]
\end{lemma}

To relate the size of a family $\mathcal F\subseteq\binom{[n]}k$ to a tail probability, fix weights $w:[n]\to[0,1]$.
Choose vertices $v_1,\ldots,v_k$ independently and uniformly from $[n]$ with replacement.
Set $X_i=w(v_i)$ for $1\le i\le k$.
Define a probability measure $\mathsf P_w$ on $[0,1]$ by assigning to each weight value the proportion of vertices having that weight.
Since each vertex is chosen with probability $1/n$, for every $1\le i\le k$ and $t\in[0,1]$ we have
\[
  \Prob(X_i=t)=\mathsf P_w(\{t\})=\frac{|\{v\in[n]:w(v)=t\}|}{n}.
\]
Equivalently, $\mathsf P_w=n^{-1}\sum_{v=1}^n\delta_{w(v)}$.
Since the choices of $v_1,\ldots,v_k$ are independent, $X_1,\ldots,X_k$ are independent.
Thus, $F_k(\mathsf P_w)=\Prob(X_1+\cdots+X_k\ge1)$.
Let $\mathcal U_w=\{E\in\mathcal F:\sum_{v\in E}w(v)<1\}$ be the edges not covered by $w$.
The mean of $\mathsf P_w$ is the average of the vertex weights $w(1),\ldots,w(n)$.
For a fixed edge $E\in\mathcal F\setminus\mathcal U_w$, the event that $\{v_1,\ldots,v_k\}=E$ consists of exactly $k!$ ordered samples, one for each ordering of the $k$ distinct vertices of $E$.
By independence, each such sample has probability $(1/n)^k$, and on this event $X_1+\cdots+X_k=\sum_{v\in E}w(v)\ge1$.
These events are pairwise disjoint as $E$ ranges over $\mathcal F\setminus\mathcal U_w$.
Consequently, we have
\begin{equation}\label{eq:cover-sampling}
  \begin{aligned}
    m(\mathsf P_w)   & =\E X_1=\frac1n\sum_{v=1}^n w(v),                                       \\
    F_k(\mathsf P_w) & =\Prob(X_1+\cdots+X_k\ge1)\ge\sum_{E\in\mathcal F\setminus\mathcal U_w}
    \Prob(\{v_1,\ldots,v_k\}=E)=\frac{k!}{n^k}\bigl(|\mathcal F|-|\mathcal U_w|\bigr).
  \end{aligned}
\end{equation}

\subsubsection{Nearly extremal families}\label{sec:sequence-reduction}
\begin{lemma}\label{lem:matching-sequence}
  Fix $k\ge3$.
  Let $(n_i,s_i)$ be integer pairs with $s_i\to\infty$, $n_i\ge k(s_i+1)$, and $n_i/s_i\to R<\infty$ as $i\to\infty$.
  Suppose $\mathcal F_i\subseteq\binom{[n_i]}k$ satisfies $\nu(\mathcal F_i)\le s_i$ and $|\mathcal F_i|\ge D_k(n_i,s_i)-o(n_i^k)$.
  For each $i$, define its distances from the two constructions by
  \[
    d_C(\mathcal F_i)=\min_{\substack{Q\subseteq[n_i]\\|Q|=ks_i+k-1}}
    \left|\mathcal F_i\triangle\binom Qk\right|
    \text{ and }
    d_B(\mathcal F_i)=\min_{\substack{S\subseteq[n_i]\\|S|=s_i}}
    \left|\mathcal F_i\triangle\mathcal A(S)\right|.
  \]
  Then the following asymptotic conclusions hold as $i\to\infty$ along the original sequence.
  \begin{enumerate}[label=\textup{(\roman*)},leftmargin=*,itemsep=0pt]
    \item\label{it:matching-sequence-dichotomy} We have $\min\{d_C(\mathcal F_i),d_B(\mathcal F_i)\}=o(n_i^k)$.
    \item\label{it:matching-sequence-clique}\label{it:matching-sequence-cover}
    If $k\le R<\rho_k$, then $d_C(\mathcal F_i)=o(n_i^k)$.
    If $R>\rho_k$, then $d_B(\mathcal F_i)=o(n_i^k)$.
    At $R=\rho_k$, the approximating construction in~\textup{(i)} may be a clique family for some indices and a cover family for others.
  \end{enumerate}
\end{lemma}
\begin{proof}
  We first handle the case $R=k$.
  Choose any $(ks_i+k-1)$-set $Q_i\subseteq[n_i]$.
  Since $n_i-|Q_i|=o(n_i)$,
  \[
    \binom{n_i}k-C_k(s_i)
    \le(n_i-|Q_i|)\binom{n_i-1}{k-1}=o(n_i^k).
  \]
  Since $|\mathcal F_i|\ge D_k(n_i,s_i)-o(n_i^k)$ and $D_k(n_i,s_i)\ge C_k(s_i)$, we have $\binom{n_i}k-|\mathcal F_i|=o(n_i^k)$.
  Consequently,
  \[
    \left|\mathcal F_i\triangle\binom{Q_i}k\right|
    \le\left(\binom{n_i}k-C_k(s_i)\right)
    +\left(\binom{n_i}k-|\mathcal F_i|\right)=o(n_i^k).
  \]
  This proves both assertions when $R=k$.
  Henceforth assume $R>k$, and write $x=1/R$ and $x_i=s_i/n_i$.

  By \cref{fact:construction-densities} and $x_i\to x$, we have $D_k(n_i,s_i)/\binom{n_i}k\to\Phi_k(x)$.
For a fixed $\xi$ with $x<\xi<1/k$ and each integer $\ell\ge1$,
  let $n_0(\ell)$ be an integer supplied by \cref{lem:approx-cover}
  with $\rho=1/\ell$.
  Since $x_i\to x<\xi$ and $n_i\to\infty$, we may choose positive integers $a_\ell$ recursively as follows.
  For each integer $\ell\ge1$, choose $a_\ell$ such that $0\le x_i\le\xi$ and $n_i\ge n_0(\ell)$ for every $i\ge a_\ell$, taking $a_\ell>a_{\ell-1}$ when $\ell\ge2$.
  For each $i\ge a_1$, let $\ell$ be the unique integer with $a_\ell\le i<a_{\ell+1}$, and set $\epsilon_i=1/\ell$.
  For each such $i$, since $\nu(\mathcal F_i)\le s_i=x_i n_i$, apply \cref{lem:approx-cover} and choose a weight function $w_i:[n_i]\to[0,1]$ with $\sum_{v=1}^{n_i}w_i(v)\le(x_i+\epsilon_i)n_i$ and $|\mathcal U_i|\le\epsilon_i n_i^k$, where $\mathcal U_i=\{E\in\mathcal F_i:\sum_{v\in E}w_i(v)<1\}$ is the family of sets not covered by $w_i$.
  We henceforth consider only $i\ge a_1$, since omitting finitely many initial terms does not affect the $o(n_i^k)$ conclusions of \cref{lem:matching-sequence}.

  As $i\to\infty$, the integer $\ell$ determined by $a_\ell\le i<a_{\ell+1}$ tends to infinity, so $\epsilon_i=1/\ell\to0$.
  Set $\mathsf P_i=\mathsf P_{w_i}=n_i^{-1}\sum_{v=1}^{n_i}\delta_{w_i(v)}$.
  By \eqref{eq:cover-sampling}, $m(\mathsf P_i)\le x_i+\epsilon_i$ and
  \begin{equation}\label{eq:sequence-success-lower-bound}
    F_k(\mathsf P_i)\ge\frac{k!}{n_i^k}(|\mathcal F_i|-|\mathcal U_i|)
    \ge\Phi_k(x)-o(1).
  \end{equation}
  The last inequality follows from $|\mathcal F_i|\ge D_k(n_i,s_i)-o(n_i^k)$, \cref{fact:construction-densities}, $k!\binom{n_i}{k}/n_i^k=1+o(1)$, and $\epsilon_i\to0$, together with $x_i\to x$ and the continuity of $\Phi_k$.

  Suppose that \textup{(i)} fails.
  After passing to a subsequence and relabelling, we may assume that there exists $\delta>0$ such that
  \begin{equation}\label{eq:sequence-distance-lower-bounds}
    \min\{d_C(\mathcal F_i),d_B(\mathcal F_i)\}\ge\delta n_i^k
    \text{ for every }i.
  \end{equation}
  By the weak sequential compactness of $\Pcal([0,1])$, we may pass to a further subsequence, again relabelling, such that $\mathsf P_i$ converges weakly to some $\mathsf P\in\Pcal([0,1])$.
  All limits in the remainder of the proof of \textup{(i)} are taken along this subsequence.
  By weak convergence on $[0,1]$, $m(\mathsf P)=\lim_{i\to\infty}m(\mathsf P_i)\le\lim_{i\to\infty}(x_i+\epsilon_i)=x$.
  By \cref{fact:weak-convergence}\textup{(ii)}, the product laws $\mathsf P_i^{\otimes k}$ converge weakly to $\mathsf P^{\otimes k}$.
  Applying \cref{thm:portmanteau}\textup{(iii)} to the closed set $H_k=\{(t_1,\ldots,t_k)\in[0,1]^k:t_1+\cdots+t_k\ge1\}$, we obtain
  \[
    \Phi_k(x)\le\limsup_{i\to\infty}\mathsf P_i^{\otimes k}(H_k)
    \le\mathsf P^{\otimes k}(H_k)=F_k(\mathsf P)\le\Phi_k(x).
  \]
  Here the first inequality follows from \eqref{eq:sequence-success-lower-bound}, and the last follows from $m(\mathsf P)\le x$ and the definition of $\Phi_k(x)$.
  Thus, $\mathsf P$ is a maximizing distribution.
  \Cref{thm:iid-tail} implies $\mathsf P\in\{\mathsf B_x,\mathsf C_{k,x}\}$.
  If $\mathsf P=\mathsf B_x$, then we will choose an $s_i$-set $S_i\subseteq[n_i]$ and take $\mathcal H_i=\mathcal A(S_i)$.
  If $\mathsf P=\mathsf C_{k,x}$, then we will choose a $(ks_i+k-1)$-set $Q_i\subseteq[n_i]$ and take $\mathcal H_i=\binom{Q_i}k$.
  The sets will be chosen so that
  \begin{equation}\label{eq:sequence-outside-edges}
    |\mathcal F_i\setminus\mathcal H_i|=o(n_i^k).
  \end{equation}

  \medskip
  \noindent{\bf Case 1. $\mathsf P=\mathsf B_x$.}
  \medskip

  We construct $S_i\subseteq[n_i]$ with $|S_i|=s_i$ and
  $|\mathcal F_i\setminus\mathcal A(S_i)|=o(n_i^k)$.
  Fix $0<\alpha<1/(2k)$.
  We first select the vertices with $w_i(v)>1-\alpha$ and then add or remove
  $o(n_i)$ vertices to obtain $|S_i|=s_i$.
  Define
  \[
    S_i'=\{v:w_i(v)>1-\alpha\}
    \text{ and }
    I_i^B=\{v:\alpha<w_i(v)\le1-\alpha\}.
  \]
  We have $\mathsf B_x((1-\alpha,1])=x$ and $\mathsf B_x((\alpha,1-\alpha])=0$.
  The intervals $(1-\alpha,1]$ and $(\alpha,1-\alpha]$ have boundaries of $\mathsf B_x$-measure zero in $[0,1]$.
  Since $\mathsf P_i$ converges weakly to $\mathsf B_x$, \cref{thm:portmanteau}\textup{(v)} implies
  \[
    \frac{|S_i'|}{n_i}=\mathsf P_i((1-\alpha,1])\to x
    \text{ and }
    \frac{|I_i^B|}{n_i}=\mathsf P_i((\alpha,1-\alpha])\to0.
  \]
  In particular, $|S_i'|=s_i+o(n_i)$ and $|I_i^B|=o(n_i)$.
  If a $k$-set $E$ avoids both $S_i'$ and $I_i^B$, then all its vertex weights are at most $\alpha$, and hence $\sum_{v\in E}w_i(v)\le k\alpha<1$.
  Thus, every edge of $\mathcal F_i\setminus\mathcal A(S_i')$ either belongs to $\mathcal U_i$ or meets $I_i^B$, and
  \[
    |\mathcal F_i\setminus\mathcal A(S_i')|
    \le|\mathcal U_i|+|I_i^B|\binom{n_i-1}{k-1}=o(n_i^k).
  \]
  For each $i$, choose an $s_i$-set $S_i$ by adding or removing exactly $\bigl||S_i'|-s_i\bigr|$ vertices from $S_i'$.
  We have $|S_i\triangle S_i'|=\bigl||S_i'|-s_i\bigr|=o(n_i)$.
  Only $k$-sets containing a changed vertex can change their membership in the cover family, so
  \[
    |\mathcal A(S_i')\triangle\mathcal A(S_i)|
    \le|S_i'\triangle S_i|\binom{n_i-1}{k-1}=o(n_i^k).
  \]
  Therefore, $\mathcal H_i=\mathcal A(S_i)$ satisfies \eqref{eq:sequence-outside-edges} and has size $M_k(n_i,s_i)$.

  \medskip
  \noindent{\bf Case 2. $\mathsf P=\mathsf C_{k,x}$.}
  \medskip

  We construct $Q_i\subseteq[n_i]$ with $|Q_i|=ks_i+k-1$ and
  $|\mathcal F_i\setminus\binom{Q_i}k|=o(n_i^k)$.
  Fix $0<\alpha<1/(2k^2)$.
  We first select the vertices with $|w_i(v)-1/k|\le\alpha$ and then add or remove
  $o(n_i)$ vertices to obtain $|Q_i|=ks_i+k-1$.
  Define
  \[
    Q_i'=\{v:|w_i(v)-1/k|\le\alpha\}
    \text{ and }
    I_i^C=\{v:w_i(v)>\alpha\}\setminus Q_i'.
  \]
  Since $1/k-\alpha>\alpha$, we have $Q_i'\subseteq\{v:w_i(v)>\alpha\}$.
  The intervals $[1/k-\alpha,1/k+\alpha]$ and $(\alpha,1]$ have boundaries of $\mathsf C_{k,x}$-measure zero in $[0,1]$.
  Both intervals contain $1/k$ and exclude $0$, so their $\mathsf C_{k,x}$-measure is $kx$.
  Since $\mathsf P_i$ converges weakly to $\mathsf C_{k,x}$, \cref{thm:portmanteau}\textup{(v)} implies
  \[
    \frac{|Q_i'|}{n_i}
    =\mathsf P_i([1/k-\alpha,1/k+\alpha])\to kx
    \text{ and }
    \frac{|I_i^C|}{n_i}
    =\mathsf P_i((\alpha,1])-\frac{|Q_i'|}{n_i}\to0.
  \]
  In particular, $|Q_i'|=ks_i+o(n_i)$ and $|I_i^C|=o(n_i)$.
  If a $k$-set $E$ is not contained in $Q_i'$ and avoids $I_i^C$, then at least one of its vertices has weight at most $\alpha$, and all the others have weight at most $1/k+\alpha$.
  Thus,
  \[
    \sum_{v\in E}w_i(v)\le\alpha+(k-1)(1/k+\alpha)
    =\frac{k-1}{k}+k\alpha<1.
  \]
  Every edge of $\mathcal F_i\setminus\binom{Q_i'}k$ therefore either belongs to $\mathcal U_i$ or meets $I_i^C$, so
  \[
    \left|\mathcal F_i\setminus\binom{Q_i'}k\right|
    \le|\mathcal U_i|+|I_i^C|\binom{n_i-1}{k-1}=o(n_i^k).
  \]
  Since $ks_i+k-1\le n_i$, for each $i$ choose a $(ks_i+k-1)$-set $Q_i$ by adding or removing exactly $\bigl||Q_i'|-(ks_i+k-1)\bigr|$ vertices from $Q_i'$.
  Since $k$ is fixed, we have $|Q_i\triangle Q_i'|=\bigl||Q_i'|-(ks_i+k-1)\bigr|=o(n_i)$.
  As only $k$-sets containing a changed vertex can change their membership in the clique family,
  \[
    \left|\binom{Q_i'}k\triangle\binom{Q_i}k\right|
    \le|Q_i'\triangle Q_i|\binom{n_i-1}{k-1}=o(n_i^k).
  \]
  Therefore, $\mathcal H_i=\binom{Q_i}k$ satisfies \eqref{eq:sequence-outside-edges} and has size $C_k(s_i)$.

  \medskip
  In either case, $|\mathcal H_i|\le D_k(n_i,s_i)$.
  Since $|\mathcal F_i|\ge D_k(n_i,s_i)-o(n_i^k)$, the bound $|\mathcal H_i|\le D_k(n_i,s_i)$ and \eqref{eq:sequence-outside-edges} give
  \[
    |\mathcal F_i\triangle\mathcal H_i|
    =2|\mathcal F_i\setminus\mathcal H_i|+|\mathcal H_i|-|\mathcal F_i|\le2|\mathcal F_i\setminus\mathcal H_i|+o(n_i^k)
    =o(n_i^k).
  \]
  This contradicts \eqref{eq:sequence-distance-lower-bounds}, proving \textup{(i)} along the whole original sequence.

  To prove \textup{(ii)}, we return to the original sequence and compare the sizes of the cover and clique families.
  Taking the minima in the definitions of $d_B(\mathcal F_i)$ and $d_C(\mathcal F_i)$, and using $|\mathcal F_i|\ge D_k(n_i,s_i)-o(n_i^k)$, we obtain
  \[
    \begin{aligned}
      d_B(\mathcal F_i) & \ge|\mathcal F_i|-M_k(n_i,s_i)\ge D_k(n_i,s_i)-M_k(n_i,s_i)-o(n_i^k), \\
      d_C(\mathcal F_i) & \ge|\mathcal F_i|-C_k(s_i)\ge D_k(n_i,s_i)-C_k(s_i)-o(n_i^k).
    \end{aligned}
  \]
  If $k<R<\rho_k$, then $x=1/R>\beta_k$, and \cref{fact:construction-densities} implies
  \[
    \frac{C_k(s_i)-M_k(n_i,s_i)}{n_i^k}
    \longrightarrow\frac{(kx)^k-1+(1-x)^k}{k!}>0.
  \]
  Thus, $d_B(\mathcal F_i)\ge c n_i^k$ for some $c>0$ and all sufficiently large $i$.
  By \textup{(i)}, we have $d_C(\mathcal F_i)=o(n_i^k)$.
  If $R>\rho_k$, then $x=1/R<\beta_k$, and \cref{fact:construction-densities} implies
  \[
    \frac{M_k(n_i,s_i)-C_k(s_i)}{n_i^k}
    \longrightarrow\frac{1-(1-x)^k-(kx)^k}{k!}>0.
  \]
  Thus, $d_C(\mathcal F_i)\ge c n_i^k$ for some $c>0$ and all sufficiently large $i$.
  By \textup{(i)}, we have $d_B(\mathcal F_i)=o(n_i^k)$.
  Together with the case $R=k$ proved above, this establishes \textup{(ii)}.
\end{proof}

\subsection{Exact bounds and equality cases}\label{sec:emc-proof}\label{sec:sparse-range}
We now prove \cref{thm:emc}.
By \cref{lem:matching-sequence}, we reduce to the cases where the family is close to either a cover family or a clique family.
Near a cover family, we extend matchings greedily and apply \cref{lem:sparse-emc} to the remaining family to obtain $|\mathcal F|\le M_k(n,s)$.
Near a clique family, shifting produces a large clique, and \cref{lem:large-clique-emc} gives $|\mathcal F|\le C_k(s)$.
We first use \cref{thm:bde} to handle the sparse range and restrict $n/s$ to a bounded interval.

\begin{theorem}[Bollob\'as, Daykin and Erd\H{o}s~{\cite[Theorem~1]{BDE}}]\label{thm:bde}
  Let $k\ge3$, $b\ge1$, and $n>2k^3b$ be integers.
  If $\mathcal J\subseteq\binom{[n]}k$ satisfies $\nu(\mathcal J)\le b$ and no $b$-set meets every edge of $\mathcal J$, then
  \[
    |\mathcal J|\le M_k(n,b)-\binom{n-b-k}{k-1}+1.
  \]
\end{theorem}

We shall use the following consequence.

\begin{lemma}\label{lem:sparse-emc}
  Let $k\ge3$, $n\ge k$, and $b\ge0$ be integers with $b\le n/(4k^3)$.
  If $\mathcal J\subseteq\binom{[n]}k$ satisfies $\nu(\mathcal J)\le b$, then $|\mathcal J|\le M_k(n,b)$.
  Equality holds if and only if $\mathcal J=\mathcal A(B')$ for some $b$-set $B'\subseteq[n]$.
\end{lemma}

\begin{proof}
  The case $b=0$ is immediate.
  Suppose $b\ge1$.
  If no $b$-set meets every edge of $\mathcal J$, then $n\ge4k^3b>2k^3b$, and by \cref{thm:bde}, we have
  \[
    |\mathcal J|\le M_k(n,b)-\binom{n-b-k}{k-1}+1<M_k(n,b).
  \]
  The strict inequality follows from $\binom{n-b-k}{k-1}>1$.
  Otherwise, $\mathcal J\subseteq\mathcal A(B')$ for some $b$-set $B'$, and equality holds exactly when $\mathcal J=\mathcal A(B')$.
\end{proof}

Now we are ready to show \cref{thm:emc}.
It suffices to show that, for all sufficiently large $s$, every family satisfying its hypotheses and $|\mathcal F|\ge D_k(n,s)$ is either a cover family or a clique family.
Suppose this assertion fails for arbitrarily large $s$.
Choose positive integers $s_i\to\infty$, integers $n_i$, and families $\mathcal F_i\subseteq\binom{[n_i]}k$ such that, for every $i$,
\begin{equation}\label{eq:emc-counterexample-sequence}
  n_i\ge k(s_i+1),\text{ } \nu(\mathcal F_i)\le s_i\text{ and }
  |\mathcal F_i|\ge D_k(n_i,s_i).
\end{equation}
We may choose each $\mathcal F_i$ to be neither $\mathcal A(S)$ for any $s_i$-set $S\subseteq[n_i]$ nor $\binom Qk$ for any $(ks_i+k-1)$-set $Q\subseteq[n_i]$.
If $n_i\ge4k^3s_i$, then by \cref{lem:sparse-emc}, we have
\[
  D_k(n_i,s_i)\le|\mathcal F_i|\le M_k(n_i,s_i)\le D_k(n_i,s_i).
\]
All inequalities are equalities, so the equality case of \cref{lem:sparse-emc} gives $\mathcal F_i=\mathcal A(S_i')$ for some $s_i$-set $S_i'\subseteq[n_i]$, a contradiction.
Hence, $k<n_i/s_i<4k^3$ for every $i$.
Passing to a subsequence and relabelling, we may therefore assume that
$
  \frac{n_i}{s_i}\longrightarrow R_*$
  as $i\to\infty
$
for some $R_*\in[k,4k^3]$.
By \cref{lem:matching-sequence}\textup{(ii)}, we have $d_C(\mathcal F_i)=o(n_i^k)$ if $k\le R_*<\rho_k$, and $d_B(\mathcal F_i)=o(n_i^k)$ if $R_*>\rho_k$.
If $R_*=\rho_k$, then \cref{lem:matching-sequence}\textup{(i)} implies $\min\{d_B(\mathcal F_i),d_C(\mathcal F_i)\}=o(n_i^k)$.
At $R_*=\rho_k$, we pass to a subsequence on which either $d_B(\mathcal F_i)\le d_C(\mathcal F_i)$ for every $i$ or $d_C(\mathcal F_i)\le d_B(\mathcal F_i)$ for every $i$.
After relabelling when necessary, we obtain one of the following two cases.
\begin{enumerate}[label=\textup{Case~\arabic*.},leftmargin=*,itemsep=0pt]
  \item There are $s_i$-sets $S_i\subseteq[n_i]$ such that
        \begin{equation}\label{eq:emc-cover-alternative}
          R_*\ge\rho_k\text{ and }
          |\mathcal F_i\triangle\mathcal A(S_i)|=o(n_i^k).
        \end{equation}
  \item There are $(ks_i+k-1)$-sets $Q_i\subseteq[n_i]$ such that
        \begin{equation}\label{eq:emc-clique-alternative}
          k\le R_*\le\rho_k\text{ and }
          \left|\mathcal F_i\triangle\binom{Q_i}k\right|=o(n_i^k).
        \end{equation}
\end{enumerate}
We treat these two cases separately.

\subsubsection{Case 1: The cover case}\label{sec:local-cover-bounds}\label{sec:local-matching-bounds}
In this case, \eqref{eq:emc-cover-alternative} holds.
Since $R_*\ge\rho_k>k$, choose constants $R_0,L_0$ with $k<R_0<R_*<L_0$.
Then $R_0s_i\le n_i\le L_0s_i$ for all sufficiently large $i$.
We will prove $|\mathcal F_i|\le M_k(n_i,s_i)$ by removing vertices of $S_i$ through which a matching can be extended greedily.
To this end, we prove the following lemma.

\begin{lemma}\label{lem:local-cover}
  Fix an integer $k\ge3$ and real constants $k<R<L$.
  There exist $\delta=\delta(k,R,L)>0$ and an integer $s_0=s_0(k,R,L)$ such that the following holds whenever $s,n$ are integers with $s\ge s_0$ and $Rs\le n\le Ls$.
  If $\mathcal F\subseteq\binom{[n]}k$ satisfies $\nu(\mathcal F)\le s$ and $|\mathcal A(S)\setminus\mathcal F|\le\delta n^k$ for some $s$-set $S\subseteq[n]$, then
  \[
    |\mathcal F|\le M_k(n,s),
  \]
  with equality exactly when $\mathcal F=\mathcal A(S')$ for some $s$-set $S'\subseteq[n]$.
  Moreover, for every $s$-set $S\subseteq[n]$ satisfying $|\mathcal A(S)\setminus\mathcal F|\le\delta n^k$, there is a set $B\subseteq S$ such that
  \[
    \nu(\mathcal F[U])\le |B|\le\frac{|U|}{4k^3},
  \]
  where $U=([n]\setminus S)\cup B$.
\end{lemma}

\begin{proof}
  Let $\gamma=R-k>0$.
  Choose $\delta=\delta(k,R,L)>0$ sufficiently small, and then choose $s_0=s_0(k,R,L)$ sufficiently large for the estimates below.
  Fix any $s$-set $S\subseteq[n]$ with $|\mathcal A(S)\setminus\mathcal F|\le\delta n^k$, and set $W=[n]\setminus S$.
  For each $v\in S$, let $m(v)$ be the number of missing edges whose unique vertex in $S$ is $v$.
  Thus,
  \[
    m(v)=\left|\left\{A\in\binom W{k-1}:
    A\cup\{v\}\notin\mathcal F\right\}\right|.
  \]
  Set $D_0=\binom{\lfloor\gamma s/2\rfloor}{k-1}$ and $B=\{v\in S:m(v)\ge D_0\}$.
  Write $b=|B|$.
  Since no missing edge is counted for two vertices of $S$, we have $bD_0\le|\mathcal A(S)\setminus\mathcal F|\le\delta n^k$.
  Set $U=W\cup B$ and $N=|U|=n-s+b$.
  Since $D_0=\Theta_{k,R}(s^{k-1})$ and $n\le Ls$, we have $b\le\delta L^k s^k/D_0=O_{k,R,L}(\delta s)$.
  Since $N\ge(R-1)s$ and $\delta$ is sufficiently small, we may assume that
  \begin{equation}\label{eq:cover-b-small}
    b\le s/3,\text{ } b\le\gamma s/4\text{ and }b\le N/(4k^3).
  \end{equation}
  We also have $\gamma s\ge4$ and $N\ge k$ by the choice of $s_0$.

  We claim that the family $\mathcal F[U]$ has matching number at most $b$.
  Suppose for a contradiction that $\mathcal F[U]$ contains a matching $\mathcal M_0$ of $b+1$ edges.
  Write $S\setminus B=\{v_1,\ldots,v_{s-b}\}$.
  Since $U=W\cup B$, the edges of $\mathcal M_0$ avoid $S\setminus B$.
  Fix any integer $j$ with $0\le j<s-b$ and any edges $E_t=\{v_t\}\cup A_t\in\mathcal F$, with $A_t\in\binom W{k-1}$ for $1\le t\le j$, such that $\mathcal M_0\cup\{E_1,\ldots,E_j\}$ is a matching.
  The edges of $\mathcal M_0$ use at most $k(b+1)$ vertices of $W$, and the $j$ added edges use $j(k-1)$ more.
  Since $j\le s-b-1$, $n\ge(k+\gamma)s$, $b\le\gamma s/4$ by \eqref{eq:cover-b-small}, and $\gamma s\ge4$, the number of unused vertices of $W$ is at least
  \[
    \begin{aligned}
      |W|-k(b+1)-j(k-1)
       & \ge |W|-k(b+1)-(s-b-1)(k-1) \\
       & =n-ks-b-1
      \ge\gamma s-b-1
      \ge\lfloor\gamma s/2\rfloor.
    \end{aligned}
  \]
  Among the unused vertices there are at least $D_0$ choices of $A\in\binom W{k-1}$.
  Since $v_{j+1}\notin B$, we have $m(v_{j+1})<D_0$, so one such $A$ satisfies $\{v_{j+1}\}\cup A\in\mathcal F$.
  Choose one such set as $A_{j+1}$, and set $E_{j+1}=\{v_{j+1}\}\cup A_{j+1}$.
  The edge $E_{j+1}$ is disjoint from the current matching because $v_{j+1}$ is new and $A_{j+1}$ uses only unused vertices of $W$.
  After all $s-b$ vertices have been processed, the matching has $(b+1)+(s-b)=s+1$ edges, a contradiction.

  Every edge of $\mathcal F$ either meets $S\setminus B$ or belongs to $\mathcal F[U]$.
  We have $\nu(\mathcal F[U])\le b$ and $b\le N/(4k^3)$.
  By \cref{lem:sparse-emc}, we have
  \[
    |\mathcal F|\le\binom nk-\binom Nk+|\mathcal F[U]|
    \le\binom nk-\binom{N-b}k=M_k(n,s),
  \]
  where $N-b=n-s$.
  Equality forces $\mathcal F[U]=\mathcal A_U(B')$ for a $b$-set $B'\subseteq U$ and requires every $k$-set meeting $S\setminus B$ to belong to $\mathcal F$.
  Thus, $\mathcal F=\mathcal A((S\setminus B)\cup B')$, whose center has size $s$.
  Conversely, $|\mathcal A(S')|=M_k(n,s)$ for every $s$-set $S'$.
\end{proof}

We now apply \cref{lem:local-cover} to complete the proof in the cover case.
By \eqref{eq:emc-cover-alternative}, $|\mathcal A(S_i)\setminus\mathcal F_i|=o(n_i^k)$, so \cref{lem:local-cover} applies with $R=R_0$ and $L=L_0$ for all sufficiently large $i$.
Consequently,
\[
  D_k(n_i,s_i)\le|\mathcal F_i|
  \le M_k(n_i,s_i)\le D_k(n_i,s_i).
\]
Thus, $|\mathcal F_i|=M_k(n_i,s_i)$.
The equality case of \cref{lem:local-cover} then implies $\mathcal F_i=\mathcal A(S_i')$ for some $s_i$-set $S_i'\subseteq[n_i]$, a contradiction.

\subsubsection{Case 2: The clique case}\label{sec:local-clique-bounds}
In this case, \eqref{eq:emc-clique-alternative} holds.
We first state a bound for families containing a large clique.
Write $\omega(\mathcal H)$ for the size of a largest clique of a $k$-graph $\mathcal H$.

\begin{lemma}\label{lem:large-clique-emc}
  For each fixed $k\ge3$, there exist $\eta>0$ and $s_0$ such that the following holds whenever $s\ge s_0$ and $k(s+1)\le n\le(k+1/2)s$.
  If $\mathcal H\subseteq\binom{[n]}k$ satisfies $\nu(\mathcal H)\le s$ and $\omega(\mathcal H)\ge(1-\eta)ks$, then $|\mathcal H|\le C_k(s)$.
  Equality holds if and only if $\mathcal H=\binom Uk$ for some $(ks+k-1)$-set $U\subseteq[n]$.
\end{lemma}

We first use \cref{lem:large-clique-emc} to prove the following theorem.
The proof of \cref{lem:large-clique-emc} is deferred to \cref{sec:large-clique-emc-proof}.

\begin{theorem}\label{lem:local-clique}
  For each fixed $k\ge3$, there exist $\delta>0$ and $s_0$ such that the following holds whenever $s\ge s_0$ and $k(s+1)\le n\le(k+1/2)s$.
  If $\mathcal F\subseteq\binom{[n]}k$ satisfies $\nu(\mathcal F)\le s$ and $|\binom Qk\setminus\mathcal F|\le\delta n^k$ for some $(ks+k-1)$-set $Q\subseteq[n]$, then $|\mathcal F|\le C_k(s)$.
  Equality holds if and only if $\mathcal F=\binom{Q'}k$ for some $(ks+k-1)$-set $Q'\subseteq[n]$.
\end{theorem}

\begin{proof}
  Choose $0<\eta<1$ as in \cref{lem:large-clique-emc}, and choose $s_0$ sufficiently large.
  Choose $0<\delta<(\eta/(2k+1))^k$.
  Let $s,n$ and $\mathcal F$ satisfy the hypotheses of \cref{lem:local-clique}.
  Fix a $(ks+k-1)$-set $Q\subseteq[n]$ such that $|\binom Qk\setminus\mathcal F|\le\delta n^k$.
  By relabelling, we may assume that $Q=[m]$, where $m=ks+k-1$.
  For distinct vertices $a,b\in[n]$ and a family $\mathcal G\subseteq\binom{[n]}k$, define the shift $S_{ab}(\mathcal G)$ by replacing each edge $E\in\mathcal G$ with $b\in E$, $a\notin E$ and $(E\setminus\{b\})\cup\{a\}\notin\mathcal G$ by $(E\setminus\{b\})\cup\{a\}$.
  All other edges remain unchanged.
  Let $\mathcal F'$ be obtained from $\mathcal F$ by repeatedly applying shifts $S_{ab}$ with $1\le a<b\le m$ until no such shift changes the family.
  Each shift preserves the size of the family and does not increase its matching number.
  Since both shifted vertices lie in $Q$, the number of missing edges in $Q$ also remains unchanged.

  Set $h=\lceil(1-\eta)ks\rceil$.
  We show that every $k$-set in $[h]$ belongs to $\mathcal F'$.
  Suppose for a contradiction that there exists $A\in\binom{[h]}k\setminus\mathcal F'$, and fix one such $A$.
  Fix an arbitrary $B\in\binom{\{h+1,\ldots,m\}}k$.
  If $B\in\mathcal F'$, then replacing its vertices one at a time by the vertices of $A$ would keep every resulting $k$-set in $\mathcal F'$, since every vertex of $A$ is smaller than every vertex of $B$ and $\mathcal F'$ is unchanged by all shifts $S_{ab}$ with $1\le a<b\le m$.
  This would give $A\in\mathcal F'$, contrary to the choice of $A$.
  Since $B$ was arbitrary, every $k$-set in $\{h+1,\ldots,m\}$ is missing from $\mathcal F'$.
  By our choice of $s_0$, we have $m-h\ge\eta ks+k-2\ge\max\{k,\eta ks/2\}$.
  Therefore,
  \[
    \left|\binom Qk\setminus\mathcal F'\right|
    \ge\binom{m-h}k
    \ge\left(\frac{m-h}{k}\right)^k
    \ge(\eta s/2)^k
    >\delta n^k,
  \]
  where the last inequality follows from $n\le(k+1/2)s$ and $\delta<(\eta/(2k+1))^k$.
  This contradicts $|\binom Qk\setminus\mathcal F'|=|\binom Qk\setminus\mathcal F|\le\delta n^k$.
  Hence, $[h]$ is a clique in $\mathcal F'$, so $\omega(\mathcal F')\ge(1-\eta)ks$.
  We also have $\nu(\mathcal F')\le s$.
  By \cref{lem:large-clique-emc}, we have $|\mathcal F|=|\mathcal F'|\le C_k(s)$.

  Suppose that $|\mathcal F|=C_k(s)$.
  The equality case of \cref{lem:large-clique-emc} implies $\mathcal F'=\binom{U_t}k$ for some $m$-set $U_t$.
  Write the sequence of shifts as $\mathcal F=\mathcal F^{(0)},\mathcal F^{(1)},\ldots,\mathcal F^{(t)}=\mathcal F'$, where $\mathcal F^{(\ell)}=S_{a_\ell b_\ell}(\mathcal F^{(\ell-1)})$ for $1\le\ell\le t$.
  Every $\mathcal F^{(\ell)}$ has $C_k(s)$ members and matching number at most $s$.
  We prove by reverse induction that every $\mathcal F^{(\ell)}$ is a clique family on $m$ vertices.
  The assertion holds for $\ell=t$.
  Suppose that $1\le\ell\le t$ and $\mathcal F^{(\ell)}=\binom{U_\ell}k$, where $|U_\ell|=m$.
  Since $S_{a_\ell b_\ell}$ only affects sets containing $a_\ell$ or $b_\ell$, every $k$-set in $U_\ell\setminus\{a_\ell,b_\ell\}$ also belongs to $\mathcal F^{(\ell-1)}$.
  Thus, $\mathcal F^{(\ell-1)}$ contains a clique on $U_\ell\setminus\{a_\ell,b_\ell\}$, whose size satisfies $|U_\ell\setminus\{a_\ell,b_\ell\}|\ge ks+k-3\ge ks\ge(1-\eta)ks$.
  The equality case of \cref{lem:large-clique-emc} therefore implies $\mathcal F^{(\ell-1)}=\binom{U_{\ell-1}}k$ for some $m$-set $U_{\ell-1}$.
  The induction shows that $\mathcal F=\binom{U_0}k$.
  Conversely, every clique family on $m$ vertices has $C_k(s)$ members.
\end{proof}

We now apply \cref{lem:local-clique} to finish the clique case.
By \eqref{eq:emc-clique-alternative},
\[
  \left|\binom{Q_i}k\setminus\mathcal F_i\right|
  \le\left|\mathcal F_i\triangle\binom{Q_i}k\right|=o(n_i^k).
\]
By \eqref{eq:emc-counterexample-sequence} and $n_i/s_i\to R_*\le\rho_k<k+1/2$, we have $k(s_i+1)\le n_i\le(k+1/2)s_i$ for all sufficiently large $i$.
By \cref{lem:local-clique}, for all sufficiently large $i$ we have
\[
  D_k(n_i,s_i)\le|\mathcal F_i|
  \le C_k(s_i)\le D_k(n_i,s_i).
\]
Equality holds throughout, so \cref{lem:local-clique} gives $\mathcal F_i=\binom{Q_i'}k$ for some $(ks_i+k-1)$-set $Q_i'\subseteq[n_i]$, a contradiction.

Thus, for all sufficiently large $s$, every family satisfying the hypotheses of \cref{thm:emc} and $|\mathcal F|\ge D_k(n,s)$ is a cover family or a clique family.
These have sizes $M_k(n,s)$ and $C_k(s)$, respectively, so the claimed bound and equality cases follow, including both types when the sizes coincide.
This proves \cref{thm:emc}.

\subsection{An exact bound for families containing a large clique}\label{sec:large-clique-emc-proof}
We first prove a quantitative lower bound for $C_k(s)-|\mathcal H'|$ and then deduce \cref{lem:large-clique-emc}.

A $k$-graph is \emph{$s$-saturated} if its matching number is $s$ and adding any missing edge creates a matching of size $s+1$.
We need the following fact.

\begin{fact}\label{fact:large-clique-deficit}
  Let $k\ge3$ be an integer and $0<\eta\le1/(64k^4)$.
  Let $n,s$ be positive integers such that $s\ge60k^4$, $\eta ks\ge k-1$, and $k(s+1)\le n\le(k+1/2)s$.
  Let $\mathcal H'\subseteq\binom{[n]}k$ be $s$-saturated, with $u=\omega(\mathcal H')\ge(1-\eta)ks$, and set $d=ks+k-1-u$.
  Then
  \[
    C_k(s)-|\mathcal H'|
    \ge d\binom{ks}{k-1}
    \left[2-\left(\frac n{ks}\right)^{k-1}-5k^3\eta-\frac{3k^3}{s}\right].
  \]
\end{fact}

In \cref{sec:clique-hm-proof}, we will apply \cref{fact:large-clique-deficit} with $d\ge2$ to families whose largest clique has at most $ks+k-3$ vertices.

\begin{proof}
  Set $m=ks+k-1$, and let $Q$ be a largest clique in $\mathcal H'$, of order $u$.
  Then $(1-\eta)ks\le u\le m$.

  Set $d=m-u\ge0$.
  The range $k(s+1)\le n\le(k+1/2)s$ implies $s\ge2k$.
  Since $u<n$ and $Q$ is a largest clique, there is a missing $k$-set with exactly $k-1$ vertices in $Q$.
  Since $\mathcal H'$ is $s$-saturated, adding this set creates an $(s+1)$-edge matching containing it, and the remaining $s$ edges form a matching in $\mathcal H'$.
  Choose an $s$-edge matching $\mathcal M$ in $\mathcal H'$ maximizing $|Q\cup V(\mathcal M)|$, and let $\mathcal M'=\{E\in \mathcal M:E\not\subseteq Q\}$.
  Every $s$-edge matching leaves at most $k-1$ vertices of $Q$ uncovered, since otherwise one more clique edge could be added.
  In particular, $|Q\cup V(\mathcal M)|=ks+|Q\setminus V(\mathcal M)|\le m$.
  The matching obtained by deleting the added missing set leaves its $k-1$ vertices in $Q$ uncovered.
  By the choice of $\mathcal M$, we have $|Q\cup V(\mathcal M)|=ks+k-1=m$.
  Write $T=Q\cup V(\mathcal M)$, $B=V(\mathcal M')$, and $r=|\mathcal M'|$.
  Then $|T\setminus Q|=d$.

  Since $|T|=ks+k-1$, we have
  \[
    C_k(s)-|\mathcal H'|
    =\left|\binom Tk\setminus\mathcal H'\right|
    -\left|\mathcal H'\setminus\binom Tk\right|.
  \]
  We first bound $\left|\binom Tk\setminus\mathcal H'\right|$ from below and then $\left|\mathcal H'\setminus\binom Tk\right|$ from above.

  Each edge of $\mathcal M'$ contains a vertex outside $Q$, so $r\le d$.
  Since $\eta ks\ge k-1$ and $\eta\le1/(64k^4)$, we have $0\le d\le\eta ks+k-1\le2\eta ks\le s/2$.
  Also, $|Q\setminus B|=u-(kr-d)=k(s-r)+k-1$.
  If $E\in\mathcal H'$ satisfies $E\not\subseteq Q$ and $E\cap B=\emptyset$, then $|Q\setminus(B\cup E)|\ge|Q\setminus B|-(k-1)=k(s-r)$.
  Choose $k(s-r)$ vertices of $Q\setminus(B\cup E)$ and partition them into $s-r$ clique edges.
  Together with $\mathcal M'$ and $E$, these form a matching of size $s+1$.
  Thus, every edge not contained in $Q$ meets $B$.

  To count the missing $k$-sets in $T$, for each $v\in T\setminus Q$ define
  \[
    d_Q(v)=\left|\left\{D\in\binom Q{k-1}:D\cup\{v\}\in\mathcal H'\right\}\right|.
  \]

  We claim that $d_Q(v)\le(k-1)d\binom u{k-2}$ for every $v\in T\setminus Q$.
  Suppose for a contradiction that $d_Q(v)>(k-1)d\binom u{k-2}$ for some $v\in T\setminus Q$.
  By the maximality of $Q$, there is a missing set $e=\{v\}\cup D$ with $D\in\binom Q{k-1}$.
  Let $\mathcal N$ be an $(s+1)$-edge matching in $\mathcal H'\cup\{e\}$ containing $e$, and let $\mathcal N'=\{E\in \mathcal N\setminus\{e\}:E\not\subseteq Q\}$.
  Since $\mathcal N\setminus\{e\}$ is an $s$-edge matching in $\mathcal H'$, the choice of $\mathcal M$ implies $|V(\mathcal N')\setminus Q|\le d$.
  Consequently, $|\mathcal N'|\le d$ and $|V(\mathcal N')\cap Q|\le(k-1)d$.
  Choose $D'\in\binom{Q\setminus V(\mathcal N')}{k-1}$ such that $e'=\{v\}\cup D'\in\mathcal H'$.
  Indeed, for each vertex of $V(\mathcal N')\cap Q$, choosing the other $k-2$ vertices produces exactly $\binom{u-1}{k-2}$ members of $\binom Q{k-1}$ containing it.
  Since $v\in T\setminus Q$, we have $d\ge1$, so the number of members of $\binom Q{k-1}$ meeting $V(\mathcal N')$ is at most $|V(\mathcal N')\cap Q|\binom{u-1}{k-2}<(k-1)d\binom u{k-2}$.
  Since $d_Q(v)>(k-1)d\binom u{k-2}$, at least one of the sets counted by $d_Q(v)$ avoids $V(\mathcal N')$.
  Since $v\notin V(\mathcal N')$ and $D'$ avoids $V(\mathcal N')$, the family $\mathcal N'\cup\{e'\}$ is a matching.
  Note that $\mathcal N$ consists of $\mathcal N'$, the added set $e$, and $s-|\mathcal N'|$ edges contained in $Q$.
  Both $e$ and $e'$ use exactly $k-1$ vertices of $Q$, so
  \[
    |Q\setminus(V(\mathcal N')\cup e')|
    \ge k(s-|\mathcal N'|).
  \]
  Choose $k(s-|\mathcal N'|)$ of these unused vertices and partition them into $s-|\mathcal N'|$ clique edges.
  Together with $\mathcal N'$ and $e'$, these edges form a matching of size $s+1$ in $\mathcal H'$, a contradiction.
  This shows that $d_Q(v)\le(k-1)d\binom u{k-2}$ for every $v\in T\setminus Q$.

  The sets consisting of one vertex of $T\setminus Q$ and $k-1$ vertices of $Q$ are distinct for different outside vertices.
  Hence, summing the bound on $d_Q(v)$ over these $d$ vertices, we obtain
  \begin{equation}\label{eq:glp-missing-edges}
    \left|\binom Tk\setminus\mathcal H'\right|
    \ge d\binom u{k-1}-(k-1)d^2\binom u{k-2}.
  \end{equation}

  We now bound the edges outside $\binom Tk$.
  Let $W=[n]\setminus T$.
  Every edge of $\mathcal H'\setminus\binom Tk$ meets both $W$ and $B$, so we split these edges into
  \[
    \mathcal E_1=\left\{E\in\mathcal H'\setminus\binom Tk:|E\cap B|=1\right\}
    \text{ and }
    \mathcal E_2=\left\{E\in\mathcal H'\setminus\binom Tk:|E\cap B|\ge2\right\}.
  \]
  Since $r\le d$, we obtain
  \begin{equation}\label{eq:glp-multiple-intersections}
    |\mathcal E_2|\le\binom{kr}{2}\binom n{k-2}
    \le\frac{k^2d^2}{2}\binom n{k-2}.
  \end{equation}

  Each edge of $\mathcal E_1$ has a unique vertex $b\in B$, and its remaining $k-1$ vertices avoid $B$ and meet $W$.
  We therefore define
  \[
    \mathcal L=\left\{L\in\binom{[n]\setminus B}{k-1}:L\cap W\ne\emptyset\right\}.
  \]
  Every edge of $\mathcal E_1$ has a unique representation $\{b\}\cup L$ with $b\in B$ and $L\in\mathcal L$.
  For $b\in B$, let $a(b)=\left|\{L\in\mathcal L:\{b\}\cup L\in\mathcal H'\}\right|$.
  To bound $|\mathcal E_1|=\sum_{b\in B}a(b)$, we estimate $\sum_{b\in F}a(b)$ for each $F\in \mathcal M'$.
  Fix $F\in \mathcal M'$ and write $\ell_F=|F\setminus Q|$.
  We first prove the following claim.

  \begin{claim}\label{clm:large-clique-e1-matching}
    There is no matching of size $\ell_F+1$ in $\mathcal E_1$ such that every edge meets $F$.
  \end{claim}

  \begin{proof}
    Suppose for a contradiction that $E_1,\ldots,E_{\ell_F+1}$ form such a matching.
    Each $E_j$ meets $B$ only in a single vertex of $F$, so it is disjoint from all the edges in $\mathcal M'\setminus\{F\}$.
    Replacing $F$ by these edges produces a matching $\widetilde{\mathcal M}=(\mathcal M'\setminus\{F\})\cup\{E_1,\ldots,E_{\ell_F+1}\}$ in $\mathcal H'$, so $|\widetilde{\mathcal M}|\le s$.
    Deleting $F$ removes $\ell_F$ vertices outside $Q$, while the new edges add at least $\ell_F+1$ distinct vertices of $W$.
    This implies $|V(\widetilde{\mathcal M})\setminus Q|\ge d+1$.
    Thus, $|Q\cup V(\widetilde{\mathcal M})|\ge u+d+1=k(s+1)$.
    Hence, we have $|Q\setminus V(\widetilde{\mathcal M})|=|Q\cup V(\widetilde{\mathcal M})|-k|\widetilde{\mathcal M}|\ge k(s+1-|\widetilde{\mathcal M}|)$.
    Choose $k(s+1-|\widetilde{\mathcal M}|)$ of these unused vertices and partition them into $s+1-|\widetilde{\mathcal M}|$ clique edges.
    Together with $\widetilde{\mathcal M}$, they form an $(s+1)$-matching in $\mathcal H'$, a contradiction.
  \end{proof}

  Set $A=(k-1)^2\binom n{k-2}$.
  We next show that at most $\ell_F$ vertices $b\in F$ satisfy $a(b)>A$.
  Otherwise, choose distinct $b_1,\ldots,b_{\ell_F+1}\in F$ with $a(b_j)>A$.
  We can choose pairwise disjoint sets $L_j\in\mathcal L$ successively, with $\{b_j\}\cup L_j\in\mathcal H'$.
  Fix any integer $j$ with $1\le j\le\ell_F+1$ and any pairwise disjoint sets $L_1,\ldots,L_{j-1}\in\mathcal L$ satisfying $\{b_t\}\cup L_t\in\mathcal H'$ for every $1\le t<j$.
  Note that $j\le\ell_F+1\le k$, since the $b_j$'s are distinct vertices of the $k$-set $F$.
  At step $j$, the previously chosen $L_1,\ldots,L_{j-1}$ occupy at most $(j-1)(k-1)\le(k-1)^2$ vertices.
  For each such vertex $v$, at most $\binom{n-1}{k-2}\le\binom n{k-2}$ candidates counted by $a(b_j)$ contain $v$.
  Thus, at most $A$ sets counted by $a(b_j)$ meet $\bigcup_{t=1}^{j-1}L_t$, and $a(b_j)>A$ allows us to choose $L_j$ disjoint from $\bigcup_{t=1}^{j-1}L_t$.
  Every $L_j$ avoids $B$, so the $\ell_F+1$ edges $\{b_j\}\cup L_j$ form a matching in $\mathcal E_1$ in which every edge meets $F$, contradicting \cref{clm:large-clique-e1-matching}.
  Thus, $\sum_{b\in F}a(b)\le\ell_F|\mathcal L|+kA$ for every $F\in \mathcal M'$.

  Summing over $F\in \mathcal M'$, we obtain
  \[
    |\mathcal E_1|=\sum_{b\in B}a(b)
    =\sum_{F\in \mathcal M'}\sum_{b\in F}a(b)
    \le d|\mathcal L|+krA
    \le d|\mathcal L|+k(k-1)^2d\binom n{k-2}.
  \]
  Here $\sum_{F\in \mathcal M'}\ell_F=\sum_{F\in \mathcal M'}|F\setminus Q|=|T\setminus Q|=d$, since every vertex of $T\setminus Q$ lies in exactly one edge of $\mathcal M'$.
  The last inequality uses $r\le d$.
  Every $L\in\mathcal L$ meets $W\subseteq[n]\setminus Q$, so $|\mathcal L|\le\binom n{k-1}-\binom u{k-1}$.
  Therefore,
  \begin{equation}\label{eq:glp-single-intersections}
    |\mathcal E_1|
    \le d\left(\binom n{k-1}-\binom u{k-1}\right)
    +k(k-1)^2d\binom n{k-2}.
  \end{equation}

  Combining \eqref{eq:glp-missing-edges}, \eqref{eq:glp-multiple-intersections} and \eqref{eq:glp-single-intersections}, and using $\binom u{k-2}\le\binom n{k-2}$, we obtain
  \[
    C_k(s)-|\mathcal H'|
    \ge d\left[2\binom u{k-1}-\binom n{k-1}
      -\left(\left(k-1+\frac{k^2}{2}\right)d+k(k-1)^2\right)\binom n{k-2}\right].
  \]
  We compare the three binomial coefficients with $\binom{ks}{k-1}$.
  For $s\ge2k$, we have $ks-j\ge ks/2$ for $0\le j\le k-2$.
  Since $u\ge(1-\eta)ks$ and $ks-j\ge ks/2$ for
$0\le j\le k-2$, a direct calculation gives
\[
  \binom{u}{k-1}
  \ge (1-2(k-1)\eta)\binom{ks}{k-1}.
\]
  Using $\prod_{j=0}^{k-2}(1-j/(ks))\ge1-(k-1)(k-2)/(2ks)$ and $1/(1-a)\le1+2a$ for $0\le a\le1/2$, we obtain
  \[
    \binom n{k-1}\le\left[\left(\frac n{ks}\right)^{k-1}+\frac{2(k-1)}s\right]\binom{ks}{k-1}.
  \]
  Since $k(s+1)\le n\le(k+1/2)s$ and $s\ge2k$,
a direct calculation gives
\[
  \binom n{k-2}<\frac3s\binom{ks}{k-1}.
\]
Combining the three estimates above with $d\le2\eta ks$, we obtain
\[
  \begin{aligned}
    C_k(s)-|\mathcal H'|
    &\ge d\binom{ks}{k-1}
    \biggl[
      2-\left(\frac n{ks}\right)^{k-1}
      -\bigl(4(k-1)+6k(k-1)+3k^3\bigr)\eta\\
    &\qquad\qquad
      -\frac{2(k-1)+3k(k-1)^2}{s}
    \biggr]\\
    &\ge d\binom{ks}{k-1}
    \left[
      2-\left(\frac n{ks}\right)^{k-1}
      -5k^3\eta-\frac{3k^3}{s}
    \right].
  \end{aligned}
\]
The second inequality follows from
$4(k-1)+6k(k-1)+3k^3\le5k^3$ and
$2(k-1)+3k(k-1)^2\le3k^3$ for $k\ge3$.
  This proves \cref{fact:large-clique-deficit}.\qedhere
\end{proof}

\begin{proof}[Proof of \cref{lem:large-clique-emc}]
  Choose $0<\eta\le1/(64k^4)$, and then choose $s_0=s_0(k,\eta)$ such that $s_0\ge60k^4$ and $\eta ks_0\ge k-1$.
  Let $\mathcal H$ satisfy the hypotheses of \cref{lem:large-clique-emc}.

  Add edges to $\mathcal H$ while the matching number remains at most $s$, obtaining a maximal family $\mathcal H'$.
  Since $n\ge k(s+1)$, the family $\mathcal H'$ has a missing edge.
  By maximality, adding any missing edge creates an $(s+1)$-matching.
  Thus, $\nu(\mathcal H')=s$ and $\mathcal H'$ is $s$-saturated.
  Let $Q$ be a largest clique in $\mathcal H'$, set $u=|Q|$ and $d=ks+k-1-u$, and apply \cref{fact:large-clique-deficit}.
  Since $n/(ks)\le1+1/(2k)$, we have
  \[
    2-\left(\frac n{ks}\right)^{k-1}
    \ge2-\left(1+\frac1{2k}\right)^{k-1}
    >2-\mathrm{e}^{1/2}>0.
  \]
  By our choices of $\eta$ and $s_0$, we have $5k^3\eta+3k^3/s\le5/(64k)+1/(20k)<2-\mathrm{e}^{1/2}$.
  By \cref{fact:large-clique-deficit}, we therefore have $|\mathcal H|\le|\mathcal H'|\le C_k(s)$, with strict inequality whenever $d>0$.
  If $d=0$, then $|Q|=ks+k-1$.
  Any edge $E\in\mathcal H'$ not contained in $Q$ leaves at least $ks$ vertices of $Q$ unused.
  Partitioning $ks$ of these vertices into $s$ clique edges gives an $(s+1)$-matching together with $E$, a contradiction.
  Thus, $\mathcal H'=\binom Qk$, and equality forces $\mathcal H=\binom Qk$.
  Conversely, every clique family on $ks+k-1$ vertices has $C_k(s)$ members and satisfies the hypotheses.
\end{proof}

\section{Stability and Hilton--Milner refinements of the Erd\H{o}s matching bound}\label{sec:concluding}\label{sec:matching-consequences}
For fixed $k\ge3$ and sufficiently large $s$, we consider families
$\mathcal F\subseteq\binom{[n]}k$ with $\nu(\mathcal F)\le s$.
In \cref{sec:unified-stability}, we show that, when $n/s$ is bounded,
$|\mathcal F|\ge D_k(n,s)-\delta n^k$ with sufficiently small $\delta>0$
implies that $\mathcal F$ can be changed into a cover or clique family
by adding or deleting at most $\varepsilon n^k$ edges.
In \cref{sec:hm-proofs}, we determine the maximum size and all equality cases
when no $s$-set meets every edge, which excludes the cover families.
In \cref{sec:clique-hm-proof}, we instead require at least $k(s+1)$ non-isolated
vertices, which excludes the clique families, and determine the maximum size
and all equality cases for $k(s+1)\le n\le(\rho_k-\varepsilon)s$.

\subsection{A stability result for the Erd\H{o}s matching conjecture}\label{sec:unified-stability}
We use \cref{lem:matching-sequence} to prove the following unified stability result.
As in \cref{sec:emc}, write $D_k(n,s)=\max\{C_k(s),M_k(n,s)\}$.

\begin{proposition}\label{prop:unified-stability}
  Let $k\ge3$ be an integer and let $L>k$.
  For every $\varepsilon>0$ and $\eta>0$, there exist $\delta=\delta(k,L,\varepsilon,\eta)>0$ and an integer $s_0=s_0(k,L,\varepsilon,\eta)$ such that the following holds.
  Let $n$ and $s$ be integers with $s\ge s_0$ and $k(s+1)\le n\le Ls$, and let $\mathcal F\subseteq\binom{[n]}k$ satisfy
  \[
    \nu(\mathcal F)\le s
    \text{ and }
    |\mathcal F|\ge D_k(n,s)-\delta n^k.
  \]
  Then the following hold.
  \begin{enumerate}[label=\textup{(\roman*)},leftmargin=*,itemsep=0pt]
    \item If $n\ge(\rho_k+\eta)s$, then there is an $s$-set $S\subseteq[n]$ such that $|\mathcal F\triangle\mathcal A(S)|\le\varepsilon n^k$.
    \item If $n\le(\rho_k-\eta)s$, then there is a $(ks+k-1)$-set $Q\subseteq[n]$ such that $\left|\mathcal F\triangle\binom Qk\right|\le\varepsilon n^k$.
    \item If $(\rho_k-\eta)s<n<(\rho_k+\eta)s$, then there is either an $s$-set $S\subseteq[n]$ such that $|\mathcal F\triangle\mathcal A(S)|\le\varepsilon n^k$, or a $(ks+k-1)$-set $Q\subseteq[n]$ such that $\left|\mathcal F\triangle\binom Qk\right|\le\varepsilon n^k$.
  \end{enumerate}
\end{proposition}

\begin{proof}
  Suppose for a contradiction that \cref{prop:unified-stability} fails for some fixed $k$, $L$, $\varepsilon$ and $\eta$ as in the statement.
  Then there are integer pairs $(n_i,s_i)$ with $s_i\ge i$ and $k(s_i+1)\le n_i\le Ls_i$, and families $\mathcal F_i\subseteq\binom{[n_i]}k$ satisfying
  \[
    \nu(\mathcal F_i)\le s_i
    \text{ and }
    |\mathcal F_i|\ge D_k(n_i,s_i)-\frac{n_i^k}{i},
  \]
  such that the conclusion of \cref{prop:unified-stability} fails for every $i$.

  Since $n_i/s_i$ is bounded, after passing to a subsequence, we may assume that $n_i/s_i\to R$ for some $R\in[k,L]$.
  By \cref{lem:matching-sequence}\textup{(i)}, we have $\min\{d_B(\mathcal F_i),d_C(\mathcal F_i)\}=o(n_i^k)$.
  Since $\varepsilon>0$, for every sufficiently large $i$, at least one of these distances is at most $\varepsilon n_i^k$.
  Thus, \textup{(iii)} of \cref{prop:unified-stability} holds whenever $(\rho_k-\eta)s_i<n_i<(\rho_k+\eta)s_i$.
  Since every $\mathcal F_i$ is a counterexample, after omitting finitely many initial terms and passing to a further subsequence, we may assume that either $n_i\ge(\rho_k+\eta)s_i$ for every $i$ or $n_i\le(\rho_k-\eta)s_i$ for every $i$.
  In the first case, failure of the cover conclusion means that $d_B(\mathcal F_i)>\varepsilon n_i^k$.
  But $R\ge\rho_k+\eta>\rho_k$, so \cref{lem:matching-sequence}\textup{(ii)} implies $d_B(\mathcal F_i)=o(n_i^k)$, a contradiction.
  In the second case, failure of the clique conclusion means that $d_C(\mathcal F_i)>\varepsilon n_i^k$.
  But $k\le R\le\rho_k-\eta<\rho_k$, so \cref{lem:matching-sequence}\textup{(ii)} implies $d_C(\mathcal F_i)=o(n_i^k)$, again a contradiction.
  This proves \cref{prop:unified-stability}.
\end{proof}

\subsection{A cover-type Hilton--Milner refinement}\label{sec:hm-proofs}

For an $s$-set $S\subseteq[n]$ and a $(ks+k-1)$-set $Q\subseteq[n]$, the cover family $\mathcal A(S)$ and the clique family $\binom Qk$ prevent an $(s+1)$-matching in different ways.
In the cover family, every edge meets a fixed $s$-set, while the clique family has only $ks+k-1$ non-isolated vertices.
We first consider how large a family with matching number at most $s$ can be if it is not contained in any cover family.
Writing $\tau(\mathcal F)$ for the minimum size of a vertex set meeting every edge of $\mathcal F$, we therefore seek the maximum size under $\nu(\mathcal F)\le s<\tau(\mathcal F)$.
For $s=1$, this is equivalent to the problem of nontrivial intersecting families studied by Hilton and Milner~\cite{HiltonMilner}.

\subsubsection{Extremal constructions and statement}\label{sec:hm-constructions}

We first describe the extremal constructions.
Under the condition that $\nu(\mathcal F)\le s<\tau(\mathcal F)$, the clique family $\binom Qk$ remains an admissible lower bound construction since its cover number is $ks>s$.
A competing construction is the following family $\mathcal H(S_0,u,E_0)$, introduced by Bollob\'as, Daykin and Erd\H{o}s~\cite{BDE}.
For integers $k\ge3$, $s\ge1$ and $n\ge k(s+1)$, choose an $(s-1)$-set $S_0\subseteq[n]$ and set $V=[n]\setminus S_0$.
Choose $u\in V$ and $E_0\in\binom{V\setminus\{u\}}k$, and define
\[
  \mathcal H(S_0,u,E_0)
  =\mathcal A(S_0)\cup\{E_0\}
  \cup\left\{E\in\binom Vk:u\in E,\ E\cap E_0\ne\emptyset\right\}.
\]
The edges avoiding $S_0$ are pairwise intersecting and have no common vertex.
A matching therefore has at most one edge avoiding $S_0$ and at most $s-1$ edges meeting it, so $\nu(\mathcal H(S_0,u,E_0))\le s$.
Equivalently, $\mathcal H(S_0,u,E_0)$ can be obtained from $\mathcal A(S_0\cup\{u\})$ by deleting the $\binom{n-s-k}{k-1}$ edges through $u$ that avoid $S_0\cup E_0$ and adding $E_0$.
Consequently, its size is
\[
  |\mathcal H(S_0,u,E_0)|
  =\mathcal{HM}_k(n,s)
  :=M_k(n,s)-\binom{n-s-k}{k-1}+1.
\]
The family $\mathcal H(S_0,u,E_0)$ is a
Hilton--Milner-type construction.
Moreover, $\tau(\mathcal H(S_0,u,E_0))=s+1$.
Indeed, $S_0\cup\{u,v\}$ is a cover for any $v\in E_0$.
Any cover $X$ of size at most $s$ must contain $S_0$.
Otherwise, since $n-|X|\ge n-s\ge k$, there is an edge through a vertex of $S_0\setminus X$ that avoids $X$, a contradiction.
But then $X\cap V$ has size at most one and cannot cover the edges avoiding $S_0$, which have no common vertex.

When $k=3$, the third construction occurs.
For $T\in\binom V3$, define
\[
  \mathcal T(S_0,T)
  =\mathcal A(S_0)\cup
  \left\{E\in\binom V3:|E\cap T|\ge2\right\},
\]
where $S_0$ is again an $(s-1)$-set in $[n]$ and $V=[n]\setminus S_0$.
The triples in $V$ meeting $T$ in at least two vertices are pairwise intersecting and have cover number two.
Thus, $\nu(\mathcal T(S_0,T))\le s$.
Any cover of size at most $s$ would again have to contain $S_0$, leaving at most one vertex to cover the triples in $V$.
Hence, $\tau(\mathcal T(S_0,T))=s+1$.
Since $|V|=n-s+1$, we have
\[
  |\mathcal T(S_0,T)|
  =\binom n3-\binom{n-s+1}3+3(n-s-2)+1
  =\mathcal{HM}_3(n,s).
\]

Write $\mathcal R_k(n,s)$ for the collection of all families $\mathcal H(S_0,u,E_0)$, together with the families $\mathcal T(S_0,T)$ when $k=3$.
As shown above, every member has size $\mathcal{HM}_k(n,s)$, matching number at most $s$, and cover number $s+1$.

For $s=1$, the theorem of Hilton and Milner~\cite{HiltonMilner} establishes the sharp bound $|\mathcal F|\le\mathcal{HM}_k(n,1)$ whenever $k\ge3$ and $n>2k$.
For general $s$, \cref{thm:bde} implies $|\mathcal F|\le\mathcal{HM}_k(n,s)$ whenever $k\ge3$ and $n>2k^3s$.
For $k\ge3$ and $\nu(\mathcal F)=s<\tau(\mathcal F)$, Frankl and Kupavskii~\cite{FranklKupavskii2019} proved $|\mathcal F|\le\mathcal{HM}_k(n,s)$ when $n\ge(s+\max\{25,2s+2\})k$, and also when $n\ge(2+o(1))ks$ as $s\to\infty$, with $o(1)$ depending only on $s$.
For the full range $n\ge k(s+1)$, the problem remains open in general.

Our results determine the sharp bound and all equality cases throughout the full range $n\ge k(s+1)$, for every fixed $k\ge3$ and all sufficiently large $s$.

\begin{theorem}\label{cor:hilton-milner}
  For every fixed integer $k\ge3$, there is an integer $s_{\mathcal{HM}}(k)$ such that, for all integers $s\ge s_{\mathcal{HM}}(k)$ and $n\ge k(s+1)$, every family $\mathcal F\subseteq\binom{[n]}k$ with $\nu(\mathcal F)\le s<\tau(\mathcal F)$ satisfies
  \[
    |\mathcal F|\le\max\{C_k(s),\mathcal{HM}_k(n,s)\}.
  \]
  Equality holds if and only if one of the following statements holds.
  \begin{enumerate}[label=\textup{(\roman*)},leftmargin=*,itemsep=0pt]
    \item $C_k(s)>\mathcal{HM}_k(n,s)$ and $\mathcal F=\binom Qk$ for some $(ks+k-1)$-set $Q\subseteq[n]$.
    \item $C_k(s)<\mathcal{HM}_k(n,s)$ and $\mathcal F\in\mathcal R_k(n,s)$.
    \item $C_k(s)=\mathcal{HM}_k(n,s)$, $\mathcal F=\binom Qk$ for some $(ks+k-1)$-set $Q\subseteq[n]$ or $\mathcal F\in\mathcal R_k(n,s)$.
  \end{enumerate}
\end{theorem}

\subsubsection{The Hilton--Milner bound in the sparse range}

We will use the following lemma in the proof of \cref{cor:hilton-milner}.

\begin{lemma}\label{lem:sparse-hm}
  Let $k\ge3$, $n\ge k$, and $b\ge1$ be integers with $b\le n/(4k^3)$.
  If $\mathcal J\subseteq\binom{[n]}k$ satisfies $\nu(\mathcal J)\le b<\tau(\mathcal J)$, then $|\mathcal J|\le\mathcal{HM}_k(n,b)$.
  Equality holds if and only if $\mathcal J\in\mathcal R_k(n,b)$.
\end{lemma}

\begin{proof}
  Since $n\ge4k^3b>2k^3b$, \cref{thm:bde} implies $|\mathcal J|\le\mathcal{HM}_k(n,b)$.
  We now determine the equality cases by induction on $b$.
  When $b=1$, we have $n>2k$, so equality implies $\mathcal J\in\mathcal R_k(n,1)$ by~\cite[Theorem~1.2]{HanKohayakawa}.

  Suppose $b\ge2$, $\nu(\mathcal J)\le b<\tau(\mathcal J)$, and $|\mathcal J|=\mathcal{HM}_k(n,b)$.
  We first find a vertex $v$ with $d_{\mathcal J}(v)>kb\binom{n-2}{k-2}$, and then prove $\nu(\mathcal J-v)\le b-1$.
  To this end, let $\mathcal M$ be a maximal matching in $\mathcal J$ and set $U=\bigcup_{E\in\mathcal M}E$.
  Since $|\mathcal M|\le b$, we have $|U|=k|\mathcal M|\le kb$.
  Every edge of $\mathcal J$ meets $U$, since otherwise it could be added to $\mathcal M$, contradicting maximality.
  Thus, counting incidences between $U$ and the edges of $\mathcal J$, we obtain
  \[
    \sum_{u\in U}d_{\mathcal J}(u)
    =\sum_{E\in\mathcal J}|E\cap U|
    \ge|\mathcal J|.
  \]
  By averaging, there is a vertex $v\in U$ such that
  \[
    d_{\mathcal J}(v)\ge\frac{|\mathcal J|}{|U|}
    \ge\frac{|\mathcal J|}{kb}
    =\frac{\mathcal{HM}_k(n,b)}{kb}>\frac{b-1}{kb}\binom{n-b+1}{k-1}
    \ge\frac1{2k}\binom{n-b+1}{k-1},
  \]
  where the strict inequality follows from
  \[
    \mathcal{HM}_k(n,b)
    =\mathcal{HM}_k(n-b+1,1)+\sum_{i=1}^{b-1}\binom{n-i}{k-1},
  \]
  obtained from the definition of $\mathcal{HM}_k(n,s)$.
  Calculation shows that
  \[
    \begin{aligned}
      \binom{n-b+1}{k-1}
       & =\binom{n-1}{k-1}-\sum_{i=1}^{b-2}\binom{n-i-1}{k-2}   \\
       & \ge\binom{n-1}{k-1}-(b-2)\binom{n-2}{k-2}              \\
       & =\left(1-\frac{(b-2)(k-1)}{n-1}\right)\binom{n-1}{k-1}
      \ge\frac12\binom{n-1}{k-1},
    \end{aligned}
  \]
  where the last inequality follows from $n\ge4k^3b$.
  Consequently,
  \begin{equation}\label{eq:sparse-hm-high-degree}
    d_{\mathcal J}(v)
    >\frac1{4k}\binom{n-1}{k-1}
    =\frac{n-1}{4k(k-1)}\binom{n-2}{k-2}
    >kb\binom{n-2}{k-2}.
  \end{equation}
  The last inequality follows from $n-1\ge4k^3b-1>4k^2(k-1)b$.

  Let $\mathcal J-v=\{E\in\mathcal J:v\notin E\}$.
  We claim that $\nu(\mathcal J-v)\le b-1$.
  Suppose for a contradiction that $\mathcal J-v$ contains a matching $\mathcal N$ of $b$ edges.
  Let $W=\bigcup_{E\in\mathcal N}E$.
  Then $|W|=kb$ and $v\notin W$.
  For each $w\in W$, at most $\binom{n-2}{k-2}$ edges of $\mathcal J$ contain both $v$ and $w$.
  Thus, at most $kb\binom{n-2}{k-2}$ edges of $\mathcal J$ contain $v$ and meet $W$.
  By \eqref{eq:sparse-hm-high-degree}, there exists an edge $F\in\mathcal J$ that contains $v$ and is disjoint from $W$.
  Then $\mathcal N\cup\{F\}$ is a matching of $b+1$ edges in $\mathcal J$, contradicting $\nu(\mathcal J)\le b$.
  Thus, $\nu(\mathcal J-v)\le b-1$.
  Also, $\tau(\mathcal J-v)>b-1$, since a cover of size at most $b-1$ together with $v$ would cover $\mathcal J$ with at most $b$ vertices.
  Since $n-1\ge4k^3(b-1)>2k^3(b-1)$, \cref{thm:bde} applies to $\mathcal J-v$, and hence
  \[
    |\mathcal J|
    =d_{\mathcal J}(v)+|\mathcal J-v|
    \le\binom{n-1}{k-1}+\mathcal{HM}_k(n-1,b-1)
    =\mathcal{HM}_k(n,b).
  \]
  Since $|\mathcal J|=\mathcal{HM}_k(n,b)$, equality holds throughout.
  Thus, $\mathcal J$ contains every $k$-set through $v$, and $\mathcal J-v$ attains the bound $\mathcal{HM}_k(n-1,b-1)$.
  By induction, $\mathcal J-v\in\mathcal R_k(n-1,b-1)$ after relabeling.
  Adding all $k$-sets through $v$ replaces $S_0$ by $S_0\cup\{v\}$ in the construction $\mathcal H(S_0,u,E_0)$, or in $\mathcal T(S_0,T)$ when $k=3$.
  Hence, $\mathcal J\in\mathcal R_k(n,b)$.
  Every family in $\mathcal R_k(n,b)$ attains equality, as verified in \cref{sec:hm-constructions}.
  This completes the proof.
\end{proof}

\subsubsection{Proof of Theorem~\ref{cor:hilton-milner}}

Fix $k\ge3$ and choose $\eta>0$ so that $k<\rho_k-\eta<\rho_k+\eta<k+1/2$.
Choose $\varepsilon>0$ sufficiently small for \cref{lem:local-cover} with $R=\rho_k-\eta$ and $L=4k^3$, for \cref{lem:local-clique}, and for the comparisons below.
Let $s$ be sufficiently large in terms of $k$ and let $\mathcal F$ satisfy the hypotheses of \cref{cor:hilton-milner} with $|\mathcal F|\ge\max\{C_k(s),\mathcal{HM}_k(n,s)\}$.
We show that $\mathcal F$ is either a clique family or a member of $\mathcal R_k(n,s)$.
This holds by \cref{lem:sparse-hm} when $n\ge4k^3s$, so assume $n<4k^3s$ from now on.

Direct calculation shows that
\[
  \begin{aligned}
    0 & \le D_k(n,s)-\max\{C_k(s),\mathcal{HM}_k(n,s)\} \\
      & \le M_k(n,s)-\mathcal{HM}_k(n,s)
    =\binom{n-s-k}{k-1}-1\le n^{k-1},
  \end{aligned}
\]
and thus $|\mathcal F|\ge D_k(n,s)-n^{k-1}=D_k(n,s)-o(n^k)$.
By \cref{prop:unified-stability}, there is either an $s$-set $S\subseteq[n]$ with $|\mathcal F\triangle\mathcal A(S)|\le\varepsilon n^k$, or a $(ks+k-1)$-set $Q\subseteq[n]$ with $\left|\mathcal F\triangle\binom Qk\right|\le\varepsilon n^k$.

\medskip
\noindent{\bf Case 1.} There is an $s$-set $S\subseteq[n]$ with $|\mathcal F\triangle\mathcal A(S)|\le\varepsilon n^k$.
\medskip

We first show that $n\ge(\rho_k-\eta)s$ in this case.
If $n<(\rho_k-\eta)s$, then \cref{fact:construction-densities} and the definition of $\rho_k$ imply $C_k(s)>M_k(n,s)+\varepsilon n^k$ for sufficiently small $\varepsilon$ and sufficiently large $s$.
This contradicts
\[
  C_k(s)\le|\mathcal F|
  \le|\mathcal A(S)|+\varepsilon n^k
  =M_k(n,s)+\varepsilon n^k.
\]
By \cref{lem:local-cover}, there is a set $B\subseteq S$ such that $\nu(\mathcal F[U])\le b\le N/(4k^3)$,
where $b=|B|$, $U=([n]\setminus S)\cup B$, and $N=|U|=n-s+b$.
Also, $N\ge n-s\ge k$.
We claim that $\tau(\mathcal F[U])>b$.
Suppose for a contradiction that $\mathcal F[U]$ has a cover $T\subseteq U$ with $|T|\le b$.
Then $T\cup(S\setminus B)$ covers $\mathcal F$, since every member of $\mathcal F$ avoiding $S\setminus B$ belongs to $\mathcal F[U]$.
But $|T\cup(S\setminus B)|\le b+(s-b)=s$, contradicting $\tau(\mathcal F)>s$.
In particular, $b\ge1$, since $b=0$ would force $\mathcal F[U]=\emptyset$.
Applying \cref{lem:sparse-hm} to $\mathcal F[U]$, we obtain $|\mathcal F[U]|\le\mathcal{HM}_k(N,b)$.
Every member of $\mathcal F$ either meets $S\setminus B$ or belongs to $\mathcal F[U]$, so
\[
  \begin{aligned}
    \max\{C_k(s),\mathcal{HM}_k(n,s)\}
     & \le|\mathcal F|                           \\
     & \le\binom nk-\binom Nk+|\mathcal F[U]|    \\
     & \le\binom nk-\binom Nk+\mathcal{HM}_k(N,b)
    =\mathcal{HM}_k(n,s),
  \end{aligned}
\]
where the last equality uses $N-b=n-s$.
Equality holds throughout.
Thus, every $k$-set meeting $S\setminus B$ belongs to $\mathcal F$.
Moreover, \cref{lem:sparse-hm} implies $\mathcal F[U]\in\mathcal R_k(N,b)$ after relabeling $U$.
Hence, $\mathcal F\in\mathcal R_k(n,s)$, with the $(b-1)$-set in the construction of $\mathcal F[U]$ enlarged to an $(s-1)$-set by adding $S\setminus B$.

\medskip
\noindent{\bf Case 2.} There is a $(ks+k-1)$-set $Q\subseteq[n]$ with $\left|\mathcal F\triangle\binom Qk\right|\le\varepsilon n^k$.
\medskip

We first show that $n\le(\rho_k+\eta)s$.
If $n>(\rho_k+\eta)s$, then \cref{fact:construction-densities}, the definition of $\rho_k$, and $M_k(n,s)-\mathcal{HM}_k(n,s)\le n^{k-1}$ imply $\mathcal{HM}_k(n,s)>C_k(s)+\varepsilon n^k$ for sufficiently small $\varepsilon$ and sufficiently large $s$.
This contradicts
\[
  \mathcal{HM}_k(n,s)\le|\mathcal F|
  \le\left|\binom Qk\right|+\varepsilon n^k
  =C_k(s)+\varepsilon n^k.
\]
Since $n\le(\rho_k+\eta)s<(k+1/2)s$, \cref{lem:local-clique} implies that
\[
  \max\{C_k(s),\mathcal{HM}_k(n,s)\}
  \le|\mathcal F|\le C_k(s).
\]
Again equality holds throughout, and the equality case of \cref{lem:local-clique} shows that $\mathcal F$ is a clique family.

\medskip
Thus, in either case, $|\mathcal F|=\max\{C_k(s),\mathcal{HM}_k(n,s)\}$, and $\mathcal F$ has one of the forms stated in \cref{cor:hilton-milner}.
This completes the proof of \cref{cor:hilton-milner}.

\subsection{A clique-type Hilton--Milner refinement}\label{sec:clique-hm-proof}

\Cref{cor:hilton-milner} refines the Erd\H{o}s matching bound by imposing $\tau(\mathcal F)>s$, which excludes the cover families but leaves the clique families admissible.
For a family $\mathcal F\subseteq\binom{[n]}k$, write $V^+(\mathcal F)=\bigcup_{E\in\mathcal F}E$ for its set of non-isolated vertices.
A clique family has exactly $ks+k-1$ non-isolated vertices.
We now seek the maximum size of $\mathcal F$ subject to
\[
  \nu(\mathcal F)\le s
  \text{ and }
  |V^+(\mathcal F)|\ge k(s+1).
\]
The cover family has no isolated vertices and remains admissible, so a strict improvement on the Erd\H{o}s matching bound is possible only when $C_k(s)>M_k(n,s)$, which requires $n\le(k+c_k+o(1))s$ as $s\to\infty$, since $\rho_k=k+c_k$.

\subsubsection{Extremal constructions and statement}\label{sec:clique-constructions}

We now describe the extremal constructions.
For integers $k\ge3$, $s\ge1$, and $n\ge k(s+1)$, set $q=ks+k-2$ and $t=n-q\ge2$.
Choose a $q$-set $Q\subseteq[n]$, a set $D\in\binom Q{k-1}$, a vertex $v\in[n]\setminus Q$, and a vertex $u\in Q$, and define
\[
  \begin{split}
    \mathcal J_1(Q,D,v)={}&\binom Qk
    \cup\left\{\{v\}\cup A:A\in\binom Q{k-1},\ A\cap D\ne\emptyset\right\}\\
    &\quad\cup\left\{\{w\}\cup D:w\in[n]\setminus(Q\cup\{v\})\right\},\\
    \mathcal J_2(Q,u)={}&\binom Qk
    \cup\left\{\{w\}\cup A:w\in[n]\setminus Q,\
    A\in\binom Q{k-1},\ u\in A\right\}.
  \end{split}
\]
Both families have no isolated vertices.
Indeed, they contain $\binom Qk$, and each vertex $w\in[n]\setminus Q$ belongs to $\{w\}\cup D$ in $\mathcal J_1$ and to $\{w,u\}\cup A$ in $\mathcal J_2$, where $A$ is any $(k-2)$-subset of $Q\setminus\{u\}$.
In either construction, the members meeting $[n]\setminus Q$ are pairwise intersecting and each contains exactly one vertex outside $Q$.
Thus, every matching covers at most $q+1=k(s+1)-1$ vertices and has at most $s$ members.
Their respective sizes are
\begin{align*}
  J_k^{(1)}(n,s) & =\binom qk+\binom q{k-1}
  -\binom{q-k+1}{k-1}+t-1,                       \\
  J_k^{(2)}(n,s) & =\binom qk+t\binom{q-1}{k-2}.
\end{align*}
We write
$
  J_k(n,s)=\max\{J_k^{(1)}(n,s),J_k^{(2)}(n,s)\}.
$
For fixed $k\ge3$ and sufficiently large $s$, Frankl~\cite{Frankl2020Nontrivial} proved that, when $n=k(s+1)$, the maximum size of a family with no isolated vertices and matching number at most $s$ is $J_k^{(1)}(n,s)$.
For each fixed $k\ge3$, Guo and Lu~\cite{GuoLu2026} proved the sharp bound $|\mathcal F|\le J_k(n,s)$ for families with $\nu(\mathcal F)\le s$ and $|V^+(\mathcal F)|\ge k(s+1)$ throughout $k(s+1)\le n\le(1+\beta)ks+k-2$ for some sufficiently small constant $0<\beta=\beta(k)<1/(3^{k+3}k!\,k^k)$ and all sufficiently large $s$.
For $k=3$, Guo, Lu, and Peng~\cite{GLP} proved the sharp bound $|\mathcal F|\le J_3(n,s)$ throughout $3s/(1-c_0)\le n\le45s/13$ for any sufficiently small fixed constant $c_0>0$ and all sufficiently large $s$.

\begin{theorem}\label{cor:clique-hilton-milner}
  For every fixed integer $k\ge3$ and every fixed $0<\varepsilon<\rho_k-k$, there is an integer $s_{\mathrm{cl}}(k,\varepsilon)$ such that, for all integers $s\ge s_{\mathrm{cl}}(k,\varepsilon)$ and $k(s+1)\le n\le(\rho_k-\varepsilon)s$, every family $\mathcal F\subseteq\binom{[n]}k$ with $\nu(\mathcal F)\le s$ and $|V^+(\mathcal F)|\ge k(s+1)$ satisfies
  \[
    |\mathcal F|\le\max\{J_k^{(1)}(n,s),J_k^{(2)}(n,s)\}.
  \]
  Equality holds if and only if one of the following statements holds.
  \begin{enumerate}[label=\textup{(\roman*)},leftmargin=*,itemsep=0pt]
    \item $J_k^{(1)}(n,s)>J_k^{(2)}(n,s)$ and $\mathcal F=\mathcal J_1(Q,D,v)$.
    \item $J_k^{(1)}(n,s)<J_k^{(2)}(n,s)$ and $\mathcal F=\mathcal J_2(Q,u)$.
    \item $J_k^{(1)}(n,s)=J_k^{(2)}(n,s)$ and either $\mathcal F=\mathcal J_1(Q,D,v)$ or $\mathcal F=\mathcal J_2(Q,u)$.
  \end{enumerate}
\end{theorem}

For each fixed $k\ge3$, the limiting coefficient $\rho_k$ in \cref{cor:clique-hilton-milner} is asymptotically optimal.
Indeed, the cover family $\mathcal A(S)$ has no isolated vertices and there is a constant $K=K(k)>0$ such that $M_k(n,s)>J_k(n,s)$ for all sufficiently large $s$ and all $n\ge\rho_k s+K$.

\subsubsection{The refined bound for families containing a large clique}

To prove $|\mathcal H|\le J_k(n,s)$ in \cref{lem:large-clique}, we use the following theorem of Shi, Frankl and Qian.

\begin{theorem}[Shi--Frankl--Qian~{\cite{SFQ}}]
  \label{thm:cross-intersection}
  Let $r\ge1$, $q\ge2r$ and $t\ge2$.
  Suppose $\mathcal L_1,\ldots,\mathcal L_t\subseteq\binom{[q]}r$ are nonempty and every member of $\mathcal L_i$ meets every member of $\mathcal L_j$ whenever $i\ne j$.
  Then
  \[
    \sum_{i=1}^t|\mathcal L_i|
    \le\max\left\{\binom qr-\binom{q-r}r+t-1,\ t\binom{q-1}{r-1}\right\}.
  \]
\end{theorem}

To show that $|\mathcal H|=J_k(n,s)$ in \cref{lem:large-clique} forces $\mathcal H=\mathcal J_1(Q,D,v)$ or $\mathcal H=\mathcal J_2(Q,u)$, we use the following special case of a theorem of Shi, Frankl and Qian.

\begin{theorem}[Shi--Frankl--Qian~{\cite{SFQ}}]\label{thm:weighted-cross-intersection}
  Let $k\ge3$ and $q>2(k-1)$ be integers, let $Q$ be a $q$-set, and let $c>0$.
  Suppose that $\mathcal A,\mathcal B\subseteq\binom Q{k-1}$ are nonempty, every member of $\mathcal A$ meets every member of $\mathcal B$, and $|\mathcal B|\le\binom{q-1}{k-2}$.
  Then
  \[
    |\mathcal A|+c|\mathcal B|
    \le\max\left\{\binom q{k-1}-\binom{q-k+1}{k-1}+c,\ (1+c)\binom{q-1}{k-2}\right\}.
  \]
  If equality holds, then either $\mathcal B=\{D\}$ and $\mathcal A=\{A\in\binom Q{k-1}:A\cap D\ne\emptyset\}$ for some $D\in\binom Q{k-1}$, or $\mathcal A=\mathcal B=\{A\in\binom Q{k-1}:z\in A\}$ for some $z\in Q$.
\end{theorem}

To prove \cref{cor:clique-hilton-milner}, we need the following lemma.

\begin{lemma}\label{lem:large-clique}
  For each fixed $k\ge3$, there exist $\eta>0$ and $s_0$ such that the following holds whenever $s\ge s_0$ and $k(s+1)\le n\le(k+1/2)s$.
  If $\mathcal H\subseteq\binom{[n]}k$ satisfies $\nu(\mathcal H)\le s$, $|V^+(\mathcal H)|\ge k(s+1)$, and $\omega(\mathcal H)\ge(1-\eta)ks$, then $|\mathcal H|\le J_k(n,s)$.
  Equality holds if and only if one of the following statements holds for some admissible choice of the parameters defining $\mathcal J_1$ and $\mathcal J_2$.
  \begin{enumerate}[label=\textup{(\roman*)},leftmargin=*,itemsep=0pt]
    \item $J_k^{(1)}(n,s)>J_k^{(2)}(n,s)$ and $\mathcal H=\mathcal J_1(Q,D,v)$.
    \item $J_k^{(1)}(n,s)<J_k^{(2)}(n,s)$ and $\mathcal H=\mathcal J_2(Q,u)$.
    \item $J_k^{(1)}(n,s)=J_k^{(2)}(n,s)$ and either $\mathcal H=\mathcal J_1(Q,D,v)$ or $\mathcal H=\mathcal J_2(Q,u)$.
  \end{enumerate}
\end{lemma}

\begin{proof}
  Choose $0<\eta=\eta(k)\le1/(64k^4)$ sufficiently small for the estimates in Case~2, and then choose $s_0=s_0(k,\eta)\ge60k^4$ sufficiently large that $\eta ks_0\ge k-1$ and the estimates below hold.
  Recall that $q=ks+k-2$ and $t=n-q\ge2$.

  Extend $\mathcal H$ to an $s$-saturated family $\mathcal H'$.
  Let $Q$ be a largest clique in $\mathcal H'$.
  Write $u$ for the order of $Q$.
  Since $\mathcal H\subseteq\mathcal H'$, we have $u\ge\omega(\mathcal H)\ge(1-\eta)ks$.
  If $u\ge q+2=k(s+1)$, then a partition of $k(s+1)$ vertices of $Q$ into $s+1$ sets of size $k$ forms an $(s+1)$-matching in $\mathcal H'$, contradicting $\nu(\mathcal H')=s$.
  Thus, $u\le q+1$.
  If $u=q+1$, then the condition $|V^+(\mathcal H)|\ge q+2$ ensures that there is an edge $E\in\mathcal H$ not contained in $Q$.
  Since $|Q\setminus E|\ge q+1-(k-1)=ks$, we can choose $s$ pairwise disjoint $k$-sets in $Q\setminus E$.
  Together with $E$, they form an $(s+1)$-matching in $\mathcal H'$, a contradiction.
  Hence, $(1-\eta)ks\le u\le q$.
  We now consider two cases according to the value of $u$.

  \medskip
  \noindent{\bf Case 1.} $u=q$.
  \medskip

  In this case, we claim that every edge of $\mathcal H'$ not contained in $Q$ has exactly one vertex outside $Q$.
  Indeed, an edge $E\in\mathcal H'$ containing at least two vertices outside $Q$ would leave at least $q-(k-2)=ks$ vertices of $Q$ unused.
  Choose $s$ pairwise disjoint $k$-sets among these vertices.
  Together with $E$, they form an $(s+1)$-matching in $\mathcal H'$, a contradiction.

  For each $v\in[n]\setminus Q$, define
  \[
    \mathcal L_v=\left\{A\in\binom Q{k-1}:A\cup\{v\}\in\mathcal H'\right\}.
  \]
  For distinct $v,w\in[n]\setminus Q$, we have $A\cap B\ne\emptyset$ for all $A\in\mathcal L_v$ and $B\in\mathcal L_w$.
  Otherwise, disjoint $A\in\mathcal L_v$ and $B\in\mathcal L_w$ would extend to two disjoint edges $\{v\}\cup A$ and $\{w\}\cup B$ of $\mathcal H'$.
  The remaining $q-2(k-1)=k(s-1)$ vertices of $Q$ can be partitioned into $s-1$ clique edges, completing an $(s+1)$-matching in $\mathcal H'$, a contradiction.

  Let $t'$ be the number of vertices $v\in[n]\setminus Q$ for which $\mathcal L_v\ne\emptyset$.
  We have $t'\ge2$, since otherwise $|V^+(\mathcal H)|\le q+1$.
  The $t'$ nonempty families $\mathcal L_v\subseteq\binom Q{k-1}$ satisfy the hypotheses of \cref{thm:cross-intersection}, since $q\ge2(k-1)$.
  Since $Q$ is a clique and every edge of $\mathcal H'$ outside $Q$ is counted by exactly one link, \cref{thm:cross-intersection} implies
  \[
    \begin{aligned}
      |\mathcal H|
       & \le|\mathcal H'|=\binom qk+\sum_{v\in[n]\setminus Q}|\mathcal L_v|                         \\
       & \le\binom qk+\max\left\{\binom q{k-1}-\binom{q-k+1}{k-1}+t'-1,\ t'\binom{q-1}{k-2}\right\} \\
       & \le J_k(n,s),
    \end{aligned}
  \]
  where the last inequality uses $t'\le t$ and the definition of $J_k(n,s)$.
  This proves the upper bound in this case.

  We now characterize the equality cases.
  Suppose that $|\mathcal H|=J_k(n,s)$.
  Equality holds throughout the bound on $|\mathcal H|$, so $\mathcal H=\mathcal H'$ and $t'=t$.
  In particular, $\mathcal L_x$ is nonempty for every $x\in[n]\setminus Q$.
  Choose distinct $v,w\in[n]\setminus Q$ whose links have the largest and second largest sizes, respectively.

  To apply \cref{thm:weighted-cross-intersection} to $\mathcal L_v$ and $\mathcal L_w$, we first show that $|\mathcal L_w|\le\binom{q-1}{k-2}$.
  Suppose otherwise and choose $\mathcal B\subseteq\mathcal L_w$ with $|\mathcal B|=\binom{q-1}{k-2}$.
  Every member of $\mathcal L_v$ meets every member of $\mathcal B$, and $q=ks+k-2>2(k-1)$ for $s\ge2$.
  Thus, \cref{thm:weighted-cross-intersection} applies to $\mathcal L_v$ and $\mathcal B$.
  Since $\binom{q-1}{k-2}>1$, for sufficiently large $c$ we have
  \[
    (1+c)\binom{q-1}{k-2}
    >\binom q{k-1}-\binom{q-k+1}{k-1}+c.
  \]
  Hence, \cref{thm:weighted-cross-intersection} gives
  \[
    |\mathcal L_v|+c\binom{q-1}{k-2}
    \le(1+c)\binom{q-1}{k-2}.
  \]
  Hence, $|\mathcal L_v|\le\binom{q-1}{k-2}$, contradicting $|\mathcal L_v|\ge|\mathcal L_w|>\binom{q-1}{k-2}$.
  Thus, $|\mathcal L_w|\le\binom{q-1}{k-2}$.

  To determine the possible forms of the two largest links, we show that $(\mathcal L_v,\mathcal L_w)$ attains equality in \cref{thm:weighted-cross-intersection}.
  By the choice of $w$, each of the $t-1$ links other than $\mathcal L_v$ has size at most $|\mathcal L_w|$.
  Hence, by \cref{thm:weighted-cross-intersection},
  \[
    \begin{aligned}
      J_k(n,s)-\binom qk
       & =|\mathcal H|-\binom qk \\
       & =\sum_{x\in[n]\setminus Q}|\mathcal L_x| \\
       & \le|\mathcal L_v|+(t-1)|\mathcal L_w| \\
       & \le\max\left\{\binom q{k-1}-\binom{q-k+1}{k-1}+t-1,\ t\binom{q-1}{k-2}\right\} \\
       & =J_k(n,s)-\binom qk.
    \end{aligned}
  \]
  Thus, both inequalities are equalities.
  The first implies that $|\mathcal L_x|=|\mathcal L_w|$ for every $x\in[n]\setminus(Q\cup\{v\})$.
  By \cref{thm:weighted-cross-intersection}, the second implies that one of the following statements holds.
  \begin{enumerate}[label=\textup{(\roman*)},leftmargin=*,itemsep=0pt]
    \item $\mathcal L_w=\{D\}$ and $\mathcal L_v=\{A\in\binom Q{k-1}:A\cap D\ne\emptyset\}$ for some $D\in\binom Q{k-1}$.
    \item $\mathcal L_v=\mathcal L_w=\{A\in\binom Q{k-1}:z\in A\}$ for some $z\in Q$.
  \end{enumerate}

  If \textup{(i)} holds, then we show that $\mathcal L_x=\{D\}$ for every $x\in[n]\setminus(Q\cup\{v\})$.
  Fix such an $x$ and suppose that $B\ne D$ for some $B\in\mathcal L_x$.
  Since $|B|=|D|$, we can choose some $d\in D\setminus B$.
  As $q>2(k-1)$, there is a $(k-1)$-set $A\subseteq Q\setminus B$ containing $d$.
  Since $d\in A\cap D$, the set $A$ meets $D$ and therefore belongs to $\mathcal L_v$ by \textup{(i)}.
  However, $A\subseteq Q\setminus B$ implies $A\cap B=\emptyset$, contradicting the fact that every member of $\mathcal L_v$ meets every member of $\mathcal L_x$.
  Thus, no member of $\mathcal L_x$ can differ from $D$.
  Since $x$ was arbitrary and each of these links is nonempty, $\mathcal L_x=\{D\}$ for every $x\in[n]\setminus(Q\cup\{v\})$.
  Recall that every edge of $\mathcal H$ not contained in $Q$ has exactly one vertex outside $Q$.
  By \textup{(i)}, the edges containing $v$ are precisely the sets $\{v\}\cup A$ with $A\in\binom Q{k-1}$ and $A\cap D\ne\emptyset$.
  For each $x\in[n]\setminus(Q\cup\{v\})$, the equality $\mathcal L_x=\{D\}$ means that the only edge containing $x$ is $\{x\}\cup D$.
  Since $\mathcal H$ also contains all $k$-subsets of $Q$, we have $\mathcal H=\mathcal J_1(Q,D,v)$.

  If \textup{(ii)} holds, then $\mathcal L_v$ consists of all $(k-1)$-subsets of $Q$ containing $z$.
  Fix $x\in[n]\setminus(Q\cup\{v\})$ and take any $B\in\mathcal L_x$.
  Since $B$ meets every member of $\mathcal L_v$, it must meet every $(k-1)$-subset of $Q$ containing $z$.
  We show that this forces $z\in B$.
  Suppose otherwise.
  Since $z\in Q\setminus B$ and $q>2(k-1)$, there is a $(k-1)$-set $A\subseteq Q\setminus B$ containing $z$.
  This set contains $z$ but is disjoint from $B$, contradicting the fact that $B$ meets every $(k-1)$-subset of $Q$ containing $z$.
  Thus, every member of $\mathcal L_x$ contains $z$, so $\mathcal L_x\subseteq\mathcal L_v$.
  We already concluded that $|\mathcal L_x|=|\mathcal L_w|$, and \textup{(ii)} states that $\mathcal L_w=\mathcal L_v$, so $\mathcal L_x=\mathcal L_v$.
  Therefore, for every $x\in[n]\setminus Q$, the edges containing $x$ are precisely the sets $\{x\}\cup A$ with $A\in\binom Q{k-1}$ and $z\in A$.
  Since $\mathcal H$ also contains all $k$-subsets of $Q$, we have $\mathcal H=\mathcal J_2(Q,z)$.

  Among $\mathcal J_1(Q,D,v)$ and $\mathcal J_2(Q,z)$, only those attaining the maximum defining $J_k(n,s)$ can occur, since $|\mathcal H|=J_k(n,s)$.
  This finishes the proof of Case 1.

  \medskip
  \noindent{\bf Case 2.} $(1-\eta)ks\le u\le q-1$.
  \medskip

  In this case, we show that $|\mathcal H'|<J_k^{(2)}(n,s)$.
  Set $d=q+1-u\ge2$ and $x=(k-1)(n/(ks)-1)$.
  Since $k(s+1)\le n\le(k+1/2)s$, we have $0\le x\le(k-1)/(2k)<1/2$.
  Since $q=ks+k-2$ and $t=n-ks-k+2$, we obtain $J_k^{(2)}(n,s)-C_k(s)=t\binom{q-1}{k-2}-\binom q{k-1}$ from Pascal's identity.
  For $s\ge2k$, we have $\binom{q-1}{k-2}\ge\binom{ks}{k-2}$.
  By the product formula, $\binom{ks+k-3}{k-2}<2\binom{ks}{k-2}<4\binom{ks}{k-1}/s$.
  Therefore, Pascal's identity yields
  \[
    \binom q{k-1}-\binom{ks}{k-1}
    =\sum_{j=1}^{k-2}\binom{ks+j-1}{k-2}
    <\frac{4k}{s}\binom{ks}{k-1}.
  \]
  As $t=n-ks-k+2$ and $\binom{ks}{k-2}/\binom{ks}{k-1}=(k-1)/(ks-k+2)$, we also have
  \[
    \frac{t\binom{q-1}{k-2}}{\binom{ks}{k-1}}
    \ge\frac{t(k-1)}{ks}
    =x-\frac{(k-1)(k-2)}{ks}
    \ge x-\frac{k-1}{s}.
  \]
  Consequently,
  \[
    \frac{J_k^{(2)}(n,s)-C_k(s)}{\binom{ks}{k-1}}
    \ge-1+x-\frac{5k}{s}.
  \]
  \Cref{fact:large-clique-deficit} applies to the $s$-saturated family $\mathcal H'$ with $d=q+1-u$.
  Since $(n/(ks))^{k-1}\le\mathrm{e}^x$ and $5k^3\eta+3k^3/s\le5/(64k)+1/(20k)<2-\mathrm{e}^{1/2}$, the quantity $2-(n/(ks))^{k-1}-5k^3\eta-3k^3/s$ is positive.
  Using $d\ge2$ and $\mathrm{e}^x\le1+x+x^2$ for $x\le1/2$, we obtain
  \[
    \begin{aligned}
      \frac{J_k^{(2)}(n,s)-|\mathcal H'|}{\binom{ks}{k-1}}
       & \ge d\left[2-\left(\frac n{ks}\right)^{k-1}-5k^3\eta-\frac{3k^3}s\right]
      -1+x-\frac{5k}s                                                             \\
       & \ge 3-2\mathrm{e}^x+x-10k^3\eta-\frac{6k^3+5k}s                          \\
       & \ge (1-2x)(1+x)-10k^3\eta-\frac{6k^3+5k}s                                \\
       & \ge\frac1k-10k^3\eta-\frac{6k^3+5k}s>0.
    \end{aligned}
  \]
  Indeed, $x\le(k-1)/(2k)$ implies $(1-2x)(1+x)\ge1/k$, while $10k^3\eta+(6k^3+5k)/s\le5/(32k)+1/(10k)+1/(12k^3)<1/k$.
  Hence, $|\mathcal H|\le|\mathcal H'|<J_k^{(2)}(n,s)\le J_k(n,s)$.
  Equality can therefore occur only in Case~1.

  \medskip
  This completes the proof.
\end{proof}

\subsubsection{Proof of Theorem~\ref{cor:clique-hilton-milner}}\label{sec:support-refinement-proof}
\begin{proof}
  Fix $k\ge3$ and $0<\varepsilon<\rho_k-k$, and set $R=\rho_k-\varepsilon$, so $k<R<\rho_k<k+1/2$.
  We first consider families without isolated vertices.
  We show that, for all sufficiently large $s$ and $k(s+1)\le n\le Rs$, any such family $\mathcal F\subseteq\binom{[n]}k$ satisfying $\nu(\mathcal F)\le s$ and $|\mathcal F|\ge J_k(n,s)$ equals $\mathcal J_1(Q,D,v)$ or $\mathcal J_2(Q,u)$ for a choice of parameters giving size $J_k(n,s)$.
  This will establish both the upper bound and its equality cases for families without isolated vertices.
  We will then delete isolated vertices to obtain the full statement.

  Suppose that the assertion for families without isolated vertices fails for arbitrarily large $s$.
  Choose a sequence of counterexamples $\mathcal F_i\subseteq\binom{[n_i]}k$ with $s_i\to\infty$ and $k(s_i+1)\le n_i\le Rs_i$.
  Each $\mathcal F_i$ has no isolated vertices, satisfies $\nu(\mathcal F_i)\le s_i$, and has
  \begin{equation}\label{eq:support-counterexample-sequence}
    |\mathcal F_i|\ge J_k(n_i,s_i),
  \end{equation}
  but equals neither $\mathcal J_1(Q,D,v)$ nor $\mathcal J_2(Q,u)$ for any choice of parameters giving size $J_k(n_i,s_i)$.

  Ordinary shifts may create isolated vertices, so we restrict ourselves to shifts that preserve the absence of isolated vertices.
  For a family $\mathcal H$, set $\mathcal D(\mathcal H)=\sum_v d_{\mathcal H}(v)^2$.
  For each $i$, start with $\mathcal H=\mathcal F_i$.
  Whenever there exists an ordered pair of distinct vertices $(u,v)$ satisfying $d_{S_{uv}(\mathcal H)}(v)>0$ and $\mathcal D(S_{uv}(\mathcal H))>\mathcal D(\mathcal H)$, choose one such pair and replace $\mathcal H$ by $S_{uv}(\mathcal H)$.
  Only the degree of $v$ can decrease under $S_{uv}$, so $d_{S_{uv}(\mathcal H)}(v)>0$ ensures that no vertex becomes isolated.
  Repeat this step until no ordered pair satisfies both conditions, and denote the resulting family by $\mathcal G_i$.
  The process terminates because $\mathcal D$ takes integer values and is bounded above on the finite set of families on $[n_i]$.
  By our choice of shifts, every family along the sequence has $|\mathcal F_i|$ edges, no isolated vertices, and matching number at most $s_i$.

  Passing to a subsequence and relabeling, we may assume that
  \begin{equation}\label{eq:support-limiting-ratio}
    n_i/s_i\longrightarrow R_*\in[k,R].
  \end{equation}
  Since $R_*<\rho_k$, we have $(k/R_*)^k>1-(1-1/R_*)^k$.
  Thus, \eqref{eq:support-limiting-ratio} and \cref{fact:construction-densities} imply $D_k(n_i,s_i)=C_k(s_i)$ for all sufficiently large $i$.
  Also, the formulas for $J_k^{(1)}(n,s)$ and $J_k^{(2)}(n,s)$ imply $J_k(n,s)=C_k(s)+O_k(s^{k-1})$ uniformly for $k(s+1)\le n\le Rs$.
  Since $n_i=\Theta_k(s_i)$, \eqref{eq:support-counterexample-sequence} implies
  \[
    |\mathcal G_i|=|\mathcal F_i|
    \ge D_k(n_i,s_i)-O_k(s_i^{k-1})
    =D_k(n_i,s_i)-o(n_i^k).
  \]
  By \cref{lem:matching-sequence}\textup{(ii)}, there are $(ks_i+k-1)$-sets $Q_i\subseteq[n_i]$ such that
  \begin{equation}\label{eq:support-clique-approximation}
    \left|\mathcal G_i\triangle\binom{Q_i}k\right|=o(n_i^k)=o(s_i^k).
  \end{equation}

  Let $0<\eta<1$ be the constant supplied by \cref{lem:large-clique}.
  The following claim shows that the approximation in \eqref{eq:support-clique-approximation} yields a clique large enough to apply \cref{lem:large-clique}.

  \begin{claim}\label{lem:support-preserving-compression}
    For all sufficiently large $i$, we have $\omega(\mathcal G_i)\ge(1-\eta)ks_i$.
  \end{claim}

  \begin{proof}
    Set $\Delta_i=\binom{ks_i+k-2}{k-1}$, and let
    \begin{equation}\label{eq:high-core-degrees}
      T_i=\left\{v\in Q_i:d_{\mathcal G_i}(v)>\frac12\binom{n_i-2}{k-1}\right\}.
    \end{equation}
    We first show that $|Q_i\setminus T_i|=o(s_i)$.
    A short calculation shows that
    \[
      \frac{\Delta_i}{\binom{n_i-2}{k-1}}
      =\prod_{j=0}^{k-2}\frac{ks_i+k-2-j}{n_i-2-j}.
    \]
    Every numerator is at least $ks_i$ and every denominator is at most $n_i$, so each factor is at least $ks_i/n_i\ge2k/(2k+1)$.
    By Bernoulli's inequality $(1-t)^{k-1}\ge1-(k-1)t$ for $0\le t\le1$, we therefore have
    \[
      \frac{\Delta_i}{\binom{n_i-2}{k-1}}
      \ge\left(1-\frac1{2k+1}\right)^{k-1}
      \ge1-\frac{k-1}{2k+1}
      =\frac{k+2}{2k+1}>\frac12.
    \]
    For $n_i\ge k(s_i+1)$, a similar argument shows that $\binom{n_i-2}{k-1}\ge(ks_i)^{k-1}/(k-1)!$.
    Hence, $\Delta_i-\frac12\binom{n_i-2}{k-1}=\Omega_k(s_i^{k-1})$.
    Each vertex $v\in Q_i\setminus T_i$ lies in $\Delta_i$ members of $\binom{Q_i}k$, of which at most $d_{\mathcal G_i}(v)\le\frac12\binom{n_i-2}{k-1}$ belong to $\mathcal G_i$.
    Thus, at least $\Delta_i-\frac12\binom{n_i-2}{k-1}$ of these members are missing.
    Summing the missing-member counts over $v\in Q_i\setminus T_i$ counts each missing member at most $k$ times, once for each of its vertices in $Q_i\setminus T_i$.
    Together with \eqref{eq:support-clique-approximation}, this implies
    \[
      |Q_i\setminus T_i|\left(\Delta_i-\frac12\binom{n_i-2}{k-1}\right)
      \le k\left|\binom{Q_i}k\setminus\mathcal G_i\right|
      =o(s_i^k).
    \]
    Hence, $|T_i|=ks_i+k-1-o(s_i)$.

    Order the vertices of $T_i$ as $v_1,\ldots,v_{|T_i|}$ in nonincreasing order of their degrees in $\mathcal G_i$.
    Fix distinct $u,v\in T_i$ with $u$ preceding $v$ in this order.
    If $S_{uv}$ made $v$ isolated, then no member of $\mathcal G_i$ could contain both $u$ and $v$, and no $(k-1)$-set $A\subseteq[n_i]\setminus\{u,v\}$ could have both $A\cup\{u\}$ and $A\cup\{v\}$ in $\mathcal G_i$.
    Otherwise, in either case, a member containing $v$ would remain unchanged under $S_{uv}$ by the definition of shifting.
    These two exclusions would imply $d_{\mathcal G_i}(u)+d_{\mathcal G_i}(v)\le\binom{n_i-2}{k-1}$, contrary to \eqref{eq:high-core-degrees}.
    Thus, $S_{uv}$ creates no isolated vertex.
    Let $a$ be the number of members moved by $S_{uv}$.
    If $a>0$, then
    \[
      \mathcal D(S_{uv}(\mathcal G_i))-\mathcal D(\mathcal G_i)
      =2a(d_{\mathcal G_i}(u)-d_{\mathcal G_i}(v))+2a^2>0.
    \]
    Since $d_{S_{uv}(\mathcal G_i)}(v)>0$, the construction of $\mathcal G_i$ gives $\mathcal D(S_{uv}(\mathcal G_i))\le\mathcal D(\mathcal G_i)$, a contradiction.
    Therefore, $a=0$, so $S_{uv}(\mathcal G_i)=\mathcal G_i$ whenever $u$ precedes $v$ in $T_i$.

    Set $h_i=\lceil(1-\eta)ks_i\rceil$.
    For all sufficiently large $i$, we have $h_i\le|T_i|$ and may set $U_i=\{v_1,\ldots,v_{h_i}\}$.
    Since $T_i\subseteq Q_i$, we have
    \begin{equation}\label{eq:dense-high-degree-core}
      \left|\binom{T_i}k\setminus\mathcal G_i\right|
      \le\left|\binom{Q_i}k\setminus\mathcal G_i\right|
      =o(s_i^k).
    \end{equation}
    Also, $|T_i\setminus U_i|=\eta ks_i+o(s_i)$, so $\binom{|T_i\setminus U_i|}{k}=\Theta_{k,\eta}(s_i^k)$.
    By \eqref{eq:dense-high-degree-core}, for each sufficiently large $i$ choose a member $B\in\mathcal G_i\cap\binom{T_i\setminus U_i}k$.
    Fix an arbitrary set $A\in\binom{U_i}k$.
    Now we show that $A\in \mathcal G_i$.
    To this end, write $A=\{a_1,\ldots,a_k\}$ and $B=\{b_1,\ldots,b_k\}$.
    Since $A\subseteq U_i$ and $B\subseteq T_i\setminus U_i$, the sets $A,B$ are disjoint and every vertex of $A$ precedes every vertex of $B$ in the ordering of $T_i$.
    For $0\le j\le k$, set
    \[
      B_j=\{a_1,\ldots,a_j,b_{j+1},\ldots,b_k\}.
    \]
    We have $B_0=B\in\mathcal G_i$ and $B_k=A$.
    Suppose that $1\le j\le k$ and $B_{j-1}\in\mathcal G_i$.
    Since $b_j\in B_{j-1}$ and $a_j\notin B_{j-1}$, if $B_j\notin\mathcal G_i$, then $S_{a_jb_j}$ would replace $B_{j-1}$ by $B_j$.
    This would contradict $S_{a_jb_j}(\mathcal G_i)=\mathcal G_i$, which holds because $a_j$ precedes $b_j$ in $T_i$.
    Thus, $B_j\in\mathcal G_i$, and induction shows that $A=B_k\in\mathcal G_i$.
    Since $A$ was arbitrary, $U_i$ is a clique of order $h_i\ge(1-\eta)ks_i$.
  \end{proof}

  Since $\mathcal G_i$ has no isolated vertices, $|V^+(\mathcal G_i)|=n_i\ge k(s_i+1)$.
  For all sufficiently large $i$, \cref{lem:support-preserving-compression,lem:large-clique} therefore imply
  \[
    J_k(n_i,s_i)\le|\mathcal F_i|=|\mathcal G_i|\le J_k(n_i,s_i).
  \]
  Equality holds throughout.
  Since shifts preserve the number of members, every family along the sequence of shifts from $\mathcal F_i$ to $\mathcal G_i$ has exactly $J_k(n_i,s_i)$ members.
  The equality case of \cref{lem:large-clique} shows that $\mathcal G_i$ is one of the constructions $\mathcal J_1$ and $\mathcal J_2$ attaining $J_k(n_i,s_i)$ and hence contains a clique on $q_i=ks_i+k-2$ vertices.

  It remains to transfer the equality classification from $\mathcal G_i$ back to $\mathcal F_i$ along the chosen shift sequence.
  Fix any sufficiently large $i$ such that $q_i-2\ge(1-\eta)ks_i$.
  We now argue backwards along the shift sequence for this fixed $i$.
  For the initial step, we have already shown that $\mathcal G_i$ contains a clique on $q_i$ vertices.
  For the induction step, consider two consecutive families $\mathcal K,\mathcal K'$ in the shift sequence, where $\mathcal K'=S_{uv}(\mathcal K)$.
  Assume inductively that $\mathcal K'$ contains a clique on a $q_i$-set $Q$.
  We show that the preceding family $\mathcal K$ satisfies an equality case of \cref{lem:large-clique}, which will also establish the existence of a clique on $q_i$ vertices in $\mathcal K$.
  A shift does not change membership of any $k$-set avoiding both $u$ and $v$, so $\binom{Q\setminus\{u,v\}}k\subseteq\mathcal K$.
  Hence,
  \[
    \omega(\mathcal K)\ge q_i-2=ks_i+k-4\ge(1-\eta)ks_i.
  \]

  The family $\mathcal K$ has no isolated vertices, satisfies $\nu(\mathcal K)\le s_i$, and has $|\mathcal K|=J_k(n_i,s_i)$.
  \Cref{lem:large-clique} again shows that $\mathcal K$ is one of the constructions $\mathcal J_1$ and $\mathcal J_2$ attaining $J_k(n_i,s_i)$ and, in particular, contains a clique on $q_i$ vertices.
  This completes the backward induction, so $\mathcal F_i$ is one of the constructions $\mathcal J_1$ and $\mathcal J_2$ attaining $J_k(n_i,s_i)$, a contradiction.
  This proves the bound and equality classification for families without isolated vertices.

  It remains to extend these conclusions to families with isolated vertices.
  Let $\mathcal F\subseteq\binom{[n]}k$ satisfy the hypotheses of \cref{cor:clique-hilton-milner}, and set $n'=|V^+(\mathcal F)|$.
  Delete its isolated vertices and relabel the remaining vertices to obtain a family $\mathcal F'\subseteq\binom{[n']}k$ with no isolated vertices.
  Deleting isolated vertices and relabeling preserve both the number of members and the matching number, so $\nu(\mathcal F')\le s$ and $|\mathcal F'|=|\mathcal F|$.
  For fixed $k,s$, $J_k^{(1)}(n,s)$ and $J_k^{(2)}(n,s)$ strictly increase with $n$, so their maximum $J_k(n,s)$ does as well.
  Since $k(s+1)\le n'\le n\le Rs$, the case without isolated vertices gives $|\mathcal F'|\le J_k(n',s)$, and hence
  \[
    |\mathcal F|=|\mathcal F'|\le J_k(n',s)\le J_k(n,s).
  \]
  Equality forces $n'=n$, so $\mathcal F$ has no isolated vertices and is one of the constructions $\mathcal J_1$ and $\mathcal J_2$ attaining $J_k(n,s)$.
\end{proof}

\section*{Declaration on the use of generative AI}
Generative AI was used for language polishing and for developing parts of the mathematical arguments in this manuscript.
For \cref{thm:iid-tail}, the authors independently developed the proof framework and proofs for $k=2,3,4,5$ and $k\ge10^{100}$ without any AI assistance.
Building on the authors' proof framework, generative AI developed a unified argument valid for every $k\ge6$ in early September 2026, thereby resolving the remaining cases $6\le k<10^{100}$.
The main refinements were a sharper estimate for the conditional shortfall in {\cref{prop:finite-shortfall-margin}}, an exact counting lower bound without the additive error term in the earlier estimate in \cref{lem:exact-upper-count}, and refined variance and tail bounds for the conditional distributions in \cref{lem:conditional-variance-bounds,lem:averaged-contact-normal,lem:conditional-support,lem:small-conditional-tail}.
For \cref{thm:emc}, the authors' initial argument left a gap of width $o(s)$ in the range of $n$ around $n=\rho_k s$ as $s\to\infty$ for fixed $k$, where the cover and clique constructions have asymptotically equal sizes.
This gap was closed following a suggestion of using weak convergence, given by generative AI (see \cref{lem:matching-sequence,sec:emc-proof}).
All other arguments proving \cref{thm:emc} were developed independently by the authors without any AI assistance.
The authors have independently checked  all AI-generated arguments in detail and take full responsibility for the mathematical correctness and content of the paper.

\section*{Acknowledgement}
The authors would like to express sincere thanks to Ziyan Liu and Chen L\"{u} for their kind support and encouragement during the completion of this manuscript.

\appendix
\section{Proofs of auxiliary results}
\subsection{Proof of Lemma~\ref{lem:crossing-upper}}
\label{app:crossing-upper}

We first establish the derivative estimates used in the proof of \cref{lem:crossing-upper}.

\begin{lemma}\label{lem:log-power-curvature}
  Let $z\in\mathbb R$, $C>z^2/4$, and $T>0$.
  Define
  \[
    f_z(t)=\frac{\log(1+zt+Ct^2)}t\text{ for }t>0, \text{ and }f_z(0)=z.
  \]
  Write $\beta_z^2=C-z^2/4$.
  The following statements hold.
  \begin{enumerate}[label=\textup{(\roman*)},leftmargin=*,itemsep=0pt]
    \item The function $f_z$ has continuous derivatives of all orders on $[0,\infty)$ with
          \[
            f_z'(0)=C-z^2/2 \text{, }f_z''(0)=2z(z^2/3-C) \text{ and }
            f_z'''(0) =-3z^4/2+6Cz^2-3C^2.
          \]
    \item If $z>0$ and $\beta_z^2<(z/2+Cts)^2<3\beta_z^2$ for all $(t,s)\in[0,T]\times[0,1]$, then $0<f_z'''(t)\le3C^2$ for $0\le t\le T$.
    \item If $z<0$ and $(z/2+Cts)^2>6\beta_z^2$ for all $(t,s)\in[0,T]\times[0,1]$, then $0<-f_z'''(t)\le\frac{-f_z'''(0)}{(1+zT+CT^2)^2}$ for $0\le t\le T$.
  \end{enumerate}
\end{lemma}
\begin{proof}
  Write $B(v)=1+zv+Cv^2$ and $L(v)=\log B(v)$.
  Since $C>z^2/4$, we have $B(v)\ge1-z^2/(4C)>0$ for every $v\in\mathbb R$.
  Since $L(0)=0$, we have $f_z(t)=\int_0^1L'(ts)\dd s$ for $t>0$.
  At $t=0$, the integral equals $L'(0)=z=f_z(0)$.
  Since $L$ is smooth on $\mathbb R$, all derivatives of $L'(ts)$ with respect to $t$ are continuous and bounded on $[-M,M]\times[0,1]$ for every $M>0$.
  Differentiating under the integral sign, we obtain
  \[
    f_z^{(j)}(t)=\int_0^1s^jL^{(j+1)}(ts)\dd s
  \]
  for every integer $j\ge0$ and every $t\ge0$.
  In particular, $f_z$ has continuous derivatives of all orders on $[0,\infty)$.
  At zero, we have $f_z^{(j)}(0)=L^{(j+1)}(0)/(j+1)$.
  Direct differentiation yields
  \[
    L''(0)=2C-z^2 \text{, }
    L'''(0)=2z^3-6Cz \text{ and }
    L^{(4)}(0)=-6z^4+24Cz^2-12C^2.
  \]
  This proves \textup{(i)}.

  To prove \textup{(ii)} and \textup{(iii)}, define
  \[
    \xi(v)=\frac{(z/2+Cv)^2-\beta_z^2}{(z/2+Cv)^2+\beta_z^2}
    =1-\frac{2\beta_z^2}{CB(v)}.
  \]
  Since $B''=2C$ and $B'''=0$, differentiation yields
  \[
    L''=-\frac{2C\xi}{B} \text{ and }
    L^{(4)}=\frac{12C^2}{B^2}-6(L'')^2
    =\frac{12C^2}{B^2}(1-2\xi^2).
  \]
  Hence,
  \begin{equation}\label{eq:log-power-third-derivative}
    f_z'''(t)=12C^2\int_0^1
    \frac{s^3\bigl(1-2\xi(ts)^2\bigr)}{B(ts)^2}\dd s.
  \end{equation}

  Under the hypotheses of \textup{(ii)}, we have $B(v)\ge1$ and $0<\xi(v)<1/2$ for $0\le v\le T$.
  By \eqref{eq:log-power-third-derivative}, for $0\le t\le T$ we obtain
  \[
    0<f_z'''(t)\le12C^2\int_0^1s^3\dd s=3C^2.
  \]
  This proves \textup{(ii)}.

  Under the hypotheses of \textup{(iii)}, we have $(z/2+Cv)^2>6\beta_z^2>0$, so $z/2+Cv\ne0$ for $0\le v\le T$.
  Since $z/2<0$ and $v\mapsto z/2+Cv$ is continuous, the intermediate value theorem implies $z/2+Cv<0$ for $0\le v\le T$.
  Thus, $B'(v)=z+2Cv<0$, and $\xi(v)=1-2\beta_z^2/(CB(v))$ also decreases with $v$.
  Since $5/7<\xi(v)<1$, we have $0<2\xi(v)^2-1\le2\xi(0)^2-1$ for $0\le v\le T$.
  Using $B(ts)\ge B(T)>0$ and $2\xi(ts)^2-1\le2\xi(0)^2-1$ in \eqref{eq:log-power-third-derivative}, for $0\le t\le T$ we obtain
  \[
    0<-f_z'''(t)
    \le\frac{3C^2\bigl(2\xi(0)^2-1\bigr)}{B(T)^2}
    =\frac{-f_z'''(0)}{(1+zT+CT^2)^2},
  \]
  proving \textup{(iii)}.
\end{proof}

We now prove \cref{lem:crossing-upper}.

\begin{proof}
  Recall that $\gamma=-\log(1-\mathrm{e}^{-1})$.
  Set $C=41/500$, $T=1/5$, and $\rho=37/100$.
  Write $g_-=45867/100000$ and $g_+=11467/25000$.
  The alternating Taylor bounds for the exponential imply
  \[
    \begin{aligned}
      \mathrm{e}^{-g_-}+\mathrm{e}^{-1}
       & >\sum_{j=0}^{7}\frac{(-g_-)^j}{j!}
      +\sum_{j=0}^{9}\frac{(-1)^j}{j!}>1,   \\
      \mathrm{e}^{-g_+}+\mathrm{e}^{-1}
       & <\sum_{j=0}^{8}\frac{(-g_+)^j}{j!}
      +\sum_{j=0}^{12}\frac{(-1)^j}{j!}<1.
    \end{aligned}
  \]
  Since $\mathrm{e}^{-\gamma}+\mathrm{e}^{-1}=1$, we obtain
  \begin{equation}\label{eq:crossing-gamma-bounds}
    g_-<\gamma<g_+.
  \end{equation}

  For $0\le t\le T$, define
  \[
    D(t)=f_\gamma(t)-\log\bigl(1+\mathrm{e}^{f_{\gamma-1}(t)}\bigr),
  \]
  where $f_z$ is the function from \cref{lem:log-power-curvature}.
  Since $\mathrm{e}^\gamma=1+\mathrm{e}^{\gamma-1}$, we have $D(0)=0$.
  For every $k\ge5$, the inequality $D(1/k)>0$ is equivalent to $g_k(\gamma+C/k)>1$.
  We will prove that $D$ is strictly concave on $[0,T]$ and $D(T)>0$.

  Set $p(t)=\mathrm{e}^{f_{\gamma-1}(t)}/(1+\mathrm{e}^{f_{\gamma-1}(t)})$.
  Differentiating $D$ twice, we obtain
  \begin{equation}\label{eq:crossing-second-derivative}
    D''(t)=f_\gamma''(t)-p(t)f_{\gamma-1}''(t)
    -p(t)(1-p(t))\bigl(f_{\gamma-1}'(t)\bigr)^2.
  \end{equation}
  We first show that $f_{\gamma-1}''(t)<0$ and $p(t)\le\rho$ on $[0,T]$.
  Since $9/20<\gamma<23/50$, we have $C>z^2/4$ for $z=\gamma,\gamma-1$.
  For $(t,s)\in[0,T]\times[0,1]$, we have $\gamma/2+Cts>0$ and $(\gamma-1)/2+Cts<0$.
  As $ts$ increases from $0$ to $T$, the square $(\gamma/2+Cts)^2$ increases and the square $((\gamma-1)/2+Cts)^2$ decreases.
  Using $9/20<\gamma<23/50$ and the values of $z/2+Cts$ for $z=\gamma,\gamma-1$ at $ts=0$ and $ts=T=1/5$, we obtain
  \[
    \begin{aligned}
      \beta_\gamma^2
       & <\frac{\gamma^2}{4}
      \le (\gamma/2+Cts)^2
      \le\left(\frac\gamma2+\frac{41}{2500}\right)^2
      <3\beta_\gamma^2,                                       \\
      ((\gamma-1)/2+Cts)^2
       & \ge\left(\frac{\gamma-1}{2}+\frac{41}{2500}\right)^2
      >6\beta_{\gamma-1}^2.
    \end{aligned}
  \]
  Hence, \cref{lem:log-power-curvature}\textup{(ii),(iii)} imply $0<f_\gamma'''(t)\le3C^2$ and $f_{\gamma-1}'''(t)<0$ for $0\le t\le T$, respectively.
  By \cref{lem:log-power-curvature}\textup{(i)}, we also have $f_{\gamma-1}'(0)<0$ and $f_{\gamma-1}''(0)<0$.
  It follows that $f_{\gamma-1}'(t)<0$ and $f_{\gamma-1}''(t)<0$ throughout $[0,T]$.
  Since $f_{\gamma-1}$ decreases, we have $0<p(t)\le p(0)=\mathrm{e}^{-1}<\rho$.
  Using $f_{\gamma-1}''(t)<0$ and $0<p(t)\le\rho<1$ in \eqref{eq:crossing-second-derivative}, we obtain
  \[
    D''(t)\le f_\gamma''(t)-\rho f_{\gamma-1}''(t).
  \]
  It therefore suffices to prove $f_\gamma''(t)-\rho f_{\gamma-1}''(t)<0$ for $0\le t\le T$.

  By \cref{lem:log-power-curvature}\textup{(i)}, for $g_--1\le z\le g_+-1$ we have $-f_z'''(0)>0$ and $\frac{\partial}{\partial z}(-f_z'''(0))=6z(z^2-2C)<0$.
  The denominator $(1+zT+CT^2)^2$ is positive and increasing in $z$ on $[g_--1,g_+-1]$.
  Thus, \cref{lem:log-power-curvature}\textup{(iii)} implies, for every $t\in[0,T]$,
  \[
    0<-f_{\gamma-1}'''(t)
    \le\frac{-f_{\gamma-1}'''(0)}{(1+(\gamma-1)T+CT^2)^2}
    \le\frac{-f_{g_--1}'''(0)}{(1+(g_--1)T+CT^2)^2}
    <\frac1{150}.
  \]
  Also, $\frac{\partial}{\partial z}f_z''(0)=2(z^2-C)>0$ on both $[g_-,g_+]$ and $[g_--1,g_+-1]$.
  For every $t\in[0,T]$, integrating $f_\gamma'''(u)\le3C^2$ and $-f_{\gamma-1}'''(u)<1/150$ over $0\le u\le t$, we obtain
  \begin{equation}\label{eq:crossing-relative-curvature}
    \begin{aligned}
      f_\gamma''(t)-\rho f_{\gamma-1}''(t)
       & \le f_\gamma''(0)-\rho f_{\gamma-1}''(0)
      +t\left(3C^2+\frac\rho{150}\right)          \\
       & \le f_{g_+}''(0)-\rho f_{g_--1}''(0)
      +T\left(3C^2+\frac\rho{150}\right)
      <-\frac1{12500}.
    \end{aligned}
  \end{equation}
  By \eqref{eq:crossing-relative-curvature}, we have $D''(t)<0$ for $0\le t\le T$.

  Finally, the strict monotonicity of $g_5$ implies
  \[
    g_5(\gamma+C/5)
    >g_5(g_-+C/5)
    =\left(\frac{547507}{500000}\right)^5
    -\left(\frac{447507}{500000}\right)^5>1.
  \]
  Hence, $D(T)>0$.
  For $0<t\le T$, concavity yields
  \[
    D(t)\ge(1-t/T)D(0)+(t/T)D(T)>0.
  \]
  Thus, $g_k(\gamma+C/k)>1=g_k(c_k)$ for every $k\ge5$.
  The strict monotonicity of $g_k$ and \eqref{eq:crossing-gamma-bounds} imply $c_k<\gamma+C/k<g_++C/k$, as required.
\end{proof}

\subsection{Proof of Lemma~{\ref{lem:coupled-multipliers}}}
\label{app:coupled-multipliers}

\begin{proof}
  Under the assumptions $k\ge6$ and $11/24<c<1/2$, we have $12/13<a<1$, $5/12<ac<1/2$, and $a=1-ac/k$.

  By \cref{fact:comparison-arrays} and the definition of $\alpha_j$, for $1\le l<s$ we have
  \begin{equation}\label{eq:sec3-explicit-slopes}
    \alpha_{k-l}
    =\frac{1}{s(s+1)[1-(1-s^{-2})^{s-l}]}-\frac1s.
  \end{equation}
  Set $h_l=k\Gamma_{k-l}/v_{k-l}$ for every integer $l$ with $1\le l\le s$.
  For every integer $l$ with $1\le l<s$, we have
  \[
    \frac{v_{k-l-1}}{v_{k-l}}=\frac cl\left(1-\frac lk\right)
    \text{ and }\frac{a^kb_{k-l}}{v_{k-l}}=\frac{ac}{l}.
  \]
  The initial value and the defining recurrence for $\Gamma_j$ therefore give
  \begin{equation}\label{eq:sec3-h-recurrence}
    \begin{aligned}
      h_s & =kL(w),                                                              \\
      (1+\alpha_{k-l})h_l
          & =\frac cl\left[\left(1-\frac lk\right)h_{l+1}-a\right]
      +kL(w)\alpha_{k-l}
          &                                                        & (1\le l<s).
    \end{aligned}
  \end{equation}

  We first bound $\alpha_{k-l}$ from below.
  For $0<q<1$ and an integer $m\ge1$, summing $q^j+q^{m-j}\le1+q^m$ for $1\le j<m$, we obtain $(1+q)/(1-q)\le2m/(1-q^m)-m$.
  Hence,
  \begin{equation}\label{eq:sec3-geometric-reciprocal}
    \frac1{1-q^m}
    \ge\frac1{m(1-q)}+\frac{m-1}{2m},
  \end{equation}
  with strict inequality when $m\ge2$.
  Combining \eqref{eq:sec3-explicit-slopes} with \eqref{eq:sec3-geometric-reciprocal} for $q=1-s^{-2}$ and $m=s-l$, we obtain
  \begin{equation}\label{eq:sec3-multiplier-slope-lower}
    \alpha_{k-l}\ge
    \frac{(2s+1)l-s-1}{2s(s+1)(s-l)},
  \end{equation}
  for $1\le l<s$, with strict inequality when $l\le s-2$.

  We first show that $h_l>1$ for $3\le l\le s$.
  Recall that $kL(w)=s(s-1)/2$.
  By \eqref{eq:sec3-explicit-slopes}, $\alpha_{k-l}$ increases with $l$, so, by \eqref{eq:sec3-multiplier-slope-lower}, for $3\le l<s$,
  \[
    (kL(w)-1)\alpha_{k-l}
    \ge\frac{(s-2)(5s+2)}{4s(s-3)}
    =\frac54+\frac{7s-4}{4s(s-3)}>\frac54.
  \]
  For $l=2$ and $l=1$, \eqref{eq:sec3-multiplier-slope-lower} implies
  \[
    \begin{aligned}
      \left(kL(w)-\frac23\right)\alpha_{k-2}
                        & >\frac34+\frac{3s^2+3s-4}{12s(s+1)(s-2)}>\frac34, \\
      kL(w)\alpha_{k-1} & >\frac{s}{4(s+1)}\ge\frac15.
    \end{aligned}
  \]
  For $3\le l<s$, suppose that $h_{l+1}>0$.
  Since $l<k$, the term $(c/l)(1-l/k)h_{l+1}$ is positive.
  Subtracting $1+\alpha_{k-l}$ from \eqref{eq:sec3-h-recurrence}, we therefore obtain
  \[
    (1+\alpha_{k-l})(h_l-1)
    >-\frac{ac}{l}-1+(kL(w)-1)\alpha_{k-l}
    >-\frac16-1+\frac54>0.
  \]
  Starting from $h_s=kL(w)>1$, downward induction shows that $h_l>1$ for every $3\le l\le s$.
  Taking $l=2$ in \eqref{eq:sec3-h-recurrence} and using $h_3>1$, we obtain
  \[
    \begin{aligned}
      (1+\alpha_{k-2})\left(h_2-\frac23\right)
       & >\frac c2\left(1-\frac2k-a\right)+\frac34-\frac23 \\
       & >-\frac ck+\frac1{12}>0.
    \end{aligned}
  \]
  The last line uses $a<1$ and $c/k<1/12$, so $h_2>2/3$.
  Finally, taking $l=1$ in \eqref{eq:sec3-h-recurrence} and using $h_2>2/3$ and $a=1-ac/k$, we obtain
  \[
    \begin{aligned}
      (1+\alpha_{k-1})h_1
       & >\frac15+c\left[\frac23\left(1-\frac1k\right)-a\right] \\
       & =\frac15-c\left[\frac13+\frac{2/3-ac}{k}\right]        \\
       & >\frac15-\frac12\left(\frac13+\frac1{24}\right)>0.
    \end{aligned}
  \]
  Here $(2/3-ac)/k<1/24$ because $ac>5/12$ and $k\ge6$, and we also used $c<1/2$.
  Thus, $h_l>0$ for every $1\le l\le s$.
  Since $v_{k-l}>0$, it follows that $\Gamma_j>0$ for every $k-s\le j<k$.

  We next prove $r_b>x$ and $L(w)x/(r_b-x)^2<9/10$, as asserted in \ref{item:coupled-ratio}.
  Since $c<1/2$,
  \[
    p_0=(1-x)^k
    <\left(\frac{2k-1}{2k+1}\right)^k
    <\mathrm{e}^{-1}<\frac38.
  \]
  The node values $H(W_s)=(1-1/s)^s$ increase with $s$, so $H(W_s)\ge H(W_4)=81/256$.
  Since $c>11/24>4/9$, the definition of $r_b$ implies
  \[
    \frac{r_b}{x}-1
    =\frac{(s+c)H(W_s)}{p_0}-1
    >\frac{27}{32}\left(s+\frac49\right)-1
    =\frac{27s-20}{32}>0.
  \]
  Consequently,
  \begin{equation}\label{eq:sec3-L-ratio}
    \frac{L(w)x}{(r_b-x)^2}
    =\frac{s(s-1)}{2a(r_b/x-1)^2}
    <\frac{1664s(s-1)}{3(27s-20)^2}
    \le\frac{104}{121}<\frac9{10}.
  \end{equation}
  The function $s\mapsto s(s-1)/(27s-20)^2$ is decreasing for real $s\ge4$, since
  \[
    \frac{\rd}{\rd s}
    \frac{s(s-1)}{(27s-20)^2}
    =\frac{20-13s}{(27s-20)^3}<0.
  \]
  At $s=4$, we have $1664s(s-1)/[3(27s-20)^2]=104/121$, proving \eqref{eq:sec3-L-ratio}.

  For \ref{item:coupled-comparison}, the weights $v_j$ are increasing, since $v_{j+1}/v_j=k(k-1-j)/(c(j+1))>1$ for $1\le j\le k-2$.
  Thus, {\cref{lem:coupled-energy}\textup{\ref{item:coupled-profile-general}}} gives $F_{k-1}(\mathsf P)\le V_*(w)$.
  By \ref{item:coupled-multipliers}, every $\Gamma_j$ is positive, so \cref{lem:coupled-energy}\textup{\ref{item:coupled-energy-bound}} yields
  \[
    \E(S_k-1)_+-L(w)F_{k-1}(\mathsf P)\ge E_*(w)-L(w)V_*(w).
  \]
\end{proof}

\subsection{Proof of Lemma~{\ref{lem:crossing-lower-finite}}}
\label{app:crossing-lower-finite}

\begin{proof}
  Set $c_*=11/24+2/(25k)$ and $\eta=(1/2-c_*)/k=1/(24k)-2/(25k^2)$.
  Then $0<\eta<1/(24k)$.
  Write $b=1-\eta$, so $0<b<1$.
  Since $g_k$ is strictly increasing and $g_k(c_k)=1$, it suffices to prove $g_k(c_*)<1$.
  By the definition of $g_k$, we have
  \[
    g_k(c_*)=\left(b+\frac{1}{2k}\right)^k
    -\left(b-\frac{1}{2k}\right)^k.
  \]
  Expanding the two powers and cancelling the even terms, we obtain
  \[
    g_k(c_*)
    =b^{k-1}\sum_{j=0}^{\lfloor(k-1)/2\rfloor}
    \frac{\binom{k-1}{2j}}{2j+1}(2kb)^{-2j}.
  \]
  For $1\le j\le\lfloor(k-1)/2\rfloor$, a calculation yields
  \[
    \frac{\binom{k-1}{2j}}{2j+1}
    =\frac{\prod_{r=0}^{j-1}(k-1-2r)(k-2-2r)}
    {\prod_{r=1}^{j}(2r)(2r+1)}
    \le\frac{1}{j!}\left(\frac{(k-1)(k-2)}{6}\right)^j,
  \]
  since each numerator pair is at most $(k-1)(k-2)$ and $(2r)(2r+1)\ge6r$ for $r\ge1$.
  At $j=0$, both sides equal $1$.
  Using the bound for $\binom{k-1}{2j}/(2j+1)$ term by term and adding the nonnegative terms with $j>\lfloor(k-1)/2\rfloor$, we obtain
  \begin{equation}\label{eq:crossing-lower-exponential-bound}
    \begin{aligned}
      0<g_k(c_*)
       & \le b^{k-1}\sum_{j=0}^{\lfloor(k-1)/2\rfloor}
      \frac1{j!}\left(\frac{(k-1)(k-2)}{24k^2b^2}\right)^j     \\
       & \le b^{k-1}\sum_{j=0}^{\infty}
      \frac1{j!}\left(\frac{(k-1)(k-2)}{24k^2b^2}\right)^j     \\
       & =b^{k-1}\exp\left(\frac{(k-1)(k-2)}{24k^2b^2}\right).
    \end{aligned}
  \end{equation}

  The inequality $\log(1-\eta)\le-\eta-\eta^2/2$ implies
  \[
    b^2\log b
    \le-(1-\eta)^2\left(\eta+\frac{\eta^2}{2}\right)
    =-\eta+\frac32\eta^2-\frac12\eta^4
    <-\eta+\frac32\eta^2.
  \]
  Taking logarithms in \eqref{eq:crossing-lower-exponential-bound} and using $\eta^2<1/(576k^2)$, we obtain
  \[
    \begin{aligned}
      \frac{b^2}{k-1}\log g_k(c_*)
       & \le b^2\log b+\frac{k-2}{24k^2}<-\eta+\frac32\eta^2+\frac{k-2}{24k^2} \\
       & <\frac{1}{k^2}
      \left(\frac{2}{25}-\frac{1}{12}+\frac{1}{384}\right)
      =-\frac{7}{9600k^2}<0.
    \end{aligned}
  \]
  Thus, $g_k(c_*)<1=g_k(c_k)$, and the strict monotonicity of $g_k$ implies $c_k>c_*$.
\end{proof}
\subsection{Proof of Lemma~{\ref{lem:node-shortfall-margin}}}
\label{app:node-shortfall-margin}

\begin{proof}
  We first prove \ref{item:profile-positive}.
  Recall that $E_s=E_*(W_s)$ and $V_s=V_*(W_s)$ for $4\le s\le k$.
  For $4\le s\le k$, write $q=(s-1)/s$.
  Let $N\sim Bin(k-1,\delta)$ and, independently, $I\sim Bernoulli(\delta)$.
  Substituting the comparison values from \cref{fact:comparison-arrays} into the definitions of $E_*$ and $V_*$ and then averaging over $I$, we obtain
  \begin{equation}\label{eq:node-energy-linear}
    \begin{aligned}
      1-V_s & =\E q^{(s-1-N)_+},                       \\
      k(\delta+E_s)
            & =\E\left[sq^{(s-N-I)_+}+(N+I-s)_+\right] \\
            & =(s-a)(1-V_s)+\E(N-s+1)_+.
    \end{aligned}
  \end{equation}
  By {\cref{lem:sec4-parameter-bounds}}, $p_0<\mathrm{e}^{-1}$.
  Since $(s-1)\log(1-1/s)>-1$, we have $1-V_s\ge q^{s-1}>\mathrm{e}^{-1}>p_0$, proving $V_s<a^k$.

  Since $V_*$ is affine on $[W_s,W_{s+1}]$ for every integer $s$ with $4\le s<k$, the node inequalities $V_s<a^k$ imply $V_*(t)<a^k$ throughout $[W_4,W_k]$.
  We now prove \ref{item:profile-margin} at the nodes.
  Set $z_s=(a^k-V_s)/p_0>0$.
  By \eqref{eq:node-energy-linear} and the definitions of $r_*$ and $M_*$,
  \[
    kr_*(W_s)\ge(s-a)(1+z_s)\text{ and }
    kM_*(W_s)=a\left(1+\frac1{z_s}\right).
  \]
  Thus, increasing $z_s$ improves the lower bound for $r_*(W_s)$ and decreases $M_*(W_s)$.
  For integers $n\ge3$, $0<a\le1$, and $z>0$, define
  \[
    \mathcal G_n(a,z)
    =n+1-2a+(n+1-a)z-\frac{19a}{20z}.
  \]
  By \eqref{eq:node-energy-linear} and the definitions of $r_*$ and $M_*$,
  \[
    \begin{aligned}
      k\left(r_*(W_s)-\frac{19M_*(W_s)}{20}-\frac{x}{20}\right)
       & =\mathcal G_{s-1}(a,z_s)+\frac{\E(N-s+1)_+}{p_0} \\
       & \ge\mathcal G_{s-1}(a,z_s).
    \end{aligned}
  \]
  Moreover, $\partial_z\mathcal G_n(a,z)=n+1-a+19a/(20z^2)>0$.
  Since $kx=a$, it suffices to prove $\mathcal G_{s-1}(a,z_s)>3a/20$ for every $4\le s\le k$.

  We first consider $s=4$.
  By \cref{fact:comparison-arrays} and $H(W_4)=(3/4)^4$,
  \begin{equation}\label{eq:fourth-node-remainder}
    (4-a)(1-V_4)-4p_0
    =4a^k\bigl[H(W_4)-A(W_4)\bigr]-a\Prob(N\ge4).
  \end{equation}
  Since consecutive summands have ratio $(k-1-h)c/[k(h+1)]\le c/5$ for $4\le h\le k-2$, comparison with a geometric series gives
  \[
    \begin{aligned}
      \frac{a\Prob(N\ge4)}{a^k}
       & =\sum_{h=4}^{k-1}\binom{k-1}{h}(c/k)^h \\
       & \le\frac{c^4}{24(1-c/5)}
      \prod_{j=1}^4\left(1-\frac jk\right)      \\
       & <\frac3{1250}\left(1-\frac5k\right)
      <\frac3{1250}-\frac1{100k}
      <4\{H(W_4)-A(W_4)\}.
    \end{aligned}
  \]
  For the second inequality, we use $c<19/40$ and
  \[
    \prod_{j=1}^4\left(1-\frac jk\right)
    \le\left(1-\frac3k\right)\left(1-\frac4k\right)
    =1-\frac7k+\frac{12}{k^2}
    \le1-\frac5k,
  \]
  where the last step uses $k\ge6$.
  The last inequality in the estimate for $a\Prob(N\ge4)/a^k$ follows from \cref{lem:initial-nodes}\textup{(iii)}.
  Thus, $z_4>a/(4-a)$ by \eqref{eq:fourth-node-remainder}, and hence
  \[
    \mathcal G_3(a,z_4)
    >\mathcal G_3\left(a,\frac a{4-a}\right)
    =\frac{4-a}{20}>\frac{3a}{20}.
  \]

  Now let $5\le s<k$ and set $n=s-1$, so $q=n/(n+1)$ and $k\ge n+2$.
  For every integer $s\ge5$, define $\widehat z_s$ on $[0,1]$ by
  \[
    \widehat z_s(u)=\frac{u(19s+4u)}{20s(s-u)}.
  \]
  We first check that $z_s>\widehat z_s(a)$ would imply $\mathcal G_{s-1}(a,z_s)>3a/20$.
  Substituting $\widehat z_s(a)$ into $\mathcal G_{s-1}$ gives
  \[
    \mathcal G_{s-1}(a,\widehat z_s(a))-\frac{3a}{20}
    =\frac{a(s-4a)(s-a)}{5s(19s+4a)}>0.
  \]
  Since $\mathcal G_{s-1}(a,z)$ increases with $z$, it remains to prove
  \begin{equation}\label{eq:node-unified-denominator}
    z_s>\widehat z_s(a).
  \end{equation}

  Since $q^{(n-N)_+}=q^{n-N}-(q^{n-N}-1)\1_{\{N>n\}}$, we may write $1-V_s=H_n-T_n$, where
  \[
    H_n=q^n(1+\delta/n)^{k-1}\text{ and }
    T_n=a^{k-1}\sum_{h=n+1}^{k-1}\binom{k-1}{h}(c/k)^h
    (q^{n-h}-1).
  \]
  Since $a^{k-1}/p_0=(1-p_0)/(ap_0)<3$ and $q^{-\ell}-1\le(\ell/n)q^{1-\ell}$ for $\ell\ge1$, we obtain
  \[
    \begin{aligned}
      \frac{T_n}{p_0}
       & <\frac{3c^{n+1}}{n(n+1)!}
      \sum_{\ell\ge1}\ell
      \left(\frac{c}{q(n+2)}\right)^{\ell-1} \\
       & =\frac{3c^{n+1}}{n(n+1)!}
      \left(1-\frac{c}{q(n+2)}\right)^{-2}.
    \end{aligned}
  \]
  Since $c<1/2$, it follows that
  \begin{equation}\label{eq:node-short-tail}
    \frac{T_n}{p_0}
    <\frac{3}{n2^{n+1}(n+1)!}
    \left(1-\frac{n+1}{2n(n+2)}\right)^{-2}
    <\begin{cases}
      1/4000,     & n=4,  \\
      1/(500n^3), & n\ge5
    \end{cases}.
  \end{equation}
  The case $n=4$ follows by substitution.
  For $n\ge5$, the quantity
  \[
    \frac{3n^2}{2^{n+1}(n+1)!}
    \left(1-\frac{n+1}{2n(n+2)}\right)^{-2}
  \]
  decreases with $n$ and is less than $1/500$ at $n=5$.

  To obtain a lower bound for $H_n/p_0$, we next prove
  \begin{equation}\label{eq:node-short-log}
    \log(H_n/p_0)>n\log q+1-\frac1{2n}+\frac{23a}{24n}
    -\begin{cases}
      (1-a)/70, & n=4,   \\
      0,        & n\ge5.
    \end{cases}
  \end{equation}
  For $0<v<1$, we have
  \[
    \log\frac{1+v}{1-v}>2v+\frac{2v^3}{3}>2v.
  \]
  Hence,
  \[
    -\log p_0>\frac{k}{k+c-1/2}+\frac1{12k^2}
    \text{ and }
    \log(1+\delta/n)>\frac{\delta}{n+\delta/2},
  \]
  where we use $c<1/2$.
  For $0<u<1/2$, define
  \[
    P_n(u)=n(1/2-u)+k(u-11/24)
    -u\left(u+\frac1{24}+\frac{u}{2n}\right)+\frac{n}{12k}.
  \]
  Then
  \[
    \begin{aligned}
       & nk\left[\log\frac{H_n}{p_0}-n\log q-1
      +\frac1{2n}-\frac{23a}{24n}\right]       \\
       & \qquad>n(1/2-c)-\frac{11k}{24}
      +\frac{kc(k-1/24)}{k+c+c/(2n)}+\frac{n}{12k}
      \ge P_n(c).
    \end{aligned}
  \]
  Since $k\ge n+2$, we have $P_n'(u)=k-n-1/24-2u(1+1/(2n))>0$ for every $u\in(0,1/2)$.
  For $n\ge5$, using $c>\bar c=11/24+2/(25k)$ and $\bar c<17/36$, we obtain
  \[
    P_n(c)>P_n(\bar c)>
    \frac2{25}+\frac5{24}
    -\frac{17}{36}\left(\frac{17}{36}+\frac1{24}
    +\frac{17}{360}\right)>0.
  \]
  For $n=4$, using $k(1-a)=ac>11/26$, we obtain
  \[
    P_4(c)+\frac{4k(1-a)}{70}
    >\frac2{25}+\frac16+\frac{11}{455}
    -\frac{17}{36}\left(\frac{17}{36}+\frac1{24}
    +\frac{17}{288}\right)>0.
  \]
  The bounds for $P_n(c)$ prove \eqref{eq:node-short-log}.

  Since $z_s=H_n/p_0-T_n/p_0-1$, it suffices to prove
  \[
    \log\frac{H_n}{p_0(1+\widehat z_s(a))}>\frac{T_n}{p_0}.
  \]
  For every integer $n\ge4$, define $\Lambda_n$ on $[0,1]$ by
  \[
    \Lambda_n(u)
    =-n\log(1+1/n)+1+\frac{23u-12}{24n}
    -\log\bigl(1+\widehat z_{n+1}(u)\bigr).
  \]
  For the fixed $n=s-1\in\{4,\ldots,k-2\}$, \eqref{eq:node-short-log} gives
  \[
    \log\frac{H_n}{p_0(1+\widehat z_{n+1}(a))}
    >\begin{cases}
      \Lambda_4(a)-(1-a)/70, & n=4,   \\
      \Lambda_n(a),          & n\ge5.
    \end{cases}
  \]
  For $u\ge0$, we have
  \[
    \log(1+u)\le Q(u):=\frac{u(6+u)}{6+4u}.
  \]
  Indeed, $Q(0)=0$ and
  \[
    Q'(u)-\frac1{1+u}
    =\frac{u^3}{(1+u)(2u+3)^2}\ge0.
  \]

  To apply endpoint estimates at the fixed parameter $a$, we will prove $\Lambda_n''(u)\le-1/(n+1)^2$ for every integer $n\ge4$ and every $u\in[0,1]$.
  For every integer $s\ge5$ and every $u\in[0,1]$, write
  \[
    1+\widehat z_s(u)=\frac{U_s(u)}{s-u}\text{ and }
    U_s(u)=s-\frac u{20}+\frac{u^2}{5s}.
  \]
  Then
  \[
    U_s''(u)U_s(u)-U_s'(u)^2
    =\frac{159}{400}+\frac{u}{50s}\left(1-\frac{4u}{s}\right)>0.
  \]
  Consequently,
  \[
    \frac{\rd^2}{\rd u^2}
    \log(1+\widehat z_s(u))
    \ge\frac1{(s-u)^2}\ge\frac1{s^2},
  \]
  and hence $\Lambda_n''(u)\le-1/(n+1)^2$ for every integer $n\ge4$ and every $u\in[0,1]$.

  First suppose $n=4$.
  The function $u\mapsto\Lambda_4(u)-(1-u)/70$ is concave on $[12/13,1]$, so it suffices to prove that both endpoint values exceed $1/4000$.
  Using $\log(1+u)\le Q(u)$, $Q(1/4)=25/112$, $\widehat z_5(12/13)=3849/17225$, and $\widehat z_5(1)=99/400$, we obtain
  \[
    \begin{aligned}
      \Lambda_4(12/13)-\frac{1-12/13}{70}
       & >\frac{92}{455}-Q\left(\frac{3849}{17225}\right) \\
       & >\frac{92}{455}-\frac{21}{104}
      =\frac1{3640}>\frac1{4000},                         \\
      \Lambda_4(1)
       & >\frac{149}{672}-Q\left(\frac{99}{400}\right)    \\
       & >\frac{149}{672}-\frac{31}{140}
      =\frac1{3360}>\frac1{4000}.
    \end{aligned}
  \]
  The bounds $Q(3849/17225)<21/104$ and $Q(99/400)<31/140$ follow by substituting into $Q(u)=u(6+u)/(6+4u)$.
  Concavity therefore implies $\Lambda_4(u)-(1-u)/70>1/4000$ for every $u\in[12/13,1]$.
  In particular, $\Lambda_4(a)-(1-a)/70>1/4000$ because $a\in(12/13,1)$.

  Now suppose $n\ge5$.
  Since $\Lambda_n''(u)\le-1/(n+1)^2$ for every $u\in[0,1]$, the function $u\mapsto\Lambda_n(u)-u(1-u)/(2(n+1)^2)$ is concave on $[0,1]$.
  Applying concavity at $u=a$ between the endpoints $0$ and $1$, we obtain
  \[
    \Lambda_n(a)
    \ge a\Lambda_n(1)
    +(1-a)\left[\Lambda_n(0)+\frac{a}{2(n+1)^2}\right].
  \]
  Applying $\log(1+u)\le Q(u)$ with $u=1/n$, we obtain
  \[
    \Lambda_n(0)
    >1-nQ(1/n)-\frac1{2n}
    =-\frac1{n(3n+2)}.
  \]
  Since $a\ge12/13$, we have
  \[
    \Lambda_n(0)+\frac{a}{2(n+1)^2}
    >\frac{5n^2-14n-13}{13n(3n+2)(n+1)^2}>0,
  \]
  where $5n^2-14n-13=5(n-4)^2+26(n-4)+11>0$ for $n\ge4$.
  Hence, $\Lambda_n(a)\ge a\Lambda_n(1)$.
  At $u=1$, we have $\widehat z_{n+1}(1)=(19n+23)/(20n(n+1))$.
  Applying $\log(1+u)\le Q(u)$ to both logarithms in the definition of $\Lambda_n$ gives
  \[
    \begin{aligned}
      \Lambda_n(1)
       & >1+\frac{11}{24n}-nQ(1/n)
      -Q\left(\frac{19n+23}{20n(n+1)}\right)               \\
       & =\frac3{6n+4}+\frac{11}{24n}
      -Q\left(\frac{19n+23}{20n(n+1)}\right)               \\
       & =\frac{n^2\{180(n-5)^2+621(n-5)+553\}+5145n+1886}
      {240n(n+1)(3n+2)(30n^2+49n+23)}
      >\frac1{100n^3}.
    \end{aligned}
  \]
  For the last inequality, the numerator exceeds $500n^2$ for $n\ge5$, while the denominator is less than $50000n^5$.
  Here we use $n+1<5n/4$, $3n+2<7n/2$, and $30n^2+49n+23<42n^2$.
  It follows that
  \[
    \Lambda_n(a)>\frac{a}{100n^3}>\frac1{500n^3}.
  \]

  Comparing $\Lambda_4(a)-(1-a)/70>1/4000$ for $n=4$ and $\Lambda_n(a)>1/(500n^3)$ for $n\ge5$ with \eqref{eq:node-short-tail}, we obtain, for every integer $n$ with $4\le n\le k-2$,
  \[
    \log\frac{H_n}{p_0(1+\widehat z_{n+1}(a))}>\frac{T_n}{p_0}\ge0.
  \]
  Exponentiating and using $\exp(u)\ge1+u$ for $u\ge0$ and $\widehat z_s(a)>0$, we obtain
  \[
    z_s=\frac{H_n}{p_0}-\frac{T_n}{p_0}-1
    >\widehat z_s(a)\left(1+\frac{T_n}{p_0}\right)
    \ge\widehat z_s(a).
  \]
  This proves \eqref{eq:node-unified-denominator}.

  It remains to consider $s=k$.
  Here $q=(k-1)/k$, and $\E(N-k+1)_+=0$ because $N\le k-1$.
  Using the binomial distribution of $N$, we obtain
  \[
    1-V_k=\E q^{k-1-N}=(\delta+aq)^{k-1}=(1-x)^{k-1}.
  \]
  Equation \eqref{eq:node-energy-linear} implies $r_*(W_k)=M_*(W_k)=1$.
  Thus,
  \[
    r_*(W_k)-\frac{19M_*(W_k)}{20}
    =\frac1{20}>\frac{x}{5},
  \]
  since $x<1/k\le1/6<1/4$.
  Thus, $r_*(W_s)-19M_*(W_s)/20>x/5$ for every integer $s$ with $4\le s\le k$.

  Fix an arbitrary $t\in[W_4,W_k]$.
  Choose an integer $s\in\{4,\ldots,k-1\}$ and $\theta\in[0,1]$ such that $t=(1-\theta)W_s+\theta W_{s+1}$.
  On $[W_s,W_{s+1}]$, the functions $r_*$ and $a^k-V_*$ are affine, and $a^k-V_*>0$ by \ref{item:profile-positive}.
  The reciprocal of a positive affine function is convex, so $M_*$ is convex and $r_*-19M_*/20$ is concave on the same interval.
  Applying concavity and then the margin at $W_s$ and $W_{s+1}$, we obtain
  \[
    \begin{aligned}
      r_*(t)-\frac{19M_*(t)}{20}
       & \ge(1-\theta)\left(r_*(W_s)-\frac{19M_*(W_s)}{20}\right)        \\
       & \quad+\theta\left(r_*(W_{s+1})-\frac{19M_*(W_{s+1})}{20}\right)
      >\frac{x}{5}.
    \end{aligned}
  \]
  This proves \ref{item:profile-margin}.
\end{proof}

\subsection{Proof of Claim~\ref{clm:large-defect-endpoint-positive}}
\label{app:large-defect-endpoint-positive}

\begin{proof}
  In \eqref{eq:large-defect-endpoint-functions}, write $s=\eta/a$ and $u=x_0/A$.
  Throughout, $(t,s,u)\in[0,1/6]\times[1/10,1]\times[0,1/3]$ and $\beta(s)=\sqrt{(1+s)/11}$.
  We have $\mathcal V_0(t,s,u)=s-1/20+13u/20+t(1-s)/6$ and $\mathcal V_1(t,s,u)=\mathcal V_0(t,s,u)+(1-s)/9$.

  We first prove $P_1>0$.
  For fixed $s,u$, the functions $Q_1$ and $R$ are nondecreasing in $t$, whereas
  \[
    \partial_t[(1-t)\mathcal V_1]
    =-\frac{19}{18}\left(s-\frac1{10}\right)
    -\frac{13u}{20}-\frac{t(1-s)}3\le0.
  \]
  Since $\beta(s)>0$, we have $P_1(t,s,u)\ge P_1(1/6,s,u)$.
  At $t=1/6$, we have
  \[
    \beta(s)\sqrt{\frac56\mathcal V_1(1/6,s,u)}
    =\sqrt{\frac{(1+s)(155s+16+117u)}{2376}}.
  \]
  Since $1+s$ and $155s+16+117u$ are positive and affine in $(s,u)$, concavity of $(r_1,r_2)\mapsto\sqrt{r_1r_2}$ on $(0,\infty)^2$ implies that $\beta(s)\sqrt{(5/6)\mathcal V_1(1/6,s,u)}$ is jointly concave in $(s,u)$.
  Moreover,
  \[
    \begin{aligned}
      Q_1(1/6,s,u) & =\left(\frac{39s}{40}-\frac{11u}{24}-\frac18\right)_+, \\
      R(1/6,s,u)   & =\frac{\max\{u,(7-5s)/32\}}3
      +\frac{u+(7-5s)/32}{14}.
    \end{aligned}
  \]
  Thus, $Q_1(1/6,s,u)$ and $R(1/6,s,u)$ are jointly convex, so $P_1(1/6,s,u)$ is jointly concave and
  \[
    P_1(t,s,u)\ge
    \min_{(s_0,u_0)\in\{1/10,1\}\times\{0,1/3\}}
    P_1(1/6,s_0,u_0).
  \]
  Substituting $(s,u)\in\{1/10,1\}\times\{0,1/3\}$ gives
  \[
    \begin{aligned}
      P_1(1/6,1/10,0)
       & =\sqrt{\frac7{480}}-\frac{221}{2688}
      >\frac1{10}-\frac{221}{2688}
      =\frac{239}{13440}>0,                                      \\
      P_1(1/6,1/10,1/3)
       & =\sqrt{\frac{47}{1440}}-\frac{1205}{8064}
      >\frac16-\frac{1205}{8064}
      =\frac{139}{8064}>0,                                       \\
      P_1(1/6,1,0)
       & =\sqrt{\frac{19}{132}}-\frac{17}{50}-\frac{17}{672}
      >\frac38-\frac{17}{50}-\frac{17}{672}
      =\frac{163}{16800}>0,                                      \\
      P_1(1/6,1,1/3)
       & =\sqrt{\frac{35}{198}}-\frac{251}{900}-\frac{281}{2016}
      >\frac{21}{50}-\frac{251}{900}-\frac{281}{2016}
      =\frac{29}{16800}>0.
    \end{aligned}
  \]
  Hence, $P_1>0$ on $[0,1/6]\times[1/10,1]\times[0,1/3]$.

  We now prove $P_0>0$.
  For fixed $t$, the function $\frac12\sqrt{(1-t)\mathcal V_0(t,s,u)}$ is jointly concave in $(s,u)$, whereas $Q_0$ and $R$ are jointly convex.
  Hence, $P_0$ is jointly concave in $(s,u)$, so
  \[
    P_0(t,s,u)\ge
    \min_{(s_0,u_0)\in\{1/10,1\}\times\{0,1/3\}}
    P_0(t,s_0,u_0).
  \]
  At $s=1$, we have $\mathcal V_0=\mathcal V_1$, $Q_0-Q_1=3/20$, $(1-t)\mathcal V_1\ge19/24$, and $\beta(1)=\sqrt{2/11}<3/7$.
  Therefore,
  \[
    \begin{aligned}
      P_0(t,1,u)-P_1(t,1,u)
       & >\frac1{14}\sqrt{\frac{19}{24}}-\frac3{50} \\
       & >\frac1{14}\cdot\frac89-\frac3{50}
      =\frac{11}{3150}>0.
    \end{aligned}
  \]
  Since $P_1>0$, we obtain $P_0(t,1,u)>0$.
  At $(s,u)=(1/10,1/3)$, we have $Q_0=0$, $E\le13/64$, and $(1-t)\mathcal V_0\ge(5/6)(4/15)=2/9$, so
  \[
    P_0(t,1/10,1/3)>
    \frac16-\frac19-\frac{1/3+13/64}{14}
    =\frac{139}{8064}>0.
  \]
  Finally, at $(s,u)=(1/10,0)$, we have $Q_0=1/10$, $R=17(2+27t)/1344$, and $(1-t)\mathcal V_0=(1-t)(1+3t)/20$.
  The function $t\mapsto(1-t)(1+3t)/20$ is positive and concave on $[0,1/6]$, so $t\mapsto P_0(t,1/10,0)$ is concave.
  At $t=0$ and $t=1/6$, we have
  \[
    \begin{aligned}
      P_0(0,1/10,0)   & >\frac1{10}-\frac1{25}-\frac{17}{672}>0, \\
      P_0(1/6,1/10,0) & =\frac{187}{67200}>0.
    \end{aligned}
  \]
  Thus, $P_0>0$ on $[0,1/6]\times[1/10,1]\times[0,1/3]$.
\end{proof}

\end{document}